\documentclass[12pt,a4paper]{report}
\usepackage[english]{babel}

\usepackage{amsfonts}
\usepackage{amsmath}
\usepackage{tikz}
\usetikzlibrary{calc}
\usepackage{amssymb}
\usepackage{amsthm}
\usepackage{indentfirst}
\usepackage[utf8]{inputenc}
\usepackage{varioref}
\usepackage{empheq,xcolor}
\usepackage{graphicx}
\usepackage{tcolorbox}

\usepackage{multirow}
\usepackage{booktabs}

\usepackage{subcaption}
\usepackage{tikz,pgfplots}

\graphicspath{ {./BBFigures/} }
\labelformat{equation}{\textrm{(#1)}}
\labelformat{thm}{Teorema {#1}}

\newcommand*{\boxcolor}{black}
\makeatletter
\renewcommand{\boxed}[1]{\textcolor{\boxcolor}{%
\tikz[baseline={([yshift=-1ex]current bounding box.center)}] \node [rectangle, minimum width=1ex,rounded corners,draw] {\normalcolor\m@th$\displaystyle#1$};}}
 \makeatother
 
 \makeindex

\usepackage[toc,page]{appendix}

\theoremstyle{plain}
\newtheorem{thm}{Theorem}
 
\newtheorem{cor}{Corollary}
\newtheorem{lem}{Lemma}
\newtheorem{prop}{Proposition}
\newtheorem{propert}{Properties}
\newtheorem{defn}{Definition}
\newtheorem{Oss}{Ossservation}
\newtheorem{defn*}{Definition}
\newtheorem{exmp}{Example}

\theoremstyle{remark}
\newtheorem{rem}{\textbf{Remark}}

\allowdisplaybreaks
\newtheorem{proo}{\textbf{Proof.}}{\nonumber}

\begin{document}
\begin{titlepage} 
	\newcommand{\HRule}{\rule{\linewidth}{0.5mm}} 
	
	\center 
	
	
\begin{figure*}
	\centering
		\includegraphics[width=0.65\linewidth]{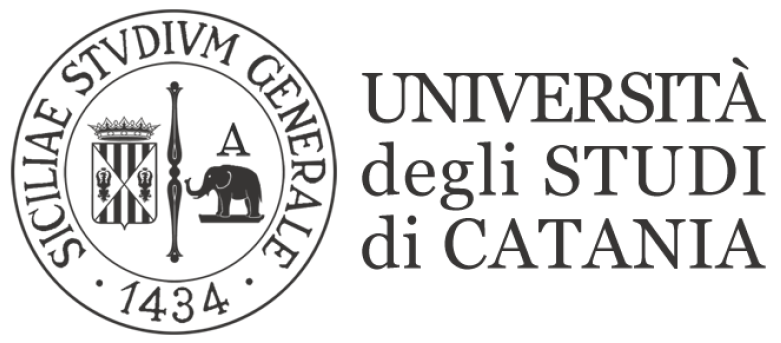}
\end{figure*}


	\vspace*{0.3cm}

		\textsc{\large Department of Mathematics and Computer Sciences}\\[0.6cm]
\textsc{\large PhD in Mathematics and Computer Sciences}\\[0.6cm]
\textsc{\large XXIX Cycle}\\[1.cm]	

	\HRule\\[0.4cm]
	
	{\huge\bfseries  Umbral Calculus}\\[0.4cm] 
		{\Large\bfseries \textit{A Different Mathematical Language}}\\[0.4cm]
		
	\HRule\\[1.5cm]
	
	
	\begin{minipage}{0.4\textwidth}
		\begin{flushleft}
			\large
			\textsc{Author}\\
			 \textit{Silvia Licciardi} 
		\end{flushleft}
	\end{minipage}
	~
	\begin{minipage}{0.4\textwidth}
		\begin{flushright}
			\large
			\textsc{Supervisors}\\
			Prof.  \textit{Vittorio Romano}\\
			Prof.  \textit{Giuseppe Dattoli}
		\end{flushright}
	\end{minipage}

	
	\vfill 
	\textsc{\large \textit{Session} 2017/2018}
	
	
	
	
\end{titlepage}

\begin{flushright}
	\textit{"La matematica è il linguaggio con cui riusciamo a interpretare e capire, in parte, il mondo fisico. Che ci permette di descrivere e di prevedere, in parte,  gli eventi del mondo. E’ il linguaggio che ci ricorda che per capire il visibile dobbiamo supporre l’invisibile. 
	Perché ci sono numeri che si possono scrivere e altri che nessuno, nemmeno il calcolatore più grande dell’universo, potrebbe scrivere. \\
	La natura del numero è misteriosa, cosa sono i numeri? "}\\[3mm]

Professor Ennio De Giorgi \\
(Tratto dallo spettacolo in memoria del Professor Marcello Anile \\
Testo originale Pamela Toscano \\
Conference Advances in Mathematics for Technology 2017.)
\end{flushright}
\newpage
\begin{flushright}
\textit{A mio Padre, che sempre crede in me...}
\end{flushright}
\newpage
\section*{\textit{Acknowledgements}}
\addcontentsline{toc}{chapter}{Acknowledgements}{}
\markboth{\textsc{Acknowledgements}}{}

\textit{During these important years of my PhD programme, I have had two important advisors, the illustrious Professor Vittorio Romano and the illustrious Professor Giuseppe Dattoli.}\\

\textit{I would first like to thank Professor Giuseppe Dattoli  of the Enea Frascati research center, where I spent the greatest part of my PhD studies. The door to Prof. Dattoli's office has always been open whenever I have run into trouble or had a question about my research or writing. Under his professional, expert and scrupulous didactic guidance and through the no less important friendship he has shown to me, I have been able to conduct and carry forward my research and grow professionally. Thanks to his constant encouragement and his vast knowledge of Mathematics and Physics, I have been able not only to learn scientific notions but also to extend the knowledge I have acquired over a wide range of research fields. The work at Enea over these years has required considerable flexibility in dealing with different research topics according to the specific demands of the moment and this would not have been possible if I had not been working with him and with a team as highly professional and compact as the one created by him, especially in the persons of Dr. Emanuele Di Palma, great expert mathematical researcher who always, always helped me, and of Dr. Ivan Spassovsky, experimental physicist and good collegue. I feel great esteem and friendship for them.} \\

\textit{I'm also grateful for the far-sighted synergy that my advisors have managed to create between them. I would like to thank Professor Vittorio Romano, my tutor at the University of Catania, where I performed most parts of my PhD coursework. I am gratefully indebted to him for his great experience as teacher, mathematician and especially as tutor. His indications on my papers and thesis have always been precise and fruitful and his availability and esteem have never lacked, even times of intense personal tribulation for me. I therefore address to him a truly special thanks.} \\

\textit{I would also like to thank all of the teachers at the Universities of Catania, Palermo and Messina whom I have met over these years of my PhD as lecturers or simply in dialogue, for the precious exchange of viewpoints. In particular the PhD coordinator, the illustrious Professor Giovanni Russo, who introduced me into the world of numerical analysis.} 
\newpage 

\textit{Last but not least, a particular thanks to my friend, collegue and confidant Dr. Rosa Maria Pidatella at the University of Catania. She has advised me, supported and helped me. We have published papers and enjoyed working together...  without her I could not have done this.}\\

\textit{Finally, I must express my very deep gratitude to my family for providing me with unfailing support and continuous encouragement throughout my years of study, through the process of researching and writing this thesis and over the course of my life. This accomplishment would not have been possible without them. Thanks to you all.}\\

\textit{Thank you very much to everyone, to my dear friends and especially to God.}

\tableofcontents

\chapter*{Preface}
\addcontentsline{toc}{chapter}{Preface}{}
\markboth{\textsc{Preface}}{}

The term \textbf{\textsl{"Umbra"}} has been initially proposed by \textit{S. Roman and G.C. Rota} \cite{S.Roman,SMRoman} to stress, in an (at the time) emerging field of operational calculus, the common practice  of replacing a series of the type
	
\begin{equation}
	 \sum_{n=0}^{\infty}c_n \frac{x^n}{n!}, 
\end{equation}
representing a certain function $f(x)$ (with its $x$ domain), with the formal exponential series

\begin{equation}
 \sum_{n=0}^{\infty}\hat{c}^n  \frac{x^n}{n!}=e^{\hat{c}x}.
\end{equation}
 
 The "promotion" of the index $n$ in $c_n$ to the status of a power exponent of the operator $\hat{c}$, namely the umbral operator, is the essence of "umbra", since it is a kind of projection of one into the other.
  Even though we adopt the same starting point, the conception of umbra and of the associated technicalities  developed in this thesis are different. 
  We will see that the possibility of replacing a function by a conveniently chosen formal series expansion provides significant advantages, among which that of treating special functions as the \textit{"umbral image"} of elementary functions. We will prove, for example, that \textit{Bessel Functions} are the umbral images of the \textit{Gaussian}.
   Albeit an apparently sterile exercise, such a point of view offers a wealth of new perspectives, based on a wise combination of "umbra", either for the study of the properties of old and new special functions and for the introduction of novel computational methods, differential operational calculus and algebraic manipulations.\\

In Fig. \ref{figUmbra} we have provided an iconographic idea of the concept of umbra according to Rota and coworkers.\\

\begin{figure}
	\centering
	\includegraphics[width=0.5\linewidth]{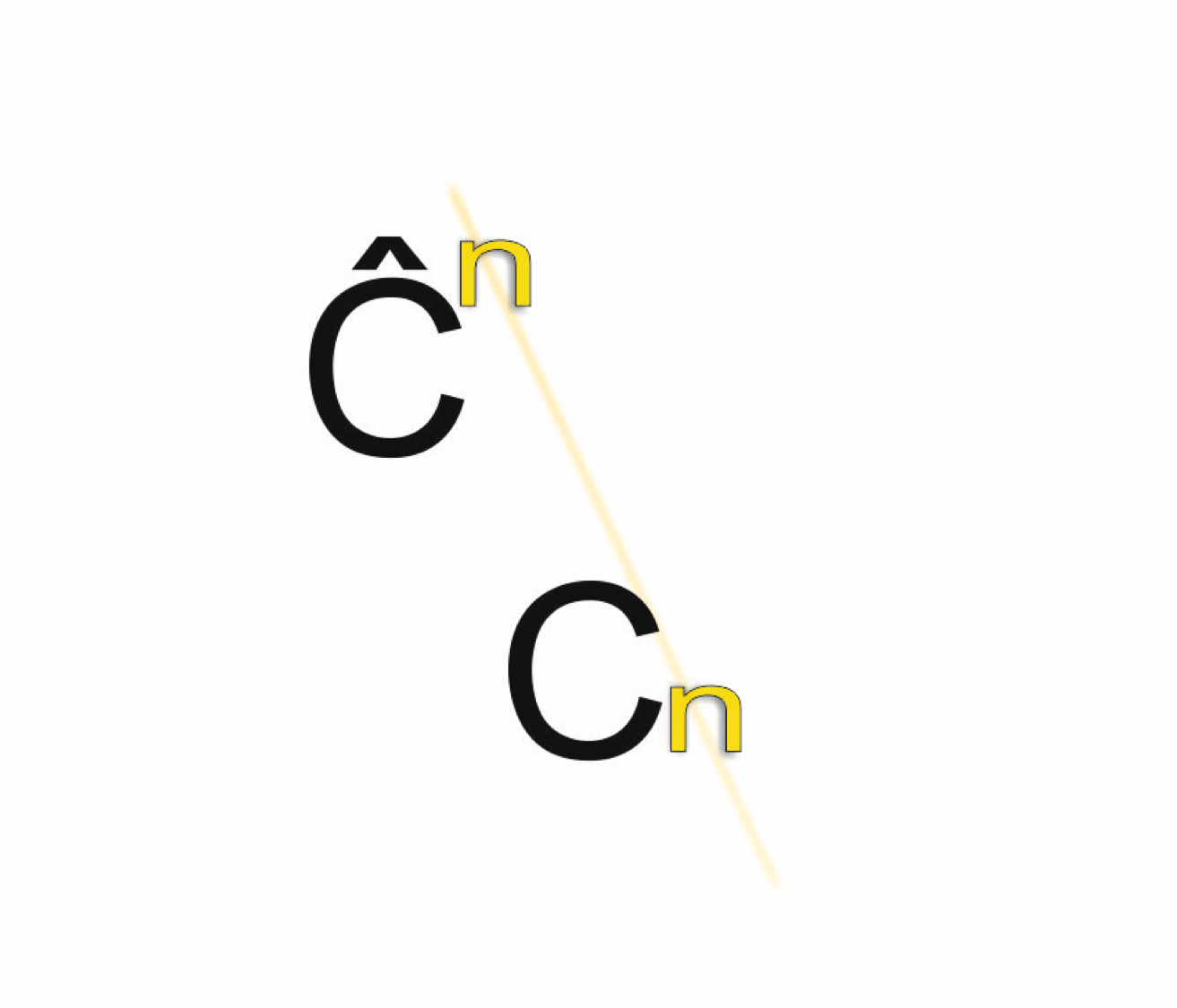}
	\caption{Origins of the term UMBRA according to Rota \& Roman.}
	\label{figUmbra}
\end{figure}
   
   The thesis is aimed at a thorough exposition of the method, relevant in the theory of special functions, for the solution of ordinary and partial differential equations, including those of fractional nature. It will provide an account of the theory and applications of  Operational Methods allowing the “translation” of the theory of special functions and polynomials into a “different” mathematical language. The language we are referring to is that of symbolic methods, largely based on a formalism of umbral type which provides a tremendous simplification of the derivation of the associated properties, with significant advantages from the computational point of view, either analytical or to derive efficient numerical methods to handle integrals, ordinary and partial differential equations, special functions and physical problems solutions. The strategy we will follow is that of establishing the rules to \textit{replace higher trascendental functions in terms of elementary functions}, taking advantage from  such a recasting. \\
   
Albeit the point of view discussed here is not equivalent to that developed by Rota and coworkers, we emphasize that it deepens its root into  the \textit{Heaviside} operational calculus \cite{Nahin} and into the methods introduced by the operationalists (\textit{Sylvester, Boole, Glaisher, Crofton and Blizard} \cite{Cartier,DelFranco,DattArt,Bell}) of the second half of the XIX century.\\

Going back to the seminal paper by \textit{Heaviside} in $1887$, we quote the statement \cite{Heaviside}, \cite{Nahin}\\
\textit{"There is a universe of mathematics lying in between the complex differentiation and integration"} . \\
The Heaviside breackthrough in the theory of electric circuits was that of finding a way to treat resistance, capacitor and inductor as the same mean, by the introduction of a specific operator, treated as an ordinary algebraic quantity. Within the framework of the Heaviside operational calculus, the analisys of whatever complex electric network can be reduced to straightforward algebraic manipulations.\\

The method has opened new avenues to deal with rational, trascendental and higher order trascendental functions, by the use of the same operational forms. The technique had been formulated in general enough terms to be readily extended to the fractional calculus.
 The starting point of our theory is the use of the \textit{Borel} transform methods to put the relevant mathematical foundation on rigorous grounds.\\

Our target is the search for a common thread between special functions, the relevant integral representation, the differential equations they satisfy and their group theoretical interpretation, by embedding all the previously quoted features within the same umbral formalism.\\

The procedure we envisage allows the straightforward derivation of (not previously known) integrals involving e.g. the combination of special functions or the Cauchy type partial differential equations (PDE) by means of new forms of solution of evolution operator, which are extended to fractional PDE. It is worth noting that our methods allow a new definition of fractional forms of Poisson distributions different from those given in processes involving fractional kinetics.\\

A noticeable amount of work has been devoted to the rigorous definition of the evolution operator and in particular the problem of its hermiticity properties and more in general of its invertibility. Much effort is devoted to the fractional ordering problem, namely the use of non-commuting operators in fractional evolution equations and to time ordering.\\

We underscore the versatility and the usefulness of the proposed procedure by presenting a large number of applications of the method in different fields of Mathematics and Physics. \\
 
In the following we provide the detailed layout of the thesis, along with an account of the topics which have been treated and of the new obtained results.\\
  
The thesis consists of six chapters. Each chapter contains a general overview of the topic we introduce and new findings based on articles published or submitted to peer review journals.\\

In Chapter \ref{Chapter1} we fix the rules underlying our point of view to umbral methods. We define the concept of umbral image function and develop a case study regarding the \textit{Bessel} functions which, within the present context, are \textit{Gaussian} functions. For this purpose, we define an umbral version of Gauss-Weierstrass integral and of Laplace transform, based on rigorous mathematical methods such as the Ramanujan Master Theorem and the Principle of Permanence of Formal Properties. We provide the first examples of how such an identification is helpful to derive all the relevant properties including the computation of infinite integrals. A great deal of effort is devoted to the theory of \textit{Mittag-Leffler} functions, with different umbral images provided by an exponential function, a rational function or an integral form which allow us to treat fractional evolution problems including  time-fractional diffusive equation. The method we propose is shown to be of noticeable importance to obtain the solution of fractional evolution equations of \textit{Schr\"{o}dinger} (\textit{FSE}) type and are naturally suited to develop methods allowing the extension to the fractional calculus of disentanglement theorems.  In particular we study the \textit{FSE} ruling the process of photon absorption/emission and introduce the associated probability distribution, which results to be a fractional \textit{Poisson} type distribution. \\

The original parts of the Chapter, containing an adequate bibliography to the relevant scientific literature, are included in the papers specified below.\\

$\star$ \textbf{G. Dattoli, E. Di Palma, E. Sabia, S. Licciardi; "Products of Bessel Functions and Associated Polynomials"; Applied Mathematics and Computation, vol.  266,  Issue C, September 2015, pages 507-514, Elsevier Science Inc. New York, NY, USA}.\\

$\star$ \textbf{D. Babusci, G. Dattoli, M. Del Franco, S. Licciardi; “Mathematical Methods for Physics”, invited Monograph by World Scientific, Singapore, 2017, in press}.\\

$\star$ \textbf{G. Dattoli, K. Gorska, A. Horzela, S. Licciardi, R.M. Pidatella; “Comments on the Properties of Mittag-Leffler Function”,\\ arxiv.org/abs/1707.01135 [math-ph], submitted for publication to European Physical Journal, 2017}.\\

$\star$ \textbf{G. Dattoli, S. Licciardi; “Book on Bessel Functions and Umbral Calculus”, work in progress}.\\

Chapter \ref{Chapter2} consists of two parts. In the first one we discuss new aspects of operational calculus theory and outline its evolution into differintegral and umbral calculi and introduce further umbral rules associated with the properties of negative and fractional derivatives. The reliability and the usefulness of the formalism we propose is benchmarked by a "revisitation" of the theory of \textit{Hermite polynomials}, which are considered in view of their twofold role of orthogonal polynomials and of solution of heat type equations. Due to the remarkable versatility of these polynomials and to the particular exponential expression of their generating function, they are a powerful tool for numerous applications, calculus of not known integrals, PDE and ODE, fractional too. We show how the method we develop are naturally suited to study physical problems associated with the definition of the Galilei group and to study the solutions (analytical and numerical)
 of problems occurring in applications. We study e.g. the evaluation of the so called Pearcey integral, often occurring in problems of wave propagation and diffraction, the solution of equations of pivotal importance in the theory of Free Electron Laser and the computing of the propagation of Super Gaussian beams. The complex of rules emerging from this new handling of Hermite polynomials is shown to evolve in an authonomous system of calculus (which we call \textit{Hermite Calculus}), whose rules are carefully described along with the introduction of the associated integral transforms. In the second part of the chapter, we show how the conceptions underlying the definition of Hermite calculus can be extended to other families of orthogonal polynomials, like those belonging of the \textit{Laguerre} type.\\

 The Chapter is based on the following original papers.\\

$\star$ \textbf{G. Dattoli, B. Germano, S. Licciardi, M.R. Martinelli; “Hermite Calculus”;  Modeling in Mathematics, Atlantis Transactions in Geometry, vol 2. pp. 43-52, J. Gielis, P. Ricci, I. Tavkhelidze (eds), Atlantis Press, Paris, Springer 2017}.\\

$\star$ \textbf{M. Artioli et al; “A 250 GHz Radio Frequency CARM Source for Plasma Fusion”, Conceptual Design Report, ENEA, pp. 154, 2016, ISBN: 978-88-8286-339-5}.\\

$\star$ \textbf{E. Di Palma, E. Sabia, G. Dattoli, S. Licciardi and I. Spassovsky; “Cyclotron auto resonance maser and free electron laser devices: a unified point of view”, Journal of Plasma Physics, Volume 83, Issue 1 , February 2017}. \\

$\star$ \textbf{D. Babusci, G. Dattoli, M. Del Franco, S. Licciardi; “Mathematical Methods for Physics”, invited Monograph by World Scientific, Singapore, 2017, in press}.\\

$\star$ \textbf{M. Artioli, G. Dattoli, S. Licciardi, S. Pagnutti; “Fractional Derivatives, Memory kernels and solution of Free Electron Laser Volterra type equation”,  Mathematics 2017, 5(4), 73; doi: 10.3390/\\math5040073}.\\


$\star$ \textbf{G. Dattoli, S. Licciardi, E. Sabia; “Operator Ordering and Solution of Umbral and Fractional Dfferential Equations”, work in progress}.\\

Chapter \ref{ChapterOP} contains an application of the methods described in the previous two Chapters to the theory of special polynomials. The content of the Chapter is essentially a recasting in umbral terms of the relevant properties and it is shown that this point of view allows the derivation of the relevant properties in a straightforaward and unified way. 
 We show that the properties of \textit{Gegenbauer, Laguerre, Legendre, Jacobi and Chebyshev polynomials} are derived from those of Hermite type. We use the proposed formalism to finally derive the properties of \textit{Voigt} functions, used in spectroscopy. \\ 
 
The original reference papers of this Chapter are:\\

$\star$ \textbf{G. Dattoli, B. Germano, S. Licciardi, M.R. Martinelli; “On an umbral treatment of Gegenbauer, Legendre and Jacobi polynomials”, International Mathematical Forum, vol. 12, 2017, no. 11, pp. 531-551}. \\

$\star$ \textbf{M. Artioli, G. Dattoli, S. Licciardi, R.M. Pidatella; “Hermite and Laguerre Functions: a unifying point of view”, work in progress}.\\

$\star$ \textbf{D. Babusci, G. Dattoli, M. Del Franco, S. Licciardi; “Mathematical Methods for Physics”, invited Monograph by World Scientific, Singapore, 2017, in press}.\\

$\star$ \textbf{C. Cesarano, G. Dattoli, S. Licciardi; “Generating Functions for Lacunary Legendre and Legendre-like Polynomials”, work in progress}.\\

In Chapter \ref{ChapterTR} we show how the use of our techniques can be used to redefine and treat classic and \textit{"new trigonometries"}. We provide a recasting in umbral terms of ordinary circular functions. 
The recasting procedure allows new and unsuspected links with either special functions and special polynomials, thus providing the introduction of new polynomial families. We introduce special functions like pseudo hyperbolic functions, Hermite Bessel functions and Laguerre Bessel functions. The procedure we propose for the introduction of generalized forms of "Trigonometries" is based on two different points of view. The first is aimed at finding a thread between ordinary circular and trigonometric functions. It is indeed shown that the umbral form of the ordinary trigonometric functions allows a natural transition from the circular to Bessel functions, while the inclusion of further umbral generalization yields important consequences on a more general definition of the semi-group property, along with the possibility of defining a new point of view to the Laguerre Bessel addition theorems.
The second way we envisage to present a generalization of the trigonometry, is a revisitation of the \textit{Euler} exponential formula. The latter benefits from non standard form of imaginary numbers, realized using different types of matrices. The importance in application is also stressed. We discuss in particular the generalization of the \textit{Courant-Snyder} method, used in accelerator Physics to treat the beam transport in magnetic lenses and the solution of problems involving the Pauli equations.\\

The original parts of the Chapter, with their adequate bibliography, are contained in the papers specified below.\\

$\star$ \textbf{G. Dattoli, S. Licciardi, E. Di Palma, E. Sabia; “From circular to Bessel functions: a transition through the umbral method”, Fractal Fract 2017, 1(1), 9; doi:10.3390/fractalfract1010009}.\\

$\star$ \textbf{G. Dattoli, S. Licciardi, E. Sabia; "Generalized Trigonometric Functions and Matrix Parameterization”; Int. J. Appl. Comput. Math 2017, https://doi.org/10.1007/s40819-017-0427-0, pp. 1-14}.\\

$\star$ \textbf{G. Dattoli, S. Licciardi, F. Nguyen, E. Sabia; “Evolution equations involving Matrices raised to non-integer exponents”; Modeling in Mathematics, Atlantis Transactions in Geometry, vol 2. pp. 31-41, J. Gielis, P. Ricci, I. Tavkhelidze (eds), Atlantis Press, Paris, Springer 2017}.\\

$\star$ \textbf{G. Dattoli, S. Licciardi, R.M. Pidatella; “Theory of Generalized Trigonometric functions: From Laguerre to Airy forms“; arXiv: 1702.08520, 2017, submitted for publication to Electronic Journal of Differential Equations (EJDE) 2017}.\\

$\star$ \textbf{G. Dattoli, S. Licciardi, E. Sabia; "New Trigonometries", work in progress}.\\

Chapter \ref{Chapter3} deals with an umbral reformulation of the theory of \textit{Bessel functions}. We study the relevant properties by means of their umbral image, introduced in the first and second Chapter. We show that the relevant theory (including differential equations, recurrences, generating functions of linear  and quadratic type...) can all be reduced to straightforward elementary calculus computations. Such a re-elaboration aims at providing a unified treatment of the various Bessel type forms, widely documented in the mathematical literature. \\

This Chapter is based on the following originals papers.\\

$\star$ \textbf{Giuseppe Dattoli, Elio Sabia, Emanuele Di Palma, Silvia Licciardi; “Products of Bessel functions and associated polynomials”; Applied Mathematics and Computation, Vol 266 Issue C, September 2015, pages 507-514, Elsevier Science Inc. New York, NY, USA}.\\

$\star$ \textbf{D. Babusci, G. Dattoli, M. Del Franco, S. Licciardi; “Lectures on Mathematical Methods for Physics”, invited Monograph by World Scientific, Singapore, 2017, in press}.\\

$\star$ \textbf{G. Dattoli, S. Licciardi; “Book on Bessel Functions and Umbral Calculus”, work in progress}.\\

In Chapter \ref{ChapterNumberTh}, deals with the umbral treatment of the theory of \textit{harmonic and Motzkin numbers}. The Chapter contains topics of different nature with respect to those treated in the previous parts of the thesis. It has been added to prove the flexibility and the generality of the methods we have proposed and to show how our point of view provides a simple way to get results (like the generating functions of harmonic numbers) hardly achieveble with conventional means. \\
	
The original parts of the Chapter, containing an adequate bibliography to the relevant scientific literature, are included in the papers specified below.\\

$\star$ \textbf{M. Artioli, G. Dattoli, S. Licciardi, S. Pagnutti; “Motzkin numbers: an operational point of view"; arXiv:1703.07262 2017, submitted for publication  to Online Electronic Integer Sequences, 2017.}\\

$\star$ \textbf{M. Artioli, G. Dattoli, S. Licciardi; “Motzkin Numbers and their \;\;Geometrical \;I\;nterpretation”;\;\;\; \;Wolfram \;Demonstrations\\
	 Project, 2017.}\\

$\star$ \textbf{G. Dattoli, B. Germano,  S. Licciardi, M.R. Martinelli; “Umbral methods and Harmonic Numbers”, researchgate 2017, submitted for publication to Mediterranean Journal of Mathematics, 2017.}\\

$\star$ \textbf{G. Dattoli, S. Licciardi, E. Sabia; “On the properties of Generalized Harmonic numbers” , work in progress.}\\

At the end of the thesis, two Appendices are also provided. We treat extensions of arguments presented in the main body of the Chapters or specific demonstrations which are required but not necessary in the main subject.\\

Finally, we want underline the idea put forward in this thesis, where the meaning of the umbra and of the associated
umbral calculus deepens its roots into a phylosophycal conception tracing back to a platonic view of the Mathematics itself. For this reason we have referred to the myth of cave and conceive the functions in an abstract sense as a convenient image of a given reference form, like summarized in  Fig. \ref{figPlatone}. \\

\begin{figure}[b]
	\centering
	\includegraphics[width=0.8\linewidth]{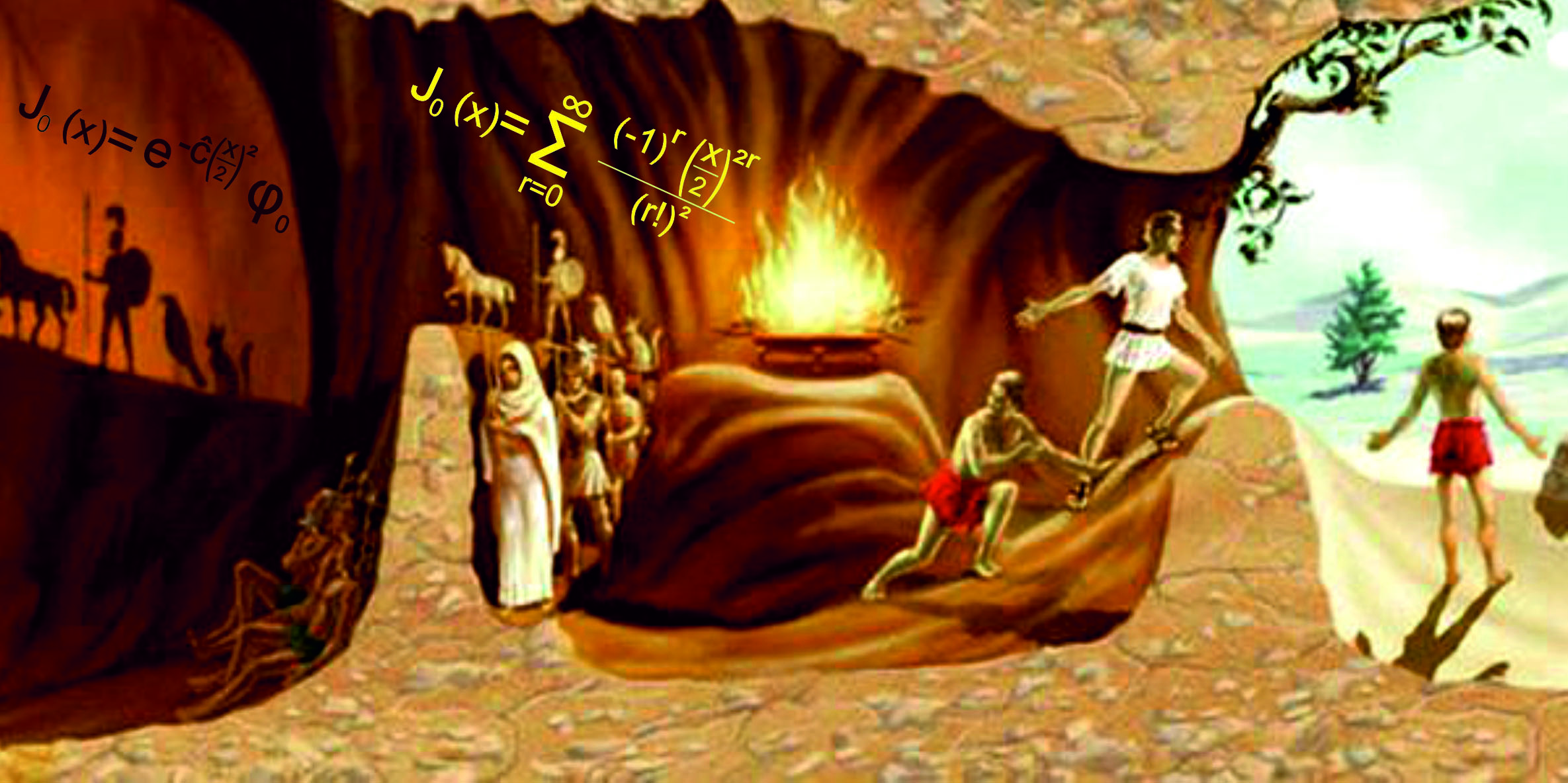}
	\caption{The "Platonic" conception of UMBRA according to the point of view developed in this thesis (imagine by \textit{www.multytheme.com/cultura/multimedia/didattmultitema/scuoladg/filosofia}
		\textit{mitocaverna.htm} and modified by the author).}
	\label{figPlatone}
\end{figure}
\chapter{Operator Theory and Umbral Calculus}\label{Chapter1}
\numberwithin{equation}{section}
\markboth{\textsc{\chaptername~\thechapter. Operator Theory and Umbral Calculus}}{}

In Chapter \ref{Chapter1} we give, at the beginning,  an introduction to the \textit{umbral calculus, definition and operational rules}. We define the concept of umbral image function and develop a case study regarding the Bessel functions which, within the present context, are
\textit{Gaussian functions}. In the second part we provide some  examples of  versatility of the method, in particular a great deal of effort is devoted to the theory of \textit{Mittag-Leffler functions}, with umbral image provided
by an exponential function. We present a solution of fractional evolution equations
of \textit{Schr\"{o}dinger} (\textit{FSE}) type which is extensible to the fractional case.  We study the \textit{FSE} ruling the process of photon absorption/emission and introduce the associated probability distribution, which results to be a fractional \textit{Poisson} type distribution.\\

The original parts of the Chapter, containing their adequate bibliography, are based on the following original papers.\\

\cite{ProdB} \textit{G. Dattoli, E. Di Palma, E. Sabia, S. Licciardi; "Products of Bessel Functions and Associated Polynomials"; Applied Mathematics and Computation, vol.  266,  Issue C, September 2015, pages 507-514, Elsevier Science Inc. New York, NY, USA}.\\

\cite{Babusci} \textit{D. Babusci, G. Dattoli, M. Del Franco, S. Licciardi; “Mathematical Methods for Physics”, invited Monograph by World Scientific, Singapore, 2017, in press}.\\

\cite{ML} \textit{G. Dattoli, K. Gorska, A. Horzela, S. Licciardi, R.M. Pidatella; “Comments on the Properties of Mittag-Leffler Function”,arxiv.org/abs/\\
	1707.01135 [math-ph], submitted for publication to European Physical Journal, 2017}.\\

$\star$ \textit{G. Dattoli, S. Licciardi; “Book on Bessel Functions and Umbral Calculus”, work in progress}.\\

 According to our procedure, the \textbf{\textit{Bessel and Gaussian}}  are reciprocal images of one function onto the other.\\

 The understanding of the methods we will develop along the course of this and forthcoming Chapters needs of a few introductory remarks clarifying the notation we are going to use throughout the text.\\

We will make large use, e.g.,  of the \textit{\textbf{Euler Gamma Function}}, defined by the following integral representation \cite{Abramovitz}:

\begin{equation}\begin{split}\label{FunzGamma}
& \Gamma (z)=\int_{0}^{\infty}e^{-\xi}\xi^{z-1}d\xi\;,\\
& Re(z)>0\;.
\end{split}\end{equation}
and, more in general, by \cite{Abramovitz}

\begin{equation}
 \Gamma(z)=\lim_{n\rightarrow\infty}\dfrac{n!\;n^{z}}{\prod_{r=0}^{n}(z+r)}, \quad z\in \left\lbrace  \mathbb{C}\smallsetminus\mathbb{Z}^{-}\right\rbrace .
\end{equation}

It is well known that this function, for natural values of the variable $z$, reduces to the ordinary factorial,  it can accordingly be viewed as a generalization of such operation. \\

\noindent The well known following identities are easily derived

\begin{align}
& \Gamma (n+1)=\int_{0}^{\infty}e^{-\xi}\xi^{n}d\xi=n!,\quad \forall n\in\mathbb{N},\label{Gpropa}\\
& \Gamma\left( \dfrac{1}{2}\right)= \int_{0}^{\infty}e^{-\xi}\xi^{-\frac{1}{2}}d\xi=\sqrt{\pi}\;.\label{Gpropb}
\end{align}
The eq. \ref{Gpropa} is proved by repeated integration by parts and eq. \ref{Gpropb} is proved after setting $\xi=\mu^{2}$ and reducing the integral to a standard Gaussian integration.\\

\begin{figure}[h]
	\centering
	\includegraphics[width=0.7\linewidth]{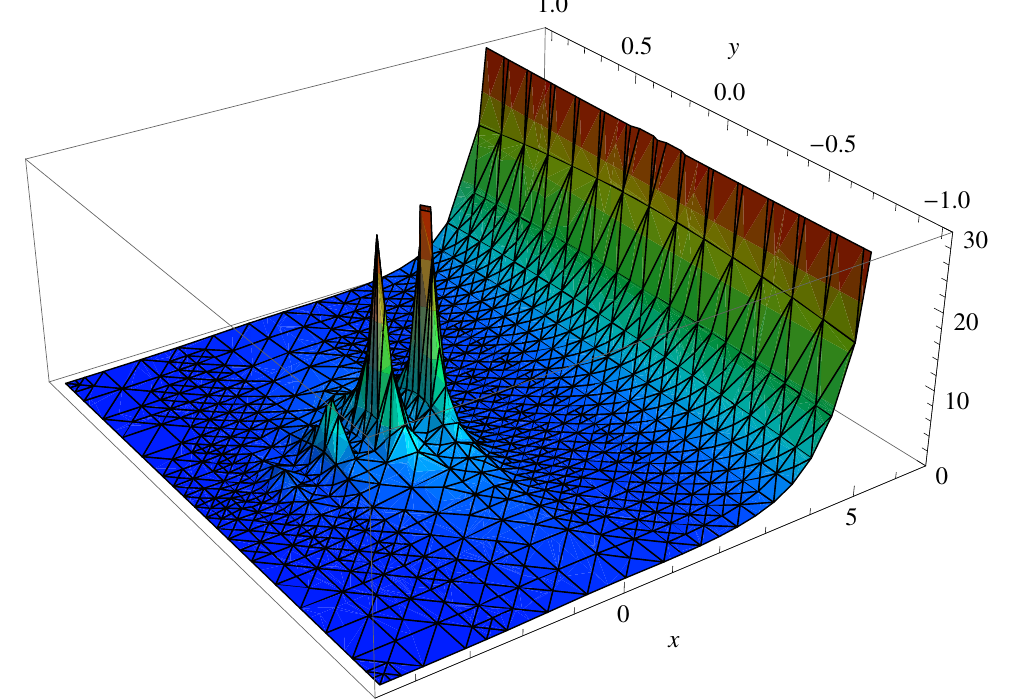}
	\caption{Euler Gamma Function $\Gamma(z)$ in the complex plane.  The poles are present for  $z\in\mathbb{Z}^-$. }
	\label{figfunctgamma}	
\end{figure}

\noindent The relevant plot in the complex plane is given in Fig. \ref{figfunctgamma}, which shows the poles for negative integer values of the argument.\\

Along with the Euler Gamma, another function (also introduced by Euler) will often be exploited here, namely the \textit{\textbf{Beta-Function}}  
defined in terms of the Gamma function as \cite{Abramovitz}

\begin{equation}\label{Betaf}
B(x,y)=\dfrac{\Gamma(x)\Gamma(y)}{\Gamma(x+y)}
\end{equation}
which reads, for $Re(x), Re(y)>0$, 
\begin{equation}\label{BetaF}
B(x,y)=\int_{0}^{1}t^{x-1}(1-t)^{y-1}dt\;.
\end{equation}
Most of the functions we will discuss in the following are expressed in terms of the Gamma function. 

\section{From Special Function to its Umbral Image}\label{UmbralImage}

We illustrate the "transition" from a \textit{Special Function} to its umbral image by using a fairly straightforward example.  The series \ref{BW} defines  the \textit{\textbf{Bessel-Wright}} (\textit{BW}) function \cite{E.M.Wright} 

\begin{equation}\label{BW}
W_{\beta}^{\alpha}(x)=\sum_{r=0}^{\infty}\dfrac{x^{r}}{r!\;\Gamma(\alpha r+\beta+1)}, \quad \forall  x\in \mathbb{R}, \forall  \alpha,\beta\in \mathbb{R}^+_0 .
\end{equation}
For $\alpha=1$ and $x\rightarrow -x$, the above function is known as the \textit{\textbf{Tricomi-Bessel}} (\textit{TB}) function of order $\beta$ \cite{Tricomi} 

\begin{equation}\label{TrBalfa1}
C_{\beta}(x)=\sum_{r=0}^{\infty}\dfrac{(-x)^{r}}{r!\Gamma(r+\beta+1)},\quad \forall\; x,\;\beta\in \mathbb{R}.
\end{equation}
The  \textit{\textbf{Umbral Formalism}}, which we use, relies on the simple assumption that functions of the type \ref{BW} can be treated as an \textbf{\textit{ordinary exponential function}}, provided that we adopt the following notation for the \textit{BW} function of $0$-order.

\begin{exmp}\label{exmpleBW}
\begin{equation}\label{UmbralC}
W_{0}^{\alpha}(x)=e^{\hat{c}^{\alpha}x}\varphi_{0},\quad \forall  x,\alpha\in \mathbb{R} .
\end{equation}
\end{exmp}
To proof this statement we define the following.
\begin{defn}
	The function\footnote{We remind that the inverse $\Gamma$-function is an entire function.}
	\begin{equation}\label{vacuum}
	\varphi(\mu):=\varphi_{\mu}=\dfrac{1}{\Gamma(\mu+1)}, \quad \forall \mu\in\mathbb{R},
	\end{equation}
 is called 	umbral \textbf{"vacuum"}.
\end{defn}
This term, borrowed from physical language, is used to stress that \textbf{the action of the operators $\mathbf{\hat{c}}$, raised to some power, is that of acting on an appropriate set of functions} (in this case the Euler Gamma function),  \textbf{by "filling" the initial "state"} $\varphi_{0}=\dfrac{1}{\Gamma(1)}$.

\begin{defn}
	 We define the  \textbf{Operator} $\mathbf{\hat{c}}$, called \textbf{Umbral},
\begin{equation}\label{vso}
\hat{c} = e^{ \partial_z},
\end{equation}	
 the \textbf{vacuum shift operator}, being z the domain's variable of the function on which the operator acts.
\end{defn}	 
\begin{thm}\label{thOpc}
The umbral operator, $\mathbf{\hat{c}^{\mu}}$, $\forall\; \mu\in  \mathbb{R}$, is the action of the operator $\hat{c}$ on the vacuum $\varphi_{0}$ such that 

\begin{equation}\label{Opc}
\hat{c}^\mu \varphi_{0}:=\varphi_{\mu}=\dfrac{1}{\Gamma(\mu+1)}.
\end{equation}
\end{thm}
\begin{proof}[\textbf{Proof}]
	$\forall\; \mu\in  \mathbb{R}$, applying eqs. \ref{vso} and \ref{vacuum}, we obtain
	\begin{equation*}
	\hat{c}^{\mu}\varphi_0 =e^{\mu \partial_z}\varphi_z \mid_{z=0}=\left. \varphi_{z+\mu} \right| _{z=0}=\left.  \dfrac{1}{\Gamma(z+\mu+1)}\right| _{z=0}=\dfrac{1}{\Gamma(\mu+1)}.
	\end{equation*}
\end{proof}
\noindent The umbral operator so defined satisfies the following

\begin{propert}
	$\forall \mu,\nu \in  \mathbb{R}$
\begin{align}
& i) \;\;\;\;\hat{c}^{\;\pm\mu}\hat{c}^{\;\nu}=\hat{c}^{\;\pm\mu+\nu}, \label{propertCa}\\[1.1ex]
& ii) \;\;\left( \hat{c}^{\;\pm \mu}\right) ^{\nu}=\hat{c}^{\;\pm\mu\; \nu}. \label{propertCb}
\end{align}
\end{propert}

\begin{proof}[\textbf{Proof}]
	$\forall\; \mu\in  \mathbb{R}$
	\begin{equation}\begin{split}\label{key}
& i)\;\;\;	\hat{c}^{\;\mu} \hat{c}^{\;\nu}=e^{\;\mu\partial_z}e^{\;\nu\partial_z}=e^{(\mu+\nu)\partial_z}=\hat{c}^{\;\mu+\nu},\\
& ii) \;\; analogous.
\end{split}	\end{equation}
	\end{proof}
We underline that the action of the operator on the vacuum cannot be separated, it has to work in a \underline{unique action}.\\

We now provide the proof of \textbf{Example} \ref{exmpleBW}.\\

 \begin{proof}[\textbf{Proof}] 	
	 We give a meaning to eq. $W_{0}^{\alpha}(x)=e^{\hat{c}^{\alpha}x}\varphi_{0},\; \forall  x,\alpha\in \mathbb{R}$,  by treating the r.h.s. as the exponential function of the operator $\hat{c}$ and thus, using an ordinary \textit{\textbf{Mac Laurin}} expansion, we end up with (see \ref{propertCb})
	
	\begin{equation}\label{exandUC}
	e^{\hat{c}^{\alpha}x}\varphi_{0}=\sum_{r=0}^{\infty}\dfrac{\hat{c}^{\;\alpha\; r}}{r!}x^{r}\varphi_{0}.
	\end{equation}
	The operator $\hat{c}$ acts on $\varphi_{0} $ only, \textbf{\underline{leaving $x$-unaffected}}, then \underline{\textbf{$\mathbf{\hat{c}}$ and x}}\\ \underline{\textbf{commute}} and we can cast the r.h.s. of eq. \ref{exandUC} in the form
	
	\begin{equation}\label{exandUCbis}
	e^{\hat{c}^{\alpha}x}\varphi_{0}=\sum_{r=0}^{\infty}\dfrac{x^{r}}{r!}\left( \hat{c}^{\;\alpha\; r}\varphi_{0}\right),
	\end{equation}
	therefore, by appling the rule \ref{Opc}, we end up with 
	
	\begin{equation}
	e^{\hat{c}^{\alpha}x}\varphi_{0}=\sum_{r=0}^{\infty}\dfrac{x^{r}}{r!\;\Gamma(\alpha r+1)}=W_{0}^{\alpha}(x), \quad \forall x,\alpha\in  \mathbb{R}.
	\end{equation}
\end{proof}
	It is also easily understood that, within such a formalism, the $\beta$-order \textit{BW} function can be written as
	
	\begin{equation}\label{BWamu}
	W_{\beta}^{\alpha}(x)=\hat{c}^{\;\beta}e^{\hat{c}^{\alpha}x}\varphi_{0}, \quad \forall  x\in \mathbb{R}, \forall  \alpha,\beta\in \mathbb{R}^+_0.
	\end{equation} 

In the forthcoming part of the thesis, we will take the freedom of treating \textit{\textbf{$\mathbf{\hat{c}}$-like operator as \underline{ordinary algebraic quantities}}} and, in the following section, we will see how the "associated" calculus finds its formal justification on the properties of the \textit{\textbf{Borel Transform} }.

\subsection{Borel Transform}\label{SecBorel}

The theory of integral transforms is one of the fundamentals of the operational calculus. We therefore make a further step in this direction, by providing a more rigorous environment to formulate the umbral technicalities established in the previus sections  using as support the  formalism underlying the theory of \textit{\textbf{Borel Transform}} \textit{(BT)} \cite{ProdB}.\\

\noindent We show that, for the present purposes, the Borel transform can be conveniently expressed in terms of Gamma function and of simple differential operators. Therefore, before proceeding further, we remind the identity \cite{ProdB}, $\forall \lambda\in\mathbb{R}, \forall x \in f's\;domain$,

\begin{equation}\begin{split}
& e^{\lambda x\hat{D}_{x}}f(x)=f(e^{\lambda}x),\\
& \hat{D}_{x}=\dfrac{\partial}{\partial x},
\end{split}\end{equation}
whose proof is easily obtained after setting $x=e^{\zeta}$ and noting that \cite{Babusci}

\begin{equation}
e^{\lambda x\hat{D}_{x}}f(x)=e^{\lambda\hat{D}_ \zeta}f(e^{\zeta})=f(e^{\lambda}e^{\zeta})=f(e^{\lambda}x).
\end{equation}
As a consequence, the further identity\footnote{Or, in alternative way, by setting   $\left(\alpha \right)^{x\, \partial _{x} } =e^{\ln (\alpha )\, x\, \partial _{x} } $  and by making the change of variables   $x=e^{y} $,  we get   $\left(\alpha \right)^{x\, \partial _{x} } f(x)=e^{\ln (\alpha )\, \partial _{y} } f(e^{y} )=f(e^{y+\ln (\alpha )} )$,   finally going back to the original variable we end up with eq. \ref{tderf}. } \cite{Babusci}

\begin{equation}\label{tderf}
t^{x\hat{D}_{x}}f(x)=f(tx)
\end{equation}
holds true. Furthermore it is also important to stress that the monomial $x^{n}$ is an eigenfunction of the operator $x\hat{D}_{x}$ in the sense that

\begin{equation}\label{eigen}
(x\hat{D}_{x})x^{n}=nx^{n}.
\end{equation}
According to the previous discussion, the \textit{BT} \cite{Ehrenpreis}, expressed by the integral

\begin{equation}\label{fBxInt}
f_{B}(x)=\int_{0}^{\infty}e^{-t}f(tx)dt,
\end{equation}
can be recast in the operational form \cite{ProdB}

\begin{equation}\begin{split}\label{Brecast}
& f_{B}(x)=\hat{B}\left( f(x)\right), \\
& \hat{B}=\int_{0}^{\infty}e^{-t}t^{x\hat{D}_{x}}dt=\Gamma(x\hat{D}_{x}+1).
\end{split}\end{equation}

A paradigmatic example, displaying how the $ \hat{B}$ operator acts on a specific function, is provided by the $0$-order Tricomi-Bessel function (eq. \ref{TrBalfa1}) \cite{Tricomi}, as showed in the following

\begin{exmp}
\begin{equation}\label{TrBzero}
C_{0}(x)=
\sum_{r=0}^{\infty}(-1)^{r}\dfrac{x^{r}}{(r!)^{2}}, \quad \forall x\in\mathbb{R}.
\end{equation}
The use of the identities \ref{Brecast} and \ref{eigen} yields, $\forall x\in\mathbb{R}$,

\begin{equation}\begin{split}\label{opBhat}
\hat{B}\left( C_{0}(x)\right) &=\Gamma(x\hat{D}_{x}+1)\left( C_{0}(x)\right) =\sum_{r=0}^{\infty}(-1)^{r}\Gamma(r+1)\dfrac{x^{r}}{(r!)^{2}}=\\
& =e^{-x}=\sum_{r=0}^{\infty}(-1)^{r}\dfrac{x^{r}}{(r!)^{2}}\int_{0}^{\infty}e^{-t}t^rdt=
\int_{0}^{\infty}e^{-t}C_{0}(tx)dt.
\end{split}\end{equation}
The $\hat{B}$ operator has evidently acted on the Bessel type function $C_{0}(x)$ by providing a kind of “\textbf{downgrading}” from higher transcendental function to the “simple” exponential.
\end{exmp}

\begin{exmp}
 The successive application of the Borel operator to the same previous function  produces the further result reported below:

\begin{equation}\label{B2}
\hat{B}^{2}[C_{0}(x)]=\hat{B}[e^{-x}]=\sum_{r=0}^{\infty}(-1)^{r}\Gamma(r+1)\dfrac{x^{r}}{r!} =\dfrac{1}{1+x} ,\;\;\; \mid x\mid<1.
\end{equation}
Again, we notice the same behaviour: the exponential function has been reduced to a rational function.
\end{exmp}

The further application of $\hat{B}$ yields a diverging series, namely

\begin{exmp}
\begin{equation}\label{B3}
\hat{B}^{3}[C_{0}(x)]=\sum_{r=0}^{\infty}(-1)^{r}r!x^{r}, \quad \forall x\in\mathbb{R}. 
\end{equation}
\end{exmp}

We have interchanged Borel operators and series summation without taking too much caution. In the case of eq. \ref{opBhat} such a procedure is fully justified, in eq. \ref{B2} the method is limited to the convergence region, while in the case of eq. \ref{B3} the procedure is not justified since it gives rise to a diverging series. In the following we will take  some freedom in handling these problems and include in our treatment also the case of diverging series.\\

Since the repeated application of \textit{BT} is associated with the Borel operator raised to some integer power, we extend the definition to a fractional \textit{BT} and, more in general, to a real positive and negative power \textit{BT}.\\

\noindent We introduce indeed the operator

\begin{equation}\label{Balpha}
\hat{B}_{\alpha}=\int_{0}^{\infty}e^{-t}t^{\alpha\; x\;\partial_{x}}dt=\Gamma(\alpha \;x\;\partial_{x}+1),\quad \alpha\in\mathbb{R}^+,
\end{equation}
which will be referred as the $\alpha$-order Borel transform.

\begin{exmp}
 By using the \textbf{\textit{cylindrical Bessel Special Function}} \cite{Babusci} (which will have a dedicated wide discussion in Chapter \ref{Chapter3}) $\forall x\in\mathbb{R}$

\begin{equation}\label{cBf}
J_{0}(x)=\sum_{r=0}^{\infty}\dfrac{(-1)^{r}\left(\frac{x}{2} \right)^{2r} }{(r!)^{2}},
\end{equation}
 we find that the $\dfrac{1}{2}$-order \textit{BT} applied to the $0$-order Bessel yields

\begin{equation}\begin{split}\label{B1/2}
& \hat{B}_{\frac{1}{2}}[J_{0}(x)]=\Gamma\left( \dfrac{1}{2}x\partial_{x}+1\right)\sum_{r=0}^{\infty}\dfrac{(-1)^{r}}{(r!)^{2}}\left( \dfrac{x}{2}\right)^{2r}=\\
& =\sum_{r=0}^{\infty}\dfrac{(-1)^{r}}{r!}\left( \dfrac{x}{2}\right)^{2r}=e^{-\left( \frac{x}{2}\right)^{2} } . 
\end{split}\end{equation}
By assuming that $ \alpha>0$, exists an operator $\left( \hat{B}^{(\alpha)}\right)^{-1} $ such that 

\begin{equation}\begin{split}\label{negaBorel}
& \left( \hat{B}_{\alpha}\right)^{-1}\hat{B}_{\alpha}=\hat{1},\\
& \left(\hat{B}_{\alpha}\right)^{-1}=\dfrac{1}{\Gamma(\alpha x\partial_{x}+1)},
\end{split}\end{equation}
then we can invert eq. \ref{B1/2} and write

\begin{equation}\label{foraProdB}
\left( \hat{B}_{\frac{1}{2}}\right)^{-1}\left[ e^{-\left( \frac{x}{2}\right)^{2} }\right]=J_{0}(x). 
\end{equation}
\end{exmp}

\begin{Oss}
The extension of eq. \ref{negaBorel} to negative  $\alpha$ yields

\begin{equation}
\hat{B}_{(-\alpha)}=\Gamma(-\alpha x \partial_x+1)=\frac{1}{\Gamma(\alpha x \partial_x)}\frac{\pi}{sin(\alpha \pi x \partial_x)}
\end{equation}
and it is worth stressing that \cite{ProdB}

\begin{equation}
\hat{B}_{(-\alpha)}\neq \left[\hat{B}_{(\alpha)}  \right]^{-1}.
\end{equation}
\end{Oss}

A definition of the inverse of the operator $\hat{B}_{\alpha}$ may be achieved through the use of the \textit{Hankel} contour integral, namely 

\begin{equation}\begin{split}
& \dfrac{1}{\Gamma(z)}=-\dfrac{i}{2\pi}\int_{C}\dfrac{e^{-t}}{(-t)^{z}}dt,\\
& \mid z\mid<1,
\end{split}\end{equation}
which can be exploited to write

\begin{equation}
\left( \hat{B}_{\alpha}\right)^{-1}f(x)=-\dfrac{i}{2\pi}\int_{C}\dfrac{e^{-t}}{t}f\left( \dfrac{x}{(-t)^{\alpha}}\right) dt.
\end{equation}

After the previous remarks we can state the following Theorem.\\

\begin{thm}
	Let  $f(x)$ a function such that  $\int_{-\infty}^{+\infty}f(x)dx=k,\; \forall k\in \mathbb{R}$, then
\end{thm}
\begin{equation}
\int_{-\infty}^{+\infty}\hat{B}_{\alpha}[f(x)]dx=k\;\Gamma(1-\alpha), \quad \mid \alpha\mid<1 .
\end{equation}

\begin{proof}[\textbf{Proof}]
	The proof is fairly straightforward by applying eq. \ref{fBxInt} and the variable change $t^\alpha x=\sigma$. $\forall k\in \mathbb{R}$, $\mid \alpha\mid<1$, we find
	
	\begin{equation*}\begin{split}
	& \int_{-\infty}^{+\infty}\hat{B}^{(\alpha)}[f(x)]dx=\int_{-\infty}^{+\infty}\left( \int_{0}^{+\infty}e^{-t}f(t^{\alpha}x)dx\right) dt=\\
	& =\int_{-\infty}^{+\infty}e^{-t}\left( \int_{0}^{+\infty}f(t^{\alpha}x)dx\right) dt=\int_{-\infty}^{+\infty}e^{-t}t^{-\alpha}\left( \int_{0}^{+\infty}f(\sigma)d\sigma\right) dt=\\
	& =\int_{-\infty}^{+\infty}f(\sigma)\left(\int_{0}^{+\infty}e^{-t}t^{-\alpha}dt \right) d\sigma=
	k\;\Gamma(1-\alpha) .
	\end{split}\end{equation*}
\end{proof}
The same procedure can be exploited  for cases involving the inverse transform.\\

These remarks provide a more sounded basis for the formalism we are discussing and which will be further developed and applied in the forthcoming parts. We will corroborate our  conclusions using extensions of the concepts developed in this section.

\section{The Gaussian Function in Umbral Calculus}

In the following,  we exploit the properties of the Gaussian function in an umbral context and, in particular, we see that families of Special Functions like Bessel functions  can be viewed as \textit{Umbral "representation"} of the Gaussian itself. To this aim, it is worth reminding the properties of the function $e^{-x^2}$ which are listed below \cite{Babusci}.\\

We start with the well known \textbf{\textit{Gaussian Integral}}

\begin{equation}\label{gau}
\int_{-\infty}^{\infty}e^{-x^2}dx=\sqrt{\pi}
\end{equation}
and remind the  \textit{\textbf{Gauss-Weierstrass integral}} $(GWI)$, which  will often be exploited in the following,

\begin{equation}\label{GWi}
\int_{-\infty}^{\infty}e^{-ax^2+bx}dx=\sqrt{\dfrac{\pi}{a}}\;e^{\;\frac{b^2}{4a}},\quad \forall b\in\mathbb{R}, \forall a\in\mathbb{R}^+.
\end{equation}
A particularly useful result, strictly related to \ref{GWi}, is given below by the
 \textbf{\textit{Gaussian integral identity}} $(GII)$

\begin{equation}\label{Gii}
e^{-b^2}=\dfrac{1}{\sqrt{\pi}}\int_{-\infty}^{\infty}e^{-\xi^2-2\;i\;b\;\xi}\;d\xi , \quad \forall b\in\mathbb{R}.
\end{equation}

Along with the Gaussian function, we introduce an associated family of special poynomials, which plays a crucial role for the topics treated in this and in the forthcoming chapters.  We remind therefore that, according to the  \textbf{\textsl{Rodriguez formula}} \cite{L.C.Andrews}, we obtain the following

\begin{prop}
	Let
	
	\begin{equation}\label{classHerm}
	H_n(\xi,\mu)=n!\sum_{r=0}^{\lfloor\frac{n}{2}\rfloor}\dfrac{\xi^{n-2r}\mu^r}{r!(n-2r)!},\quad \forall \xi,\mu\in\mathbb{R}, \forall n\in\mathbb{N},
	\end{equation}
	two variable polynomials often referred as \textbf{\textsl{Hermite-Kamp\'e de F\'eri\'et}} polynomials, obtained by repetead derivatives of a Gaussian or also defined through the generating function \cite{Appell,Babusci}
	
	\begin{equation}\label{genfunctH}
	\sum_{n=0}^{\infty}\dfrac{t^n}{n!}H_n (x,y)=e^{xt+yt^2}, \quad \forall x,y\in\mathbb{R}.
	\end{equation}
Then, $\forall a\in\mathbb{R}$,	
\begin{equation}\label{GHPol}
\partial_{x}^n e^{-ax^2}=H_n (-2ax,-a)e^{-ax^2}=(-1)^n H_n (2ax,-a)e^{-ax^2},
\end{equation}
is a generalized form of Hermite Polynomials.
\end{prop}
  We remind also the link between two variable $Hermite$ polynomials and one variable $Hermite$ polynomials \cite{Babusci}. 
  
 \begin{propert}
 	$\forall x,y\in\mathbb{R}, \forall n\in\mathbb{N}$\footnote{We mention one-variable Hermite polynomial \cite{Abramovitz} \begin{equation}\label{key}
 		H_n(x)=(-1)^ne^{x^2}\partial_x^n e^{-x^2}.
 		\end{equation}}	 
  \begin{equation} 
  \begin{split}\label{idGeg}
  & i)\;\;\; H_n (x,y)=(-i) ^n y^{n/2}H_n\left( \dfrac{i\;x}{2\sqrt {y}}\right),\\
  or\; also\\
  & ii)\;\; H_n (x,y)= (-y)^{n/2}H_n\left( \dfrac{x}{2\sqrt {-y}}\right),\\
  then\\
  & iii)\; H_n (x)=H_n (2x,-1),\\
  furthermore\\
  & iv)\; 
  H_{n}(x,y)=y^{\frac{n}{2}}H_{n}\left( \dfrac{x}{\sqrt{y}},1\right). 
  \end{split} 
  \end{equation}
\end{propert}

\subsection{Umbral Bessel Function }
To give a first idea of how powerful the umbral representation is, we consider the \textit{cylindrical Bessel function of $0$-order} \ref{cBf} \cite{Babusci},  and note that
 
 \begin{lem}
  By using the operator definition \ref{Opc} and the property of $\Gamma$-function \ref{Gpropa}, we find, $\forall x\in\mathbb{R}$,
 
 \begin{equation}\label{J0op}
 J_{0}(x)=\sum_{r=0}^{\infty}\dfrac{(-1)^{r}\left(\frac{x}{2} \right)^{2r} }{(r!)^{2}}=\sum_{r=0}^{\infty}\dfrac{(-1)^{r}\left(\frac{x}{2} \right)^{2r}\hat{c}^r }{r!}\varphi_0 =e^{-\hat{c}\left(\frac{x}{2} \right)^{2}}\varphi_{0},
 \end{equation} 
 obtaining in this way a new formulation of Bessel function.
 \end{lem}
  We notice also that 
  
   \begin{lem}\label{lem2C}
  The $0$-order Tricomi-Bessel \ref{TrBalfa1} is expressible in the $0$-order cylindrical Bessel as
 
 \begin{equation}\label{C0J0}
 C_0(x)=J_0\left(2\sqrt{x} \right), \quad \forall x\in\mathbb{R}. 
 \end{equation}
 \end{lem}

\begin{proof}[\textbf{Proof}]
$\forall x\in\mathbb{R}	$, by using \ref{Gpropa} and algebraic manipulations 

\begin{equation}\label{key}
C_0(x)=\sum_{r=0}^{\infty}\dfrac{(-x)^{r}}{r!\Gamma(r+1)}=\sum_{r=0}^{\infty}\dfrac{(-1)^r(\sqrt{x})^{2r}}{r!^2}=J_0\left(2\sqrt{x} \right).
\end{equation}	
\end{proof}	

\begin{cor}
	We can write the umbral $0$-order Tricomi-Bessel function
	\begin{equation}\label{C0umbral}
	C_0(x)=e^{-\hat{c}x}\varphi_{0}, \quad \forall x\in\mathbb{R}.
	\end{equation}
\end{cor}	

\begin{exmp}
 Using the \textit{GWI} \ref{GWi}, the operator definition \ref{Opc} and the $\Gamma$-function property \ref{Gpropb}, we obtain 
 
 \begin{equation}\begin{split}\label{IntJ0cop}
 \int_{-\infty}^{\infty}J_{0}(x)dx&= \int_{-\infty}^{\infty}e^{-\hat{c}\left(\frac{x}{2} \right)^{2}}dx\;\varphi_{0}=2\sqrt{\dfrac{\pi}{\hat{c}}}\varphi_0=2\sqrt{\pi}\hat{c}^{-\frac{1}{2}}\varphi_0
 =\\
 & =2\sqrt{\pi}\dfrac{1}{\Gamma\left( -\dfrac{1}{2}+1\right) }=2.
\end{split} \end{equation}
 \end{exmp}
We end up with  \cite{FromCircular}

\begin{lem}
	By the use of the \textit{GII} \ref{Gii} and eq. \ref{UmbralC}, we obtain, $\forall x\in\mathbb{R},$ 
 \begin{equation}\begin{split}\label{J0W0} 
  J_{0} (x)&=e^{-\hat{c}\left(\frac{x}{2} \right)^{2}}\varphi_{0}=
  \dfrac{1}{\sqrt{\pi } } \int _{-\infty }^{+\infty }e^{-\xi ^{2} -i\, \hat{c}\, ^{\frac{1}{2} } \, x\;\xi}  d \xi \, \varphi _{0}  =\\
  & =
  \dfrac{1}{\sqrt{\pi } } \int _{-\infty }^{+\infty }e^{-\xi ^{2} } \left[e^{-i\, \hat{c}\, ^{\frac{1}{2} }  \; x\;\xi} \varphi _{0} \right] d \xi =\dfrac{1}{\sqrt{\pi } } \int _{-\infty }^{+\infty }e^{-\xi ^{2} }W_0^{ \frac{1}{2} }\left( -i\;x\;\xi\right)d\xi ,
 \end{split}\end{equation}
 which realizes a new integral representation of $0$-order Bessel function in terms of $0$-order \textit{BW} function.
\end{lem}

Now it is clear what is the meanng of  umbral image of a function. Whenever two functions share the same formal series by means of the introduction of an appropriate umbral operator, they are reciprocal images.

\subsection{$\mathbf{\hat{b}}$ - Operator}

To make the previous remark even more effective, it is important to stress that not only the exponential formal series is useful, but other series can be exploited too.
Introducing for example, with the same procedure of section \ref{UmbralImage}, the operator $\hat{b}$, with its opportune vacuum $\Phi$, we provide

\begin{defn}
\begin{equation}\label{bOp}
\hat{b}^r\Phi_0:=\Phi_r=\dfrac{1}{\left( \Gamma(r+1)\right)^2 }, \quad \forall r\in\mathbb{R}.
\end{equation}
\end{defn}

\noindent Then we can set 

\begin{lem}
By using an ordinary Mac Laurin expansion we get, $\forall x\in\mathbb{R}$,
\begin{equation}\begin{split}\label{J0opb}
J_0 (x)&=\dfrac{1}{1+\hat{b}\left( \frac{x}{2}\right)^2 }\Phi_0 
=\sum_{r=0}^{\infty}(-1)^r \left( \dfrac{x}{2}\right)^{2r} \left( \hat{b} ^{\;r}\;\Phi_0 \right)=\\
&  =\sum_{r=0}^{\infty} \dfrac{(-1)^r}{(r!)^2} \left( \dfrac{x}{2}\right)^{2r},
\end{split}\end{equation}
thus expressing Bessel function in another way, according to operator $\hat{b}$.
\end{lem}

 The use of the integral formula 

\begin{equation}\label{intPia}
\int_{-\infty}^{\infty}\dfrac{1}{1+a\;x^{2}}dx= \dfrac{\pi}{\sqrt{a}}, \quad \forall a\in\mathbb{R},
\end{equation}
yields, for example,

\begin{cor}
By applying the operator \ref{bOp}	
\begin{equation}
\int_{-\infty }^{\infty} J_0 (x)dx=2 \dfrac{\pi}{\sqrt{\hat{b}}}\Phi_0 = 2\; \pi\; \hat{b} ^{-\frac{1}{2}} \Phi_0 = 2 \dfrac{\pi}{\left( \Gamma  \left( \frac{1}{2}\right)\right) ^2 }=2.
\end{equation}
\end{cor}
The use of a rational function instead of an exponential as umbral image of a Bessel is therefore equally useful.\\

A conclusion to be drawn from the previous examples is that, \textbf{choosing an elementary function as the umbral image  of another, the properties of the first can be "transferred" to the second}, provided that a set of formal rules be applied.
Such a statement is referred as \textbf{\textit{"Principle of Permanence of the Formal Properties"}}.

\subsection{Principle of Permanence of the Formal Properties}\label{PPFP}

Such a “principle”, a close consequence of the $RMT$ \eqref{AppARMT} and of the umbral formalism, can be worded as follows.

\begin{prop}
	If an umbral correspondence between two different functions is established, such a correspondence can be extended to other operations including integrals.
\end{prop}

We illustrate such a statement with the following example.

\begin{exmp} Suppose we want to calculate the integral 

\begin{equation}\label{key}
I_T(a,b)=\int_{-\infty }^{\infty}C_0(bx)e^{-ax^2}dx, \quad \forall b\in\mathbb{R},\forall a\in\mathbb{R}^+,
\end{equation}
where the subscript T stands for Tricomi. According to the umbral definition of 0-order Tricomi-Bessel function \ref{C0J0} and reminding \ref{J0op} $C_0(bx)=J_0(2\sqrt{bx})=e^{-\hat{c}bx}\varphi_{0}$,  we can write 

\begin{equation}\label{key}
I_e(a,b)=\int_{-\infty }^{\infty}e^{-\hat{c}\;b\;x}e^{-ax^2}dx\;\varphi_{0},\quad \forall b\in\mathbb{R},\forall a\in\mathbb{R}^+,
\end{equation}
the subscript $e$ stands for exponential. Since (see \ref{GWi})

\begin{equation}\label{key}
I(a,b)=\int_{-\infty }^{\infty}e^{bx}e^{-ax^2}dx=\sqrt{\dfrac{\pi}{a}}e^{\frac{b^2}{4a}},\quad \forall b\in\mathbb{R},\forall a\in\mathbb{R}^+,
\end{equation}
we “invoke” the previous principle and assume that the same relation holds under the correspondence

\begin{equation}\label{key}
I_T(a,b)=I_e(a,b)=I(a,-b\hat{c})\varphi_{0}=\sqrt{\dfrac{\pi}{a}}W_0^{(2)}\left( \frac{b^2}{4a}\right),\quad \forall b\in\mathbb{R},\forall a\in\mathbb{R}^+. 
\end{equation}
\end{exmp}

In the following part of this Chapter, we will provide further examples of the importance of the properties of Gaussian (and non Gaussian as well) umbral forms. In the next section we see how the method, we have so far envisaged, is tailor suited to treat problems in the theory of fractional derivatives.

\section{Mittag-Leffler Function: an Umbral Point of View}

In this section we explore the consequence of the umbral restyling of the \textit{\textbf{Mittag-Leffler}} (\textit{ML}) function    \cite{Mittag-Leffler}-\cite{MLWolfram} 

\begin{equation}\label{ML}
E_{\alpha,\beta}(x)=\sum_{r=0}^{\infty}\dfrac{x^{r}}{\Gamma(\alpha r+\beta)},\quad\forall x\in\mathbb{R}, \forall\alpha,\beta\in\mathbb{R}^+,
\end{equation}
which has become a pivotal tool of the \textit{fractional calculus} \cite{Oldham}, namely of the branch of calculus employing derivatives or integrals of fractional order as further discussed later in this Chapter.\\

\begin{figure}[h]
	\begin{subfigure}[c]{0.48\textwidth}
		\includegraphics[width=1.05\linewidth]{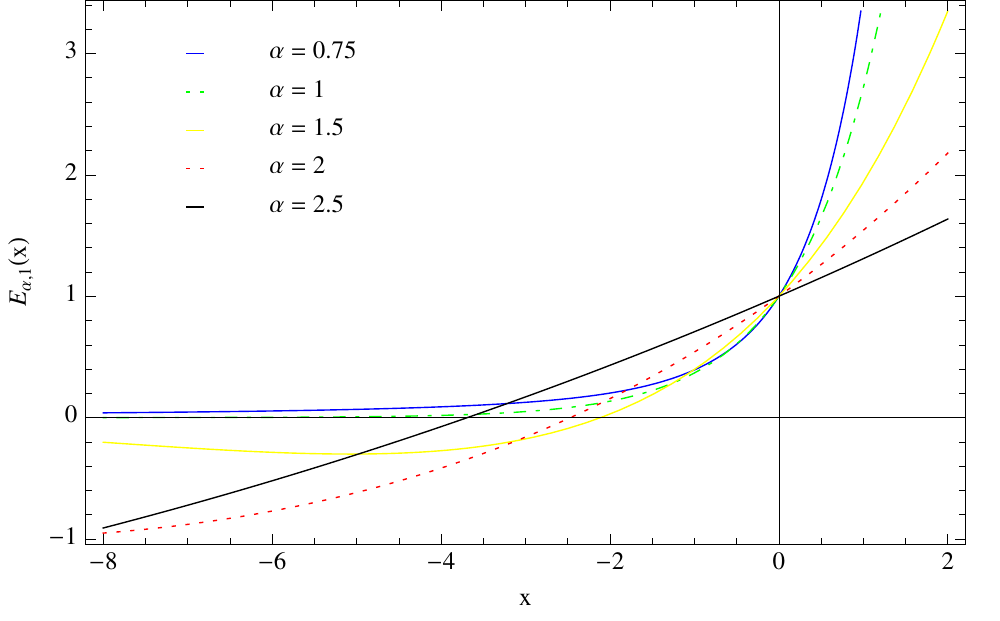}
		\caption{$\beta=1$ and different $\alpha$ values. }
	\end{subfigure}
	\begin{subfigure}[c]{0.48\textwidth}
		\includegraphics[width=1.05\linewidth]{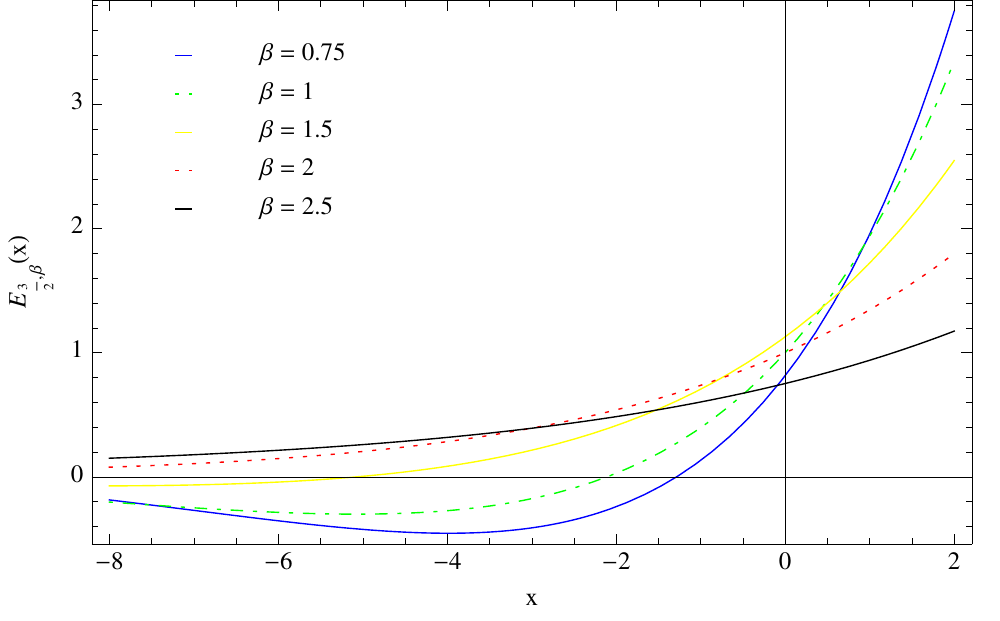}
		\caption{$\alpha=\frac{3}{2}$ and different $\beta$ values. }
	\end{subfigure}	
	\caption{Mittag Leffler Functions $E_{\alpha,\beta}(x)$ for different $\alpha$ and $\beta$ values. }
	\label{figml}
\end{figure}

According to the assumptions of the previous sections, we can cast the \textit{ML} function in the following form \cite{ML} .

\begin{prop}
	$\forall x\in\mathbb{R}, \forall\alpha,\beta\in\mathbb{R}^+$
\begin{equation}\label{EconAlfaBeta}
E_{\alpha,\beta}(x)=\hat{c}^{\beta-1}\dfrac{1}{1-\hat{c}^{\alpha}x}\varphi_{0}.
\end{equation}
\end{prop}

\begin{proof}[\textbf{Proof}]
We have, $\forall x\in\mathbb{R}, \forall\alpha,\beta\in\mathbb{R}^+$, (see \ref{Opc})

\begin{equation}\begin{split}\label{EabC}
E_{\alpha,\beta}(x) &=\sum_{r=0}^{\infty}\dfrac{x^{r}}{\Gamma(\alpha r+\beta)}=
\sum_{r=0}^{\infty}\hat{c}^{\alpha r+\beta-1}x^{r}\varphi_{0}=
\hat{c}^{\beta-1}\sum_{r=0}^{\infty}\left( \hat{c}^{\alpha}x\right)^{r} \varphi_{0}=\\
& = \hat{c}^{\beta-1}\dfrac{1}{1-\hat{c}^{\alpha}x}\varphi_{0}.
\end{split}\end{equation}
\end{proof}
We have formally reduced the trascendental function \ref{ML} to a rational form. \\
In deriving the previous results, we have not paid any attention to the radius of convergence of the series in \ref{EabC} since the expansion holds only in \textit{formal sense}, being it an operator expansion. The convergence must be checked for the final function obtained via the action of the umbral operator on the vacuum and will be defined in terms of the variable $x$ only.\\

By the same procedure, namely by treating $\hat{c}$ as an ordinary constant, we can recast the \textit{ML} function in terms of an integral representation. We write indeed \cite{ML}

\begin{cor}
	$\quad\forall x\in\mathbb{R}, \forall\alpha,\beta\in\mathbb{R}^+$, by the use of the \textit{\textbf{Laplace Transform}} identity
	
	\begin{equation}\label{L-T}
	\dfrac{1}{A}=\int_{0}^{\infty}e^{-sA}ds,
	\end{equation}
	which holds independently of the nature of $A$ (be it a number or an operator), we get
\begin{equation}\label{intEab}
E_{\alpha,\beta}(x)=\hat{c}^{\beta-1}\dfrac{1}{1-\hat{c}^{\alpha}x}\varphi_{0}=\hat{c}^{\;\beta-1}\int_{0}^{\infty}e^{-s}e^{\hat{c}^{\alpha}\;x\;s\;}ds\;\varphi_{0}.
\end{equation}
\end{cor}

\begin{cor}
According to eq.
\ref{BWamu} and to the previous discussion, we recognize that $\forall x\in\mathbb{R},\forall \alpha,\beta\in\mathbb{R}^+$

\begin{equation}
\hat{c}^{\;\beta-1}e^{\hat{c}^{\alpha}\;x}\varphi_{0}=W_{\beta-1}^{\alpha}(x)=\sum_{r=0}^{\infty}\dfrac{x^r}{r!\;\Gamma(\alpha\;r+\beta)},
\end{equation}
therefore we end up with (see \ref{intEab})

\begin{equation}\label{EW}
E_{\alpha,\beta}(x)=\int_{0}^{\infty}e^{-s}W_{\beta-1}^{\alpha}(xs)ds,
\end{equation}
which states that the \textit{ML} is the Borel transform of the \textit{BW} function (see eq. \ref{fBxInt}) \cite{ML}.
\end{cor}

In order to provide a further flavour of the flexibility of the method we are proposing, we consider the problem of evaluating the following integral \cite{ML}.

\begin{exmp}
	$\forall \alpha,\beta\in\mathbb{R}^+$
\begin{equation}\label{IntE}
I_{\alpha,\;\beta}=\int_{-\infty}^{\infty}E_{\alpha,\beta}(-x^{2})dx,
\end{equation}
which can be easily computed provided that, in the integration process, we treat as ordinary constants the operators appearing in it. We find therefore, using eq. \ref{EconAlfaBeta},

\begin{equation}
I_{\alpha,\;\beta}=\hat{c}^{\beta-1}\int_{-\infty}^{\infty}\dfrac{1}{1+\hat{c}^{\;\alpha}x^{2}}dx\;\varphi_{0},
\end{equation}
and, exploiting the integral result \ref{intPia} and the rule \ref{Opc}, we obtain

\begin{equation}\label{Iab1}
i)\; I_{\alpha,\;\beta}= \hat{c}^{\beta-1}\dfrac{\pi}{\sqrt{\hat{c}^{\alpha}}} \varphi_{0} =\pi\; \hat{c}^{\beta-\frac{\alpha}{2}-1}\varphi_{0}
 =\dfrac{\pi}{\Gamma\left(\beta-\frac{\alpha}{2}\right) }
\end{equation}
or, by using the integral representation in terms of \textit{BW} function (see Appendix \ref{AppAML}, proof \ref{solMLBW} ) we end up with the same result, namely

\begin{equation}\begin{split}\label{MLBW}
& ii)\;I_{\alpha,\;\beta}=\left( \sqrt{\pi}\hat{c}^{\beta-\frac{\alpha}{2}-1}\int_{0}^{\infty}e^{-s}s^{-\frac{1}{2}}ds\right) \varphi_{0}=\sqrt{\pi}\;\Gamma\left( \dfrac{1}{2} \right)\hat{c}^{\beta-\frac{\alpha}{2}-1}\varphi_{0}=\\
& =\dfrac{\pi}{\Gamma\left(\beta-\frac{\alpha}{2}\right) }. 
\end{split}\end{equation}
\end{exmp}
The \textit{exponential umbral image} of the \textit{ML} can be realized by the use of the following representation.

\begin{defn}
We introduce, $\forall \alpha,\beta\in\mathbb{R}^+$, the umbral vacuum
\begin{equation}\label{dvac}
\psi_\kappa:=\dfrac{\Gamma(\kappa+1)}{\Gamma(\alpha\kappa+\beta)},\quad \forall \kappa\in\mathbb{R}.
\end{equation}	
\end{defn}	

\begin{defn}\label{defndop}
We define the shift operator ${}_{\alpha,\beta}\hat{d}$, $\forall \alpha,\beta\in\mathbb{R}^+$, such that, $\forall \kappa\in\mathbb{R}$, by using the same procedure of Theorem \ref{thOpc}, we get

\begin{equation}\label{Opd}
{}_{\alpha,\;\beta}\hat{d}^{\;\kappa}\psi_0=\dfrac{\Gamma(\kappa+1)}{\Gamma(\alpha\kappa+\beta)}.
\end{equation}	
\end{defn}
\noindent Then we obtain

\begin{prop}
$\forall \alpha,\beta\in\mathbb{R}^+, \forall x\in\mathbb{R}$, the exponential umbral image of the \textit{ML}-function can be realized by
	
\begin{equation}\label{dab}
 E_{\alpha,\beta}(x)=e^{\;{}_{\alpha,\beta}\hat{d}\;x}\psi_0.
 \end{equation}
 \end{prop}
\begin{proof}[\textbf{Proof}]
	$\forall \alpha,\beta\in\mathbb{R}^+, \forall x\in\mathbb{R}$, using known results of geometrical series, the operator definition \ref{Opd} and the $\Gamma$-function property \ref{Gpropa}, we obtain
\begin{equation}\label{soldab}
 e^{{}_{\alpha,\beta}\hat{d}x}\psi_0=\sum_{r=0}^{\infty}\dfrac{x^r \left( {}_{\alpha,\beta}\hat{d}\right) ^r}{r!}\psi_0 = \sum_{r=0}^{\infty}\dfrac{x^r}{\Gamma(\alpha r+\beta)}=E_{\alpha,\beta}(x).
\end{equation}	
\end{proof}	
\noindent Now, we can obtain the same previous result (\ref{Iab1}-\ref{MLBW}) exploiting ${}_{\alpha,\beta}\hat{d}$-operator, namely

\begin{exmp}
	It is enough to use eqs. \ref{dab}-\ref{GWi}-\ref{Opd}-\ref{Gpropb} to get $\forall \alpha,\beta\in\mathbb{R}^+$
 \begin{equation}\begin{split}
 I_{\alpha,\beta}&=\int_{-\infty }^{\infty }E_{\alpha,\beta}(-x^2)dx=\int_{-\infty }^{\infty }e^{-{}_{\alpha,\beta}\hat{d}\;x^2}dx\;\psi_0 =
\sqrt{\pi}\left( {}_{\alpha,\;\beta}\hat{d}\right) ^{-\frac{1}{2}}\psi_0=\\
& =\frac{\pi }{ \Gamma \left(\beta -\frac{\alpha }{2} \right)}. 
\end{split}\end{equation}
\end{exmp}
Then, as already noted, there is no reason to privilege exponential or the rational image function, which are easily shown to be equivalent for the derivation of results of practical interest. \\

After these remarks we can appreciate the importance of \textit{ML} function in the theory of \textit{\textbf{fractional calculus}} we are going to introduce.

\subsection{The Properties of Mittag-Leffler and Fractional Calculus}\label{prML}

The fractional calculus, namely the formalism relevant to the use of derivative operators raised to a fractional exponent, will be further discussed in the forthcoming chapters of the thesis. Here we provide some introductory tools involving the use of \textit{ML} type function and, to this aim, we note that $E_{n,1}(\lambda x^{n})$ is an eigenfunction of the $\partial_{x}^{n}$ operator ($\forall x,\lambda\in\mathbb{R},\;\forall n\in\mathbb{N}$), therefore \cite{ML}

\begin{lem}\label{soleig}
\begin{equation}
\partial_{x}^{n}\;E_{n,1}(\lambda x^{n})=\lambda E_{n,1}(\lambda x^{n}),\quad\forall n\in\mathbb{N}, \forall x,\lambda\in\mathbb{R} .
\end{equation}
\end{lem}
(see the Proof in Appendix \ref{AppAML}).\\

An analogous identity can be extended also to the case of real order \textit{ML} functions. In this case \textbf{derivatives of non-integer order} should be considered.  

\begin{cor}
Using the \textbf{\textit{Euler-Riemann-Liouville}} definition \cite{Oldham} for real order derivative \cite{ML}

\begin{equation}\label{ERL}
 \partial_x ^{\alpha}x^{\nu}=\dfrac{\Gamma(\nu+1)}{\Gamma(\nu-\alpha+1)}x^{\nu-\alpha},\quad  \forall x,\alpha,\nu\in\mathbb{R},
 \end{equation}
 we find 

\begin{equation} \label{eq16} 
\partial_{x}^{\alpha }\; E_{\alpha,1 } (\lambda \, x^{\alpha } )=
\lambda \, E_{\alpha,1 } (\lambda x^{\alpha } ) +\dfrac{x^{-\alpha}}{\Gamma(1-\alpha)}\footnote{The extra-term emerges because, according to eq. \ref{ERL}, the fractional derivative of a constant does not vanish.},\quad  \forall x,\lambda\in\mathbb{R},\forall \alpha\in\mathbb{R}^+. 
\end{equation}
 \end{cor}
\noindent (See the proof in Appendix \ref{AppAML} - eq. \ref{soleq16})\\

The  \textit{ML} function, $E_{\alpha,1}(\lambda x^{\alpha})$, is an eigenfunction of the $\partial_{x}^{\alpha}$ operator $\forall \alpha\in\mathbb{R}^+$, this result can be used for various kind of applications in different fields, as e.g. for the solution of the following \textit{fractional evolution problem} \cite{FFP}.

\begin{exmp}
	The following \textbf{fractional evolution problem}, $\forall x\in\mathbb{R},\forall \alpha\in\mathbb{R}^+, \forall t\in\mathbb{R}^+_0$,
\begin{equation}\label{cauchyML} 
\left\lbrace  \begin{array}{l} \partial _{t}^{\alpha } F(x,t)=\partial _{x}^2\; F(x,t)+\dfrac{t^{-\alpha}}{\Gamma(1-\alpha)}f(x), \\[1.6ex]
F(x,0)=f(x) , \end{array}\right.
\end{equation} 
defines a  \textit{\textbf{time-fractional diffusive} equation}. According to the previous discussion, to the fact that the \textit{ML}  ''$E_{\alpha ,1} (t^{\alpha } ) $'' is an eigenfunction of the fractional derivative operator, according to the definition \ref{eq16} and considering the formalism developed so far, we can obtain the relevant solution in the form \cite{ML}-\cite{FFP}

\begin{equation} \label{eq31} 
F(x,t)=E_{\alpha ,1} (t^{\alpha } \partial _{x} ^2)\, f(x),
\end{equation} 
where, for the problem under study, we have that 
\begin{defn}\label{defPEO}
$\forall \alpha\in\mathbb{R}^+, \forall t\in\mathbb{R}^+_0$, $\mathbf{E_{\alpha ,1} (t^{\;\alpha } \partial _{x}^2 )}$ is the \textbf{pseudo-evolution operator} (PEO).
\end{defn}

 The relevant action on the initial function can be espressed as \cite{ML}
 
\begin{equation}\label{Fxt}
F(x,t)=\dfrac{1}{\sqrt{2\, \pi } } \int _{-\infty }^{+\infty }E_{\alpha ,1}(-t^{\alpha } k^2 )\, \tilde{f}(k)\, e^{i\, x\, k} dk ,
\end{equation}
where $\tilde{f}(k)$ is the Fourier transform of $f(x)$\footnote{We observe that eq. \ref{Fxt} can be recast in terms of Levy distribution according to ref. \cite{FFP}.} \cite{Sneddon}.\\

\begin{proof}[\textbf{Proof}]
	$\forall x\in\mathbb{R},\forall \alpha\in\mathbb{R}^+, \forall t\in\mathbb{R}^+_0$, the action of the PEO $E_{\alpha ,1} (t^{\;\alpha } \partial _{x}^2 ) $,  on the initial function $f(x)$, is easily obtained by defining $f(x)$ through the Fourier transform and anti-transform,
	
	\begin{equation}\begin{split}\label{Fou}
	& \tilde{f}(k)=\dfrac{1}{\sqrt{2\pi}}\int_{-\infty}^{\infty}f(x)e^{-i k x}dx,\\
	& f(x)=\dfrac{1}{\sqrt{2\pi}}\int_{-\infty}^{\infty}\tilde{f}(k)e^{i x k }dk,
	\end{split}\end{equation} 
	therefore, using the eqs. \ref{Fou} and the theorem of series integration, we have
	
	\begin{equation*}
	E_{\alpha,1}(t^{\alpha}\;\partial_{x}^{2})f(x)=\dfrac{1}{\sqrt{2\pi}}\int_{-\infty}^{\infty}E_{\alpha,1}(t^{\alpha}\;\partial_{x}^{2})\tilde{f}(k)e^{i x k }dk,
	\end{equation*}
	then, using \textit{ML} definition \ref{ML}, we obtain 
	
	\begin{equation*}\begin{split}\label{key}
	& F(x,t)=E_{\alpha ,1} (t^{\alpha } \partial _{x} ^2)\, f(x)=\dfrac{1}{\sqrt{2\pi}}\int_{-\infty}^{\infty}E_{\alpha,1}(t^{\alpha}\;\partial_{x}^{2})\tilde{f}(k)e^{i x k }dk=\dfrac{1}{\sqrt{2\pi}}\cdot\\
	& \cdot\int_{-\infty}^{\infty} \sum_{r=0}^\infty \dfrac{t^{\alpha r} \partial_{x}^{2 r}}{\Gamma(\alpha r+1)}\tilde{f}(k)e^{i x k }dk=
	\dfrac{1}{\sqrt{2\pi}}\int_{-\infty}^{\infty} \sum_{r=0}^\infty \dfrac{t^{\alpha r} }{\Gamma(\alpha r+1)}(ik)^{2r}\tilde{f}(k)e^{i x k }dk=\\
	& =	\dfrac{1}{\sqrt{2\pi}}\int_{-\infty}^{\infty} \sum_{r=0}^\infty \dfrac{t^{\alpha r}(-k^2)^r }{\Gamma(\alpha r+1)}\tilde{f}(k)e^{i x k }dk, 	
	\end{split}\end{equation*}
	thus finally getting
	
	\begin{equation*}\label{key}
	F(x,t)=\dfrac{1}{\sqrt{2\pi}}\int_{-\infty}^{\infty}E_{\alpha,1}(-t^{\alpha}\;k^{2})\tilde{f}(k)e^{i x k }dk.
	\end{equation*}
	\end{proof}
\end{exmp}

\noindent Examples of solutions \ref{Fxt} are reported in Figs. \ref{Fig1}, at different times for different $\alpha$ values, which clearly display a behaviour which is not simply diffusive but also \textit{anomalous}.\\

\begin{figure}[htp]
	\centering
	\begin{subfigure}[c]{0.48\textwidth}
		\includegraphics[width=0.9\linewidth]{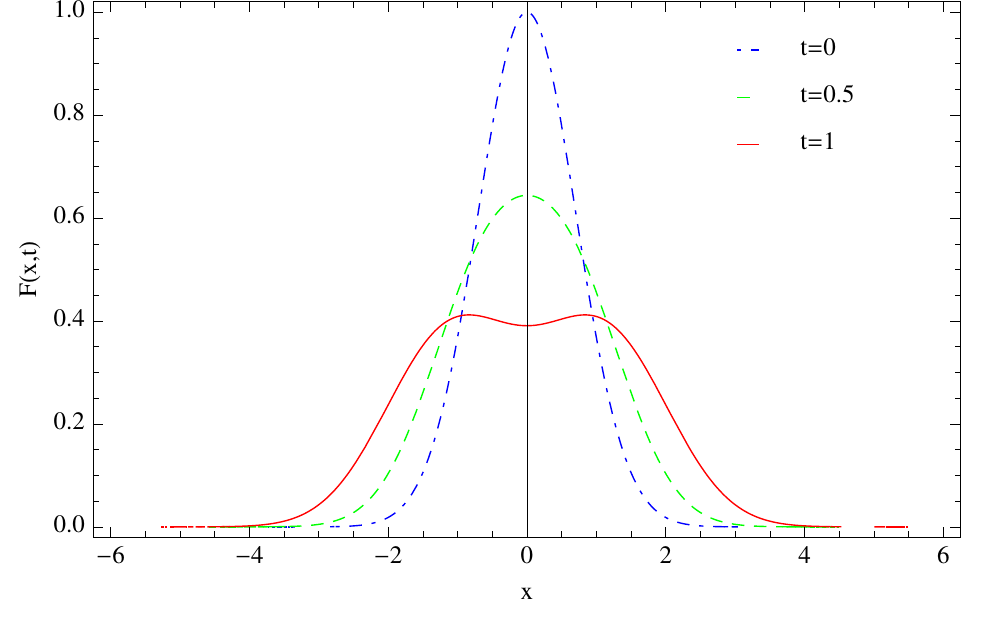}
		\caption{$\alpha=1.5$.}
		\label{Fig1a}
	\end{subfigure}
	\begin{subfigure}[c]{0.48\textwidth}
		\includegraphics[width=0.9\linewidth]{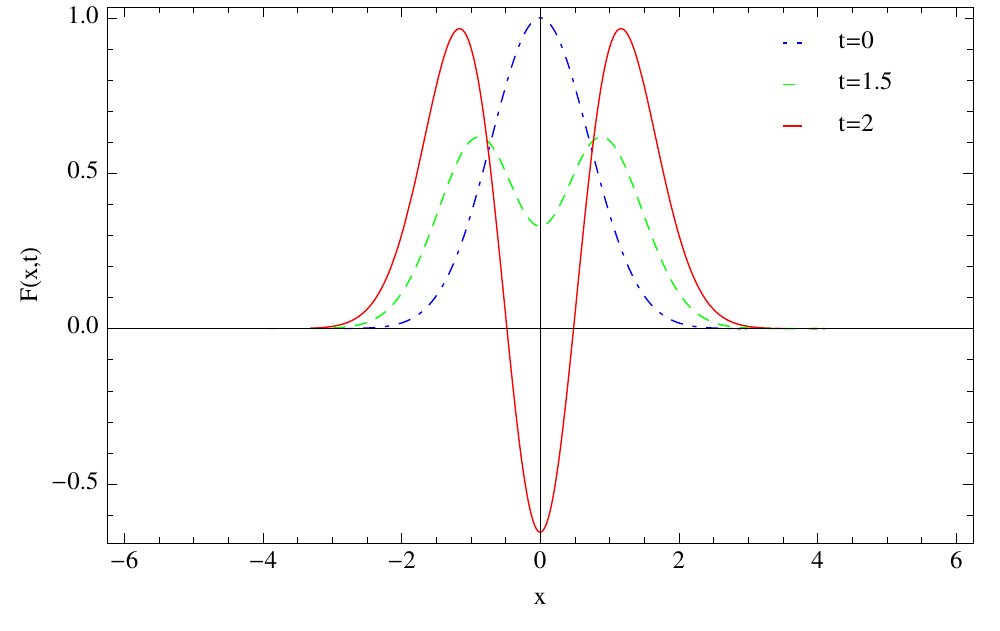}
		\caption{$\alpha=3.5$.}
		\label{Fig1b}
	\end{subfigure}
	\\[3mm]
	\caption{Solution $F(x,t)$ of eq. \ref{cauchyML} for $f(x)=e^{-x^{2}}\rightarrow\tilde{f}(k)=\dfrac{e^{-\frac{k^2}{4}}}{\sqrt{2}}$, at different times for different $\alpha$ values.}\label{Fig1} 
\end{figure}

As already stressed, eq. \ref{cauchyML} is a fractional diffusive equation. This type of equations have played an increasingly important role in the description of processes called \textit{super} or \textit{sub-diffusive}, regarding the evolution of distributions
whose mean square exhibits a dependence on time provided by a power law (faster than linear for the super diffusive case and vice-versa for the sub-diffusive counterpart).\\ 

To better appreciate how these effects emerge from the previous formalism, we consider the \textbf{\textit{ordinary heat diffusion equation}}.

\begin{exmp}
	 $\forall x\in\mathbb{R}, \forall t\in\mathbb{R}^+_0$, let
	\begin{equation}\label{key}
		\left\lbrace \begin{array}{l}
	\partial_t F(x,t)=\partial_x^2 F(x,t) \\[1.1ex] F(x,0)=f(x)
	\end{array}\right. 
	\end{equation}
	the ordinary heat diffusion equation, whose solution can be expressed in terms of the Gauss-Weierstrass transform \ref{GWi} (a direct consequence of the Fourier transform method,  \ref{Fou}), namely \cite{Babusci} 

\begin{equation}\label{Ft2}
F(x,t)=\dfrac{1}{2\sqrt{ \pi t } } \int _{-\infty }^{+\infty }e^{-\frac{(\xi-x)^2}{4t}} f(\xi)\;d\xi,
\end{equation}
where the distribution $f(x)$  is assumed to be normalized to unit with momenta 

\begin{equation}\label{momenta}
m_n (0)=\int _{-\infty }^{+\infty }x^n f(x)dx, \quad \forall n\in\mathbb{N}.
\end{equation}
The moments associated with the distribution $F(x,t)$ are therefore specified by 

\begin{equation}\begin{split}
& m_n (t)=\int _{-\infty }^{+\infty }x^n F(x,t)dx=
 \int _{-\infty }^{+\infty }e^{-\frac{\xi^2}{4t}} f(\xi) I_n (\xi)\;d\xi,\\
& I_n (\xi)=\dfrac{1}{2\sqrt{ \pi\;t } }\int _{-\infty }^{+\infty }x^n
e^{-\frac{x^2}{4t}}e^{\frac{x\xi}{2t}}dx.
\end{split}\end{equation}
The integral $ I_n (\xi)$ can be evaluated using the generating function method \cite{Babusci}, theorem of the series integration and the \textit{GWI} \ref{GWi}, in fact

\begin{equation}\begin{split}
 \sum_{n=0}^{\infty}\dfrac{u^n}{n!} I_n (\xi)&=
\dfrac{1}{2\sqrt{ \pi\;t } }\int _{-\infty }^{+\infty } \sum_{n=0}^{\infty} \dfrac{(ux)^n}{n!}e^{-\frac{x^2}{4t}}e^{\frac{x\xi}{2t}}dx=\\
& =\dfrac{1}{2\sqrt{ \pi\;t } }\int _{-\infty }^{+\infty }e^{\left( u+\frac{\xi}{2t}\right)x }e^{-\frac{x^2}{4t}}dx = e^{\frac{\xi^2}{4t}}e^{u\xi+u^2 t}
\end{split}\end{equation}
and, by the use of the generating function of two variable Hermite polynomials \ref{genfunctH},  yields 

\begin{equation}
I_n (\xi)=e^{\frac{\xi^2}{4t}}H_n (\xi,t),
\end{equation}
thus finally getting, by using eq. \ref{momenta},

\begin{equation}\begin{split}\label{Hermmu}
& m_n (t)=\int _{-\infty }^{+\infty }H_n (\xi,t)f(\xi)d\xi=H_n(\hat{m},t)\mu_0,\\
& H_n(\hat{m},t)=n!\sum_{r=0}^{\lfloor\frac{n}{2}\rfloor}\dfrac{\hat{m}^{n-2r}t^r}{(n-2r)!r!}.
\end{split}\end{equation}
In eq. \ref{Hermmu} we have assumed that $\hat{m}$ is a kind of umbral operator acting on the vacuum $\mu_0$ and defining the momenta as

\begin{equation}
\hat{m}^n \mu_0=m_n (0).
\end{equation}
In conclusion, we find

\begin{equation}
m_n(t)=n!\sum_{r=0}^{\lfloor\frac{n}{2}\rfloor}\dfrac{m_{n-2r}(0)t^r}{(n-2r)!r!},
\end{equation}
where $m_2 (t)$ shows a linear dependence on time.
\end{exmp}

The formalism we have developed so far yields the possibility of evaluating the momenta associated with the distribution \ref{eq31}-\ref{Fxt} by the use of the following substitution 

\begin{equation}
H_n(\hat{m},t)\rightarrow H_n\left(\hat{m}, {}_{\alpha,1}\hat{d}\;t^{\alpha} \right) .
\end{equation}
The second momentum is, in this case, non-linear and, reminding eqs. \ref{Opd}-\ref{dab}, we get 

\begin{defn}
 $\forall x,y\in\mathbb{R}, \forall \alpha\in\mathbb{R}^+ $, the family of polynomials

\begin{equation}\begin{split}\label{HdOp}
{}_\alpha H_n(x,y):&= H_n\left(x, {}_{\alpha,1}\hat{d}\;y \right)\psi_0 =
n!\sum_{r=0}^{\lfloor\frac{n}{2}\rfloor}\dfrac{x^{n-2r}\left( {}_{\alpha,1}\hat{d}\;y\right) ^r}{(n-2r)!r!}\psi_0=\\
& =
n!\sum_{r=0}^{\lfloor\frac{n}{2}\rfloor}\dfrac{x^{n-2r}y ^r}{(n-2r)!\Gamma(\alpha r+1)},
\end{split}\end{equation}
is called \textbf{Mittag-Leffler-Hermite} (\textit{MLH}).
\end{defn}
Its properties will be discussed later (see \ref{MLHPCC}).\\

The introduction of the  \textit{PEO} (Definition \ref{defPEO}) $E_{\alpha ,1} \left( t^{\alpha } \partial _{x}^2 \right) $,  is of central importance for our forthcoming discussion,  its role and underlying computational rules will be therefore carefully explored in the following.\\

In order to provide further elements allowing to appreciate the flexibility of the procedure employing the umbral methods, we consider the \textit{\textbf{fractional Poisson distribution}} (\textit{FPD}), discussed in ref. \cite{ML}, within the context of non Markovian stochastic processes with a non-exponential distribution of inter-arrival times.\\

 \begin{exmp}
 Without entering into the phenomenology of the fractional Poisson processes, we note that the equation governing the generating function of the distribution itself is given, $\forall\alpha\in\mathbb{R}^+,\forall s\in\mathbb{R},\forall t\in\mathbb{R}^+_0$, by \cite{Laskin}

\begin{equation}\label{GasE}
G_{\alpha}(s,t)=E_{\alpha,1}\left(-(1-s)\;\Omega \;t^{\alpha} \right),
\end{equation}
$\Omega$ has physical dimension $\left[ \Omega\right]=\frac{1}{T^\alpha}$, where $T$ is the time.\\
By the use of the umbral notation \ref{dab} we can expand the previous generating function as

\begin{equation}
G_{\alpha}(s,t)=\sum_{m=0}^{\infty}s^{m}{}_{\alpha}P(m,t),
\end{equation}
where 

\begin{equation}\label{PoiPr}
{}_{\alpha}P(m,t)=\dfrac{(\Omega t^{\alpha})^{m}}{m!}\sum_{n=0}^{\infty}\dfrac{(n+m)!}{\Gamma(\alpha(n+m)+1)}\dfrac{(-\Omega t^{\alpha})^{n}}{n!}
\end{equation}
is the FPD, introduced in ref. \cite{Zolotarev}.\\

\begin{proof}[\textbf{Proof}]
	 We use the framework of the umbral formalism and the eqs.  \ref{GasE}-\ref{dab}-\ref{propertCa}- exponential series expansion-\ref{Opd}.
	 
	\begin{equation*}\begin{split}
	& G_{\alpha}(s,t)=E_{\alpha,1}\left(-(1-s)\;\Omega \;t^{\alpha} \right)=e^{{}_{\alpha,1}\hat{d}(-(1-s)\Omega t^\alpha)}\psi_0=\\
   &	=e^{\;_{\alpha,1}\hat{d}\;s\;(\Omega\; t^{\alpha})}e^{-\;_{\alpha,1}\hat{d}\;(\Omega\; t^{\alpha})}\psi_0=
   \sum_{m=0}^{\infty}s^{m}\dfrac{_{\alpha}\hat{d}^{\;m}}{m!}(\Omega t^{\alpha})^{m}
	\sum_{n=0}^{\infty}\dfrac{_{\alpha}\hat{d}^{\;n}}{n!}(-\Omega t^{\alpha})^{n}\psi_0=\\
	& =\sum_{m=0}^{\infty}s^{m} \dfrac{(\Omega t^{\alpha})^{m}}{m!}\sum_{n=0}^{\infty}\dfrac{(-\Omega t^{\alpha})^{n}}{n!}{}_{\alpha,1}\hat{d}^{\;m+n}\psi_0=\\
	& =\sum_{m=0}^{\infty}s^{m} \dfrac{(\Omega t^{\alpha})^{m}}{m!}\sum_{n=0}^{\infty}\dfrac{(-\Omega t^{\alpha})^{n}}{n!}\dfrac{(n+m)!}{\Gamma(\alpha(n+m)+1)}=\\
	& =\sum_{m=0}^{\infty}s^{m}{}_{\alpha}P(m,t),
\end{split}\end{equation*}	
\end{proof}
	
According to the methods we have envisaged, to calculate average and r.m.s. values we use the following

\begin{cor}
By setting, $\forall \alpha,\Omega\in\mathbb{R}^+, \forall t\in\mathbb{R}^+_0, \forall m\in\mathbb{N}$,
\begin{equation}\label{PoiPrOp}
{}_{\alpha}P(m,t)= \dfrac{\left( _{\alpha,1}\hat{d}\;\Omega\;t^{\alpha}\right)^m }{m!}e^{-_{\alpha,1}\hat{d}\;\Omega\;t^{\alpha}}\psi_0,
\end{equation}
we find, for the first order moment, 

\begin{equation}\label{mFPD}
i)\; \langle\;{}_\alpha m_t\;\rangle= \dfrac{\left(\Omega\;t^{\alpha} \right) }{\Gamma(\alpha+1)}
\end{equation}
and, for the variance,

\begin{equation}\label{sFPD}
ii)\; {}_{\alpha}\sigma^{2}_t =\dfrac{2 \left(\Omega\;t^{\alpha} \right)^2}{\Gamma(2 \alpha+1)}+\dfrac{\left(\Omega\;t^{\alpha} \right)}{\Gamma(\alpha+1)}-\dfrac{\left(\Omega\;t^{\alpha} \right)^2}{\left(\Gamma(\alpha+1) \right)^2 },
\end{equation}
in agreement with the results obtained in refs. \cite{Laskin,Zolotarev}.
\end{cor}

\begin{proof}[\textbf{Proof}]
$\forall \alpha,\Omega\in\mathbb{R}^+, \forall t\in\mathbb{R}^+_0, \forall m\in\mathbb{N}$, by using eqs. \ref{dab}- \ref{Opd} and algebraic manipulation, we obtain
\begin{equation*}\begin{split}
i)\; &\langle\;{}_\alpha m_t\;\rangle= \sum_{m=0}^{\infty}m \dfrac{\left( _{\alpha}\hat{d}\;\Omega\;t^{\alpha}\right)^m }{m!}e^{-_{\alpha}\hat{d}\;\Omega\;t^{\alpha}}\psi_0 =\\
& =
\sum_{m=1}^{\infty}\left( _{\alpha}\hat{d}\;\Omega\;t^{\alpha}\right)\dfrac{\left( _{\alpha}\hat{d}\;\Omega\;t^{\alpha}\right)^{m-1} }{(m-1)!}e^{-_{\alpha}\hat{d}\;\Omega\;t^{\alpha}}\psi_0 =\\
& =\left( _{\alpha}\hat{d}\;\Omega\;t^{\alpha}\right)e^{_{\alpha}\hat{d}\;\Omega\;t^{\alpha}}e^{-_{\alpha}\hat{d}\;\Omega\;t^{\alpha}}\psi_0=\\
&=e^{_{\alpha}\hat{d}\;\Omega\;t^{\alpha}}\psi_0=
\dfrac{\left(\Omega\;t^{\alpha} \right) }{\Gamma(\alpha+1)}.
\end{split}\end{equation*}

\begin{equation*}\begin{split}
ii)\; &{}_\alpha\sigma_t^{2}= \sum_{m=0}^{\infty}m^2 {}_{\alpha}P(m,t)-\left( \sum_{m=0}^{\infty}m\; {}_{\alpha}P(m,t)\right)^2 = \\
& = \sum_{m=1}^{\infty}m \dfrac{\left( _{\alpha}\hat{d}\;\Omega\;t^{\alpha}\right)^m }{(m-1)!}e^{-_{\alpha}\hat{d}\;\Omega\;t^{\alpha}}\psi_0-\left( \dfrac{\left(\Omega\;t^{\alpha} \right) }{\Gamma(\alpha+1)}  \right)^2= \\
& =\sum_{m=1}^{\infty} \dfrac{(m-1+1)\left( _{\alpha}\hat{d}\;\Omega\;t^{\alpha}\right)^m }{(m-1)!}e^{-_{\alpha}\hat{d}\;\Omega\;t^{\alpha}}\psi_0-\left( \dfrac{\left(\Omega\;t^{\alpha} \right) }{\Gamma(\alpha+1)}  \right)^2=\\
& = \left[   \left( _{\alpha}\hat{d}\;\Omega\;t^{\alpha}\right)^2 \sum_{m=2}^{\infty} \dfrac{\left( _{\alpha}\hat{d}\;\Omega\;t^{\alpha}\right)^{m-2} }{(m-2)!}+\left( _{\alpha}\hat{d}\;\Omega\;t^{\alpha}\right) \sum_{m=1}^{\infty} \dfrac{\left( _{\alpha}\hat{d}\;\Omega\;t^{\alpha}\right)^{m-1} }{(m-1)!}\right]\cdot\\
& \cdot e^{-_{\alpha}\hat{d}\;\Omega\;t^{\alpha}}\psi_0-\left( \dfrac{\left(\Omega\;t^{\alpha} \right) }{\Gamma(\alpha+1)}  \right)^2= \\
& = \left( _{\alpha}\hat{d}\;\Omega\;t^{\alpha}\right)^2 \psi_0 + \left( _{\alpha}\hat{d}\;\Omega\;t^{\alpha}\right)\psi_0 -\left( \dfrac{\left(\Omega\;t^{\alpha} \right) }{\Gamma(\alpha+1)}  \right)^2= \\ 
& =\dfrac{2 \left(\Omega\;t^{\alpha} \right)^2}{\Gamma(2 \alpha+1)}+\dfrac{\left(\Omega\;t^{\alpha} \right)}{\Gamma(\alpha+1)}-\dfrac{\left(\Omega\;t^{\alpha} \right)^2}{\left(\Gamma(\alpha+1) \right)^2 }.
\end{split}\end{equation*}	
\end{proof}	
\end{exmp}

In these introductory sections we have shown how the formalism we are going to propose and develop in this thesis is particularly useful and flexible, to frame old and new problems within a general and easily manageable context. Further comments will be provided in the concluding parts of this Chapter devoted to applications.

\section{Mittag-Leffler Hermite Polynomials}\label{MLHPCC}
\numberwithin{equation}{section}

In eq. \ref{HdOp} we have introduced a family of polynomials that we have called $MLH$.
They play, within the context of \textbf{\textit{fractional diffusion heat equation}} (\textit{FDHE}), the same role of the heat polynomials \cite{Widder} in the case of the ordinary heat equation, therefore

\begin{prop}\label{propMLH}
Let us consider the fractional evolution problem \ref{cauchyML} with initial condition $F(x,0)=x^n, \forall n\in\mathbb{N}$, the relevant solution can accordingly be written as  

\begin{equation}\label{solCHd}
{}_{\alpha}H_n (x,t^\alpha)=\left(e^{\;{}_{\alpha,1}\hat{d}\;t^{\alpha}\partial_{x}^2}x^n \right)\psi_0\;. 
\end{equation}
\end{prop}

\begin{proof}[\textbf{Proof.}]
By the use of eq. \ref{HdOp} we can write, $\forall x\in\mathbb{R}, \forall \alpha\in\mathbb{R}^+, \forall t\in\mathbb{R}^+_0, \forall n\in\mathbb{N}$,

\begin{equation*}\label{key}
{}_{\alpha}H_n (x,t^\alpha)=H_n\left(x, {}_{\alpha,1}\hat{d}\;t^\alpha \right)\psi_0 =n!\sum_{r=0}^{\lfloor\frac{n}{2}\rfloor}\dfrac{x^{n-2r}t ^{\alpha r}}{(n-2r)!\Gamma(\alpha r+1)}
\end{equation*}	
but, on the other side, expanding the exponential, acting the successive derivatives on $x^n$ and applying the umbral operator \ref{Opd} we find

\begin{equation*}\begin{split}\label{polHa}
 \left(e^{\;{}_{\alpha,1}\hat{d}\;t^{\alpha}\partial_{x}^2}x^n \right)\psi_0&=\sum_{r=0}^\infty \dfrac{{}_{\alpha,1}\hat{d}^{\;r}\; t^{\;\alpha r}\; \partial_x^{\;2r}}{r!}\;x^n\;\psi_0=\\ &=\sum_{r=0}^{\lfloor\frac{n}{2}\rfloor}\dfrac{n!x^{n-2r}t^{\alpha r} {}_{\alpha}\hat{d}^{\;r} }{(n-2r)!r!}\psi_0=n!\sum_{r=0}^{\lfloor\frac{n}{2}\rfloor}\dfrac{x^{n-2r}t^{\alpha\;r}}{(n-2r)!\Gamma(\alpha\;r+1)}.
\end{split}\end{equation*}
\end{proof}
The \textit{MLH} belong to the \textit{\textbf{App\'el polynomial}} family \cite{Appell} and are easily shown to satisfy, in its domain, the recurrences (see \cite{Babusci} where the relevant properties have been touched on and proved.)

\begin{propert}
	\begin{align}
	& i) \;\;\;\;\partial_{x}\; {}_{\alpha}H_n (x,t^{\alpha})=n\; {}_{\alpha}H_{n-1} (x,t^{\alpha}); \label{propertMLHa}\\[1.1ex]
	& ii) \;\;\partial_{t}\; {}_{\alpha}H_n (x,t^{\alpha})=n(n-1) {}_{\alpha}H_{n-2} (x,t^{\alpha})+\dfrac{t^{-\alpha}}{\Gamma(1-\alpha)}x^n\;. \label{propertMLHb}
	\end{align}
\end{propert}
The generating function can be obtained from eq. \ref{solCHd} as follows

\begin{cor}
In the hypothesis of Proposition \ref{propMLH}, by eqs. \ref{genfunctH}-\ref{Opd}-\ref{dab}, 

\begin{equation}\begin{split}\label{genfunMLH}
\sum_{n=0}^{\infty}\dfrac{\xi^n}{n!}{}_{\alpha}H_n (x,t^{\alpha})&=
\sum_{n=0}^{\infty} \dfrac{\xi^n}{n!}H_n\left(x, {}_{\alpha,1}\hat{d}\;t^\alpha \right)\psi_0 =
\left(e^{x\;\xi}\;e^{\;{}_{\alpha,1}\hat{d}\;t^{\alpha}\xi^2} \right)\psi_0=\\
& =E_{\alpha,1}\left(t^{\alpha}\xi^2 \right)e^{x\;\xi}  .
\end{split}\end{equation}
\end{cor}
In order to write eq. \ref{genfunMLH} in a more convenient form for our purposes, we note the following identity \cite{FFP,Barkai}

\begin{equation}\begin{split}\label{key}
& \int_{0}^{\infty}n_{\alpha}(s,t)\dfrac{s^n}{n!}ds=\dfrac{t^{\alpha\;n}}{\Gamma(\alpha\;n+1)}, \quad \forall n\in\mathbb{N}, \forall \alpha\in\mathbb{R}^+, \forall t\in\mathbb{R}^+_0,\\[1.1ex]
& n_{\alpha}(s,t):=\dfrac{1}{\alpha}\;\dfrac{1}{s\sqrt[\alpha]{s}}\;g_{\alpha}\left(\dfrac{t}{\sqrt[\alpha]{s}} \right) \;.
\end{split}\end{equation}
with $g_{\alpha}(x)$  being the one sided \textit{\textbf{L\'evy}} stable distribution.

\begin{exmp}
	From the previous identities we may draw at least two conclusions (see ref. \cite{FFP}):
\begin{align}
& i)\;\;\;{}_{\alpha}H_n (x,t^{\alpha})=\int_{0}^{\infty}n_{\alpha}(s,t)H_n (x,s)ds, \label{propertHa1}\\[1.1ex] 
& ii)\;\;\; E_{\alpha,1}(b\; t^{\alpha})=\int_{0}^{\infty}n_{\alpha}(s,t)e^{b\;s}ds. \label{propertHa2}
\end{align}

The eq. \ref{propertHa2} can be exploited to cast the \textit{PEO} (Definition \ref{defPEO}) associated to the fractional evolution problem \ref{cauchyML}  in the form

\begin{equation}\label{key}
E_{\alpha,1}\left(t^{\alpha}\partial_{x}^2 \right)= \int_{0}^{\infty}n_{\alpha}(s,t)e^{\;s\; \partial_{x}^2}ds.
\end{equation}
Being $e^{s\; \partial_{x}^2}$ the evolution operator for the ordinary diffusion problem, we can write the solution \ref{eq31} of problem \ref{cauchyML} as

\begin{equation}\label{key}
F(x,t)=\int_{0}^{\infty}n_{\alpha}(s,t)\left(e^{s \partial_{x}^2} f(x) \right)ds 
\end{equation}
and, in the case in which  $f(x)=e^{-x^2}$, by the use of eq. \ref{GWi} we obtain the \textbf{Glaisher} identity \cite{Khan}

\begin{equation}\label{Glaisher}
e^{s \partial_{x}^2}e^{-x^2}=\dfrac{1}{\sqrt{1+4s}}e^{-\frac{x^2}{1+4s}}, \quad \forall x, s\in\mathbb{R},
\end{equation}
which yields

\begin{equation}\label{key}
F(x,t)=\int_{0}^{\infty}\dfrac{n_{\alpha}(s,t)}{\sqrt{1+4s}}e^{-\frac{x^2}{1+4s}}ds.
\end{equation}
\end{exmp}

A further important conclusion which may be drawn from the previous formalism concerns the question whether the family of polynomials \textit{MLH} \ref{HdOp} can be considered orthogonal. 
The question can be settled out in a fairly simple way by assuming that an expansion of the type

\begin{equation}\label{Gexpans}
G(x,\xi)=\sum_{n=0}^{\infty}a_n\;{}_{\alpha}H_n\left( x,\xi^{\alpha}\right), \quad \forall x,\xi\in\mathbb{R}, \forall \alpha\in\mathbb{R}^+.  
\end{equation}
be allowed. According to eq. \ref{propertHa1} we find

\begin{equation}\label{key}
G(x,\xi)=\sum_{n=0}^{\infty}a_n\int_{0}^{\infty}n_{\alpha}(s,\xi)H_n(x,s)ds.
\end{equation}
We can therefore state the following Theorem.

\begin{thm}
	If, $\forall x,s\in\mathbb{R}$, the series $\sum_{n=0}^{\infty}a_n\;  H_n(x,s)$  is uniformly converging to a function $g(x,s)$  then, by \textit{MLH}-definition, the expansion \ref{Gexpans} does exist and the expansion coefficients are the same as for the ordinary expansion.	
\end{thm}

Further comments and use of this last result will be presented later in this thesis.

\section{Applications}
\numberwithin{equation}{section}
\markboth{\textsc{\chaptername~\thechapter. Applications}}{}
\subsection{Fractional Schr\"{o}dinger Equation,\\ Coherent States and Associated Probability Distribution }\label{ApplicSchr} 
\numberwithin{equation}{section}
\markboth{\textsc{\chaptername~\thechapter. FSE}}{}

In section \ref{prML} we have introduced the \textit{FPD} \ref{PoiPr} as a mere consequence of the definition of the \textit{ML} function. In this section we derive a different form of  \textit{FPD} by solving the \textit{\textbf{fractional Schr\"{o}dinger equation}} (\textit{FSE}) for a physical process implying the emission and absorption of photons. We assume that the relevant dynamics is ruled by the \textit{ML} - Schr\"{o}dinger equation\footnote{According to Dirac notation we write the state $\mid \Psi \rangle$ to indicate  the function $\Psi(t)$.} \cite{Naber} 

\begin{equation}\begin{split}\label{CrAnn}
& i^{\alpha}\;\partial_{t}^{\alpha}\mid \Psi \;\rangle= \hat{H}\mid \Psi \;\rangle+i^{\alpha}\dfrac{t^{-\alpha}}{\Gamma(1-\alpha)}\mid \Psi(0)\;\rangle, \\
& \hat{H}=i^{\alpha}\;\Omega\left( \hat{a}-\hat{a}^{+}\right), \quad 0\leq\alpha\leq 1, \forall t\in\mathbb{R}^+_0\,
\end{split}\end{equation}
where  $\hat{a},\hat{a}^{+}$ are annihilation, creation  operators \cite{Louiselle} satisfying the commutation relation 

\begin{equation}\label{aapiu}
\left[ \hat{a},\hat{a}^{+}\right] = \hat{1}
\end{equation}
and the constant $\Omega$ in eq. \ref{CrAnn} has the dimension of $t^{-\alpha}$.\\
If we work in a \textit{Fock} basis \cite{Louiselle} and choose the ”physical” vacuum (namely the state of the quantized electromagnetic field
with no photons) as the initial state of our process \ref{CrAnn}, namely

\begin{equation}
\mid \Psi(t)\; \rangle   \mid_{t=0}=\mid 0\;\rangle,
\end{equation}
we can understand how the field ruled by a $FSE$ evolves from the vacuum.
The comparison with the ordinary Schr\"{o}dinger counterpart is interesting, because the field evolves into a coherent state, displaying an emission process in which the photon counting statistics follows a Poisson distribution \cite{Louiselle}.\\
The formal solution of the evolution problem provided by eq. \ref{CrAnn} is formally obtained, using the eqs. \ref{dab}-\ref{eq31} \footnote{Here we omit the ${}_\alpha\hat{d}$-vacuum $\psi_0$ \ref{dvac}.} as \cite{ML}

\begin{equation}\label{eq52a}
 \mid \Psi \;\rangle=e^{\; _{\alpha}\hat{d}\;t^{\alpha}\; \Omega\;   \left( \hat{a}-\hat{a}^{+}\right)}\mid 0\;\rangle\footnote{To semplify the writing we substitute ${}_{\alpha,1}\hat{d}$ with ${}_{\alpha}\hat{d}$.},
\end{equation}
in which the evolution operator

\begin{equation}\label{Uop}
\hat{U}(t)=e^{\; {}_{\alpha}\hat{d}\;\Omega\;t^{\alpha}\left(\hat{a}-\hat{a}^+ \right)  }
\end{equation}
cannot be straightforwardly disentangled (namely written as product of exponential operators) since the creation and annihilation operators are not commuting. \\
To this aim we remind the Weyl \cite{Weyl} disentanglement rule

\begin{equation}\begin{split}\label{rulesComp}
& e^{\hat{A}+\hat{B}}=e^{\hat{A}}e^{\hat{B}}e^{-\frac{\hat{k}}{2}},\\
& \left[  \hat{A},\hat{B}\right] =\hat{A}\hat{B}-\hat{B}\hat{A}= \hat{k},\\
& \left[  \hat{A},\hat{k}\right]=\left[  \hat{k},\hat{B}\right]=0,
\end{split}\end{equation}
which holds if $\hat{A},\hat{B}$ are not commuting each other, but their commutator “commutes” either with $\hat{A}$ and $\hat{B}$.\\
According to the previous prescription we set

\begin{equation}\begin{split}
& _{\alpha}\hat{A} = {}_{\alpha}\hat{d}\;\Omega\;t^{\alpha}\hat{a},\\
& e^{{}_{\alpha}\hat{d}\;\Omega\;t^{\alpha}\left( \hat{a}-\hat{a}^+ \right) }=e^{\;_{\alpha}\hat{A}- _{\alpha}\hat{A}^+}.
\end{split}\end{equation}
We note that 

\begin{equation}\label{key}
\left[ _{\alpha}\hat{A} ,\;  _{\alpha}\hat{A}^+ \right]= {}_{\alpha}\hat{d}^{\;2} \left( \Omega\;t^{\alpha}\right)^2.
\end{equation}

\begin{proof}[\textbf{Proof.}]
By the use of the second eq. of \ref{rulesComp}, eq. \ref{aapiu} and algebraic rules we get
\begin{equation*}\begin{split}\label{key}
&\left[ _{\alpha}\hat{A} ,\;  _{\alpha}\hat{A}^+ \right]= {}_{\alpha}\hat{A} {}_{\alpha}\hat{A}^+ - {}_{\alpha}\hat{A}^+ _{\alpha}\hat{A}=
{}_{\alpha}\hat{d}\;\Omega\;t^{\alpha}\hat{a}\; {}_{\alpha}\hat{d}\;\Omega\;t^{\alpha}\hat{a}^+ -{}_{\alpha}\hat{d}\;\Omega\;t^{\alpha}\hat{a}^+\; {}_{\alpha}\hat{d}\;\Omega\;t^{\alpha}\hat{a}=\\
& =\left( {}_{\alpha}\hat{d}\;\Omega\;t^{\alpha}\right)^2 \left(\hat{a}\;\hat{a}^+ - \hat{a}^+ \hat{a} \right) =\left( {}_{\alpha}\hat{d}\;\Omega\;t^{\alpha}\right)^2 \left[\hat{a},\hat{a}^+ \right]  ={}_{\alpha}\hat{d}^{\;2} \left( \Omega\;t^{\alpha}\right)^2.
\end{split}\end{equation*}
\end{proof}	
\noindent Being the umbral operator ${}_{\alpha}\hat{d}$   independent of the creation annihilation counterparts, we end up with the following exponential disentanglement \footnote{We remind that the Weyl identity provides \begin{equation}\begin{split}\label{Weylid}
 	& \left[ \hat{A} ,\hat{B} \right]=\hat{k} \Rightarrow \left[ \hat{B} ,\hat{A} \right]=-\hat{k};\\ 
 	&  e^{\hat{A}+\hat{B}}=\left\lbrace \begin{array}{l} e^{\hat{A}}e^{\hat{B}}e^{-\frac{\hat{k}}{2}},\\[1.3ex]
 	e^{\hat{B}}e^{\hat{A}}e^{\frac{\hat{k}}{2}}.
 	 	\end{array}\right. 
\end{split}\end{equation}}
 
 \begin{equation}
 e^{\;_{\alpha}\hat{A}- _{\alpha}\hat{A}^+}=
 e^{-\frac{\left( \Omega\;t^{\alpha}\right)^2 \;{}_\alpha \hat{d}^{\;2}}{2}}\; e^{\; - _{\alpha}\hat{A}^+}e^{\; _{\alpha}\hat{A}}.
 \end{equation}
The solution of our \textit{FSE} therefore writes \cite{ML}

\begin{equation}\begin{split}\label{eq52}
& \mid \Psi \;\rangle=e^{\; _{\alpha}\hat{d}\;t^{\alpha}\; \Omega\;   \left( \hat{a}-\hat{a}^{+}\right)}\mid 0\;\rangle=\\
& = e^{- \frac{ \left( \; _{\alpha}\hat{d}\; t^{\alpha}\; \Omega\right) ^2}{2} }e^{-\left( \; _{\alpha}\hat{d}\; t^{\alpha}\; \Omega\right)\;\hat{a}^+}e^{\left( \; _{\alpha}\hat{d}\; t^{\alpha}\; \Omega\right)\; \hat{a}}\mid 0\;\rangle .
\end{split}\end{equation}
The use of the identities \cite{Louiselle}

\begin{equation}\begin{split}
& (\hat{a}^+)^n \mid 0 \;\rangle= \sqrt{n!}\mid n\;\rangle,\\[1.1ex]
& \hat{a} \mid 0\;\rangle=0,
\end{split}\end{equation}
finally yields the solution in the form \cite{ML} 

\begin{equation}\label{solSchrpsi}
\mid \Psi \;\rangle=e^{- \frac{ \left( \; _{\alpha}\hat{d}\; t^{\alpha}\; \Omega\right) ^2}{2} }e^{-\left( \; _{\alpha}\hat{d}\; t^{\alpha}\; \Omega\right)\hat{a}^+}\mid 0\;\rangle= e^{- \frac{ \left( \; _{\alpha}\hat{d}\; t^{\alpha}\; \Omega\right) ^2}{2} }\sum_{n=0}^{\infty}\dfrac{\left( -\; _{\alpha}\hat{d}\; t^{\alpha}\; \Omega\right)^n }{\sqrt{n!}}\mid n\;\rangle .
\end{equation}
The use of the orthogonality properties of the number photon states  $\mid m\;\rangle$, namely $\langle\;n\mid m\;\rangle=\delta_{n,m}$, yields probability amplitude of finding the state $\mid \Psi\;\rangle$ in a photon number state $\mid m\;\rangle$. It is just according to the identity \cite{ML}

\begin{equation}
\langle\; m\mid \Psi\;\rangle= e^{- \frac{ \left( \; _{\alpha}\hat{d}\; t^{\alpha}\; \Omega\right) ^2}{2} }\cdot\dfrac{\left( -\; _{\alpha}\hat{d}\; t^{\alpha}\; \Omega\right)^m}{\sqrt{m!}},
\end{equation}
 which is formally equivalent to a Poisson probability amplitude \ref{PoiPrOp}. \\

\noindent The probability distribution is then found as \cite{ML}

\begin{equation}\begin{split}\label{prP}
& _{\alpha}p(m,t) =\;\mid \langle m\mid \Psi\;\rangle\mid^2\;=e^{-  \left( \; _{\alpha}\hat{d}\; t^{\alpha}\; \Omega\right) ^2}\cdot\dfrac{\left(  _{\alpha}\hat{d}\; t^{\alpha}\; \Omega\right)^{2m}}{m!}=
\dfrac{X^m}{m!}e_{m}^{(\alpha,\;2)}(-X), \\[1.1ex]
& X=\left( t^{\alpha}\; \Omega\right) ^2,\\[1.1ex]
& e_{m}^{(\alpha,\;2)}(-X) =\sum_{r=0}^{\infty}\dfrac{(-1)^{r}}{r!}\dfrac{\Gamma(2(r+m)+1)}{\Gamma(2(r+m)\alpha +1)}X^r, \quad \forall X,m,\alpha\in \mathbb{R}^+_0, \alpha\leq 1 ,
\end{split}\end{equation}
which is similar, but not equivalent, to the \textit{FPD} derived in \ref{PoiPr}.

\begin{proof}[\textbf{Proof.}]
\begin{equation*}\begin{split}\label{key}
& _{\alpha}p(m,t)=e^{-  \left( \; _{\alpha}\hat{d}\; t^{\alpha}\; \Omega\right) ^2}\cdot\dfrac{\left(  _{\alpha}\hat{d}\; t^{\alpha}\; \Omega\right)^{2m}}{m!}\psi_0= 
\sum_{r=0}^\infty \dfrac{(-1)^{r}}{r!}\dfrac{\left(t^\alpha\Omega \right)^{2r+2m}\;{}_\alpha\hat{d}^{\;2r+2m} }{m!}\psi_0=\\
& =\sum_{r=0}^\infty \dfrac{(-1)^{r}}{r!}\dfrac{\left(t^\alpha\Omega \right)^{2(r+m)}}{m!}\dfrac{\Gamma(2(r+m)+1)}{\Gamma(2(r+m)\alpha+1)}
\end{split}\end{equation*}	
\end{proof}	

\noindent It is however evident that the probability \ref{prP} is properly normalized and indeed

\begin{equation}
\sum_{m=0}^{\infty} {} _{\alpha}p(m) =1 .
\end{equation}

Furthermore, regarding the evaluation of the average number of emitted photons, we proceed as in proof of \ref{mFPD}.

\begin{equation}\begin{split}\label{medPh}
\langle \;m\;\rangle &=  \sum_{m=0}^{\infty} m\; {} _{\alpha} p(m)=\sum_{m=0}^{\infty} m \dfrac{X^m}{m!}e_{m}^{(\alpha,\;2)}(-X)=\\
& =
\sum_{m=1}^{\infty}\dfrac{X^m}{(m-1)!}\sum_{r=0}^\infty \dfrac{(-1)^{r}}{r!} {}_\alpha\hat{d}^{2(r+m)}X^r\;\psi_0=\\
& =e^{-\left( {}_{\alpha}\hat{d}^{\;2} X\right) }\sum_{m=1}^{\infty}\dfrac{\left( {}_{\alpha}\hat{d}^{\;2} X\right)^m}{(m-1)!}\psi_0= {}_{\alpha}\hat{d}^{\;2} X\;\psi_0=\dfrac{2X}{\Gamma(2\alpha+1)}
\end{split}\end{equation}
and the analogous procedure \ref{sFPD} allows the evaluation of the r.m.s. of the emitted photons, namely

\begin{equation}\begin{split}\label{rms}
\sigma_{m}^{\;2}&=\langle\;m^2\;\rangle-\langle\;m\;\rangle^2\;={}_{\alpha}\hat{d}^{\;4} X^2+ {}_{\alpha}\hat{d}^{\;2} X-\left(\dfrac{2X}{\Gamma(2\alpha+1)} \right) ^2=\\
& = 2X\left[ 2X\left( \dfrac{6}{\Gamma(4\alpha+1)}-\dfrac{1}{\left( \Gamma(2\alpha+1)\right)^2 }\right)+\dfrac{1}{ \Gamma(2\alpha+1)}\right]  .
\end{split}\end{equation}
We define the \textit{\textbf{Mandel parameter}}

\begin{equation}
Q_{\alpha}=\dfrac{\sigma_{m}^{\;2}-\langle\;m\;\rangle}{\langle\;m\;\rangle}=2X\left( \dfrac{6}{\Gamma(4\alpha+1)}-\dfrac{1}{\left( \Gamma(2\alpha+1)\right)^2 }\right)\Gamma(2\alpha+1).
\end{equation}
The behaviour of $Q_\alpha$ vs. $\alpha$, for different values of $t$, is shown in Fig. \ref{Figure3}. The understanding of the relevant physical meaning requires a discussion going beyond the scope of this thesis. We note however that a process of photon emission ruled by a fractional Schr\"{o}dinger equation is fixed by a power law (recall that $X=\left( t^{\alpha}\Omega\right)^2 $ ), furthermore the presence of a region with $Q_\alpha<0$ indicates the possibility of photon bunching. These are clearly pure speculations since the physical process ruled by  eq. \ref{CrAnn} has not been defined. The photon emission probability vs. $X$ and different $\alpha$  is given in Figs. \ref{Figure4}. \\

\begin{figure}[h]
	\centering
	\includegraphics[width=0.5\linewidth]{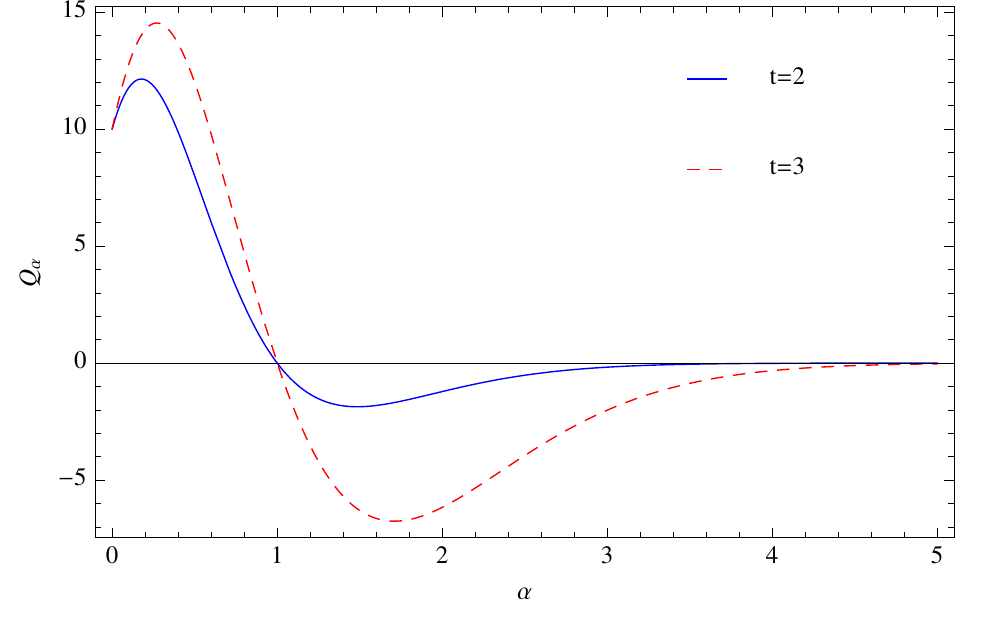}
	\caption{Mandel Parameter $Q_{\alpha}$ vs $\alpha$, for different values of $t$.}
	\label{Figure3}
\end{figure}

\begin{figure}[htp]
	\centering
	\begin{subfigure}[c]{0.48\textwidth}
		\includegraphics[width=0.9\linewidth]{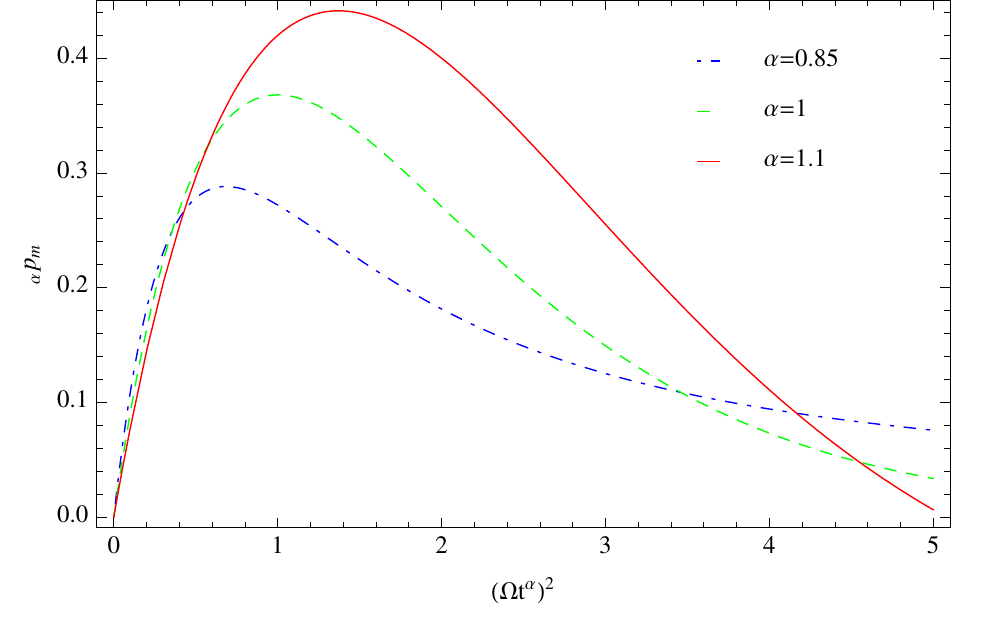}
		\caption{$m=1$.}
		\label{Fig3a}
	\end{subfigure}
	\begin{subfigure}[c]{0.48\textwidth}
		\includegraphics[width=0.9\linewidth]{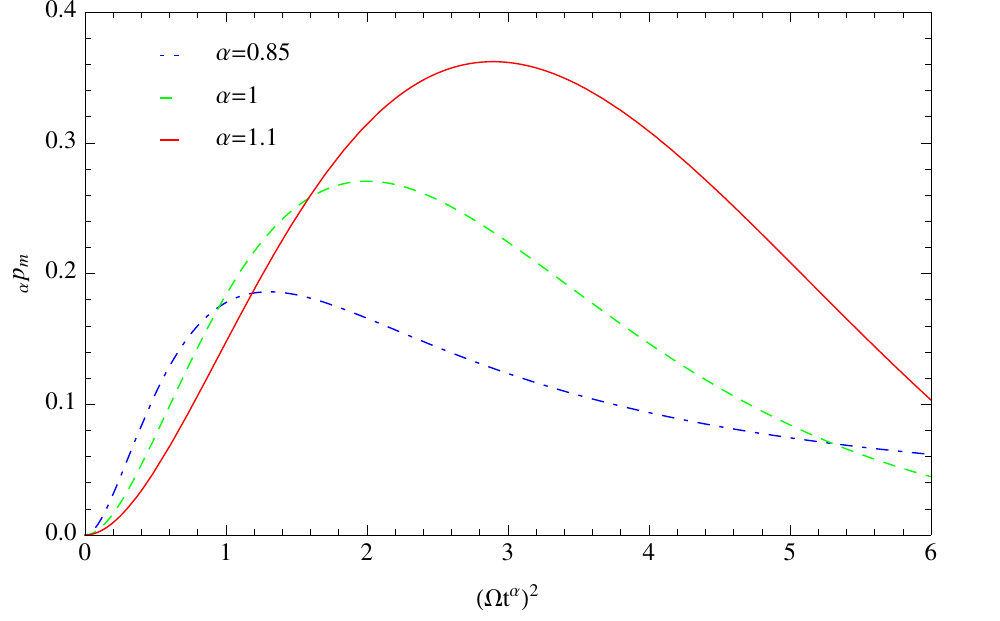}
		\caption{$m=2$.}
		\label{Fig3b}
	\end{subfigure}
\\[2mm]
	\begin{subfigure}[c]{0.48\textwidth}
		\includegraphics[width=0.9\linewidth]{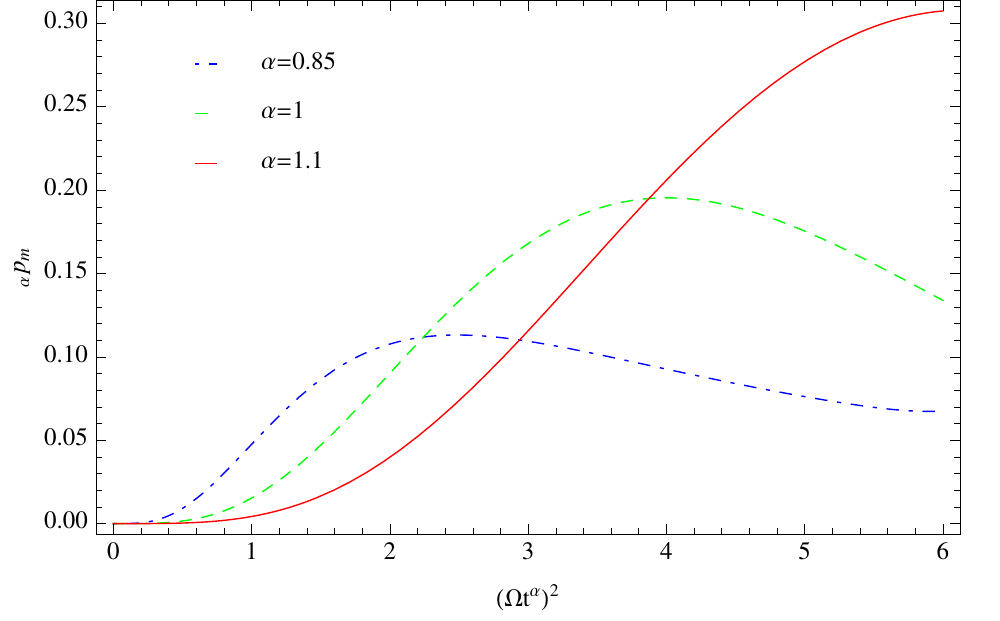}
		\caption{$m=4$.}
		\label{Fig3c}
	\end{subfigure}
	\caption{Probability distribution "$  {}_{\alpha}p(m)$" vs $\left( \Omega t^{\alpha}\right) ^2$, for different values of $\alpha$ and $m$.}\label{Figure4} 
\end{figure}

\chapter{Operator, Differintegral and Umbral Calculi}\label{Chapter2}
\numberwithin{equation}{section}
\markboth{\textsc{\chaptername~\thechapter. Operator, Differintegral and Umbral calculi}}{}

In this Chapter we discuss new aspects of  Operational Calculus Theory and of its evolution into \textit{Differintegral and Umbral calculi}. We extend the umbral theory by the use of special polynomials as \textit{Hermite polynomials} and provide several examples and physical applications showning the versatility of the method.\\

The original parts of the Chapter, containing their adequate bibliography, are based on the following original papers.\\

\cite{HermCalc} \textit{G. Dattoli, B. Germano, S. Licciardi, M.R. Martinelli; “Hermite Calculus”;  Modeling in Mathematics, Atlantis Transactions in Geometry, vol 2. pp. 43-52, J. Gielis, P. Ricci, I. Tavkhelidze (eds), Atlantis Press, Paris, Springer 2017}.\\

\cite{CDR} \textit{M. Artioli et al; “A 250 GHz Radio Frequency CARM Source for Plasma Fusion”, Conceptual Design Report, ENEA, pp. 154, 2016, ISBN: 978-88-8286-339-5}.\\

\cite{CarmFel} \textit{E. Di Palma, E. Sabia, G. Dattoli, S. Licciardi and I. Spassovsky; “Cyclotron auto resonance maser and free electron laser devices: a unified point of view”, Journal of Plasma Physics, Volume 83, Issue 1 , February 2017}. \\

\cite{Babusci} \textit{D. Babusci, G. Dattoli, M. Del Franco, S. Licciardi; “Mathematical Methods for Physics”, invited Monograph by World Scientific, Singapore, 2017, in press}.\\

\cite{FelHigh} \textit{M. Artioli, G. Dattoli, S. Licciardi, S. Pagnutti; “Fractional Derivatives, Memory kernels and solution of Free Electron Laser Volterra type equation”,  Mathematics 2017, 5(4), 73; doi: 10.3390/math5040073}.\\


$\star$ \textit{G. Dattoli, S. Licciardi, E. Sabia; “Operator Ordering and Solution of Umbral and Fractional Dfferential Equations”, work in progress}.\\

 The theoretical framework we are going to describe is relevant to a series of papers which have been developed during the last years and deepen their roots in the so called \textbf{\textit{Quasi-Monomial}}  $(QM)$ treatment of special polynomials \cite{Germano,Bell}. \\

\noindent According to such a point of view, we recall the following definition.  

\begin{defn}\label{QMdef}
A family of polynomials $p_n (x),\;\forall n\in\mathbb{N},\;\forall x\in\mathbb{R}$, is said to be a \textsl{Quasi-Monomial} if a couple of operators, $\mathbf{\hat{M}}$ and $\mathbf{\hat{P}}$, hereafter called \textbf{Multiplicative} and \textbf{Derivative} operators respectively, do exist and act according to the rules

\begin{align}
& \hat{M}\; p_n(x)=p_{n+1}(x),\label{Mop}\\
& \hat{P}\; p_n(x)=n\;p_{n-1}(x).\label{Pop}
\end{align}
\end{defn}
The previous identities clarify the role and, hence, the names of the two operators and can be exploited to derive the relevant properties.\\

\begin{cor}
 The combinations of the two identities, $\forall n\in\mathbb{N},\;\forall x\in\mathbb{R}$, yields

\begin{equation}\label{MPo}
\hat{M} \hat{P}\;p_n(x)=n\;p_n(x)
\end{equation}
 and

\begin{equation}\label{PMo}
\hat{P} \hat{M}\;p_n(x)=(n+1)\;p_n(x),
\end{equation}
which eventually allows the conclusion that the commutator of the multiplicative and derivative operators is

\begin{equation}\label{key}
\left[ \hat{P} ,\hat{M}\right] =\hat{P} \hat{M}-\hat{M} \hat{P}=\hat{1}.
\end{equation}
\end{cor}

It is therefore evident that $\hat{P},\hat{M},\hat{1}$  can be viewed as the generators of a Weyl algebra \cite{Babusci} 
. Furthermore, for a specific differential realization of the afore mentioned operators, eqs. \ref{MPo}, \ref{PMo} provide the eigenvalue equation of the $p_n(x)$ polynomials.\\ 

\begin{prop}
Assuming that the "vacuum" is such that $p_0(x)=1$, we infer, from eq. \ref{Mop},  that, $\forall n\in\mathbb{N},\;\forall x\in\mathbb{R}$, we can generate our family of polynomials $p_n(x)$ according to the “rule”

\begin{equation}\label{Mpn}
\hat{M}^n \;1=p_n(x).
\end{equation}
\end{prop}

\begin{cor}
The relevant generating function straightforwardly follows from eq. \ref{Mpn}, namely

\begin{equation}\label{genfunopM}
\sum_{n=0}^{\infty}\dfrac{t^n}{n!}\;p_n(x)=\sum_{n=0}^{\infty}\dfrac{t^n \hat{M}^n}{n!}1=e^{t\hat{M}}1.
\end{equation}
\end{cor}

The previous remarks are a nut-shell content of the $QM$ theory of special polynomials. The formalism will be further elaborated in the forthcoming part of the chapter but, in the remaining of this section, we see how some “popular” families of special polynomials can be ascribed to such a general context.\\

\begin{exmp}
In \ref{classHerm} we have introduced the two variable Hermite polynomials $(HP)$ $H_n(x,y)$  which are shown to satisfy the recurrences \cite{Babusci} 

\begin{equation}\begin{split}\label{Hnpm}
 \left( x+2y\partial_{x}\right)H_n(x,y) &=H_{n+1}(x,y),\\
 \partial_x H_n(x,y)&=nH_{n-1}(x,y),
\end{split}\end{equation}
$\forall x,y\in\mathbb{R}, \forall n\in\mathbb{N}$.
It is accordingly evident that Hermite polynomials are $QM$ and that

\begin{align}
 & \hat{M}=x+2y\partial_x,\label{MPHerma}\\
& \hat{P}=\partial_x.\label{MPHermb}
\end{align}
The second order differential equation satisfied by $H_n(x,y)$ can therefore be written in terms of the relevant product, namely

\begin{equation}\label{key}
\hat{M}\hat{P}=\left( x+2y\partial_{x}\right)\partial_{x},
\end{equation}
yielding

\begin{equation}\label{key}
x\;\partial_x H_n(x,y)+2\;y\; \partial_{x}^2 H_n(x,y)=n H_n(x,y).
\end{equation}
Regarding the generating function, by using eqs. \ref{genfunopM}-\ref{MPHerma}, we find

\begin{equation}\label{key}
\sum_{n=0}^{\infty}\dfrac{t^n}{n!}H_n(x,y)=e^{t\left( x+2y\partial_{x}\right)}1.
\end{equation}
The use of the Weyl rule \ref{Weylid} \cite{Vazquez} yields\footnote{We remind that $\left[x,\partial_x \right]=-1 $.}

\begin{equation}\label{newGenfH}
e^{t\left( x+2y\partial_{x}\right)}1=e^{-\frac{1}{2}\left[ tx, 2yt\partial_{x}\right] }e^{tx}e^{2yt\partial_{x}}1=e^{xt+yt^2},
\end{equation}
thus providing the quoted generating function \ref{genfunctH}.\\
In eq. \ref{newGenfH} the last exponential operator disappears because 

\begin{equation}\label{key}
e^{2\;y\:t\;\partial_{x}}1=1.
\end{equation}
We may however ask how the previous operator identities should be modified if, instead of a constant, it acts on a generic function of the variable $x$.\\ 
We have noted \ref{Mpn}-\ref{MPHerma} that 

\begin{equation}\label{newtondeH}
\left( x+2y\partial_{x}\right)^n 1=H_n(x,y)
\end{equation}
but, if we relax the assumption that the operator on the left is acting on unity, the result is however slightly more complicated.\\
We first note that (series expansions-\ref{genfunctH})

\begin{equation}\begin{split}\label{key}
 \sum_{n=0}^{\infty}
\dfrac{t^n}{n!}\left( x+2y\partial_{x}\right)^n&=e^{yt^2 +xt}e^{2ty\partial_{x}}=
\sum_{m=0}^{\infty}\dfrac{t^m}{m!}H_m(x,y)\sum_{s=0}^\infty\dfrac{(2ty)^s}{s!}\partial_x^s=\\
&  =\sum_{n=0}^{\infty}\dfrac{t^n}{n!}\hat{N}_n,\\
 \hat{N}_n&=\sum_{s=0}^n \binom{n}{s}(2y)^s H_{n-s}(x,y)\partial_{x}^s.
\end{split}\end{equation}

Equating the same like power "t" terms, we end up with the so called \textbf{Burchnall} identity  \cite{Vazquez}

\begin{equation}\label{key}
\left( x+2y\partial_{x}\right)^n=\sum_{s=0}^{n}\binom{n}{s}(2y)^s H_{n-s}(x,y)\partial_{x}^s.
\end{equation}
Accordingly we obtain

\begin{equation}\label{Hfs}
\left( x+2y\partial_{x}\right)^n f(x)=\sum_{s=0}^{n}\binom{n}{s}(2y)^s H_{n-s}(x,y) f^{(s)}(x),
\end{equation}
where $f^{(s)}(x)$ denotes the $s^{th}$-derivative of the function $f(x)$ and therefore eq. \ref{Hfs} reduces to eq. \ref{newtondeH} if $f(x)$  is a constant.
\end{exmp}

In the forthcoming sections we provide a description of the orthogonal properties of $HP$ within the so far developed operational framework.

\section{Hermite Polynomial Orthogonal Properties and the Operational Formalism}

In this section we discuss the orthogonal nature of the two-variable $HP$ using the rules we have established in the previous section. The point of view we describe is slightly different from the conventional treatment and is developed to obtain a more general definition of the concepts, underlying the orthogonal nature of different polynomial families, we will employ in the course of this thesis.\\

\begin{lem}
Along with the definition of \textit{QM}, we can also introduce functions expanded on quasi monomials, according to the rule (we omit the vacuum $1$, for conciseness)

\begin{equation}\label{fM}
f(\hat{M})=\sum_{n=0}^{\infty}a_n\hat{M}^n=
\sum_{n=0}^{\infty}a_n p_n(x),
\end{equation}
which can be understood as an expansion over the quasi monomial basis $p_n(x), \forall x\in\mathbb{R}$.
\end{lem}
 
The meaning of eq. \ref{fM} can be worded as it follows.

\begin{prop}
Any function having a Mac Laurin expansion on the ordinary monomial basis, can be associated to a corresponding function which can be expanded basis on its \textit{QM} counterpart. 
\end{prop}

Therefore, once we have established a correspondence between $f(\hat{M})$ and a function of $x$, we have found the expansion of that function on the \textit{QM} basis. In terms of $HP$ we find 

\begin{exmp}
$\forall x,y\in\mathbb{R}$ (see \cite{Babusci} and \ref{newtondeH}) 

\begin{equation}\label{eydx}
e^{y\partial_x^2}e^x=\sum_{n=0}^{\infty}\dfrac{e^{y\partial_x^2}x^n}{n!}=\sum_{n=0}^{\infty}\dfrac{\left( x+2y\partial_{x}\right)^n}{n!}=
\sum_{n=0}^{\infty}\dfrac{H_n(x,y)}{n!}=e^{x+y},
\end{equation}
or (see below \ref{proof213} for further details)

\begin{equation}\label{exxy}
e^{y\partial_x^2}e^{-x^2}=\sum_{n=0}^{\infty}\dfrac{(-1)^n e^{y\partial_x^2}x^{2n}}{n!}=\sum_{n=0}^{\infty}\dfrac{(-1)^n}{n!}H_{2n}(x,y)=\dfrac{1}{\sqrt{1+4y}}e^{-\frac{x^2}{1+4y}}.
\end{equation}
The second serie has a limited convergence radius $\mid y\mid <\dfrac{1}{4}$ as further discussed later.
\end{exmp}

It must be now clarified that the previous statements does not allow us to conclude that \textit{QM} polynomials are “naturally” orthogonal. It must be stressed that the previous analysis does not allow any conclusion on the Orthogonality properties of the $H_n(x,y)$ polynomials. To get general statements within such a respect, we develop a more elaborated treatment reported below. We clarify this point by discussing a specific example.

\begin{exmp}\label{exbohH}
Let us accordingly assume that a function $f(x)$ can be expanded as

\begin{equation}\label{key}
f(x)=\sum_{n=0}^{\infty}a_nH_n(x,y),
\end{equation}
where $a_n$ are the coefficients of the expansion. It is evident that, being $f$ a function of one variable only, $y$ should be regarded as a parameter. Let us note that, since $H_n(x,y)=e^{y\partial_{x}^2} x^n$ \cite{Babusci}, we find, from eqs. \ref{fM}-\ref{eydx},

\begin{equation}\label{efan}
e^{-y\partial_{x}^2}f(x)=\sum_{n=0}^{\infty}a_nx^n.
\end{equation}
To ensure the existence of a \textit{Gauss-Weierstrass} transform \ref{GWi} of  $f(x)$,  we assume that $\mathbf{y}$ be a \textbf{negative} defined parameter by setting $y=-\mid y\mid,\forall y\in\mathbb{R}$, thus, writing eq. \ref{efan}  (by applying \ref{GWi} at eq. \ref{efan} and through a variable change) as

\begin{equation}\label{key}
\sum_{n=0}^{\infty}a_nx^n=\dfrac{1}{2\sqrt{\pi \mid y\mid}}\int_{-\infty }^{\infty}e^{-\frac{(x-\xi)^2}{4\mid y\mid}}f(\xi)d\xi.
\end{equation}
It is therefore evident that

\begin{equation}\begin{split}
\sum_{n=0}^{\infty}a_nx^n&=
\dfrac{1}{2\sqrt{\pi \mid y\mid}}
\int_{-\infty }^{\infty}e^{\frac{x\xi}{2\mid y\mid}-\frac{x^2}{4\mid y\mid}}\left(e^{-\frac{\xi^2}{4\mid y\mid}} f(\xi)\right)d\xi=\\
& = \dfrac{1}{2\sqrt{\pi \mid y\mid}}\sum_{m=0}^{\infty}\dfrac{x^m}{m!} \int_{-\infty }^{\infty}H_m\left( \frac{\xi}{2\mid y\mid},-\frac{1}{4\mid y\mid}\right) \left(e^{-\frac{\xi^2}{4\mid y\mid}} f(\xi)\right)d\xi
\end{split}\end{equation}
and, by equating $x$-like power, we obtain for the expansion coefficients    

\begin{equation}\label{key}
a_m=\dfrac{1}{2m!\sqrt{\pi \mid y\mid}}\int_{-\infty }^{\infty}H_m\left( \frac{\xi}{2\mid y\mid},-\frac{1}{4\mid y\mid}\right) e^{-\frac{\xi^2}{4\mid y\mid}} f(\xi)d\xi.
\end{equation}
It is accordingly evident that the function 

\begin{equation}\label{key}
u_m\left( x,-\mid y\mid\right)= \dfrac{1}{2m!\sqrt{\pi \mid y\mid}}H_m\left( \frac{x}{2\mid y\mid},-\frac{1}{4\mid y\mid}\right) e^{-\frac{x^2}{4\mid y\mid}} 
\end{equation}
is bi-orthogonal to $H_m\left( x,-\mid y\mid\right)$, namely

\begin{equation}\label{deltamn}
\int_{-\infty }^{\infty}u_m\left( x,-\mid y\mid\right)H_n\left( x,-\mid y\mid\right)dx=\delta_{m,n}.
\end{equation}

We note that, within the present context, $\mid y\mid$ is only a parameter and that, setting e.g. $\mid y\mid=\dfrac{1}{2}$, we obtain the ordinary Hermite functions of quantum harmonic oscillator.\\

\noindent It is worth noting that, to ensure the condition \ref{deltamn}, it is necessary that the $Re\left( -\mid y\mid\right) $   be positive.
\end{exmp}

Even though trivial, we provide a further example of expansion associated with the \textit{shifted HP} following defined.

\begin{exmp}
Let the operational rule (either $y$ and $z$ are parameters) \cite{Babusci} 

\begin{equation}\label{key}
H_n(x,y;z):=H_n(x-z,y)=e^{-z\partial_{x}+y\partial_{x}^2}x^n.
\end{equation}
The conditions for the expansion on this family of polynomials occurs through the same procedure as before, we set therefore (the conditions on the sign of $z$ are not influent)

\begin{equation}\label{key}
f(x)=\sum_{n=0}^{\infty}a_nH_n\left(x,-\mid y\mid;z \right) 
\end{equation}
and therefore

\begin{equation}\label{key}
e^{\mid y\mid \partial_{x}^2+z\partial_{x}}f(x)=\sum_{n=0}^{\infty}a_nx^n.
\end{equation}

Using the same procedure of example \ref{exbohH} we end up with the following family of bi-orthogonal functions to the shifted $HP$

\begin{equation}\label{key}
u_n\left(x,-\mid y\mid;z \right) =\dfrac{1}{2m!\sqrt{\pi \mid y\mid}}
H_m\left( \frac{x-z}{2\mid y\mid},-\frac{1}{4\mid y\mid}\right) e^{-\frac{(x-z)^2}{4\mid y\mid}} ,
\end{equation}
which provides the Hermite functions of the shifted harmonic oscillator.
\end{exmp}

The expansion \ref{exxy}, on the light of the previous discussion, can be explained as it follows.

\begin{proof}[\textbf{Proof.}] \textit{of eq. \ref{exxy}.}
	By expanding $e^{-x^2}$, we get
\begin{equation}\label{proof213}
e^{y\partial_{x}^2}e^{-x^2}=e^{y\partial_{x}^2}\sum_{r=0}^{\infty}\dfrac{(-1)^r}{r!}x^{2r},
\end{equation}
interchanging the summation index and the exponential operator, we get

\begin{equation}\label{key}
e^{y\;\partial_{x}^2}e^{-x^2}=\sum_{r=0}^{\infty}\dfrac{(-1)^r}{r!}e^{y\;\partial_{x}^2}x^{2r}=
\sum_{r=0}^{\infty}\dfrac{(-1)^r}{r!}H_{2r}(x,y).
\end{equation}
From the other side, by using the \textit{GWI }\ref{GWi} we easily infer again the Glaisher identity \ref{Glaisher} $e^{y\;\partial_{x}^2}e^{-x^2}=\dfrac{1}{\sqrt{1+4y}}e^{-\frac{x^2}{1+4y}},
$ which justifies what is reported in eq. \ref{exxy}. It is important to stress that, while the transform integral of $e^{y\;\partial_{x}^2}e^{-x^2}$ converges for any $y>-\dfrac{1}{4}$, the expansion in terms of two variable Hermite has a limited range of convergence (see later \ref{genfunH2n}). The reasons are subtle and rely on the fact that the interchanging of the summation and the exponential is not always admitted.
\end{proof}

\begin{rem}
\textit{Before closing this section let us note that the expansion of a monomial in terms of HP is fairly straightforward and is expressed as}

\begin{equation}\label{key}
x^n=\sum_{r=0}^{\lfloor\frac{n}{2}\rfloor}a_{n,r}H_{n-2r}(x,y),
\end{equation}
\textit{which yields the condition}

\begin{equation}\label{key}
a_{n,r}=\dfrac{n!}{(n-2r)!r!}(-y)^r,
\end{equation}
\textit{which, being a finite sum, does require the condition }$y=-\mid y\mid$.
\end{rem}

We can further stress the role of the $y$ parameter in the theory of two variable Hermite by noting that\\

a)	Being $H_n(x,y)$ solutions of the heat equation \cite{Widder}, they play the role of the so called \textit{heat polynomials}, introduced by \textsl{Widder} in ref. \cite{Widder}, where $y$ is associated with the evolution time;\\

b)	A more subtle meaning can be inferred from the plots reported in Fig. \ref{FigHerm3D2D} in which we have shown the polynomials in an $x,y,z$ space. Within this context, the operational rule $H_n(x,y)=e^{y\partial_{x}^2}x^n$ \cite{Babusci} can be understood in a geometrical sense. \textbf{\textit{The exponential operator transforms an ordinary monomial into a Hermite type special polynomial.}} The “evolution” from an ordinary monomial to the corresponding Hermite is shown by moving the cutting plane orthogonal to the $y$ axis. For a specific value of the polynomial degree $n$, the polynomials lie on the cutting plane, as shown in the figures.\\
 It is worth stressing that \textit{only for negative values of} $y$ do the \textit{polynomials exhibit zeros}, in accordance with the fact that in this region they realize an orthogonal set.

\begin{figure}[htp]
	\centering
	\begin{subfigure}[c]{0.49\textwidth}
		\includegraphics[width=0.9\linewidth]{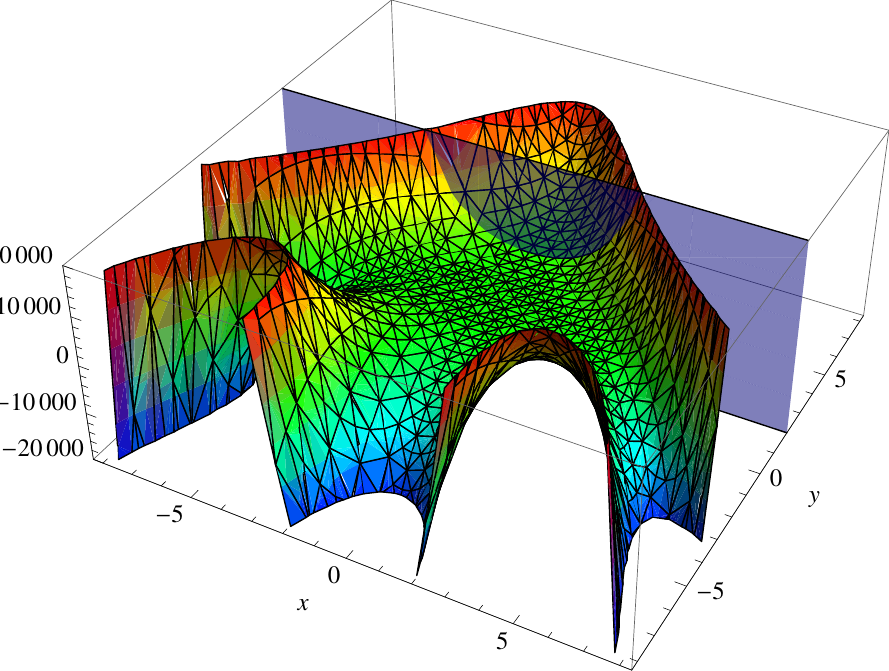}
		\caption{$H_6(x,y)$ cutted by the plane $y=2$.}
		\label{artMotz1}
	\end{subfigure}
	\begin{subfigure}[c]{0.49\textwidth}
	\includegraphics[width=0.9\linewidth]{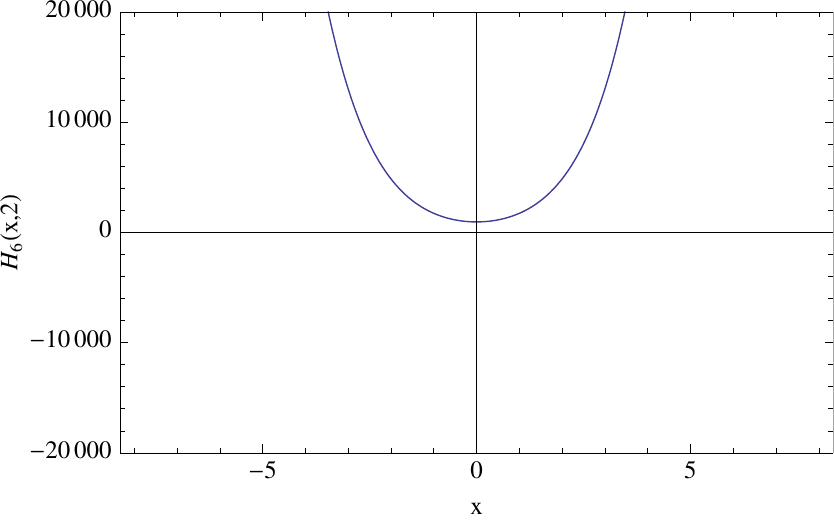}
	\caption{$H_6(x,2)$.}
\end{subfigure}
\\[2mm]
	\begin{subfigure}[c]{0.49\textwidth}
		\includegraphics[width=0.9\linewidth]{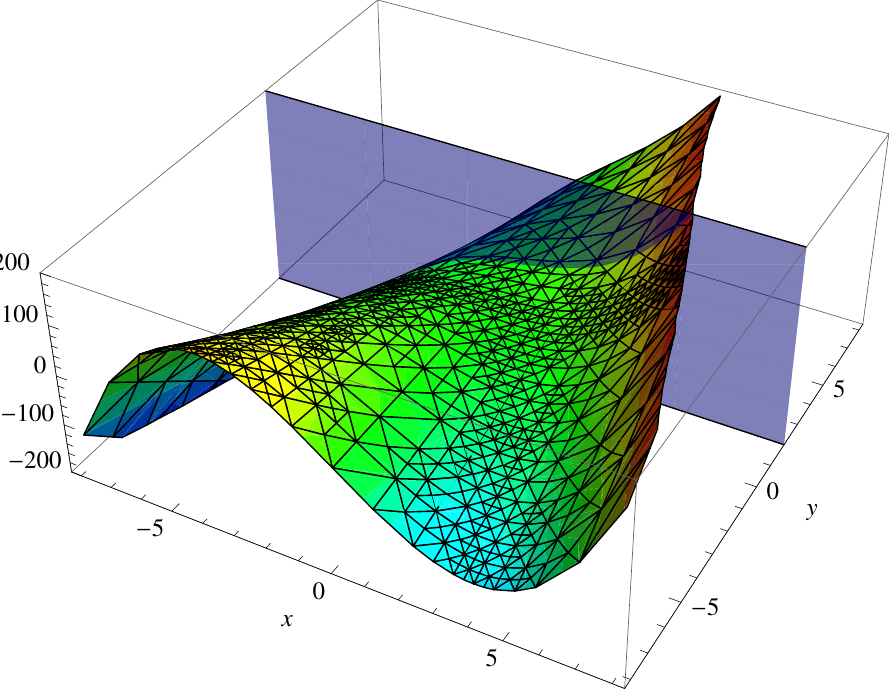}
		\caption{$H_3(x,y)$ cutted by the plane $y=2$.}
	\end{subfigure}
	\begin{subfigure}[c]{0.49\textwidth}
	\includegraphics[width=0.9\linewidth]{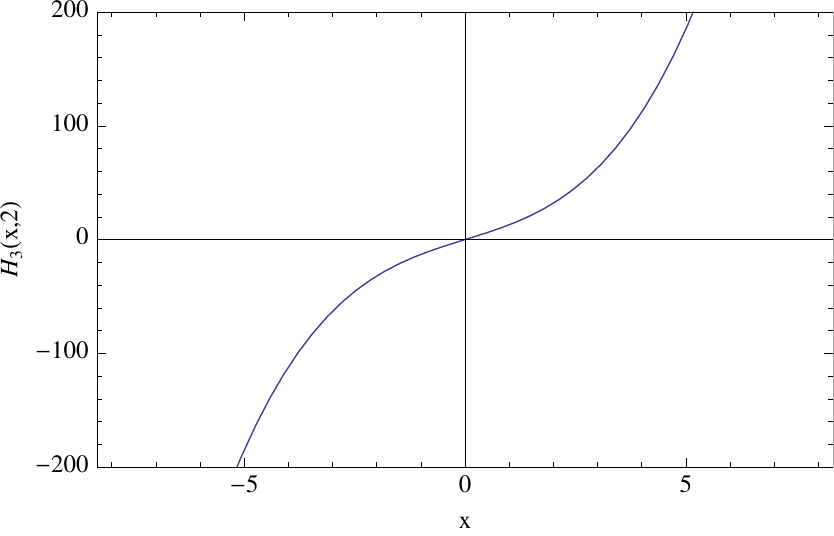}
	\caption{$H_3(x,2)$.}
\end{subfigure}
\\[2mm]
	\begin{subfigure}[c]{0.49\textwidth}
	\includegraphics[width=0.9\linewidth]{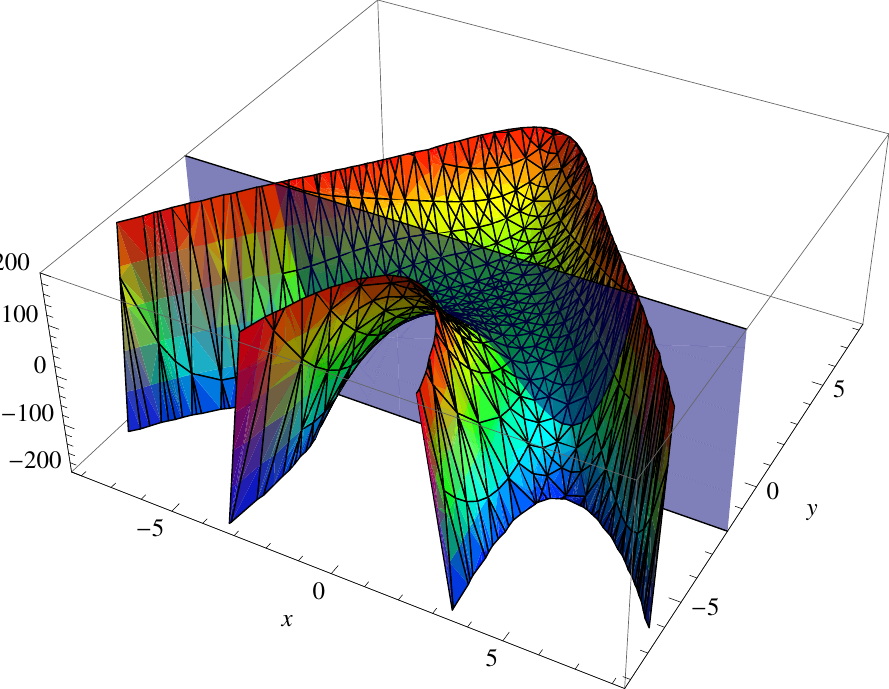}
	\caption{$H_4(x,y)$ cutted by the plane $y=-2$.}
\end{subfigure}
	\begin{subfigure}[c]{0.49\textwidth}
		\includegraphics[width=0.9\linewidth]{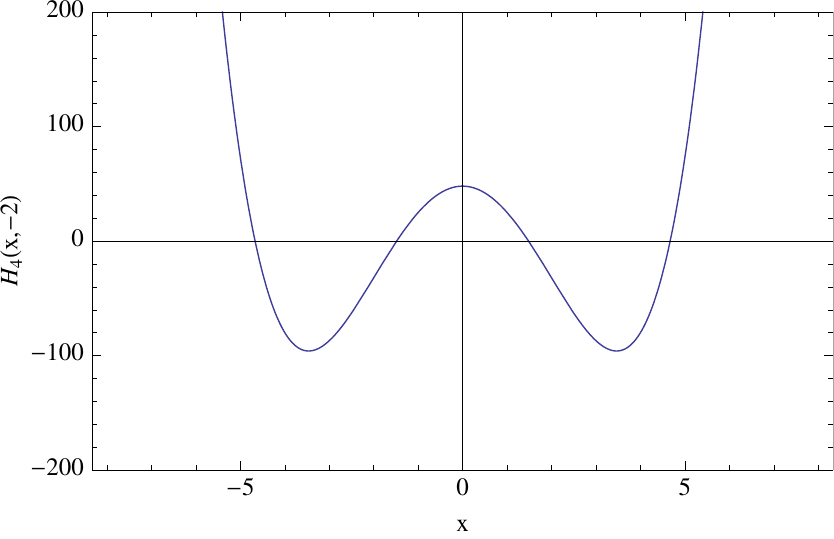}
	\caption{$H_4(x,-2)$.}
	\end{subfigure}
	\caption{Geometrical representation of two-variable Hermite polynomials in 3D and 2D, for different $n$ and $y$ values \cite{MotzkinWolfram}.}\label{FigHerm3D2D} 
\end{figure}
\newpage
The discussions and the examples we have developed so far show operational point of view on the theory of special polynomials and open many new possibilities of interpretations on the nature of the polynomials themselves as will be further discussed in the forthcoming sections of this chapter. 



\section{An Umbral Point of View on Hermite Polynomials}\label{NewtonBinHP}

In this section we consider the transition from Monomiality to Umbral interpretation of \textit{HP}. We noted e.g., in eqs. \ref{J0op}-\ref{J0opb}, that the Gaussian function is the umbral image of the cylindrical Bessel function, now we see that the \textit{Newton bynomial} realizes the umbral image of \textit{HP}. 

\begin{defn}
	We introduce
	\begin{equation}\label{key}
	\theta(z) :=\theta_z=y^{\frac{z}{2}}
	\left( \dfrac{\Gamma(z+1)}{\Gamma\left( \frac{1}{2}z+1\right) }\left|  \cos \left( \frac{\pi}{2}z\right)  \right| \right) , \quad \forall z\in\mathbb{R},
	\end{equation}
	the Hermite function vacuum.
\end{defn}
\noindent Then
\begin{prop}
	The umbral operator ${}_y \hat{h}^r$  acts on the vacuum $\theta_0$ according to the  rule

	\begin{equation}
	\begin{split}\label{eq2HermLagbis} 
	& {}_y \hat{h}^r\;\theta_0 :=\theta_r, \quad \forall r\in \mathbb{R},\\
	& \theta_r =\dfrac{y^{\frac{r}{2}}r!}{\Gamma\left( \frac{r}{2}+1\right) }\left|  \cos \left( r\dfrac{\pi}{2}\right) \right| = 
	\left\lbrace   \begin{array}{ll}
	0                           & r=2s+1 \\
	y^s \dfrac{(2s)!}{s!} & r=2s
	\end{array}\right. \;\;\forall s\in\mathbb{Z}.
	\end{split}\end{equation}  
\end{prop}
\begin{proof}[\textbf{Proof.}]
	$\forall s\in\mathbb{Z}, \forall r,y,z\in \mathbb{R}$, by using the shift operator ${}_y {\hat{h}}=e^{\partial_z}$, we get 
	\begin{equation*}\begin{split}
	{}_y\hat{h}^r\theta_0&=
	\left. {}_y\hat{h}^r \theta (z)\right| _{z=0}=
	\left. e^{r\partial_z}\theta (z)\right| _{z=0}=
	\left.\theta (z+r)\right| _{z=0}=\left. \theta_{z+r}\right| _{z=0}=\\
	& =\left. y^{\frac{z+r}{2}}
	\left( \dfrac{\Gamma(z+r+1)}{\Gamma\left( \frac{1}{2}(z+r)+1\right) }\left|  \cos \left( \frac{\pi}{2}(z+r)\right)  \right| \right)\right|_{z=0}=\\
	& =\dfrac{y^{\frac{r}{2}}r!}{\Gamma\left( \frac{r}{2}+1\right) }\left|  \cos \left( r\dfrac{\pi}{2}\right) \right| = 
	\left\lbrace   \begin{array}{ll}
	0                           & r=2s+1 \\
	y^s \dfrac{(2s)!}{s!} & r=2s
	\end{array}\right. = \theta_r. 
	\end{split}\end{equation*}
\end{proof}         
The numbers $ \dfrac{(2s)!}{s!}=1,2,12,120,1680,...$ are recognized as the \textit{quadrupal factorial numbers}, reported in $OEIS$ sequence $A001813$.\\

We remind the following Proposition already stated in \cite{G.Dattoli}.

\begin{prop}\label{propHpol}
 The Newton bynomial umbral version of Hermite polynomials is accordingly obtained from eq. \ref{eq2HermLagbis} 

\begin{equation} \label{eq1HermLagbis}
H_{n} (x,y)  = \left( x+{}_y {\hat{h}}\right) ^n \theta_0, \quad \forall x,y\in\mathbb{R},\forall n\in\mathbb{N}.
\end{equation}
\end{prop}
\begin{proof}[\textbf{Proof.}]
	$\forall x,y\in\mathbb{R},\forall n\in\mathbb{N}, \forall s\in\mathbb{Z}$, by the use of Newton bynomial and eq. \ref{eq2HermLagbis}  
	\begin{equation*}\begin{split}\label{key}
\left( x+{}_y {\hat{h}}\right) ^n \theta_0&=\sum_{r=0}^n\binom{n}{r}x^{n-r}{}_y {\hat{h}}^{r}	\theta_0=\sum_{r=0}^n\binom{n}{r}x^{n-r}y^s \dfrac{(2s)!}{s!}=\\
& =\sum_{s=0}^{\lfloor\frac{n}{2}\rfloor} x^{n-2s}y^s \dfrac{n!}{(n-2s)!s!}=H_n(x,y)
\end{split}	\end{equation*}
\end{proof}
\begin{cor}
	The correspondence between umbral \ref{eq1HermLagbis} and monomiality operators (Definition \ref{QMdef}) is 
	
\begin{equation}\label{key}
\hat{M}\leftrightarrow \left( x+{}_{y}\hat{h}\right).
\end{equation}
\end{cor}
The generating function of \textit{HP} is straightforwardly inferred (in according to \ref{genfunopM} too)

\begin{equation}\label{key}
\sum_{n=0}^{\infty}\dfrac{t^n}{n!}H_n(x,y)=
\sum_{n=0}^{\infty}\dfrac{\left(t\left( x+{}_{y}\hat{h}\right) \right) ^n}{n!}=
e^{xt}e^{{}_{y}	\hat{h}t}\theta_{0},\quad\forall t\in\mathbb{R},
\end{equation}
which yields the ordinary expression \ref{genfunctH} by noting that

\begin{Oss}
\begin{equation}\label{yhgenfun}
e^{{}_{y}\hat{h}t}\theta_{0}=\sum_{r=0}^{\infty}\dfrac{t^r}{r!} {}_{y}\hat{h}^{\;r}\theta_{0}=e^{yt^2},\quad \forall y,t\in\mathbb{R}. 
\end{equation}
\end{Oss}
\begin{proof}[\textbf{Proof.}]
$\forall y,t\in\mathbb{R}, \theta_0$ the vacuum of the ${}_{y}\hat{h}$-operator, by using series expansion and eq. \ref{eq2HermLagbis}, we obtain 
\begin{equation}\label{key}
e^{{}_{y}\hat{h}t}\theta_{0}=\sum_{r=0}^{\infty}\dfrac{t^r}{r!} {}_{y}\hat{h}^{\;r}\theta_{0}=
\sum_{r=0}^{\infty}\dfrac{t^r}{r!}y^s \dfrac{(2s)!}{s!}=
\sum_{s=0}^{\infty}\dfrac{t^{2s}}{s!}y^s =e^{yt^2}. 
\end{equation}
\end{proof}
The umbral point of view to \textit{HP} is particularly useful for a straightforward derivation of the relevant properties. We provide some examples through which we establish some of these properties and give an idea of the implication offered by the present formalism.\\

We note  that 

\begin{cor}
Regarding the derivation of generating function involving even index \textit{HP}, the following identity holds (applying \ref{eq1HermLagbis})

\begin{equation}\label{key}
\sum_{n=0}^{\infty}\dfrac{t^n}{n!}H_{2n}(x,y)=e^{t\left(x+ {}_{y}\hat{h}\right) ^2}\theta_{0}.
\end{equation}
Furthermore, on account of the \textit{GII} \ref{Gii} and eq. \ref{yhgenfun}, we can write

\begin{equation}\begin{split}\label{key}
e^{t\left(x+ {}_{y}\hat{h}\right) ^2}\theta_{0}&=\dfrac{1}{\sqrt{\pi}}\int_{-\infty }^{\infty}e^{-\xi^2 +2\sqrt{t}\; \left(x+ {}_{y}\hat{h}\right)\;\xi}\;d\xi\;\theta_{0}=\\
& =\dfrac{1}{\sqrt{\pi}}\int_{-\infty }^{\infty}e^{-\xi^2}e^{2\sqrt{t}x\xi}\left( e^{2\sqrt{t}\xi{}_{y}\hat{h}}\theta_0\right) d\xi=\\
& =\dfrac{1}{\sqrt{\pi}}\int_{-\infty }^{\infty}e^{-\xi^2}e^{2\sqrt{t}x\xi}e^{4yt\xi^2}d\xi=
\dfrac{1}{\sqrt{1-4yt}}e^{\frac{x^2t}{1-4yt}},
\end{split}\end{equation}
 thus getting

\begin{equation}\begin{split}\label{genfunH2n}
&\sum_{n=0}^{\infty}\dfrac{t^n}{n!}H_{2n}(x,y)=\dfrac{1}{\sqrt{1-4yt}}e^{\frac{x^2\;t}{1-4yt}},\\
& \mid t\mid <\dfrac{1}{\mid 4y\mid},
\end{split}	\end{equation}
which is a result analogous to that already viewed in eq. \ref{exxy}, within the framework of the orthogonal properties of two variable Hermite.
\end{cor}

\begin{exmp}
	By exploiting eq. \ref{yhgenfun} we get
\begin{equation}\label{superG}
e^{-y x^4}=e^{-i {}_{y}\hat{h}x^2}\theta_{0},\quad\forall x\in\mathbb{R},\forall y\in\mathbb{R}^+_0.
\end{equation}
According to this identity, the \textbf{super-Gaussian of order 4} can be treated as an ordinary Gaussian. It is, accordingly, instructive to note that

\begin{equation}\label{key}
\int_{-\infty }^{\infty}e^{-y x^4}dx=\int_{-\infty }^{\infty}e^{-i {}_{y}\hat{h}x^2}\;dx\;\theta_{0}=\sqrt{\dfrac{\pi}{i{}_{y}\hat{h}}}\;\theta_{0}
\end{equation}
and, being the integral a real integral (to use "i" is an artifice), we calculate  

\begin{equation}\label{key}
\sqrt{\pi}\left| \left( i^{-\frac{1}{2}}\right)\right|  {}_{y}\hat{h}^{-\frac{1}{2}}\theta_{0}=\sqrt{\pi}\dfrac{\sqrt{2}}{2}\dfrac{y^{-\frac{1}{4}}\Gamma\left( \dfrac{1}{2}\right) }{\Gamma\left( \dfrac{3}{4}\right) }=\dfrac{1}{2\sqrt[4]{y}}\Gamma\left( \dfrac{1}{4}\right) ,
\end{equation}
 obtaining, therefore, the correct (well known) result integral of the super-gaussian in eq. \ref{superG}. 
\end{exmp}

\begin{exmp}\label{18_20}
	Let
\begin{equation}\label{key}
I(\alpha,\beta)=\int_{-\infty }^{\infty}e^{-\alpha x^2-\beta x^4}dx, \quad\forall\alpha\in\mathbb{R},\forall\beta\in\mathbb{R}^+_0,
\end{equation}
which can be written as

\begin{equation}\label{key}
\begin{split}
& I(\alpha,\beta)=\int_{-\infty }^{\infty}e^{-x^2\left( \alpha+{}_{\beta}\hat{h}\right) }dx\;\theta_0,\\
& \hat{H}(\alpha,\beta)=\left( \alpha+{}_{\beta}\hat{h}\right),
\end{split}
\end{equation}
which yields

\begin{equation}\label{key}
I(\alpha,\beta)=\sqrt{\dfrac{\pi}{\hat{H}(\alpha,\beta)}}\theta_0=\sqrt{\pi}\hat{H}_{-\frac{1}{2}}(\alpha,\beta)\theta_0.
\end{equation}
\end{exmp}

The previous identity suggests the possibility of defining negative order $HP$, in which the index is not constrained to integers but may keep any real value, so we give

\begin{defn}
 According to eq. \ref{eq1HermLagbis} we define, $\forall \nu\in\mathbb{R}^+, \forall x,y\in\mathbb{R}$, the \textbf{Negative Order Hermite} (NOH)
 
\begin{equation}\label{NOHpol}
H_{-\nu}(x,y)=\left(x+{}_{y}\hat{h} \right)^{-\nu}\theta_{0}. 
\end{equation}
They are no more polynomials but Hermite \textbf{functions}.
\end{defn}

\begin{prop}
 The relevant NOH-function integral representation can be written as
 
\begin{equation}\label{NOHf}
H_{-\nu}(x,y)=
\dfrac{1}{\Gamma(\nu)}\int_{0}^{\infty}s^{\nu-1}e^{-sx}e^{-ys^2}ds, \quad \forall x\in\mathbb{R},\forall y,\nu\in\mathbb{R}^+.
\end{equation}
\end{prop}

\begin{proof}[\textbf{Proof.}]
	$\forall x\in\mathbb{R},\forall y,\nu\in\mathbb{R}^+$, by the use of Laplace transform and eq. \ref{yhgenfun}, we obtain
\begin{equation*}\begin{split}\label{key}
H_{-\nu}(x,y)&=
\dfrac{1}{\left(x+{}_{y}\hat{h} \right)^{\nu}}\theta_0=
\int_{0}^{\infty}e^{-xs}\dfrac{s^{\nu-1}e^{-{}_{y}\hat{h}s}}{\Gamma(\nu)}ds\;\theta_0=\\
& =\dfrac{1}{\Gamma(\nu)}\int_{0}^{\infty}s^{\nu-1}e^{-xs}\left( e^{-s{}_{y}\hat{h}_x}\theta_{0}\right) ds=\\
& =
\dfrac{1}{\Gamma(\nu)}\int_{0}^{\infty}s^{\nu-1}e^{-sx}e^{-ys^2}ds.
\end{split}\end{equation*}
	\end{proof}

The use of the same procedure leads to the derivation of the infinite integral $\forall y\in\mathbb{R}$

\begin{equation}\begin{split} \label{eq11HermLag}  
I_\nu(x,y\mid m)&=\int_{0 }^\infty e^{-s^m(x+ys^m)}s^{\nu -1}ds=
\int_{0 }^\infty e^{-s^m(x+{}_{-\mid y \mid }\hat{h})}s^{\nu -1}ds\;\theta_0=\\
& =\dfrac{\Gamma\left( \frac{\nu}{m}\right) }{m}H_{-\frac{\nu}{m}}(x,y).
\end{split}\end{equation} 

In these two last sections we have presented the theory of \textit{HP} in terms of a non standard procedure, which allows significant degrees of freedom from the computational point of view and opens new possibilities for the solution of practical problems. In the next section we extend this wide range of properties linked to Hermite polynomials/functions by providing various examples of applicability of the method, which we reassume by calling it Hemite Calculus.

\section{Hermite Calculus}
\numberwithin{equation}{section}
\markboth{\textsc{\chaptername~\thechapter. Hermite Calculus}}{}

The Hermite calculus investigated so far, useful to treat computations involving Hermite polynomials and their generalizations as well, can be expanded as follows.\\

\begin{exmp}\label{HermCalcImportant}
$\forall \alpha,\beta\in\mathbb{R}:\alpha+\beta>0, \forall \gamma\in\mathbb{R}$, we consider the integral

\begin{equation}\label{int1}
I(\alpha,\beta,\gamma)=\int_{-\infty}^{\infty}e^{-(\alpha+\beta)x^{2}-\gamma x}dx,
\end{equation}
which can be evaluated through the GWI \ref{GWi}, thus getting

\begin{equation}
I(\alpha,\beta,\gamma)=\sqrt{\dfrac{\pi}{\alpha + \beta}}e^{\frac{\gamma^{2}}{4(\alpha + \beta)}}.
\end{equation}
We restyle eq. \ref{int1} in umbral way

\begin{equation}\label{htilde}
I(\alpha,\beta,\gamma)=\int_{-\infty}^{\infty}e^{-\alpha x^{2}-\hat{h}_{(\gamma,-\beta)}x}dx\; \eta_0,
\end{equation}
where we have introduced the notation

\begin{equation}\label{serieh}
e^{-\hat{h}_{(\gamma,-\beta)}x}\eta_0=\sum_{r=0}^{\infty}\dfrac{(-x)^{r}}{r!}\hat{h}_{(\gamma,-\beta)}^{r}\eta_0=\sum_{r=0}^{\infty}\dfrac{(-x)^{r}}{r!}H_{r}(\gamma,-\beta),
\end{equation}
based on the use of the umbral identity  (\cite{Eric}, \cite{G.Dattoli})

\begin{equation}\label{herm}
\hat{h}_{(\gamma,-\beta)}^{r}\eta_0=\eta_r=H_{r}(\gamma,-\beta).
\end{equation}
In the integral in eq. \ref{htilde} we have treated the term which can be expended in terms of Hermite polynomials as a single block and we have enucleated the variable x raised to the first power.\\

\noindent According to \ref{GWi} we can write  \cite{G.Dattoli} \cite{H.M.Srivastava} 

\begin{equation}
I(\alpha,\beta,\gamma)=\sqrt{\dfrac{\pi}{\alpha}}e^{\frac{\hat{h}_{(\gamma,-\beta)}^{2}}{4\alpha}}\eta_0=\sqrt{\dfrac{\pi}{\alpha}}\sum_{r=0}^{\infty}\dfrac{1}{r!}\left(\dfrac{\hat{h}_{(\gamma,-\beta)}^{2}}{4\alpha} \right)^{r}\eta_0,
\end{equation}
which provides us with the correct result for the problem we are studying.  The application of the previous prescription yields, indeed, if $\left|  \dfrac{\beta}{\alpha}\right|  < 1$ 

\begin{equation}\begin{split}\label{funcgen}
& \sqrt{\dfrac{\pi}{\alpha}}\sum_{r=0}^{\infty}\dfrac{1}{r!}\left(\dfrac{\hat{h}_{(\gamma,-\beta)}^{2}}{4\alpha} \right)^{r}\eta_0=\sqrt{\dfrac{\pi}{\alpha}}\sum_{r=0}^{\infty}\dfrac{1}{r!}\dfrac{1}{(2\sqrt{\alpha})^{2r}}H_{2r}(\gamma,-\beta)=\\
& =\sqrt{\dfrac{\pi}{\alpha}}\sum_{r=0}^{\infty}\dfrac{1}{r!}H_{2r}\left( \dfrac{\gamma}{2\sqrt{\alpha}},-\dfrac{\beta}{4\alpha}\right)=\sqrt{\dfrac{\pi}{\alpha+\beta}}e^{\frac{\gamma^{2}}{4(\alpha+\beta)}},
\end{split}\end{equation}
\footnote{We remind that $a^nH_n(x,y)=H_n(ax,a^2y)$ \cite{Babusci}.} which is obtained after using the identity \ref{genfunH2n}.
\end{exmp}

\begin{propert}
The umbral operator defined in eq. \ref{herm} satisfies the identity\footnote{The subscript $\left( \gamma,-\beta\right) $ has been omitted because the identity holds for $\hat{h}$ operators with the same basis, hereafter it will be included whenever necessary. } (see \ref{propertCa})

\begin{equation}
\hat{h}^{m}\hat{h}^{r}=\hat{h}^{m+r},
\end{equation}
$\forall m,r\in\mathbb{R}$. It is also fairly natural to set

\begin{equation}
\partial_{\hat{h}}\hat{h}^{r}\eta_0=r\;\hat{h}^{r-1}\eta_0=rH_{r-1}(\gamma,-\beta).
\end{equation}
Reminding the recurrence \ref{Hnpm} $\partial_{\gamma}H_{r}(\gamma,-\beta)=rH_{r-1}(\gamma,-\beta)$,
%
the "derivative" operator can be identified with

\begin{equation}
\partial_{\hat{h}}\rightarrow \partial_{\gamma}.
\end{equation}
Furthermore since

\begin{equation}
\hat{h}\hat{h}^{r}\eta_0=\hat{h}^{r+1}\eta_0=H_{r+1}(\gamma,-\beta)
\end{equation}
and, on account of the recurrence \ref{Hnpm},

\begin{equation}
H_{r+1}(\gamma,-\beta)=\gamma H_{r}(\gamma,-\beta)-2\; \beta \;r\; H_{r-1}(\gamma,-\beta),
\end{equation}
we can also conclude that $\hat{h}$ itself can be identified with the differential operator

\begin{equation}\label{opPDE}
\hat{h}=\gamma -2\; \beta \;\partial_{\gamma}.
\end{equation}
in according to $\hat{M}$-operator \ref{MPHerma}.
\end{propert}

It is also worth noting that

\begin{lem}
	$\forall x,r,\gamma,\beta\in\mathbb{R}, $
\begin{equation}\begin{split}
\partial_{x}^{r}e^{-\hat{h}x}\eta_0&=(-1)^{r}\hat{h}^{r}e^{-\hat{h}x}\eta_0=(-1)^{r}\sum_{n=0}^{\infty}\dfrac{(-x)^{n}}{n!}\hat{h}^{n+r}\eta_0=\\
& =(-1)^{r}\sum_{n=0}^{\infty}\dfrac{(-x)^{n}}{n!}H_{n+r}(\gamma,-\beta)
\end{split}\end{equation}
and, according to the identity \cite{Babusci}

\begin{equation}\label{Hnlgf}
\sum_{n=0}^{\infty}\dfrac{t^{n}}{n!}H_{n+l}(x,y)=H_{l}(x+2yt,y)e^{xt+yt^{2}},
\end{equation}
we can establish the “rule”

\begin{equation}\label{rule}
\hat{h}^{r} e^{-\hat{h}x}\eta_0=H_{r}(\gamma +2\;\beta \;x,-\beta)e^{-(\gamma x +\beta x^{2})}.
\end{equation}
\end{lem}
We can now make a step further through the 

\begin{exmp}
	Let
\begin{equation}\begin{split}
& I(\gamma,\beta)=\int_{-\infty}^{\infty}e^{-\hat{h}x^{2}}dx\;\eta_0, \quad\forall \alpha\in\mathbb{R},\forall\beta\in\mathbb{R}^+_0,\\
&e^{-\hat{h}x^{2}}\eta_0= e^{-(\gamma x^{2}+\beta x^{4})},
\end{split}\end{equation}
which, after applying the GWI \ref{GWi} writes

\begin{equation}\label{Ih}
I(\gamma,\beta)=\sqrt{\pi}\;\hat{h}^{-\frac{1}{2}}\eta_0,
\end{equation}
which makes sense only if we can provide a meaning for $\hat{h}^{-\frac{1}{2}}$ . The most natural conclusion is that they can be understood as \textbf{fractional order Hermite}, which for our purposes can be defined as it follows \cite{R.Hermann}

\begin{equation}\label{Hermfrazcos}
H_{\nu}(x,-y)=y^{\frac{\nu}{2}}\dfrac{e^{ \frac{x^{2}}{4y}}}{\sqrt{\pi}}\int_{0}^{\infty}e^{-\frac{t^{2}}{4}}t^{\nu}\cos \left( \dfrac{x}{2\sqrt{y}}t-\dfrac{\pi}{2}\nu\right)dt,\quad\forall\nu\in\mathbb{R}, 
\end{equation}
 or as

\begin{equation}\label{Hermfraz}
H_{\nu}(x,-y)=\Gamma(\nu+1)\sum_{r=0}^{\infty}\dfrac{x^{\nu-2r}(-y)^{r}}{\Gamma(\nu+1-2r)r!}, \qquad x>>y,
\end{equation}
which has however a limited range of convergence.\\
The correctness of eq. \ref{Ih} can be readily proved involving either the definitions \ref{Hermfrazcos} and the \ref{Hermfraz} (and it confirm the result obtained in esample \ref{18_20} through another operator).
\end{exmp}
We will comment in examples \ref{Hnoninteg1} and \ref{Hnoninteg2} on the extension of the Hermite polynomials to non-integer index.\\

\begin{lem}
Let us now consider the following repeated derivatives

\begin{equation}\begin{split}
 \partial_{x}^{\; n}e^{-\hat{h}x^{2}}\eta_0&=(-1)^{n}H_{n}(2\;\hat{h}\;x,-\hat{h})e^{-\hat{h}x^{2}}\eta_0=\\
& = (-1)^{n}n!\sum_{r=0}^{\left[ \frac{n}{2}\right] }\dfrac{(-1)^{r}(2x)^{n-2r}}{(n-2r)!r!}\left( \hat{h}^{n-r}e^{-\hat{h}x^{2}}\right)\eta_0.
\end{split}\end{equation}
Thus getting, on account of eq. \ref{rule},

\begin{equation}
\partial_{x}^{\;n}e^{-\hat{h}x^{2}}\eta_0=(-1)^{n}n!\sum_{r=0}^{\left[ \frac{n}{2}\right] }\dfrac{(-1)^{r}(2x)^{n-2r}}{(n-2r)!r!}H_{n-r}(\gamma +2\;\beta x^{2},-\beta)e^{-(\gamma x^{2} +\beta x^{4})},
\end{equation}
in accordance with

\begin{equation}
\partial_{x}^{\;n}e^{-(\gamma x^{2}+\beta x^{4})}=H_{n}^{(4)}(-2\;\gamma x-4\;\beta x^{3},-\gamma-6\;\beta x^{2},-4\;\beta x,-\beta)e^{-(\gamma x^{2}+\beta x^{4})}.
\end{equation}
\end{lem}

\begin{exmp}\label{Hnoninteg1}
In ref. \cite{J.Bohacik} the following integral

\begin{equation}\begin{split}\label{intJ}
& J(a,b,c) = \int_{-\infty}^{\infty} e^{-(ax^{4}+bx^{2}+cx)}dx,\quad\forall b,c\in\mathbb{R},\forall\alpha\in\mathbb{R}^+_0,\\
& Re(a)>0,
\end{split}\end{equation}
has been considered, within the framework of problems regarding the non-perturbative treatment of the anharmonic oscillator. A possible perturbative treatment is that of setting

\begin{equation}\begin{split}
& J(a,b,c) = \sum_{n=0}^{\infty}\dfrac{(-1)^{n}}{n!}g_{n}(a)\left[ H_{n}(c,-b)+H_{n}(-c,-b)\right] ,\\
& g_{n}(a) = \int_{0}^{\infty}x^{n}e^{-ax^{4}}dx=\frac{1}{4}a^{-\frac{n+1}{4}}\Gamma\left( \dfrac{n+1}{4}\right),
\end{split}\end{equation}
which, as noted in \cite{J.Bohacik}, is an expansion with zero radius of convergence in spite of the fact that $J(a,b,c)$ is an entire function for any real or complex value of $b, c$.\\

The use of our point of view allows to write

\begin{equation}\label{Jabc}
J(a,b,c)=\int_{-\infty}^{\infty}e^{-\hat{h}_{(b,-a)}x^{2}-cx}dx\;\eta_0=\sqrt{\dfrac{\pi}{\hat{h}}}e^{\frac{c^{2}}{4\hat{h}}}\eta_0=\sqrt{\pi}\sum_{s=0}^{\infty}\dfrac{1}{s!}\left( \dfrac{c}{2}\right)^{2s}\hat{h}^{-\left( s+\frac{1}{2}\right) }\eta_0.
\end{equation}
We have omitted the subscript $(b,-a)$ in the r.h.s. of eq. \ref{Jabc} to avoid a cumbersome notation. The meaning of the operator $\hat{h}$ raised to a negative exponent is easily understood as 

\begin{equation}
\hat{h}^{-\left( s+\frac{1}{2}\right) }\eta_0=H_{-\left( s+\frac{1}{2}\right)}(b,-a),
\end{equation}
where the negative index Hermite polynomials are expressed in terms of the \textbf{parabolic cylinder functions} $D_{n}$ according to the identity \cite{Abramovitz}

\begin{equation}\label{pcf}
H_{-n}(x,-y)= (2y)^{-\frac{n}{2}}e^{\frac{x^{2}}{8y}} D_{-n}\left( \dfrac{x}{\sqrt{2y}}\right).
\end{equation}

The use of eq. \ref{pcf} in eq. \ref{Jabc} finally yields the same series expansion obtained in ref. \cite{J.Bohacik}

\begin{equation}\label{Jserie}
J(a,b,c)=\sqrt{\pi}\sum_{s=0}^{\infty}\dfrac{1}{s!}\left( \dfrac{c}{2}\right)^{2s}\left(2a \right)^{-\frac{1}{2}\left( s+\frac{1}{2}\right) }e^{\frac{b^{2}}{8a}}D_{-\left( s+\frac{1}{2}\right) }\left( \dfrac{b}{\sqrt{2a}}\right) ,
\end{equation}
which is convergent for any value of $b, c$ and $a>0$.
\end{exmp}

\begin{exmp}\label{Hnoninteg2}
Regarding the use of non-integer  Hermite polynomials it is evident that the definition adopted in eq. \ref{Hermfrazcos} can be replaced by the use of the parabolic cylinder function, it is therefore worth noting that the use of the properties of the $D_n$ functions allows the following alternative form for eq. \ref{Ih} (see ref. \cite{Weisstein} and eq. \ref{pcf})

\begin{equation}\begin{split}
I(\gamma,\beta) & = \sqrt{\pi} \left(2\;\beta \right)^{-\frac{1}{4}}e^{\frac{\gamma^{2}}{8\beta}}D_{-\frac{1}{2}}\left(\dfrac{\gamma}{\sqrt{2\beta}} \right)= \\
& =\sqrt{\dfrac{\gamma}{2\sqrt{2\beta}}} \left(2\;\beta \right)^{-\frac{1}{4}}e^{\frac{\gamma^{2}}{8\beta}} K_{\frac{1}{4}}\left(\dfrac{\gamma^{2}}{8\beta} \right) ,\\
D_{-\frac{1}{2}}(z)&=\sqrt{\dfrac{z}{2\pi}}K_{\frac{1}{4}}\left( \dfrac{1}{4}z^{2}\right) ,
\end{split}\end{equation}
where $K_{\nu}(z)$ is a \textbf{modified Bessel function of the second kind} (see Chapter \ref{Chapter3}).
\end{exmp}

A further example of application of the method developed so far is provided by

\begin{exmp}
\begin{equation}\begin{split}
\int_{0}^{\infty}e^{-(\beta x^{2n}+\gamma x^{n})}dx&=\int_{0}^{\infty}e^{-\hat{h}_{(\gamma,-\beta)}x^{n}}dx\;\eta_0=\dfrac{1}{n}\Gamma\left(\dfrac{1}{n} \right) \hat{h}_{(\gamma,-\beta)}^{-\frac{1}{n}}\eta_0=\\
& =\dfrac{1}{n}\Gamma\left(\dfrac{1}{n} \right)\left(2\beta \right)^{-\frac{1}{2n}}e^{\frac{\gamma^{2}}{8\beta}}D_{-\frac{1}{n}}\left( \dfrac{\gamma}{\sqrt{2\beta}}\right).
\end{split}\end{equation}
 
The use of this family of polynomials allows to cast the integral in eq. \ref{Jabc} in the form

\begin{equation}\label{corr}
J(a,b,c)=\int_{-\infty}^{\infty}e^{\;{}_{4}\hat{h}_{(-c,-a)}x-bx^{2}}dx\;\epsilon_0=\sqrt{\dfrac{\pi}{b}}e^{\frac{\left(\;{}_{4}\hat{h}_{(-c,-a)} \right)^{2} }{4b}}\epsilon_0,
\end{equation}
where

\begin{equation}\label{4h}
\left(\;{}_{4}\hat{h}_{(-c,-a)} \right)^{n}\epsilon_0:=\epsilon_n=H_{n}^{(4)}(-c,-a)=(-1)^{n}n!\sum_{r=0}^{\left[ \frac{n}{4}\right] }\dfrac{c^{n-4r}(-a)^{r}}{(n-4r)!r!},
\end{equation}
with $H_{n}^{(4)}(c,-a)$ being a \textbf{fourth order} Hermite Kamp\'e de F\'eri\'et \cite{Appell} polynomial.\\
The series expansion of the right hand side of eq. \ref{corr} in terms of fourth order Hermite  converges in a much more limited range than the series \ref{Jserie} and has been proposed to emphasize the possibilities of the method we have proposed so far.
\end{exmp}

The bynomial umbral procedure emphasized so far can be extended to \textbf{Higher Order Hermite Polynomials} as shown in Appendix \ref{higherHermite}.

\begin{exmp}
According to our formalism the Pearcey integral, widely studied in optics,  is easily reduced to a particular case of eq. \ref{intJ}, within the framework of diffraction problems \cite{Lopez}, namely

\begin{equation}
J(1,x,-iy)=\int_{-\infty}^{\infty}e^{-(t^{4}+xt^{2})+iyt}dt=\sqrt{\dfrac{\pi}{\hat{h}_{(x,-1)}}}e^{-\frac{y^{2}}{4\hat{h}_{(x,-1)}}}\eta_0,\quad\forall x,y\in\mathbb{R}
\end{equation}
and can be expressed in terms of parabolic cylinder functions, as indicated in example \ref{Hnoninteg2}. It is perhaps worth stressing that, in the literature a converging  series for the Pearcey integral is given in the form \cite{Berry}

\begin{equation}\begin{split}
J(1,x,-iy)&=\int_{-\infty}^{\infty}e^{-t^{4}-\hat{h}_{(iy,-x)}t}dt\;\eta_0=\int_{0}^{\infty}e^{-t^{4}}\left( e^{\hat{h}_{(iy,-x)}t}+e^{\hat{h}_{(-iy,-x)}t}\right) dt\;\eta_0=\\
& =2\sum_{n=0}^{\infty}(-1)^{n}g_{2n}(1)a_{2n}(x,y),
\end{split}\end{equation}
with

\begin{equation}\begin{split}
& a_{0}(x,y)=1,\\
& a_{1}(x,y)=y,\\
& a_{n}(x,y)=\dfrac{1}{n}\left(y\;a_{n-1}(x,y)+2\;x\;a_{n-2}(x,y) \right),
\end{split}\end{equation}
which is reconciled with our previous result, in terms of two variable Hermite polynomials, provided that one recognizes

\begin{equation}
J(1,x,-iy)=\sum_{n=0}^{\infty}\dfrac{(-1)^{n}}{n!}g_{n}(1)\left[ H_{n}(-iy,-x)+H_{n}(iy,-x)\right] .
\end{equation}
\end{exmp}

In this section we have provided some hint on the use of the Hermite calculus to study integral forms with specific application in different field of research. In Chapter \ref{ChapterOP} we will show how the method can be extended to a systematic investigation of the Voigt functions and  to their relevant generalizations \cite{Pathan}.\\

A further important application is the use of the method within the framework of \textit{evolutive PDE}. In order to start, we consider the following straightforward example.

\begin{exmp}
	Let, $\forall x,\alpha,\beta\in\mathbb{R},\forall t\in\mathbb{R}^+_0$,
\begin{equation}\label{my}
\left\lbrace  \begin{array}{l} \partial_{t}F(x,t)=\left(\alpha\partial_{x}+\beta\partial_{x}^{2} \right) F(x,t)\\[1.6ex]
F(x,0)=f(x), \end{array}\right.
\end{equation}
whose (formal) solution is easily obtained as

\begin{equation}\label{41}
F(x,t)=e^{(\alpha t)\partial_{x}+(\beta t)\partial_{x}^{2}}f(x).
\end{equation}
The use of the formalism developed so far allows to write the rhs of eq. \ref{41} in the form

\begin{equation}
F(x,t)=e^{\hat{h}_{(\alpha t, \;\beta t)}\partial_{x}}f(x)\eta_0 ,
\end{equation}
by the use of standard exponential rules we obtain

\begin{equation}\label{trasl}
F(x,t)=f\left(x+ \hat{h}_{(\alpha t, \;\beta t)}\right)\eta_0 ,
\end{equation}
which is still a formal solution unless we provide a meaning for the rhs of eq. \ref{trasl}. Let us therefore use the Fourier transform method to write

\begin{equation}\begin{split}\label{meanFourier}
f\left(x+ \hat{h}_{(\alpha t, \;\beta t)}\right)\eta_0&=\dfrac{1}{\sqrt{2\pi}}\int_{-\infty}^{\infty}\tilde{f}(k)e^{ikx+ik\hat{h}_{(\alpha t, \;\beta t)}}dk\;\eta_0=\\
& = \dfrac{1}{\sqrt{2\pi}}\int_{-\infty}^{\infty}\tilde{f}(k)e^{ik(x+\alpha t)-k^{2}\beta t}dk,
\end{split}\end{equation}
which is a kind of Gabor transform \cite{Gabor}. It is evident that the same result can be obtained with ordinary means, we have used this example to prove the correctness and flexibility of the method we propose.\\

Let us now specialize the result to the case $f(x)=x^{n}$ and write

\begin{equation}\begin{split}\label{nFourier}
f\left(x+ \hat{h}_{(\alpha t, \;\beta t)}\right)\eta_0&=\left(x+ \hat{h}_{(\alpha t, \;\beta t)}\right)^{n}\eta_0=\sum_{s=0}^{n}\binom{n}{s}x^{n-s}\hat{h}_{(\alpha t, \;\beta t)}^{s}\eta_0,\\
& =H_{n}(x+\alpha t, \;\beta t)
\end{split}\end{equation}
which is just the derivation from a different point of view of the following operational identity \cite{Babusci}

\begin{equation}
e^{\kappa\partial_{x}+\lambda \partial_{x}^{2}}x^{n}=H_{n}(x+\kappa,\lambda),\quad\forall x,\kappa,\lambda\in\mathbb{R}.
\end{equation}
\end{exmp}
%
%
%

The possibilities for the applicability of the integration method discussed in this section arise if, inside the integrand, an exponential generating function is recognized. \\

To clarify this point we note that the integral

\begin{equation}\begin{split}\label{A}
& f(a,b,c)=\int_{-\infty}^{\infty}e^{-ax^{2}+\sqrt{x^{2}+bx+c}}dx,\\
& b^{2}-4c<0,\\
& a>1,
\end{split}\end{equation}
can be written as

\begin{equation}\label{B}
f(a,b,c)=\sqrt{\dfrac{\pi}{a}}e^{\frac{\hat{R}^{2}}{4a}}=\sqrt{\dfrac{\pi}{a}}\sum_{n=0}^{\infty}\dfrac{1}{n!}\left( \dfrac{1}{4a}\right)^{n} R_{2n}(b,c),
\end{equation}
provided that

\begin{equation}\begin{split}
& e^{\sqrt{x^{2}+bx+c}}=e^{\hat{R}x}=\sum_{n=0}^{\infty}\dfrac{x^{n}}{n!}\hat{R}^{n},\\
& \hat{R}^{n}=R_{n}(b,c),
\end{split}\end{equation}
where  $R_{n}(b,c)$ are polynomials of the parameter $b, c$.\\

Even though such a polynomials expansion can be obtained using different procedure, we have tested the validity of our ansatz using the following integral definition

\begin{equation}
R_{m}(b,c)=\dfrac{m!}{2\pi}\int_{0}^{2\pi}e^{-im\phi}e^{\sqrt{e^{2i\phi}+be^{i\phi}+c}}d\phi,
\end{equation}
which has been used to benchmark the identity \ref{B}, with the full numerical integration of \ref{A}. 

\section{The Negative Derivative Operator Method and the Associated Technicalities}\label{NegDerOpMeth}

We have already stressed that the theory of fractional derivatives has very venerable roots. More than two centuries ago \textit{Euler} started his studies on this subject \cite{Oldham} and his achievements provided a powerful and universal frame for the solution of problems, involving fractional derivatives either in pure and applied mathematics.\\
We have also underscored that, to date, the fractional derivative studies have been merged within a wider context, including  the wealth of knowledge of the special functions and with such a comprehensive tool as the operational method.\\

\noindent Even though we have already presented the formalism of fractional derivative operator, in which the exponent is any positive or negative real number, we devote this section to the properties of derivative operators with negative integer exponents, whose relevant formalism has been sporadically discussed in the mathematical literature in relation with various problems, ranging from integral equations to applications in laser physics \cite{Ciocci,Kondo}. We will see, in the following, that such a formalism, originated by \cite{Ricci}, is fairly rich and has interesting consequences, when merged with the umbral procedure we are developing here.\\

\noindent Operational methods, expanded within the context of the fractional derivative formalism, have opened new possibilities in the application of Calculus.  Even classical problems, with well-known solutions, may acquire a different flavor, if viewed within such a perspective which, if properly pursued, may allow further progresses disclosing new avenues for their study and generalizations.\\
It is indeed well known that the operation of integration is the inverse of that of derivation; however such a statement, by itself, does not enable a formalism to establish rules to handle integrals and derivatives on the same footing. \\

\noindent An almost natural environment to place this specific issue is the formalism of real order derivatives, in which the distinction between integrals and derivatives becomes superfluous.\\
The use of the formalism associated with the fractional order operators offers new computational tools as e.g. the extension of the concept of integration by parts. Within such a context we get

\begin{prop}
 The integral of a function $f\in C^\infty$ can be written in terms of the series \cite{Ricci} 

\begin{equation}\label{intfunct}
\int_{0}^{x}f(\xi)d\xi=\sum_{s=0}^{\infty}(-1)^s\dfrac{x^{s+1}}{(s+1)!}f^{(s)}(x), \quad \forall x\in\mathbb{R}, 
\end{equation}
where $f^{(s)}(x)$ denotes the $s^{th}$-derivative of the integrand function. 
\end{prop}
\noindent To proof eq.  \ref{intfunct} we give the following

\begin{defn}\label{DerNegOp}
	$\forall x\in\mathbb{R}, \forall f\in C^\infty$, let
\begin{equation}\begin{split}\label{key}
& \int_{0}^{x}g(\xi)f(\xi)d\xi={}_{0}\hat{D}_x^{-1}\left(g(x)f(x )\right), \\
& {}_{\alpha}\hat{D}_x^{-1}s(x)=\int_{\alpha}^{x}s(\xi)d\xi,\\
& g(x)=1.
\end{split}\end{equation}
 ${}_{\alpha}\hat{D}_x^{-1}$ is the \textbf{negative derivative operator}.
 \end{defn}
\noindent Then
\begin{proof}[\textbf{Proof.}]
We rewrite eq. \ref{intfunct} according to Definition \ref{DerNegOp} and by the use of a slightly generalized form of the \textit{Leibniz formula}, written as

\begin{equation}\label{negder}
{}_{0}\hat{D}_x^{-1}\left(g(x)f(x )\right)=\sum_{s=0}^{\infty}\binom{-1}{s}g^{(-1-s)}(x)f^{(s)}(x),
\end{equation}
we provide the proof after taking $g(x)=1$ and  after noting that

\begin{equation}\begin{split}\label{key}
& \binom{-1}{s}=(-1)^s, \\
& g^{(-1-s)}(x)=\dfrac{x^{s+1}}{(s+1)!}.
\end{split}\end{equation}
\end{proof}

The interesting element of such an analytical tool is that it allows the evaluation of the primitive of a function in terms of an automatic procedure, analogous to that used in the calculus  of the derivative of a function. At the same time it marks the conceptual, even though not formal, difference between the two operations. The integrals give rise to a computational procedure involving, most of the times, an infinite number of steps. Eq. \ref{negder} becomes useful if e.g. the function $f(x)$ has peculiar properties under the operation of derivation, like being cyclical,  vanishing after a number of steps or other.\\

The formalism, we have just envisaged can be combined, e.g., with the properties of the special polynomials to find useful identities. In the case of two variable $HP$, $H_n(x,y)$, satisfying the property \cite{Babusci}

\begin{propert}
\begin{equation}\begin{split}\label{key}
& \partial_x^s H_n(x,y)=\dfrac{n!}{(n-s)!}H_{n-s}(x,y) , \\
& \partial_y^s H_n(x,y)=\dfrac{n!}{(n-2s)!}H_{n-2s}(x,y),
\end{split}\end{equation}
\end{propert}
we obtain the following definite integrals 

\begin{exmp}
	$\forall x,y\in\mathbb{R}$, we get
\begin{equation}\begin{split}\label{key}
\int_{0}^xH_n(\xi,y)d\xi&=\sum_{s=0}^n \dfrac{(-1)^s\; x^{s+1}}{(s+1)!}\dfrac{n!}{(n-s)!}H_{n-s}(x,y)=\\
& =\sum_{s=0}^n \dfrac{ x}{(s+1)}\binom{n}{s}(-x)^s H_{n-s}(x,y), \\
\int_{0}^y H_n(x,\eta)d\eta&=\sum_{s=0}^n \dfrac{(-1)^s\; y^{s+1}}{(s+1)!}\dfrac{n!}{(n-2s)!}H_{n-2s}(x,y)
\end{split}\end{equation}
and

\begin{equation}\begin{split}\label{key}
& \int_{0}^xH_n(\xi,y)\cos(\xi)d\xi=\sum_{s=0}^n (-1)^s\dfrac{\cos\left(x+ s\dfrac{\pi}{2}\right) }{(s+1)!}\dfrac{n!}{(n-s)!}H_{n-s}(x,y),\\
& \int_{0}^y H_n(x,\eta)\cos(\eta)d\eta=\sum_{s=0}^n (-1)^s\dfrac{\cos\left(y+ s\dfrac{\pi}{2}\right) }{(s+1)!}\dfrac{n!}{(n-2s)!}H_{n-2s}(x,y).
\end{split}\end{equation}
\end{exmp}
Furthermore, by taking into account that (generalization of eq. \ref{GHPol}) \cite{Babusci}

\begin{equation}\label{vasteH}
\partial_x^s\;e^{ax^2+bx}=H_s(2ax+b,a)e^{ax^2+bx}, \quad\forall a,b,x\in\mathbb{R}
\end{equation}
and reminding that $(-1)^nH_n(x,y)=H_n(-x,y)$ \cite{Babusci}, we find

\begin{exmp}
\begin{equation}\label{intgauH}
\int_{0}^x e^{a\xi^2+b\xi}d\xi=\sum_{s=0}^{\infty}\dfrac{x^{s+1}}{(s+1)!}H_s(-2ax-b,a)e^{ax^2+bx}
\end{equation}
and

\begin{equation}\label{intgauHbis}
\int_{0}^x e^{a\xi^2+b\xi}\cos(\xi)d\xi=\sum_{s=0}^{\infty}\dfrac{\cos\left(x+ s\dfrac{\pi}{2}\right) }{(s+1)!}H_s(-2ax-b,a)e^{ax^2+bx}.
\end{equation}
\end{exmp}

 The right hand side of eq. \ref{intgauH} can be viewed as the primitive of the \textit{erfc} function (see Chapter \ref{ChapterOP}).\\

We can now merge umbral and negative derivative methods to get further results.\\

The use of the properties of the Gaussian functions under repeated derivatives, as e.g. the generalized form of \textit{HP} \ref{GHPol}  $\partial_x^n\;e^{ax^2}=H_n(2ax,a)e^{ax^2},
$
allows the derivation of its \textit{Bessel} umbral counterpart identity (see later Chapter \ref{Chapter3}), namely

\begin{exmp}
Let $J_0(x)$ the $0$-order Bessel function $\forall x\in\mathbb{R}$, by using eq. \ref{J0op}	and the $J_n(x)$-umbral image \ref{derivsucc} (proved in Chapter \ref{Chapter3}) we get
\begin{equation}\begin{split}\label{key}
\partial_x^n\;J_0(x)&=\partial_x^n\;e^{-\hat{c}\left( \frac{x}{2}\right)^2 }\varphi_0=H_n\left( -\hat{c}\frac{x}{2},-\frac{\hat{c}}{4}\right) e^{-\hat{c}\left( \frac{x}{2}\right)^2 }\varphi_0=\\
& =(-1)^n n!\sum_{r=0}^{\lfloor \frac{n}{2}\rfloor}\dfrac{(-1)^r x^{n-2r}\hat{c}^{n-r}}{2^{n-2r}2^r(n-2r)!r!}\sum_{s=0}^\infty\dfrac{(-1)^s\left(\frac{x}{2} \right)^{2s} \hat{c}^s}{s!} \varphi_{0}=\\
& =(-1)^n n!\sum_{r=0}^{\lfloor \frac{n}{2}\rfloor}\dfrac{(-1)^r\left(\frac{x}{2} \right)^{-r}}{2^r(n-2r)!r!}\sum_{s=0}^\infty\dfrac{(-1)^s\left(\frac{x}{2} \right)^{2s+(n-r)}}{s!(s+(n-r))!}=\\
& =
(-1)^n n!\sum_{r=0}^{\lfloor \frac{n}{2}\rfloor}\dfrac{(-2)^{-r}x^{-r}}{r!(n-2r)!}J_{n-r}(x).
\end{split} \end{equation}
 We can therefore “translate” the identity \ref{intgauH} as

\begin{equation}\label{key}
\int_{0}^xJ_0(\xi)d\xi=\sum_{s=0}^{\infty}\dfrac{x^{s+1}}{(s+1)!}\left(s! \sum_{r=0}^{\lfloor \frac{n}{2}\rfloor}\dfrac{(-2)^{-r}x^{-r}}{r!(s-2r)!}\right)J_{s-r}(x). 
\end{equation} 
\end{exmp}

%
%

Further integral transforms can be framed within the same context. 

\begin{exmp}
The well-known identities \cite{Abramovitz} $\forall x\in\mathbb{R}^+$

\begin{equation}\begin{split}\label{key}
& \int_{0}^{\infty} J_0\left(2\sqrt{xu} \right)\sin(u)du=\cos(x),\\
& \int_{0}^{\infty} J_0\left(2\sqrt{xu} \right)\cos(u)du=\sin(x),  
\end{split}\end{equation}
can be proved by following the method illustrated so far. They can indeed be easily stated as follows 

\begin{equation}\begin{split}\label{key}
\int_{0}^{\infty} J_0\left(2\sqrt{xu}\right) e^{iu}du&
=\int_{0}^{\infty} e^{-(\hat{c}x-i)u}du\;\varphi_0
=\dfrac{1}{\hat{c}x-i}\varphi_0
=i\sum_{r=0}^{\infty}(-i\hat{c}x)^r\varphi_0=\\
& =-ie^{-ix}.
\end{split}\end{equation}
\end{exmp}
Furthermore, 

\begin{exmp}
Reminding the Lemma \ref{lem2C}, which provides $
J_0\left(2\sqrt{x}\right)=C_0(x)=\sum_{r=0}^{\infty}\dfrac{(-x)^r}{r!^2},
$ where $C_0(x)$ is the $0$-order Tricomi-Bessel function, satisfying the identity $\forall x\in\mathbb{R}, \forall s\in\mathbb{R}^+_0$ \cite{Tricomi}

\begin{equation}\begin{split}\label{TrBn}
& \partial_x^s\;C_0(x)=(-1)^sC_s(x),\\
& C_s(x)=\sum_{r=0}^{\infty}\dfrac{(-x)^r}{r!(r+s)!},
\end{split}\end{equation}
we find

\begin{equation}\begin{split}\label{key}
& \int_0^{\infty }u^sC_s(xu)\sin(u)du=(-1)^s\cos\left(x+s\dfrac{\pi}{2} \right), \\
& \int_0^{\infty }u^sC_s(xu)\cos(u)du=(-1)^s\sin\left(x+s\dfrac{\pi}{2} \right)
\end{split}\end{equation}
and, by a straightforward application of eqs. \ref{J0op}-\ref{C0umbral}, we also obtain

\begin{equation}\label{key}
\int_{0}^{\infty}C_0(xu)J_0(u)du=J_0(x).
\end{equation}
\end{exmp}

In the forthcoming section we use the same formalism to introduce the \textsl{\textbf{Laguerre polynomial}} families.\\


\section{ Laguerre Polynomials and the Relevant Umbral Forms}

In this section we see how the different concepts and computational tools, developed in the previous sections, can be merged to provide a “non standard” point of view to the theory of the two variable \textbf{\textit{Laguerre polynomials}} (\textit{LP}), defined by the series (see \cite{Germano} and references therein)

\begin{equation}\label{LagP}
L_n(x,y)=n!\sum_{r=0}^n \dfrac{(-1)^ry^{n-r}x^r}{(n-r)!r!^2}, \quad \forall x,y\in\mathbb{R}, \forall n\in\mathbb{N}.
\end{equation}

\begin{prop}
On account of the fact that, by Definition \ref{DerNegOp}, 

\begin{equation}\label{DerOne}
{}_0\hat{D}_x^{-r}1=\dfrac{x^r}{r!}, \quad\forall x\in\mathbb{R}^+_0,
\end{equation}
we can cast eq. \ref{LagP} in the form

\begin{equation}\label{LPder}
L_n(x,y)=n!\sum_{r=0}^n \dfrac{(-1)^r y^{n-r}{}_0\hat{D}_x^{-r}}{(n-r)!r!}=\left( y-{}_0\hat{D}_x^{-1}\right)^n 
\end{equation}
(the $1$ on the right side of the operator has been omitted for brevity).\\

The $LP$ have therefore been transformed in a $Newton$ bynomial involving a negative derivative operator.
\end{prop}

\begin{propert}
It is fairly natural to infer from the operational definition in eq. \ref{LPder} that
\begin{itemize}
	\item [$a)$] 
	\begin{equation}\label{MultL}
	\left( y-{}_0\hat{D}_x^{-1}\right)\left( y-{}_0\hat{D}_x^{-1}\right)^n=\left( y-{}_0\hat{D}_x^{-1}\right)^{n+1},
	\end{equation}	
	\item [$b)$] being ${}_0\hat{D}_x^{-1}\partial_{x}=\hat{1}$, it is easily checked by recursivity that 
	\begin{equation}\label{DerOpL}
	-\partial_{x}x\partial_{x}\left( y-{}_0\hat{D}_x^{-1}\right)^n=nL_{n-1}(x,y)  
	\end{equation}
	and that 
	\begin{equation}\label{key}
	\left[-\partial_{x}x\partial_{x},  y-{}_0\hat{D}_x^{-1}\right] =1.
	\end{equation}
\end{itemize} 	
\end{propert}

\noindent Then, we can provide

\begin{prop}
According to Definition \ref{QMdef}, eqs. \ref{MultL}-\ref{DerOpL} define the multiplicative and derivative operators of $LP$, namely

\begin{align}
& \hat{M}=y-\hat{D}_x^{-1}, \\
& \hat{P}=-\partial_{x}x\partial_{x}.\label{prove}
\end{align}
In particular the eq. \ref{prove} is referred, in the relevant mathematical literature, as \textbf{Laguerre derivative} ${}_l\partial_x$. 
\end{prop}

A number of consequences can be drawn from the previous formalism. 

\begin{propert} We have\\
\begin{itemize}
	\item [$i)$]   \textbf{\textit{Differential equation satisfied by Laguerre operators}}  
			
	\begin{equation}\label{key}
	-\left(\left( y-\hat{D}_x^{-1}\right) \partial_{x}x\partial_{x} \right)L_n(x,y)=nL_n(x,y), 
	\end{equation}
	which yields
	\begin{equation}\begin{split}\label{key}
	& y\;x\;Z^{''}+(y-x)Z'+nZ=0, \\
	& Z=L_n(x,y).
	\end{split}\end{equation}
	Even though not explicitly mentioned the two variable Laguerre reduces to their ordinary counterpart \cite{L.C.Andrews} for $y=1$.
	
	\item [$ii)$] \textbf{\textit{Laguerre heat equation (LHE)}}\\
	
	The following identity naturally follows from the previous identities \cite{Babusci}
	
	\begin{equation}\label{LPDE}
	\left\lbrace  \begin{array}{l}
	 \partial_{y}L_n(x,y)=-\partial_{x} x\partial_{x}L_n(x,y)\\[1.1ex]
	 L_n(x,0)=\dfrac{(-x)^n}{n!}.
	\end{array}\right.
	\end{equation}
	It is evident that eq. \ref{LPDE} provides the further operational definition of $LP$
	
	\begin{equation}\label{key}
	L_n(x,y)=e^{-y\partial_{x} x\partial_{x}}\dfrac{(-x)^n}{n!},
	\end{equation}
	which can be checked after the use of the (surprising) identity
	
	\begin{equation}\label{LagDers}
	\left(\partial_{x} x\partial_{x} \right)^s=\partial_{x}^s x^s\partial_{x}^s, 
	\end{equation}
	in turn proved by induction.\\
	
	Eq. \ref{LPDE} is an initial value problem analogous to the heat equation satisfied by the $HP$, this justifies its definition of $LHE$.\\ 
	In example \ref{LHEexmp} we will discuss relevant solutions involving different initial condition.
	
	\item [$iii)$] \textbf{\textit{Associated generating function(s) as straightforward consequence of the envisaged formalism}}\\
	
	We find indeed
	
	\begin{equation}\begin{split}\label{key}
	\sum_{n=0}^{\infty}t^nL_n(x,y)&=\dfrac{1}{1-\left(y-\hat{D}_x^{-1} \right)t }=\dfrac{1}{1-yt}\left( \dfrac{1}{1+\frac{t\hat{D}_x^{-1}}{1-yt}}\right)=\\
	& = \dfrac{1}{1-yt}\sum_{r=0}^{\infty}  \dfrac{(-1)^r}{(1-yt)^r}\hat{D}_x^{-r}.
	\end{split}\end{equation}
	Finally, by the use of eq. \ref{DerOne}, we find	
	
	\begin{equation}\begin{split}\label{key}
	& \sum_{n=0}^{\infty}t^nL_n(x,y)=\dfrac{1}{1-yt}e^{-\frac{xt}{1-yt}},\\
	& \mid yt\mid<1.
	\end{split}\end{equation}
	Other generating functions are obtained by combining elementary algebraic rules and those associated with the negative derivative operator, like e.g.
	
	\begin{equation}\label{key}
	\sum_{n=0}^{\infty}t^nL_n(x,y)=\sum_{n=0}^{\infty}t^n\left(y-\hat{D}_x^{-1} \right)^n=e^{yt}e^{-t\hat{D}_x^{-1}}=e^{yt}\sum_{r=0}^{\infty} \dfrac{(-t)^r}{r!}\hat{D}_x^{-r},
	\end{equation}
	thus allowing the conclusion
	
	\begin{equation}\label{key}
	\sum_{n=0}^{\infty}t^nL_n(x,y)=e^{yt}\sum_{r=0}^{\infty} \dfrac{(-tx)^r}{r!^2}=e^{yt}C_0(xt).
	\end{equation}
\end{itemize} 
	\end{propert}
	
	\begin{exmp}\label{LHEexmp}
	Let us now apply the previous results to the solution of the $LHE$
	
	\begin{equation}\label{cauchy} 
	\left\lbrace  \begin{array}{l} \partial _{y} F(x,y)=-\partial _{x}x\partial_x\; F(x,y)\\[1.6ex]
	F(x,0)=e^{-x}. \end{array}\right.
	\end{equation} 
	
	The use of the evolution operator method yields
	
	\begin{equation}\label{key}
	F(x,y)=\dfrac{1}{1-y}e^{-\frac{x}{1-y}},
	\end{equation}
	which has a limited convergence radius, fixed by $\mid y\mid<1$.\\
	
	In the case of a generic initial function $F(x,0)=g(x)$, the solution of the $LHE$ reads (see Appendix \ref{LagTransff} for the relevant proof)
	
	\begin{equation}\label{key}
	F(x,y)=e^{\frac{x}{y}}\int_{0}^{\infty}e^{-s}C_0\left( \dfrac{x}{y}s\right)g(-xs)ds, 
	\end{equation}
	which plays the same role of the Gauss- Weierstrass transform and can be exploited to infer the LP orthoganal properties.
	\end{exmp}	
	
	Before proceeding further we comment on the geometrical properties of \textit{LP}, by taking advantage from eqs. \ref{LPDE}, the procedure is analogous to that already discussed for \textit{HP} and we find indeed what is shown in Figs. \ref{FigLP}.\\	
	The previous formulas suggest geometrical representations for the two-variable \textit{LP}, which are displayed in the graphics. The exponential operator transforms an ordinary monomial $\dfrac{(-x)^n}{n!}$ into a Laguerre type polynomial. The monomial-to-polynomial transition is shown by moving the cutting plane orthogonal to the $y$ axis. For a specific value of the polynomial degree $n$, the polynomials  lie on the cutting $L_n(x,y)$ plane.\\ 

\begin{figure}[htp]
	\centering
	\begin{subfigure}[c]{0.49\textwidth}
		\includegraphics[width=0.9\linewidth]{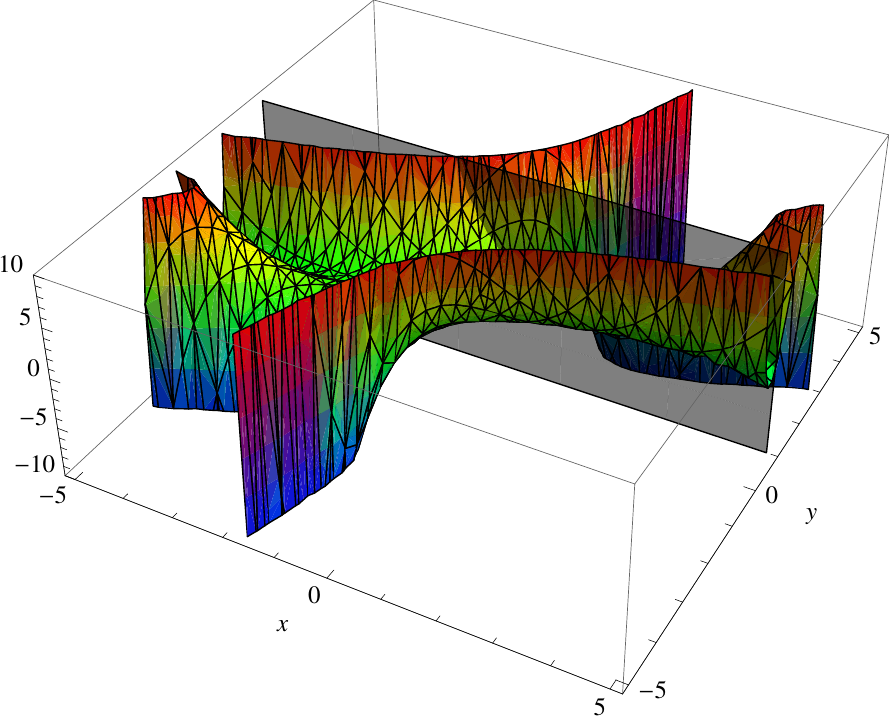}
		\caption{$L_4(x,y)$ cutted by the plane $y=1$.}
		\label{artMotz1}
	\end{subfigure}
	\begin{subfigure}[c]{0.49\textwidth}
		\includegraphics[width=0.9\linewidth]{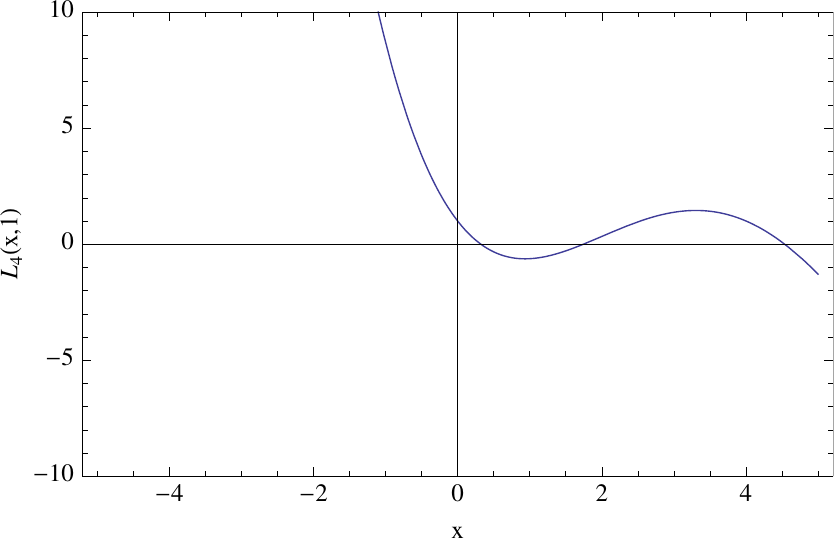}
		\caption{$L_4(x,1)$.}
	\end{subfigure}
	\\[2mm]
	\begin{subfigure}[c]{0.49\textwidth}
		\includegraphics[width=0.9\linewidth]{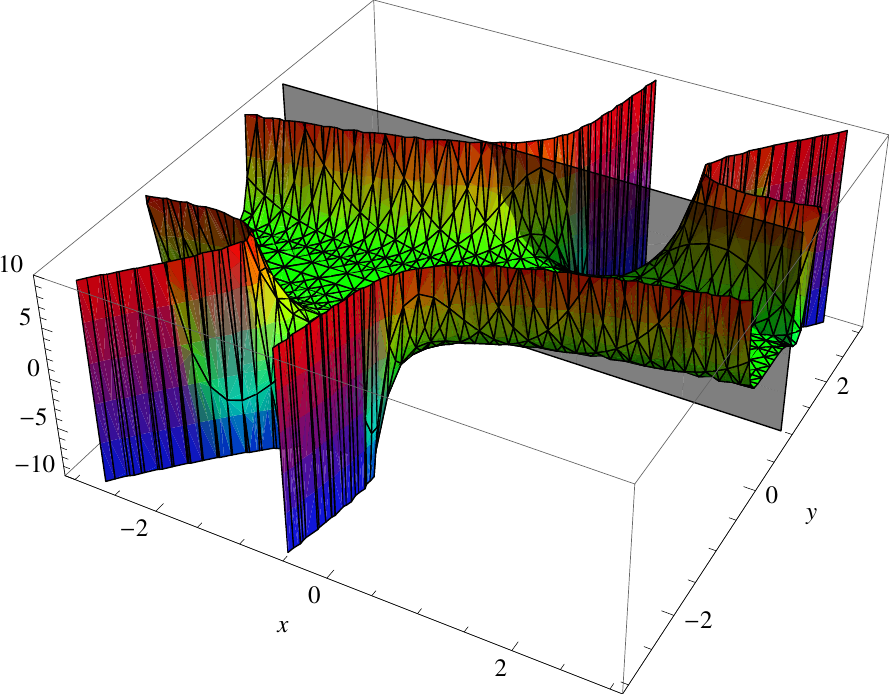}
		\caption{$L_8(x,y)$ cutted by the plane $y=1$.}
	\end{subfigure}
	\begin{subfigure}[c]{0.49\textwidth}
		\includegraphics[width=0.9\linewidth]{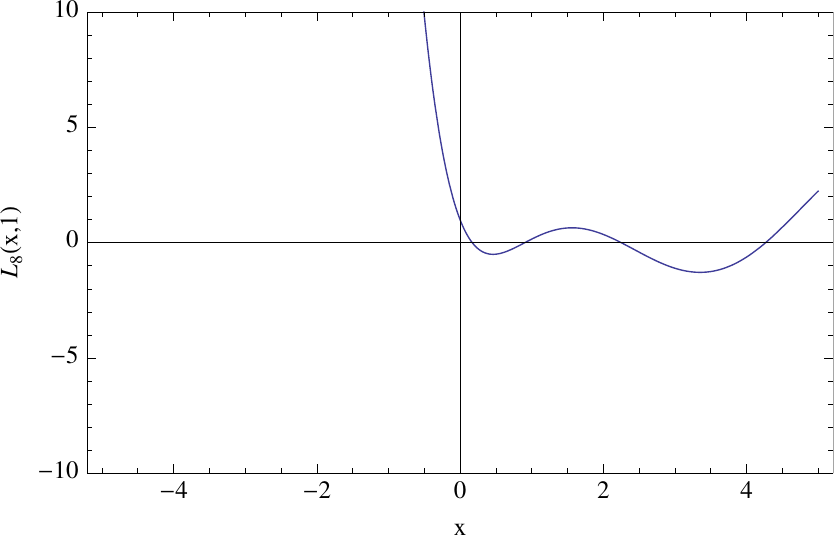}
		\caption{$L_8(x,1)$.}
	\end{subfigure}
	\\[2mm]
	\begin{subfigure}[c]{0.49\textwidth}
		\includegraphics[width=0.9\linewidth]{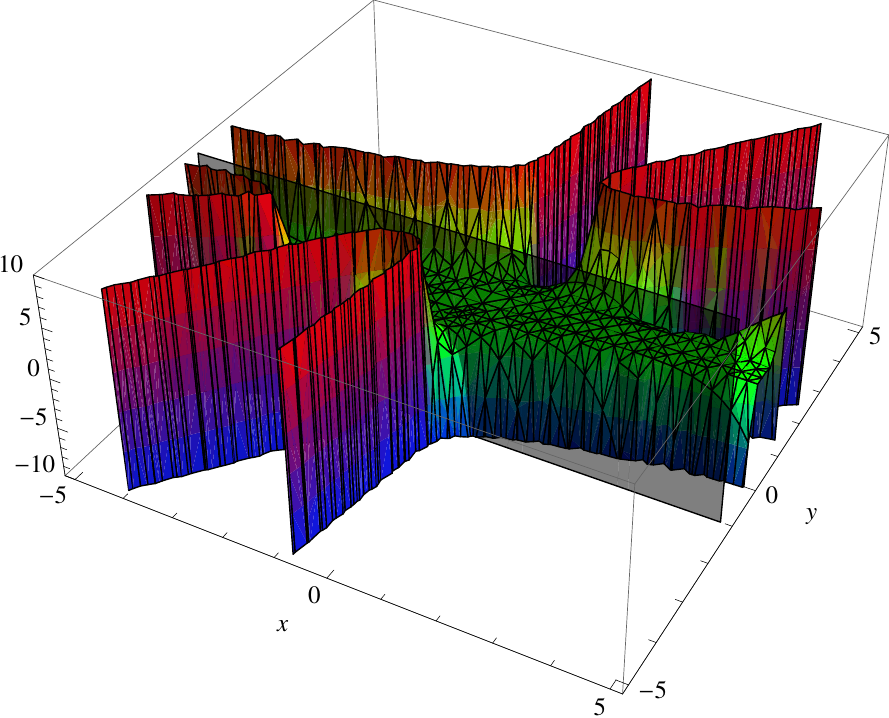}
		\caption{$L_9(x,y)$ cutted by the plane $y=-1$.}
	\end{subfigure}
	\begin{subfigure}[c]{0.49\textwidth}
		\includegraphics[width=0.9\linewidth]{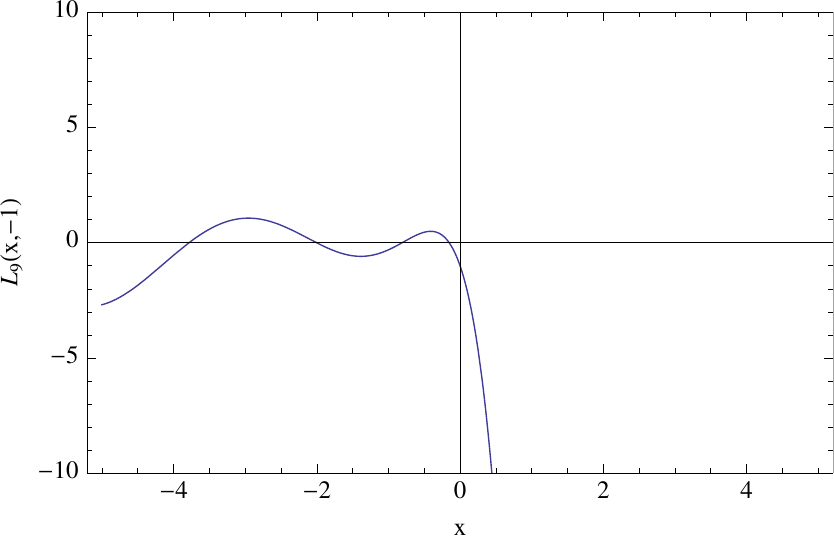}
		\caption{$L_9(x,-1)$.}
	\end{subfigure}
	\caption{Geometrical representation of two-variable Laguerre polynomials in 3D and 2D, for different $n$ and $y$ values \cite{MotzkinWolfram}.}\label{FigLP} 
\end{figure}
	\newpage		
	It is evident that either for Laguerre and Hermite polynomials the exponential operator acting on the monomial is a kind of shift in the y variable and indeed we have
	
	\begin{equation}\label{key}
	\begin{split}
	& e^{-y_2 \partial_x x \partial_x}L_n(x,y_1)=L_n(x,y_1 +y_2),\\
	& e^{y_2 \partial_x^2}H_n(x,y_1)=H_n(x,y_1 +y_2).
	\end{split}
	\end{equation}
	
	Furthermore, to recover monomials from $HP$ and $LP$ it will be enough to set
	
	\begin{equation}\label{key}
	\begin{split}
	& e^{y \partial_x x \partial_x}L_n(x,y)=L_n(x,0)=\dfrac{(-x)^n}{n!},\\
	& e^{y \partial_x^2}H_n(x,y)=H_n(x,0)=x^n.
	\end{split}
	\end{equation}

In the forthcoming section we discuss the transition towards the Umbral treatment.

\section{The Umbral Version of Laguerre and Hermite Associated Polynomials} 

Most of the results we have derived in the previous section can be deduced using the umbral formalism as it results evident from the following\footnote{We note that in similar way we can write  from eq. \ref{Opc}
	
	\begin{equation} \label{GrindEQ__6_OpOrd} 
	L_{n}^{} (x,\, y)=\sum _{s=0}^{n}\binom{n}{s} \, y^{n-s} \left(-\hat{c}x\right)^{s} \, \, \varphi _{\, 0} =\sum _{s=0}^{n}\,  \frac{(-1)^s n!}{\left(s!\right)^{ 2} (n-s)!} \, y^{n-s} \, x^{s} . 
	\end{equation} }

\begin{prop}
	$\forall x,y\in\mathbb{R}, \forall n\in\mathbb{N}$, by using eq. \ref{Opc} and the same procedure of Proposition \ref{propHpol}, we provide
\begin{equation}\label{LagDefPhi}
L_n(x,y)=(y-\hat{c}x)^n\varphi_{0}.
\end{equation} 
\end{prop}

The use of the previous proposition allows a significant simplification of the theory of $LP$ and eq. \ref{LagDefPhi} will be largely exploited in the forthcoming chapters. We note that such a point of view has opened new avenues in the derivation of \textit{Lacunary  Generating Functions} \cite{lacunary} and for the relevant combinatorial interpretation \cite{Strehl}.\\

Let us note that eq. \ref{LagDefPhi} suggests that the definition can be extended to negative real value of the index. In this case will not deal with polynomials any more but with Laguerre functions.\\

\begin{prop}
We provide the integral form of negative order LP 

\begin{equation}\label{Lneg}
L_{-\nu}(x,y)=\dfrac{1}{\Gamma(\nu)}\int_{0}^{\infty}s^{\nu-1}e^{-sy}C_0(-sx)ds, \quad\forall\nu\in\mathbb{R}^+,\forall x,y\in\mathbb{R}.
\end{equation}
\end{prop}
\begin{proof}[\textbf{Proof.}]
	By applying  Laplace transform and eq. \ref{C0umbral} we get
	\begin{equation}\begin{split}\label{key}
L_{-\nu}(x,y)&=(y-\hat{c}x)^{-\nu}\varphi_{0}=\dfrac{1}{\Gamma(\nu)}\int_{0}^{\infty}s^{\nu-1}e^{-sy}e^{\hat{c}xs}\varphi_0ds\\
& =\dfrac{1}{\Gamma(\nu)}\int_{0}^{\infty}s^{\nu-1}e^{-sy}C_0(-sx)ds,	
\end{split}	\end{equation}
\end{proof}

\begin{cor}
	In similar way, by applying identity \ref{C0J0}, we can recast \ref{Lneg} as
\begin{equation} \label{eq15HermLag}   
L_{-\nu}(x,y)=\frac{1}{\Gamma (\nu )}\int_{0 }^\infty s^{\nu -1}e^{-sy}J_0\left( 2\sqrt{-xs}\right) ds.
\end{equation} 
\end{cor}

 It has been shown in \cite{L.C.Andrews} that the negative order $LP$ satisfy the same recurrences and differential equation of their polynomial counterpart.
 
 \begin{exmp}
It is evident that the derivation of integrals of the type 

\begin{equation}\begin{split}\label{eq17HermLag}
\int_{-\infty }^{\infty}e^{-ax^2}J_0(bx)dx&=
\int_{-\infty }^{\infty}e^{-x^2\left( a+\frac{b^2}{4}\hat{c}\right) }dx\;\varphi_{0}=\sqrt{\pi}\left(a+ \frac{b^2}{4}\hat{c}\right)^{-\frac{1}{2}}\varphi_{0}=\\ 
& =\sqrt{\pi}L_{-\frac{1}{2}}\left( -\frac{b^2}{4},a\right)
\end{split}\end{equation}
are a straightforward consequence of the previous formalism.
\end{exmp}
Eq. \ref{eq17HermLag}  is particularly interesting since can also be interpreted in terms of two variable generalized Bessel functions (see Chapter \ref{ChapterOP}).\\

A further point deserving attention is the definition of \textbf{\textit{Associated Laguerre polynomials}} ($ALP$), which within the present framework are introduced as

\begin{defn}
	We give a slightly different definition of the two-variable Associated Laguerre polynomials, $\forall x,y\in\mathbb{R},\forall n\in\mathbb{N},\forall \alpha\in\mathbb{R}$, 
\begin{equation}\label{diffAHP}
\Lambda_n^{(\alpha)}(x,y):=\hat{c}^\alpha(y-\hat{c}x)^n\varphi_{0}=n!\sum_{r=0}^n\dfrac{(-1)^r y^{n-r}x^r}{(n-r)!r!\Gamma(r+\alpha+1)},
\end{equation}
linked to the ordinary definition\footnote{Ordinary \textit{AHP} \begin{equation}\label{OrdAHP}
	L_n^{(m)}(x,y)=(n+m)!\sum_{r=0}^\infty\dfrac{(-1)^r y^{n-r}x^r}{(n-r)!r!(r+m)!}.
	\end{equation}} \cite{L.C.Andrews} by

\begin{equation}\label{assAHP}
L_n^{(\alpha)}(x,y)=\dfrac{\Gamma(n+\alpha+1)}{n!}\Lambda_n^{(\alpha)}(x,y).
\end{equation}
\end{defn}

\begin{cor}
The structural similarity between $HP$ and $LP$ due to the “unifying” formalism developed so far, suggests  (see eqs. \ref{eq2HermLagbis}-\ref{eq1HermLagbis}) the definition of \textbf{\textit{Associated Hermite polynomials}} ($AHP$) as

\begin{equation}\label{key}
H_n(x,y\mid p)={}_y\hat{h}^{p}\left(x+ {}_y\hat{h}\right) ^n\theta_{0}=n!\sum_{r=0}^n\dfrac{(r+p)!y^{\frac{r+p}{2}}x^{n-r}}{\Gamma\left( \frac{r+p}{2}+1\right)(n-r)!r!}\left|\cos\left( (r+p)\frac{\pi}{2}\right)  \right| .
\end{equation}
\end{cor}
According to the previous definition we can state the following 

\begin{cor}
	We provide the index-duplication formula
\begin{equation}\begin{split}\label{key}
H_{2n}(x,y)&=\left(x+ {}_y\hat{h} \right)^n \left( x+{}_y\hat{h} \right)^n\theta_{0}=n!\sum_{s=0}^n\dfrac{x^{n-s}}{(n-s)!s!}{}_y\hat{h}^s\left( x+{}_y\hat{h} \right)^n=\\
& =n!\sum_{s=0}^n\dfrac{x^{n-s}}{(n-s)!s!}H_n(x,y\mid s)
\end{split}\end{equation}
and argument duplication formula 

\begin{equation}\begin{split}\label{NT1}
H_n(2x,y)&=\left[\left(x+\dfrac{{}_y\hat{h}}{2}\right) +\left(x+\dfrac{{}_y\hat{h}}{2} \right) \right]^n \theta_{0}=\\
& =\sum_{s=0}^n\binom{n}{s}\sum_{r=0}^s\binom{s}{r}\dfrac{x^r}{2^{s-r}}H_{n-s}\left( x,\dfrac{y}{4}\mid s-r\right) .
\end{split}\end{equation}
It is furthermore easily checked that

\begin{equation}\label{NT2}
x^n=\left[ \left( x+{}_y\hat{h}\right)-{}_y\hat{h} \right]^n \theta_{0}= \sum_{r=0}^n\binom{n}{r}(-1)^r H_{n-r}(x,y\mid r)
\end{equation}
or similarly
\begin{equation}\label{NT2bis}
x^n=\left[ \left( x+{}_y\hat{h}\right)-{}_y\hat{h} \right]^n \theta_{0}= \sum_{r=0}^n\binom{n}{r}(-1)^{n-r} H_{r}(x,y\mid n-r)
\end{equation}
and that

\begin{equation}\label{NT3}
H_{n+m}(x,y)=\left( x+{}_y\hat{h}\right)^m \left( x+{}_y\hat{h}\right)^n\theta_{0}=\sum_{r=0}^m\binom{m}{r}x^{m-r}H_n(x,y\mid r).
\end{equation}
\end{cor}

The last identity is a reformulation of the \textit{Nilsen} Theorem, concerning the sum of the indices of \textit{HP} \cite{Carpanese}. The previous results (\ref{NT1}-\ref{NT3}) occur in the literature in different forms \cite{Carpanese}.\\
Even though not explicitly mentioned, the Hermite umbra can be raised to \textit{any real power} and this allows noticeable freedom in guessing possible generalizations. A fairly direct example is provided by the following extension.

\begin{lem}
	$\forall \alpha,\beta\in\mathbb{R}$
\begin{equation}\label{key}
H_n(x,y\mid \beta;\alpha)={}_y\hat{h}^{\beta}\left( x+{}_y\hat{h}^{\alpha}\right)^n\theta_{0}, 
\end{equation}
yielding a family of polynomials with generating function 

\begin{equation}\begin{split}\label{key}
& \sum_{n=0}^{\infty}\dfrac{t^n}{n!}H_n(x,y\mid \beta;\alpha)=e^{xt}y^{\frac{\beta}{2}}e_{\alpha,\beta}\left( y^{\frac{\alpha}{2}}t\right),\\
& e_{\alpha,\beta}(x)=\sum_{r=0}^{\infty}\dfrac{\Gamma(\alpha r+\beta+1)x^r}{\Gamma\left( \frac{\alpha r+\beta}{2}+1\right) r!}\left|\cos \left( \frac{\alpha r+\beta}{2}\pi\right)\right| . 
\end{split}\end{equation}

Their properties can easily be studied and they are framed within the context of the \textit{Sheffer family}. They can accordingly be defined through the operational rule

\begin{equation}\label{key}
H_n(x,y\mid \beta;\alpha)=y^{\frac{\beta}{2}}e_{\alpha,\beta}\left( y^{\frac{\alpha}{2}}\partial_{x}\right) x^n.
\end{equation}
\end{lem}

\section{A Note on Operator Ordering and Laguerre Umbral Operators}

In Chapter \ref{Chapter1} we have touched on the problem of \textit{operator ordering} arising in the solution of PDE when \textit{umbral exponential functions (UEF)} containing sum of non commuting operators are involved. We have also argued that the problem of finding appropriate ordering forms is, in these cases, complicated by the fact that \textit{UEF}’s do not possess the semi group property of the ordinary exponential function (as it has been noticed in application \ref{ApplicSchr} and as we will see later in paragraph \ref{LagAiry}), which should be properly redefined.\\

The problem of merging ordering procedures and the umbral formalism has been introduced in ref. \cite{Babusci} and further elaborated in \cite{Airy}-\cite{ML}. The break-through of the procedure may be shown using a fairly straightforward example, provided by a \textit{Laguerre-derivative based PDE} \cite{Babusci},  namely

\begin{exmp}
	Let, $\forall x,\alpha,\beta\in\mathbb{R},\forall t\in\mathbb{R}^+_0$,
\begin{equation}\label{GrindEQ__1_OpOrd}
\left\lbrace  \begin{array}{l} 
 _{l} \partial _{t} F(x,t)=-(\alpha \, x-\beta \, \partial _{x} )F(x,t),\\[1.6ex] 
 F(x,0)=f(x),
\end{array}\right. \end{equation} 
where we indicate with $_{l} \partial _{t}:=\partial _{t} t\, \partial _{t} $ the laguerre derivative (\textbf{l-derivative}) \ref{prove}. We remind the property \ref{LagDers} $_{l} \partial _{t}^{\, n}=\partial _{t} ^{\, n} t^{\, n}\partial _{t}^{\, n},\;\forall n\in\mathbb{R}$\footnote{Albeit we are not making any explicit reference to fractional derivatives, we quote this example to underscore the generality of the procedure we envisage, whose operator nature allows a fairly straightforward extension to the fractional case} and we naturally introduce the \textbf{Laguerre exponential} (l-exponential) (a "compromise" between ordinary exponential and Laguerre function) \cite{DatTo, Mancho,Babusci}

\begin{equation} \label{GrindEQ__5_Airy} 
_{l} e(\eta ):=\sum _{r=0}^{\infty }\frac{\eta ^{\;r} }{\left(r!\right)^{2} }, \quad \forall\eta\in\mathbb{R},   
\end{equation} 
or in umbral version (by using \ref{Opc})

\begin{equation} \label{GrindEQ__16_OpOrd} 
_{l} e(x)=\sum _{r=0}^{\infty }\frac{\left(x\right)^{\, r} }{\left(r!\right)^{2} }=e^{\hat{c}\, x} \varphi _{0}.  
\end{equation}
The function ${}_{l} e(\eta )$ is a $0$-order Bessel Tricomi function \cite{Tricomi}\footnote{It can be expressed in terms of the $0$-order modified Bessel function (see Chapter \ref{Chapter3}) through the identity   ${}_{l} e(\eta )=I_{0} (2\, \sqrt{\eta } )$.} and satisfies the "Laguerre"-eigenvalue equation $\forall \lambda\in\mathbb{R}$

\begin{equation} \label{GrindEQ__9_Airy} 
 {}_{l} \partial _{t} \left({}_{l} e(\lambda \, t)\right)=\lambda \left({}_{l} e(\lambda \, t)\right), 
\end{equation} 
indeed, since l-derivative satisfies ordinary series integration theorem, we get
\begin{equation*}\label{key}
{}_{l} \partial _{t} \left({}_{l} e(\lambda \, t)\right)=
{}_{l} \partial _{t}\sum_{r=0}^\infty\dfrac{(\lambda\;t)^r}{r!^2}=
\sum_{r=0}^\infty{}_{l} \partial _{t}\dfrac{(\lambda\;t)^r}{r!^2}=\lambda\sum_{r=1}^\infty\dfrac{(\lambda\;t)^{r-1}}{(r-1)!^2}=\lambda{}_{l} e(\lambda \, t)
\end{equation*}
(which corroborates the interpretation of its role as that of a $l$-exponential) and furthermore

\begin{equation}\label{key}
e^{\kappa\;{}_{l} \partial _{t}}\dfrac{t^n}{n!}=
\sum_{r=0}^\infty\dfrac{\kappa^r}{r!}{}_{l} \partial _{t}^r\dfrac{t^n}{n!}=
\sum_{r=0}^\infty\dfrac{\kappa^r}{r!}\partial_t^r t^r \partial_t^r\dfrac{t^n}{n!}=
\sum_{r=0}^\infty\binom{n}{r}\dfrac{\kappa^r x^{n-r}}{(n-r)!}
\end{equation}
obtained after expanding the l-exponential and by applying the property of the l-derivative operator \ref{LagDers}.\\

The solution of eq. \ref{GrindEQ__1_OpOrd} then can be written using the evolution operator formalism in which the exponential is replaced by  

\begin{equation}\label{key}
F(x,t)={}_l e\left( - t\, (\alpha \, x-\beta \, \partial _{x} ) \right)
\end{equation}
or, in umbral form, as

\begin{equation} \label{GrindEQ__17_OpOrd} 
F(x,t)=e^{-\hat{c}\, t\, (\alpha \, x-\beta \, \partial _{x} )} f(x)\varphi_{0}. 
\end{equation} 
The (pseudo) exponential evolution operator cannot be disentangled into the product of two exponentials, because the operators in the argument of the exponential are not commuting. However by applying the Weyl ordering rule \ref{rulesComp} (see application \ref{ApplicSchr}), we obtain

\begin{equation} \label{GrindEQ__18_OpOrd} 
F(x,t)=e^{-\frac{(\hat{c}t)^{2} }{2} \alpha \beta } e^{-\hat{c}\, t\, \alpha \, x} e^{\hat{c}\beta \, t\, \partial _{x} } f(x)\varphi_{0}=e^{-\frac{(\hat{c}t)^{2} }{2} \alpha \beta } e^{-\hat{c}\, t\, \alpha \, x} f(x+\hat{c}\beta \, t)\varphi_{0}. 
\end{equation} 

\noindent It is important to note that the operational ordering brings into play a further term depending on the square of the umbral operator $\hat{c}$, which commutes with the differential operators $x,\;\partial _{x} $.\\

\noindent Assuming for simplicity $f(x)=1$, we find that

\begin{equation} \label{GrindEQ__19_OpOrd} 
F(x,t)=e^{-\frac{(\hat{c}t)^{2} }{2} \alpha \beta } e^{-\hat{c}\, t\, \alpha \, x} \varphi_{0}. 
\end{equation} 
The example we have discussed is sufficient to show how that the \textbf{umbral formalism naturally yields the solution of evolution problems involving Laguerre derivative and non-commuting operators}.
\end{exmp}

 Albeit eq. \ref{GrindEQ__19_OpOrd} is just providing the solution for a mathematical problem (without any particular meaning except that of being associated with a Laguerre evolution problem) it is worth speculating a little more on it. After expanding the exponential containing $x$, we find

\begin{equation} \label{GrindEQ__20_OpOrd} 
F(x,t)=\sum _{n=0}^\infty\frac{\left(\hat{c}t\alpha x\right)^{n} }{n!}  e^{-\frac{(\hat{c}t)^{2} }{2} \alpha \beta }\varphi_{0} =
\sum _{n=0}^\infty\frac{(\alpha \, tx)^{n} }{n!} {}_{l} e_{n}^{(2)} \left(-\frac{\alpha \beta \, t^{2} }{2} \right),  
\end{equation} 
where ${}_{l} e_{n}^{(m)} \left(x\right)$ is the Bessel like function so defined \cite{DatTo, Mancho,Babusci}

\begin{equation} \label{GrindEQ__21_OpOrd} 
{}_{l} e_{n}^{(m)} \left(x\right)=\sum _{r=0}^{\infty }\frac{x^{r} }{r!\, \Gamma (m\, r+n+1)} .  
\end{equation} 

The results obtained so far have extended to the Laguerre derivative case, analogous conclusions contained in ref. \cite{ML}-\cite{Airy}  where, among the other things, the fractional Poisson distribution has been defined, within the context of the coherent states generated in the study of a Fractional Schr\"{o}dinger equation, driving the process of one photon emission-absorption dynamics. 

\begin{Oss}
The same problem translated to a \textbf{Laguerre-Schr\"{o}dinger equation} yields the following result (see eqs. \ref{solSchrpsi}-\ref{prP})

\begin{equation} \label{GrindEQ__22_OpOrd} 
|\Psi (t)\;\rangle=\sum _{n=0}^\infty\frac{\left(\hat{c}\; \Omega \; t\right)^{n} }{\sqrt{n!} }  e^{-\frac{(\hat{c}\, \Omega \, t)^{2} }{2} }\varphi_{0} |\;n\;\rangle .
\end{equation} 
The probability of finding the state $|\Psi (t)\;\rangle $ into a \textit{"n"} state is given by the square amplitude reported below 

\begin{equation}\begin{split} \label{GrindEQ__23_OpOrd} 
& \mid\langle\;n|\Psi (t)\;\rangle\mid^{2} =\frac{\left(\hat{c}\; \Omega \; t\right)^{2n} }{n!} e^{-\left(\hat{c}\; \Omega \; t\right)^{2} } \varphi_{0}=\frac{X^{n} }{n!} {}_{l} e_{2n}^{(2)} \left(-X\right), \\ 
& X=\left(\Omega \; t\right)^{2} ,
\end{split}\end{equation} 
which is a normalized \textbf{Laguerre-Poisson ``distribution''} with mean value and variance, respectively given by

\begin{equation}\begin{split} \label{GrindEQ__24_OpOrd} 
& \langle\;n\;\rangle=\frac{X}{2} ,\\ 
& \sigma _{n}^{2} =\frac{1}{4!} \left(12-5X\right)X .
\end{split}\end{equation} 
The quotes are due to the fact that eq. \ref{GrindEQ__23_OpOrd} cannot be considered a probability distribution because it is not positively defined for all x values.\\

\noindent This last statement may appear rather surprising, since we have defined a square amplitude. It should be however understood that the ``square'' of the scalar product in eq. \ref{GrindEQ__23_OpOrd} is not a number but an operator, specified by its action on $\varphi _{0} $ and therefore it is not necessarily positive defined.
\end{Oss}

 The examples we have discussed so far are sufficiently general to allow further mathematical extension, which will be discussed in the forthcoming sections.\\

The problems associated with the operator ordering in the context of the umbral formalism open a new and important Chapter, we have just grasped the surface of the problem, further comments will be provided in the following applications.

\section{Applications}

\subsection{The Complex Galilei Group and the Principle of Monomiality}\label{CGG}

Without entering the specific details of the physical aspects of \textit{\textbf{Galilei}} $(G-)$ group, we remind that it is the space-time symmetry group of non-relativistic quantum mechanics \cite{Baym}. The generators of the (complex) $G$-group are straightforwardly obtained by noting that (we consider the case of $(1+1)$ dimension only)

\begin{itemize}
	\item [a)]	The time and spatial shift generators are associated with the respective first derivatives 
	
	\begin{equation}\begin{split}\label{key}
	& \hat{P}_0=-i\partial_{t},\\
	& \hat{P}_z=-i\partial_z.
	\end{split}\end{equation}
	\item [b)]	The Galilean boost (namely the generator of Galilean transformation 
	is specified \cite{Pashaev}	
	\begin{equation}\label{key}
	\hat{K}=z+4\;i\;t\;\partial_z .
	\end{equation}
\end{itemize}
The commutation relations between the group generators are given below 

\begin{equation}\begin{split}\label{key}
& \left[P_0,P_z \right]=0,\\
& \left[ \hat{P}_0,K\right]=4\;i\;\hat{P}_z,\\
& \left[ \hat{P}_z,K\right]=-i.
\end{split}\end{equation}
 Furthermore, the Schr\"{o}dinger operator

\begin{equation}\label{key}
\hat{S}=-\hat{P}_0+2\hat{P}_z^2,
\end{equation}
is easily seen to commute with all the $G$-group generator. A rather na\"{i}ve consequence of such a statement is that, if the function $\Phi(z,t)$ is  a solution of the Schr\"{o}dinger equation, namely

\begin{equation}\label{opS}
\hat{S}\;\Phi(z,t)=0,
\end{equation}
then any operator $\hat{W}$, commuting with $\hat{S}$,  generates further solutions of the same equation, we have indeed

\begin{equation}\label{key}
\left[\hat{W},\hat{S} \right]\Phi(z,t)=0\rightarrow\hat{W}\left(\hat{S} \;\Phi(z,t)\right)-\hat{S}\left(\hat{W} \Phi(z,t)\right)  =0.
\end{equation}
According to eq. \ref{opS} we get

\begin{equation}\begin{split}\label{key}
& \hat{S}\left( \Psi(z,t)\right) =0,\\
& \Psi(z,t)=\hat{W}\Phi(z,t).
\end{split}\end{equation}
The function $\Psi$, defined through the action of the operator $\hat{W}$ on $\Phi$, is a further solution of the Schr\"{o}dinger problem.\\

It is also evident that we can realize the operator as a $\hat{W}$ function of the generators of the $G$-group, therefore

\begin{equation}\begin{split}\label{key}
&  \Psi_1(z,t)=e^{i\;t_0\;\hat{P}_0}\Phi(z,t)=\Phi(z,t+t_0),\\
& \Psi_2(z,t)=e^{i\;t_0\;\hat{P}_z}\Phi(z,t)=\Phi(z+z_0,t),
\end{split}\end{equation}
are (trivial) examples of additional solutions.\\

The exponenziation of the  operator $\hat{K}$ yields an additional solution

\begin{equation}\label{key}
\Psi_3(z,t)=e^{i\;\lambda\;\hat{K}}\Phi(z,t)=e^{2\;i\;\lambda^2\;t+i\;\lambda\;z}\Phi(z-4\lambda t,t)
\end{equation}
and we find also

\begin{equation}\label{key}
\Psi_{4,n}(z,t)=\hat{K}^n\Phi(z,t),
\end{equation}
which by the virtue of the Burchnall identity reads

\begin{equation}\label{key}
\Psi_{4,n}(z,t)=\sum_{s=0}^{n}\binom{n}{s}(4it)^s H_{n-s}(x,2it)\Phi^{(s)}(z,t).
\end{equation}

The examples we have just provided yields an idea of the wealth of consequences associated with the above “conception” of special polynomials.\\ 

In the forthcoming section we will use the present formalism in an application regarding the so called \textbf{\textit{Free Electron Laser}} high gain equation.

\subsection{Two variable Hermite Polynomials,  Volterra Integral Equation Solution and Application to Free Electron Laser Equation}
\numberwithin{equation}{section}
\markboth{\textsc{\chaptername~\thechapter. Two variable Hermite Polynomials,  Volterra Integral Equation Solution and Application to Free Electron Laser Equation}}{}

Free Electron Laser $(FEL)$ devices can be framed within the family of coherent radiation devices, originated by the Klystron. The unifying mechanism is that of an electron beam density bunching, induced by an appropriate energy modulation. In the case of $FEL$ the energy modulation is realized through an undulator magnet, providing a transverse component of the electron motion and the consequent coupling to a co-propagating electromagnetic wave \cite{FelHigh,CDR,CarmFel}.\\

The $FEL$ is a device capable of transforming the kinetic energy of a beam of electrons into electromagnetic radiation with laser-like properties. The $FEL$ lasing mechanism does not rely on the stimulated emission by an atomic or molecular ensemble in which a population inversion has been realized. No quantum energy gap limits in principle the tunability of the device, which extends over the e.m. spectrum, from microwaves to $X$-ray, by “simply” tuning the energy of the e-beam from $MeV$ to $GeV$ region.\\

The $FEL$ dynamics in the unsaturated regime is ruled by an equation of the type \cite{Colson}

\begin{equation}\begin{split} \label{GrindEQ__1_FEL} 
& \partial _{\tau }\; a=i\; \pi \, g_{0} \int _{0}^{\tau }\tau 'e^{-i\; \nu \, \tau '-\frac{(\pi \mu _{\varepsilon } \tau ')^{2} }{2} }  a(\tau -\tau ')\; d\tau ',\\ 
& a(0)=1,
\end{split} \end{equation} 
where $a$ represents the laser field amplitude, $g_{0} $ the small signal gain coefficient, $\nu$ is linked to the laser frequency and the coefficient $\mu _{\varepsilon } $ is a parameter regulating the effects of the inhomogeneous broadening due to the electrons' energy distribution.\\

It is an integro-differential equation of Volterra type. 
The kernel of the integral part is not trivial, eq. \ref{GrindEQ__1_FEL}  cannot indeed been solved analytically, unless  $\mu_{\varepsilon }=0$, as shown in the forthcoming section.\\

In this section we discuss and compare different forms of solution, based on perturbative methods, including the concepts relevant to the Hermite calculus developed in the previous sections.\\

In order of providing an idea of the point of view we are going to develop,  we can write eq. \ref{GrindEQ__1_FEL}  in a slightly different form, more suitable for our purposes,

\begin{equation}\begin{split}\label{key}
& \partial_{\tau }\; a=\hat{V}(\tau)\;a,\\
& \hat{V}(\tau)=i\; \pi \, g_{0} \int _{0}^{\tau }\tau 'e^{-i\; \nu \, \tau '-\tau' \partial_{\tau }-(\pi\mu_{\varepsilon })^2 \tau'^2}d\tau'.
\end{split}\end{equation}
The solution can accordingly be obtained by the iteration 

\begin{equation}\begin{split}\label{key}
& \partial_{\tau }\; a_n(\tau)=\hat{V}^n a_0,\\
& a(\tau)=\sum_{n=0}^\infty a_n(\tau),\\
& a_0=a(0)=1,
\end{split}\end{equation}
where the $n-th$ term of the previous expansion (namely the Volterra series) can also be cast in the explicit form \cite{Dattoli,enea}

\begin{equation}\begin{split}\label{GrindEQ__2_FEL} 
& \partial _{\tau } \;a_{n} =i\; \pi \, g_{0} \int _{0}^{\tau }\tau 'e^{-i\; \nu \, \tau '-\frac{(\pi \mu _{\varepsilon } \tau ')^{2} }{2} }  a_{n-1} (\tau -\tau ')\; d\tau ',\\ 
& a_{0} =1.
\end{split} \end{equation} 
For later convenience we call $n$ the principal expansion index. \\

Even though the method is efficient and fastly converges the integrals cannot be done analytically but requires a numerical treatment, which increases the computation times when a larger number of term is involved.\\
We show how the wise use of generalized forms of multivariable polynomials opens new avenues for the solution of eq. \ref{GrindEQ__1_FEL} and of Volterra type equations more in general. \\
 A way to overcome the computation of integrals in analytical terms  is provided by two variable \textit{HP} \ref{classHerm} which, by the use of the relevant generating function \ref{genfunctH}, yields


\begin{equation} \label{GrindEQ__5_FEL} 
\partial _{\tau } a_{n} =i\; \pi \, g_{0} \sum _{m=0}^{\infty }\frac{1}{m!}  \int _{0}^{\tau }\tau '^{(m+1)}  H_{m} \left(-i\, \nu ,-\frac{(\pi \, \mu _{\varepsilon } )^{2} }{2} \right)a_{n-1} (\tau -\tau ')\; d\tau '. 
\end{equation} 
The Hermite expansion index $m$ is therefore nested into the principal index $n$. The integration procedure is now straightforward and we get

\begin{enumerate}
	\item  \textbf{\textit{First order solution}}
	\begin{equation} \label{GrindEQ__6_FEL} 
	a_{1} =i\; \pi \, g_{0} \sum _{m=0}^{\infty }\frac{H_{m} \left(-i\, \nu ,-\frac{(\pi \, \mu _{\varepsilon } )^{2} }{2} \right)}{m!}
	\int _{0}^{\tau }d\tau '  \int _{0}^{\tau '}\tau ''^{(m+1)}  d\tau '', 
	\end{equation} 
	which yields
	\begin{equation}\begin{split} \label{GrindEQ__7_FEL} 
	& a_{1} =i\; \pi \, g_{0} \sum _{m_{1} =0}^{\infty }\alpha _{m_{1} } \tau ^{m_{1} +3}  , \\ 
	& \alpha _{m_{1} } =\frac{H_{m_{1} } }{m_{1} !(m_{1} +2)\, (m_{1} +3)} .
	\end{split} \end{equation} 
	The arguments of the Hermite have been omitted to avoid cumbersome expressions.\\
	
	\item  \textbf{\textit{Second order solution}}
	\begin{equation} \label{GrindEQ__8_FEL} 
	a{}_{2} =\left(i\; \pi \, g_{0} \right)^{2} \sum _{m_{2} =0}^{\infty }\frac{H_{m_{2} } }{m_{2} !} \int _{0}^{\tau }d\tau '  \int _{0}^{\tau '}\tau ''^{(m_{2} +1)} a_{1}  (\tau '-\tau '')d\tau '', 
	\end{equation} 
	which after some algebra yields
	
	\begin{equation} \begin{split}\label{GrindEQ__9_FEL} 
	a_{2} &=\left(i\; \pi \, g_{0} \right)^{2} \sum _{m_{1} ,m_{2} =0}^{\infty }\alpha _{m_{1} ,m_{2} } \tau ^{m_{1} +m_{2} +6}  , \\ 
	\alpha _{m_{1} ,m_{2} }& =\alpha _{m_{1} } \sum _{s=0}^{m_{1} +3}\left(\begin{array}{c} {m_{1} +3} \\ {s} \end{array}\right) \frac{H_{m_{2} } }{m_{2} !(m_{2} +s+2)(m_{2} +m_{1} +6)} =\\
	& =\frac{(-1)^s\;H_{m_{1} } H_{m_{2} } }{m_{2} !m_{1} !(m_{1} +2)\, (m_{1} +3)(m_{2} +m_{1} +6)} \cdot\\
	& \cdot\sum _{s=0}^{m_{1} +3}\left(\begin{array}{c} {m_{1} +3} \\ {s} \end{array}\right) \frac{(-1)^s}{(m_{2} +s+2)} .
	\end{split} \end{equation} 
	We can use the previous two contributions to check whether the results we obtain concerning e.g. the $FEL$ gain agree with analogous conclusions already given in the literature. \\
	
	\item  \textbf{\textit{Higher order solution}}\\
	
	The interesting aspect of the present nested procedure is that the $n^{th}$ order can be computed in a modular way just looking at the symmetries of the expansion itself. The $n^{th}$ order term indeed reads
	
	\begin{equation} \begin{split}\label{GrindEQ__10_FEL} 
	 a_{n} &=\left(i\; \pi \, g_{0} \right)^{n} \sum _{m_{1} ,..m_{n} =0}^{\infty }\alpha _{m_{1} ,..m_{n} } \tau ^{\left(\sum _{r=1}^{n}m_{r}  +3\, n\right)}  ,\\ 
	 \alpha _{m_{1} ,..m_{n} } &=\alpha _{m_{1} ,..m_{n-1} } \sum _{s=0}^{\left(\sum _{r=1}^{n-1}m_{r}  +3\, (n-1)\right)}\left(\begin{array}{c} {\sum _{r=1}^{n-1}m_{r}  +3\, (n-1)} \\ {s} \end{array}\right)\cdot\\
	& \cdot \frac{(-1)^sH_{m_{n} } }{m_{n} !(m_{n} +s+2)\left(\sum _{r=1}^{n}m_{r}  +3\, n\right)} .
	\end{split} \end{equation} 
\noindent The square modulus of the function $a$ at $\tau =1$ reads namely
	
	\begin{equation} \label{GrindEQ__11_FEL} 
	G(\nu ,\mu _{\varepsilon } )=\left\| a\right\| ^{2} -1 
	\end{equation} 
\end{enumerate}
(we have subtracted the initial condition for convenience).\\

The computation procedure has been implemented in \textit{Mathematica} but doesn't rely on its symbolic capabilities only and thus can be implemented in any computer language, unless the analitical result is requested. In this case the calculation has beeen done using a computing a tool with symbolic capabilities.\\
The derivation of
\begin{equation} \label{GrindEQ__12_FEL} 
G(\nu ,\mu _{\varepsilon } )=\left\| a_{0} +a_{1} +a_{2} +..+a_{n} \right\| ^{2} -1 
\end{equation} 
requires the calculation of the norm of the quadratic sum of $n+1$ terms, of which one is constant (unity for $a_0$) and $n$, according to the method outlined in previous section, are the $HP$ expansion as in eqs. \ref{GrindEQ__7_FEL}, \ref{GrindEQ__9_FEL} and \ref{GrindEQ__10_FEL}. The expansion has to be truncated for computational reasons and it is easy to implement, by the literal point of view, but the performance is poor. The $j^{th}$ expansion, in fact, requires the evaluation of $M^j$ terms (if all expansions are truncated at the $M^{th}$ term) for every $j=1\dots n$, that is $\dfrac{M^{n+1}-1}{M-1}$ terms in total, which can be quite demanding. Algorithmic optimizations are thus needed.\\
\noindent It can be noted that some $HP$ are the same in different occurrencies, and could be evaluated once and then stored for the next order.\\
\noindent The first optimization is to build a lookup table containing the $HP$ or to define functions that remember their own values (which is a feature offered by Mathematica to implement transparently lookup tables).\\
\noindent The second optimization is to do index or exponent reordering with the aim to move factors in or out of the summation to avoid repetitions of the calculus, like in the r.h.s of eq. \ref{GrindEQ__9_FEL}.\\
\noindent The comparison is reported in Fig. \ref{Fig1FEL} for a $2^{nd}$ order gain function, where the present procedure (at different truncation levels) is confronted with those from a complete numerical integration and no differences are foreseeable. The truncation of the Hermite expansion obviously affects the speed and the approximation of the gain curve.\\
\noindent The method allows also taking into account the broadening effect as shown in Fig. \ref{Fig2FEL}. It has to be noted that for higher broadening effects an higher truncation level is needed to achieve a good approximation, as shown in Figs. \ref{brM10}-\ref{brM25}. \\
\noindent The Hermite expansion approach let the integrals found in eq. \ref{GrindEQ__2_FEL} become polynomials and thus very easy to solve automatically. They are, of course, high in number but tractable also by a symbolic processor.\\
\noindent It is therefore possible to obtain, with the same degree of approximation, but with longer calculation times than the numeric case, also the full analytical expression for the curve.\\

\begin{figure}[htp]
	\centering
	\begin{subfigure}[c]{0.48\textwidth}
		\includegraphics[width=0.9\linewidth]{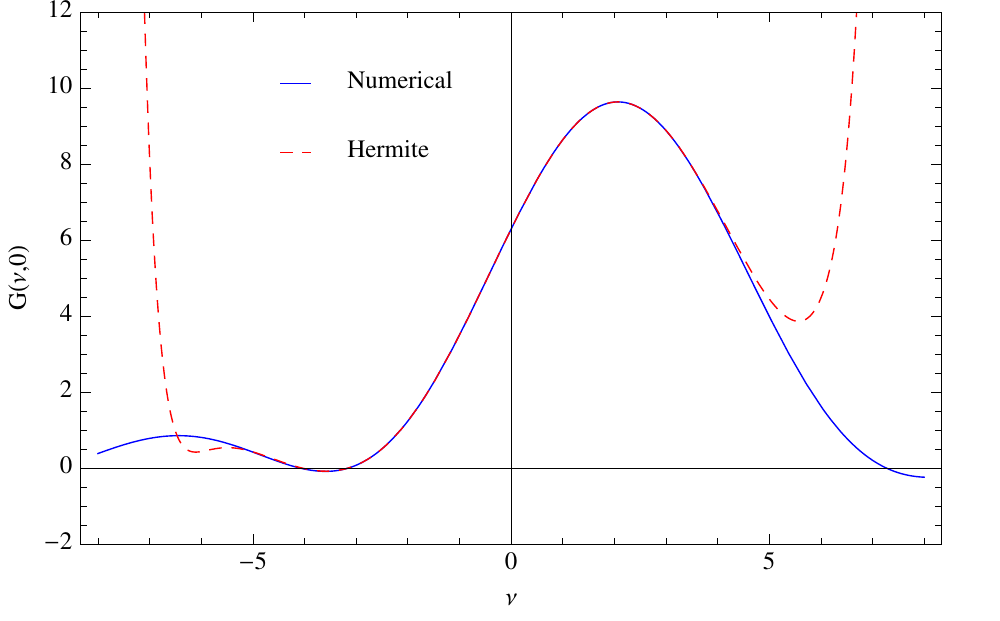}
		\caption{$M=10$}
	\end{subfigure}
	\begin{subfigure}[c]{0.48\textwidth}
		\includegraphics[width=0.9\linewidth]{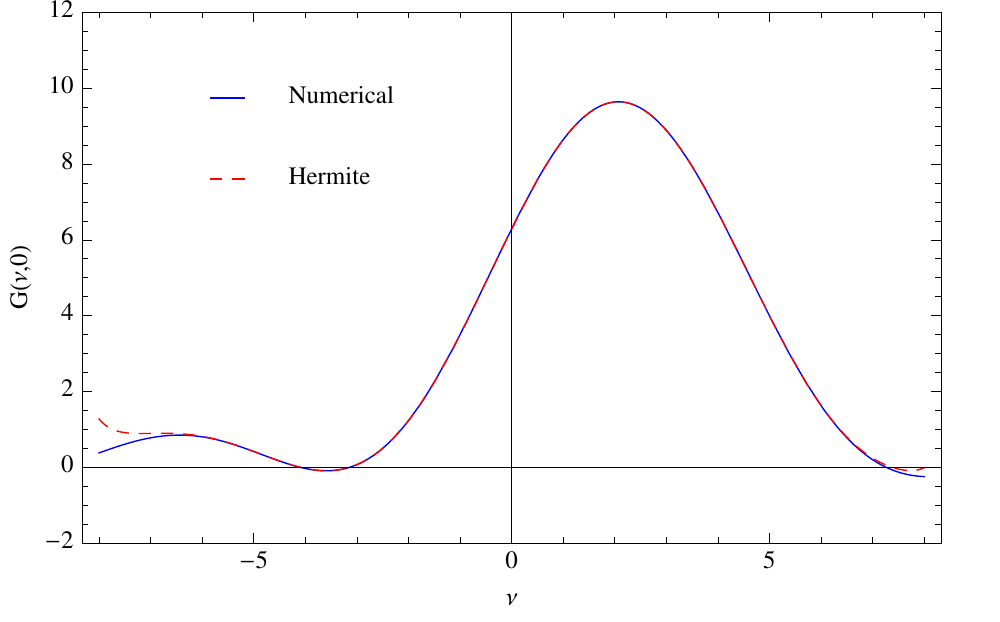}
		\caption{$M=15$}
	\end{subfigure}
\\[3mm]
\begin{subfigure}[c]{0.48\textwidth}
	\includegraphics[width=0.9\linewidth]{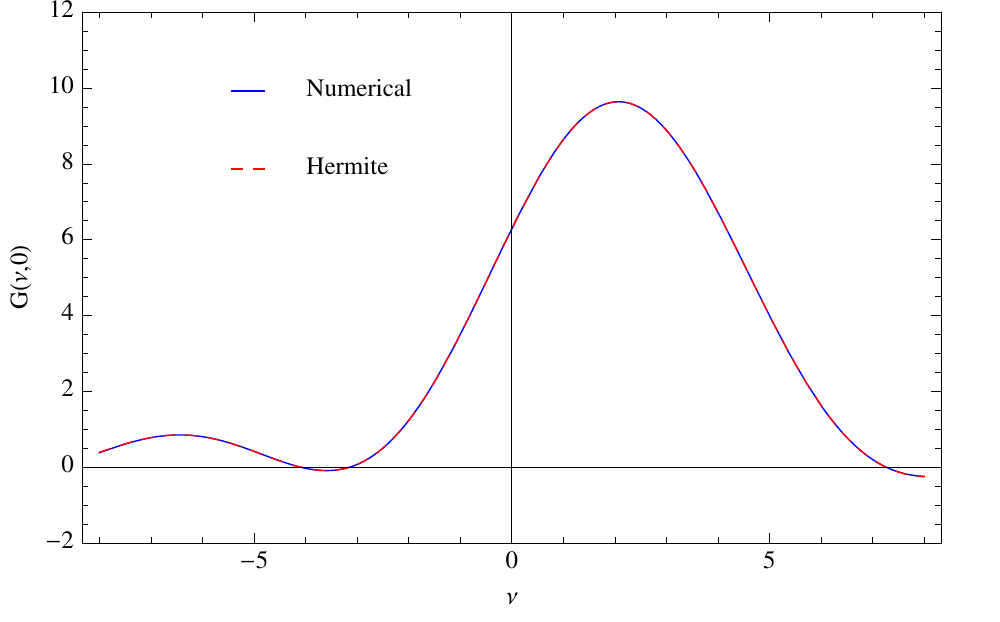}
	\caption{$M=20$}
\end{subfigure}
\begin{subfigure}[c]{0.48\textwidth}
	\includegraphics[width=0.9\linewidth]{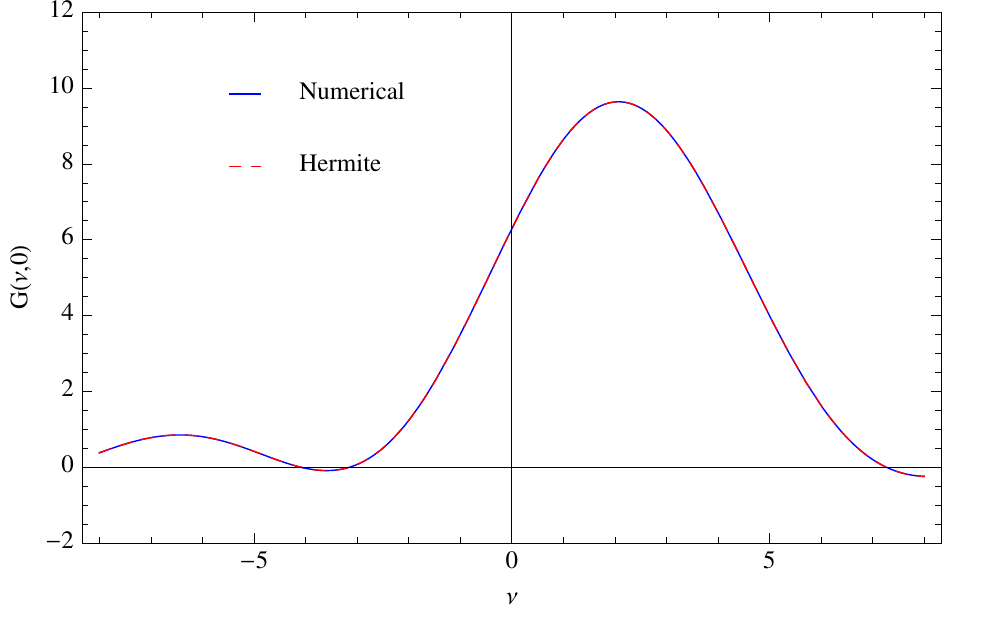}
	\caption{$M=25$}
\end{subfigure}
	\caption{Comparison between complete numerical integration with no broadening effects ($\mu_\varepsilon=0$) $G(\nu,0)=\mid\mid a\mid\mid^2-1$ with $g_0=5$ at the end time $\tau=1$, performed by Mathematica, and Hermite solution $G(\nu ,0 )=\left\| a_{0} +a_{1} +a_{2} \right\| ^{2} -1 $ at different truncation levels ($M=10, M=15, M=20,M=25$). }\label{Fig1FEL} 
\end{figure}

\begin{figure}[htp]
	\centering
	\begin{subfigure}[c]{0.48\textwidth}
		\includegraphics[width=0.9\linewidth]{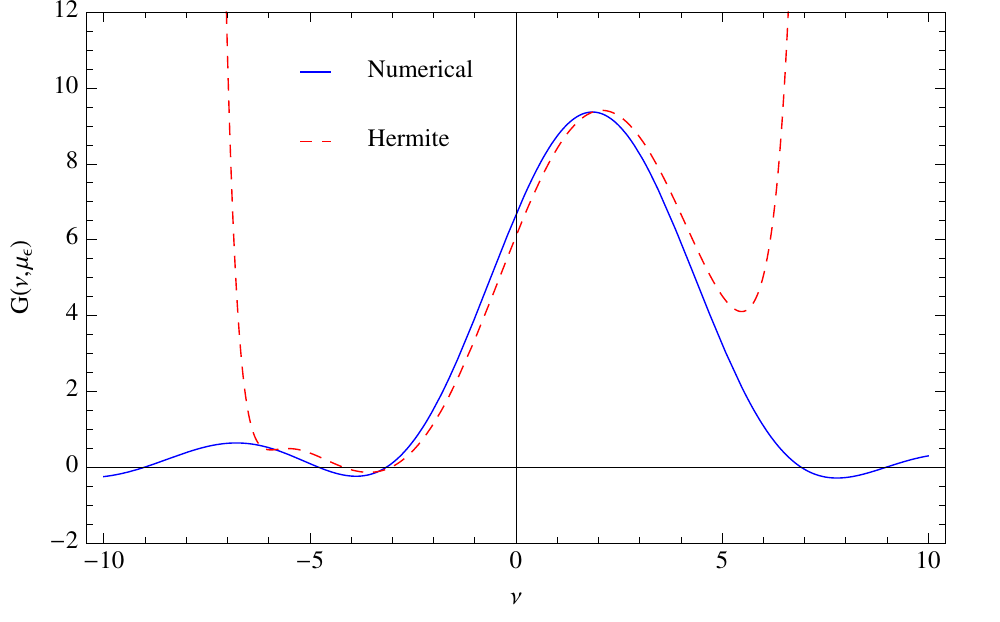}
		\caption{$\mu_\varepsilon=0.1$, $M=10$}
	\end{subfigure}
	\begin{subfigure}[c]{0.48\textwidth}
		\includegraphics[width=0.9\linewidth]{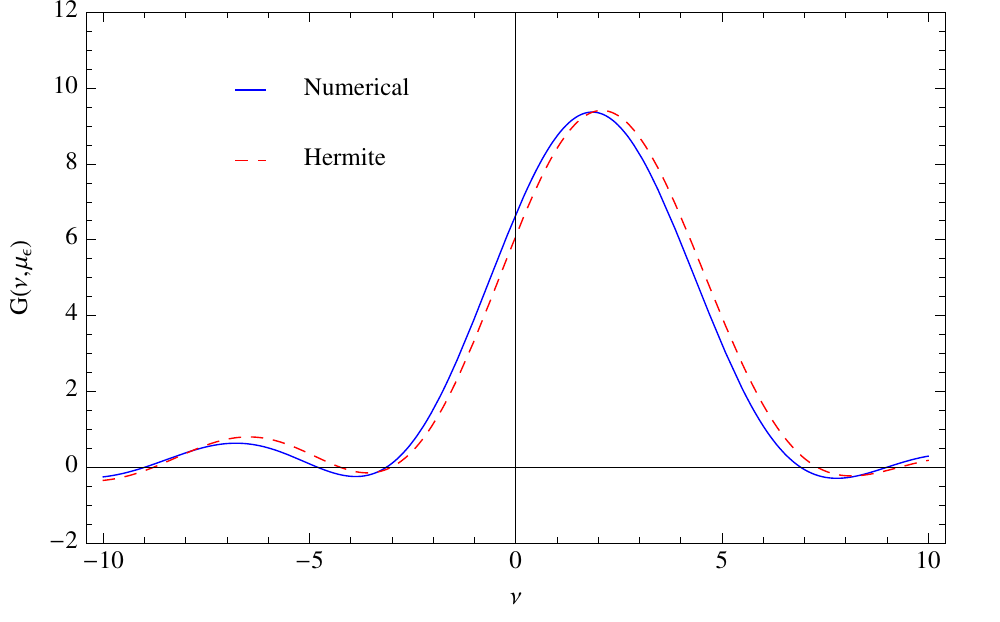}
		\caption{$\mu_\varepsilon=0.1$, $M=25$}
	\end{subfigure}
\\[3mm]
	\begin{subfigure}[c]{0.48\textwidth}
	\includegraphics[width=0.9\linewidth]{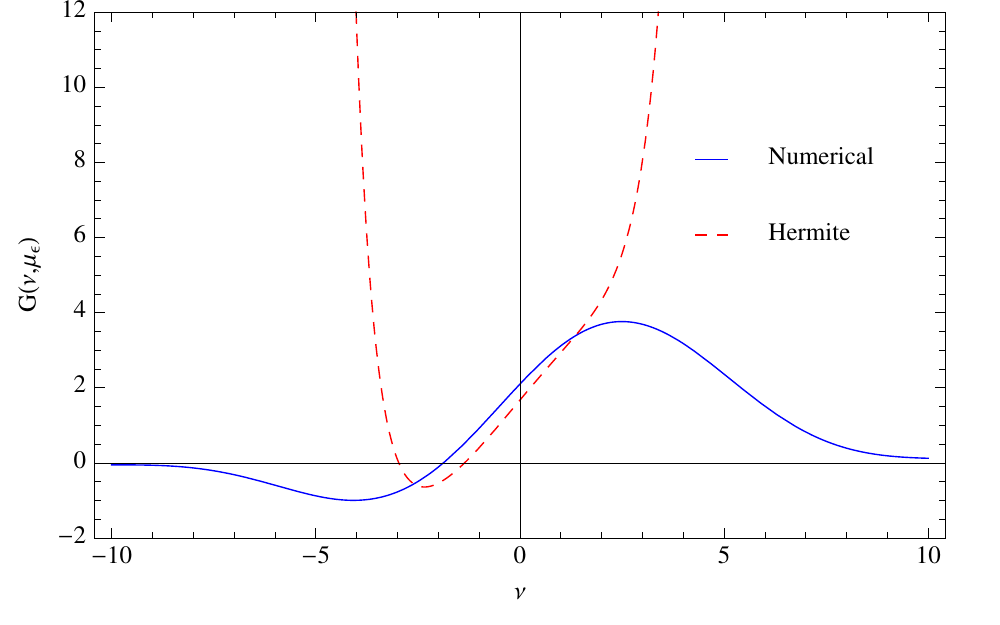}
	\caption{$\mu_\varepsilon=0.7$, $M=10$}
	\label{brM10}
\end{subfigure}
\begin{subfigure}[c]{0.48\textwidth}
	\includegraphics[width=0.9\linewidth]{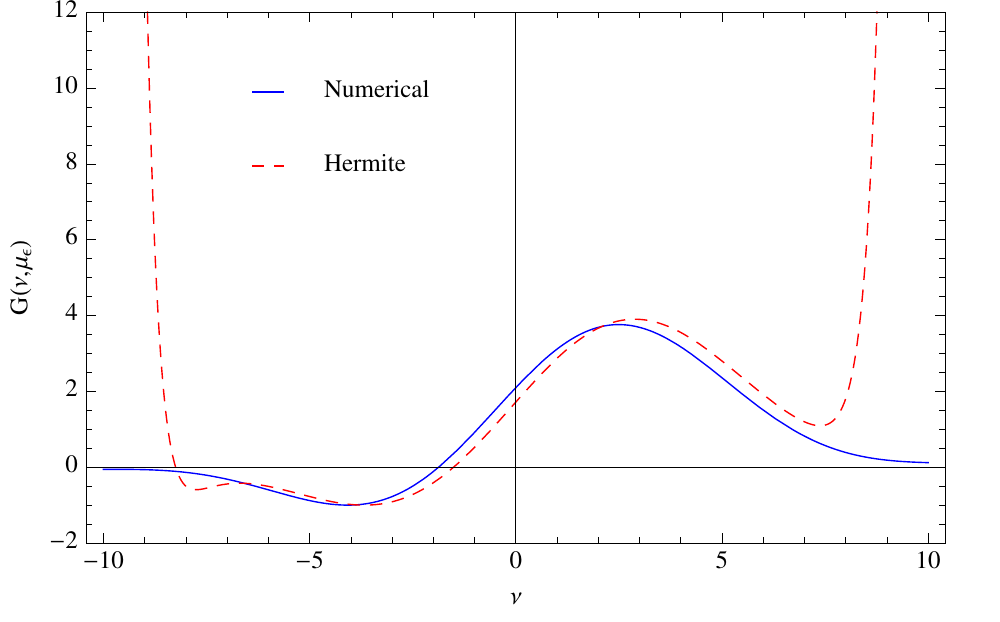}
	\caption{$\mu_\varepsilon=0.7$, $M=25$}
	\label{brM25}
\end{subfigure}
	\caption{Comparison between complete numerical integration with different broadening effects ($\mu_\varepsilon=0.1,\mu_\varepsilon=0.7$) $G(\nu,\mu_\varepsilon)=\mid\mid a\mid\mid^2-1$ with $g_0=5$ at the end time $\tau=1$, performed by Mathematica, and Hermite solution $G(\nu ,\mu_\varepsilon )=\left\| a_{0} +a_{1} +a_{2} \right\| ^{2} -1 $ at different truncation levels ($M=10, M=25$). }\label{Fig2FEL} 
\end{figure}
\newpage

The method we have described is however general enough to be applied to Volterra like equations with different kernels, provided that a corresponding expansion in terms of a suitable polynomial family be allowed.\\

\noindent To clarify the previous statement we consider the equation

\begin{equation}\label{eqInt}
\partial_{\tau }\; a(\tau)=i\pi g_0\int_{0}^{\tau}\dfrac{\tau' e^{-i\nu\tau'}}{(1-i\alpha\tau')(1-i\beta\tau')}a(\tau-\tau')d\tau' .
\end{equation}
In this case the problem can be solved by the use of a two variable Legendre like polynomials $P_n(x,y)$ (which we will see better in the next Chapter) with theirs generating function
$
\sum_{n=0}^{\infty}t^n P_n(x,y)=\dfrac{1}{1+x\;t+y\;t^2}.
$
The solution of the problem \ref{eqInt} looks therefore like the expression derived for eq. \ref{GrindEQ__1_FEL} with the two variable Legendre polynomials  replacing the Hermite, namely

\begin{equation}\label{key}
\partial_{\tau }\; a_n(\tau)=i\pi g_0\sum_{m_n=0}^{\infty}P_{m_n}\left( -i\left( \alpha+\beta\right) ,-\alpha\beta\right) \int_{0}^{\tau}\tau'^{(m_n+1)}a_{n-1}(\tau-\tau')d\tau' .
\end{equation}

\subsection{Volterra Integral Equation, FEL and  Negative Derivative Formalism}

In accordance with the formalism of paragraph \ref{NegDerOpMeth}, the repeated integration can be viewed as the $n-th$ power of the negative derivative operator, namely

\begin{equation}\label{key}
\hat{D}_x^{-n}f(x)=\int_{0}^x dx_1f(x_1)\int_{0}^{x_1} dx_2f(x_2)\dots\int_{0}^{x_{n-1}} dx_nf(x_n).
\end{equation}
The use of the $Cauchy$ repeated integral formula yields a way to express the series of integrals on the rhs of the previous equation in a fairly compact form, thus yielding 

\begin{equation}\label{key}
\hat{D}_x^{-n}f(x)=\dfrac{1}{(n-1)!}\int_{0}^x (x-\xi)^{n-1}f(\xi)d\xi.
\end{equation}

Let us therefore consider eq. \ref{GrindEQ__1_FEL}  in which we assume that the parameter $\mu_\varepsilon$ is vanishing

\begin{equation}\label{boh}
\partial_{\tau }a(\tau)=i\pi g_0\int_{0}^{\tau}\tau'e^{-i\nu\tau'}a(\tau-\tau')d\tau' .
\end{equation}
From the physical point of view, eq. \ref{boh} rules the evolution of a $FEL$ amplitude, in which the inhomogeneous broadening effects, associated with the electron beam energy spread, are negligible.\\

\noindent If we make the change of variable $\tau-\tau'=\xi$, eq. \ref{boh} can be cast in the form 

\begin{equation}\label{key}
e^{i\nu\tau}\partial_{\tau }a(\tau)=i\pi g_0\int_{0}^{\tau}\left(\tau-\xi\right)  e^{i\nu\xi}a(\xi)d\xi .
\end{equation}
The use of the of negative derivative formalism allows to write

\begin{equation}\label{key}
e^{i\nu\tau}\partial_{\tau }a(\tau)=i\pi g_0\hat{D}_{\tau}^{-2}\left(e^{i\nu\tau}a(\tau) \right), 
\end{equation}
by keeping the second derivative of both sides we end up with the third order differential equation

\begin{equation}\label{key}
\begin{split}
& \partial_{\tau }^3a(\tau)+2i\nu\partial_{\tau }^2a(\tau)-\nu^2\partial_{\tau }a(\tau)=i\pi g_0 a(\tau),\\
& \partial_{\tau }^2a(\tau)\mid_{\tau=0}=0,\\
& \partial_{\tau }a(\tau)\mid_{\tau=0}=0,\\ 
&  a(\tau)\mid_{\tau=0}=a(0),\\
\end{split}
\end{equation}
which can be solved by standard means.\\
When the inhomogeneous contributions are included, the integral equation cannot be reduced to an ordinary differential equation, notwithstanding the application of the method allows interesting applications. We note indeed that  

\begin{equation}\label{key}
e^{i\nu\tau}\partial_{\tau }a(\tau)=i\pi g_0\int_{0}^{\tau}\left(\tau-\xi\right)  e^{i\nu\xi}
e^{-\frac{1}{2}\left[ \pi\mu_{\varepsilon }\left( \tau-\xi\right)\right] ^2 }a(\xi)d\xi,
\end{equation}
by expanding at the first order the exponential inside the integral kernel we find

\begin{equation}\label{key}
e^{i\nu\tau}\partial_{\tau }a(\tau)=i\pi g_0\int_{0}^{\tau}\left(\tau-\xi\right)  e^{i\nu\xi}\left[ 1-\dfrac{1}{2}\left( \pi\mu_{\varepsilon }\right)^2(\tau-\xi)^2 \right]a(\xi)d\xi, 
\end{equation}
which in terms of higher order negative derivatives reads

\begin{equation}\label{key}
e^{i\nu\tau}\partial_{\tau }a(\tau)=i\pi g_0\hat{D}_{\tau}^{-2}
\left[e^{i\nu\tau}\left(  1-3\left( \pi\mu_{\varepsilon }\right)^2\hat{D}_{\tau}^{-2} \right) \right]a(\tau).
\end{equation}
The last equation evidently reduces to a fifth order ordinary differential equation by keeping the fourth  derivative with respect to time of both sides. The order increases by keeping further terms in the exponential expansion.\\

The study of these specific problems goes beyond the scope of this thesis and will not be further discussed, we note however that the use of these techniques, as well as those reported in previous section,  may provide a fairly helpful tool in many topics of practical nature concerning the physics of $FEL$.

\chapter{Special Polynomials and Umbral Operators}\label{ChapterOP}
\numberwithin{equation}{section}
\markboth{\textsc{\chaptername~\thechapter. Orthogonal Polynomials and Umbral Operators}}{}

In this Chapter we underline the deep link which connects the umbral operators with a wide family of orthogonal polynomials. We provide a wide range of integral or PDE results in terms of special polynomials easily treatable if in umbral forms.\\

The original parts of the Chapter, containing their adequate bibliography, are based on the following original papers.\\

\cite{Gegenbauer} \textit{G. Dattoli, B. Germano, S. Licciardi, M.R. Martinelli; “On an umbral treatment of Gegenbauer, Legendre and Jacobi polynomials”, International Mathematical Forum, vol. 12, 2017, no. 11, pp. 531-551}. \\

\cite{Babusci} \textit{D. Babusci, G. Dattoli, M. Del Franco, S. Licciardi; “Mathematical Methods for Physics”, invited Monograph by World Scientific, Singapore, 2017, in press}.\\

$\star$ \textit{M. Artioli, G. Dattoli, S. Licciardi, R.M. Pidatella; “Hermite and Laguerre Functions: A unifying point of view”, work in progress}.\\

$\star$ \textit{C. Cesarano, G. Dattoli, S. Licciardi; “Generating Functions for Lacunary Legendre and Legendre-like Polynomials”, work in progress}.

\section{On an Umbral Treatment of Gegenbauer, Legendre and Jacobi Polynomials}\label{UTGeg}

Special polynomials, ascribed to the family of Gegenbauer, Legendre, and Jacobi and of their associated forms, can be expressed in an operational way, which allows a high degree of flexibility for the formulation of the relevant theory.\\

In the following we develop the study of the properties of the Gegenbauer polynomials \cite{L.C.Andrews} and of their generalized forms in terms of umbral operators. We introduce the main concepts, associated  with the technique we are going to deal with, starting from a very simple example.

\begin{exmp}
We consider, $\forall x\in\mathbb{R}:x>-1$, the elementary function 

\begin{equation}\label{unoGeg}
e^{(\nu )} (x)=(1+x)^{-\nu }, \quad Re(\nu )>0.
\end{equation} 
 According to the use of standard Laplace transform identities \ref{L-T}, the function \ref{unoGeg} can be rewritten as

\begin{equation}\label{2Geg}
e^{(\nu )} (x)=\frac{1}{\Gamma (\nu )} \int _{0}^{\infty }e^{-s}  s^{\nu -1} e^{-sx} ds.
\end{equation} 
By following ref. \cite{Borel}, we use the operational rule \ref{tderf} $\left(\alpha \right)^{x\, \partial _{x} } f(x)=f(\alpha \, x) $ to write

\begin{equation}
e^{(\nu )} (x)=\frac{1}{\Gamma (\nu )} \int _{0}^{\infty }e^{-s}  s^{\nu -1} s^{x\, \partial _{x} } e^{-x} ds.            
\end{equation}
We can therefore recast the function \ref{unoGeg} as

\begin{equation}\begin{split}\label{GrindEQ__4_1_Geg} 
& \frac{1}{\Gamma (\nu )} \int _{0}^{\infty }e^{-s}  s^{\nu -1} s^{x\, \partial_{x} } e^{-x} ds=\left(\frac{1}{\Gamma (\nu )} \int _{0}^{\infty }e^{-s}  s^{\nu -1} s^{x\, \partial _{x} } ds\right)\, e^{-x} = \\
&  =\hat{\Gamma }_{\nu }\; e^{-x} =\sum _{r=0}^{\infty }\frac{(-1)^{r} }{r!}  \left(\nu \right)_{r} x^{r} , \\
& \hat{\Gamma }_{\nu } =\frac{\Gamma (\nu +x\, \partial _{x} )}{\Gamma (\nu )},  \\
& \left(\nu \right)_{r} =\frac{\Gamma (\nu +r)}{\Gamma (\nu )}, \quad \forall r\in\mathbb{R}^+_0 .
\end{split}\end{equation} 
The procedure of bringing the exponential outside the sign of integration, thus defining  the operator $\hat{\Gamma }_{\nu } $, is  allowed only for the values of $\left|x\right|<1$ for which the series, containing the \textbf{Pochhammer symbol} $\left(\nu \right)_{r} $, converges. The series appearing in \ref{GrindEQ__4_1_Geg} is recognized as the Newton bynomial, even though obtained in an "involved"  albeit useful way for the purposes of the present intent. In the spirit of the umbral calculus we reinterpret the function $e^{(\nu )} (x)$ as an \textbf{\emph{ordinary exponential function}}, by introducing 

\begin{defn}
The operator $\hat{\gamma }$ is defined as

\begin{equation}
\hat{\gamma }^{\,r} \nu_0:=\left(\nu \right)_{r},
\end{equation}
with $\nu_0$ the vacuum.
\end{defn}
Accordingly, we can cast the function $e^{(\nu )} (x)$ in the form 

\begin{equation}\label{defGeg}
e^{(\nu )} (x)=e^{-\hat{\gamma}\, x} \nu_{0},
\end{equation} 
thus formally treating it as an exponential (namely a transcendental function) even though the series \ref{GrindEQ__4_1_Geg} has a limited range of convergence $\left|x\right|<1$. 
\end{exmp}
As explained below, we take advantage from the previous exponential umbral restyling of the function in \ref{unoGeg} to construct a new formalism useful for the study of various family o special polynomials.

\begin{prop}
By keeping the derivative of both sides of eq. \ref{defGeg} with respect to the  $x$ variable and using the ordinary rules of calculus, we obtain $\forall x\in\mathbb{R}, Re(\nu)>0$ the identity\footnote{  The same identity can be obtained from eq. \ref{2Geg} which yields     \begin{equation}\partial _{x} e^{(\nu )} (x)=-\frac{1}{\Gamma (\nu )} \int _{0}^{\infty }e^{-s}  s^{\nu } e^{-sx} ds=-\frac{\Gamma (\nu +1)}{\Gamma (\nu )\, } e^{(\nu +1)} (x).\end{equation}  The umbral identity we have derived is not limited by any convergence  restriction. }

\begin{equation}\begin{split} \label{GrindEQ__6_1_Geg} 
 \partial _{x} e^{(\nu )} (x)&=\left(\partial _{x} e^{-\hat{\gamma}\, x} \right)\; \nu_0=-\hat{\gamma}\, e^{-\hat{\gamma}\, x} \nu_0=\sum _{r=0}^{\infty }(-1)^{r+1}\hat{\gamma}^{\,r+1} \dfrac{x^{r} }{r!}  \, \nu_0= \\ 
& =-\sum _{r=0}^{\infty }(-1)^{r} \left(\nu \right)_{r+1} \frac{x^{r} }{r!}  \, =-\nu \, e^{(\nu +1)} (x), 
\end{split}\end{equation} 
which follows as a consequence of \cite{L.C.Andrews}

\begin{equation}
\left(\nu \right)_{r+1} =\nu \, \left(\nu +1\right)_{r}  \end{equation}
and more generally, $\forall m,r\in\mathbb{R}^+_0$.
\begin{equation}
\left(\nu \right)_{r+m} =(\nu)_{m}(\nu +m)_{r} .
\end{equation}
We can accordingly state the rule

\begin{equation}\begin{split}\label{ruleGeg}
\left(\partial _{x}^{\,m} e^{-\hat{\gamma}\, x} \right)\; \nu_0&=(-1)^{m} \left(\hat{\gamma}^{\,m} e^{-\hat{\gamma}\, x} \right)\, \; \nu_0= \\
& =(-1)^{m} \left(\nu \right)_{m} e^{(\nu +m)} (x).
\end{split} \end{equation} 
\end{prop}
Before proceeding further it is worth clarifying a point, which will be more thoroughly treated later in Chapter \ref{ChapterTR}. Even though the formalism we have developed allows to treat not trivial functions in terms of elementary exponential functions, some properties like the \textbf{\textit{semigroup identities}} associated with the exponential case are not easily associated with $e^{(\nu )} (x)$. We find indeed that

\begin{Oss}
Albeit the following chain of identities is correct $\forall  x,y\in\mathbb{R}$

\begin{equation} \label{GrindEQ__9_1_Geg} 
e^{(\nu )} (x+y)=e^{-\hat{\gamma}\, \left(x+y\right)} \nu_0=e^{-\hat{\gamma}\, x} e^{-\hat{\gamma}\, y} \nu_0, 
\end{equation} 
it is also true that
\begin{equation}
e^{(\nu )} (x+y)\ne e^{(\nu )} (x)\, e^{(\nu )} (y).
\end{equation}
\end{Oss}
To overcome such an apparently paradoxical conclusion, we clarify that the concept of semi-group has to be properly framed within the \textbf{\textit{appropriate algebraic context}}. In defining the semigroup and, thereby, the associated identities, the corresponding binary operations between $x$ and $y$ need to be defined.\\

The \textbf{associative binary operation} \textit{(ABO)} $\;e^{x+y}=e^{x}e^{y}=e^{y}e^{x}$ is a consequence of the fact that 
$(x+y)^{n}=\sum_{r=0}^{n}\binom{r}{n}x^{n-r}y^{r}$. This means that we can define the opportune \textit{ABO} if we modify the Newton bynomial as follows.

\begin{defn}
	Let $B(x,y)$ the Beta-function \ref{Betaf}, we define a modified Newton bynomial
\begin{equation}\begin{split} \label{GrindEQ__11_1_Geg}  (x\oplus _{\nu } y)^{n} :&=\sum _{r=0}^{n}\binom{n} {r} \, \left\lbrace \begin{array}{c} {(\nu)_{n} } \\ {(\nu)_{r} } \end{array}\right\rbrace ^{-1} x^{n-r} y^{r} , \\
 \left\lbrace \begin{array}{c} {\left(\nu \right)_{m} } \\ {\left(\nu \right)_{p} } \end{array}\right\rbrace &=\frac{\left(\nu \right)_{m} }{\left(\nu \right)_{m-p} \left(\nu \right)_{p} } =\dfrac{B(\nu +m,\, \nu )}{B(\nu +m-p,\nu +p)}.
\end{split}\end{equation} 
Then, the corresponding \textit{ABO} is 

\begin{equation}\begin{split}\label{+oGeg}
 e^{(\nu )} (y)\, e^{(\nu )} (x)&=\sum _{r=0}^{\infty }\frac{(\nu) _{r} }{r!} (-y)^{r}  \sum _{s=0}^{\infty }\frac{(\nu) _{s} }{s!} (-x)^{s}  = \\ 
& =\sum _{n=0}^{\infty }\frac{\left(-1\right)^{n} }{n!}  \left(\nu \right)_{n} \left(\sum _{r=0}^{n}\binom {n} {r} \left\lbrace \begin{array}{c} {\left(\nu \right)_{n} } \\ {\left(\nu \right)_{r} } \end{array}\right\rbrace ^{-1} x^{n-r} y^{r} \right)= \\ 
& =e^{(\nu )} (x\oplus _{\nu } y).
\end{split}\end{equation}
\end{defn}
Accordingly we conclude that the proper environment for the algebraic semi-group property of the umbral exponential discussed in this section is the use of associative operations of the type \ref{+oGeg}.\\

The reliability of the formalism we are developing can be further checked by deriving integrals involving  the \textit{pseudo Gaussian function} as

\begin{exmp}
	If we have a function of the type
\begin{equation} \label{GrindEQ__13_1_Geg} 
e^{(\nu )} (x^{2} )=e^{-\hat{\gamma}\, x^{2} } \nu_0=(1+x^{2} )^{-\nu }, \quad \forall x\in\mathbb{R},  
\end{equation} 
according to the rules we have stipulated along with the properties of the ordinary Gaussian function \ref{GWi}, we can state that\footnote{This (well known) result is a byproduct of the outlined technique, but it could be also derived as a consequence of the (RMT) (see section \ref{AppARMT}).}  \cite{K.Gorska}

\begin{equation}\begin{split}\label{GaufunGeg}
 \int _{-\infty }^{+\infty }e^{(\nu )} (x^{2} )\,  dx\, &=\int _{-\infty }^{+\infty }e^{-\hat{\gamma}\, x^{2} } \,  dx\, \nu_0= \\ 
&=\sqrt{\frac{\pi }{\hat{\gamma}} }\, \nu_0=\sqrt{\pi } \left(\nu \right)_{-\frac{1}{2} } =\sqrt{\pi } \frac{\Gamma \left(\nu -\dfrac{1}{2} \right)}{\Gamma \left(\nu \right)} , \\ 
 Re(\nu )&>\frac{1}{2}.
\end{split}\end{equation}
 
It must be stressed that the integral in eq. \ref{GaufunGeg} is extended to all the real axis and therefore the umbral representation should be representative of the function on the r.h.s. of eq. \ref{GrindEQ__13_1_Geg} and not of the relevant series expansion  $\sum _{r=0}^{\infty }(-1)^{r} \left(\nu \right)_{r} \frac{x^{2\, r} }{r!}  $, having radius of convergence $\left|x\right|<1$.\\

To clarify this point, we note that, by exploiting again the Laplace transform method and \ref{GWi}, we can alternatively write the integral \ref{GaufunGeg} as  

\begin{equation}\begin{split}
 \int _{-\infty }^{+\infty }\frac{1}{\left(1+x^{2} \right)\, ^{\nu } }  \, dx&=\int _{-\infty }^{+\infty }dx \left[\dfrac{1}{\Gamma(\nu)}\int _{0}^{\infty }e^{-s}  s^{\nu -1} e^{-s\, x^{2} } ds\right]= \\[1.1ex] 
& =\dfrac{\sqrt{\pi }}{\Gamma(\nu)} \int _{0}^{\infty }e^{-s}  s^{\nu -\frac{3}{2} } ds=\sqrt{\pi } \frac{\Gamma \left(\nu -\frac{1}{2} \right)}{\Gamma \left(\nu \right)}, 
\end{split}\end{equation}
which confirms the reliability of use of the previously stated umbral rules.
\end{exmp}

By pushing further the formalism, we can take advantage from the wealth of properties of Gaussian integrals, by  getting e.g.

\begin{exmp}
\begin{equation}\begin{split}
& \int _{-\infty }^{+\infty }e^{(\nu )} (a\, x^{2} +i\, b\, x)\,  dx =\sqrt{\frac{\pi }{a\hat{\gamma}} } e^{-\hat{\gamma}\frac{b^{2} }{4\, a} } \, \nu_0 =\sqrt{\frac{\pi }{a} }\, \frac{\Gamma \left(\nu -\dfrac{1}{2} \right)}{\Gamma (\nu )} \, \frac{1}{\left(1+\dfrac{b^{2} }{4\, a} \right)^{\nu -\frac{1}{2} } } , \\ 
& Re(\nu )>\frac{1}{2} ,\;\;\; Re(a)>0.
\end{split}\end{equation}
\end{exmp}
Let us now consider a further application of the previous procedure, by keeping the \textit{successive derivatives} (with respect to the variable $x$) of the pseudo Gaussian function, namely

\begin{exmp}
	We introduce, $\forall n\in\mathbb{N}, \forall x\in\mathbb{R}, Re(\nu)>0$,
\begin{equation}
e_{n}^{(\nu )} (x^{2} ):=\partial _{x}^{n} e^{(\nu )} (x^{2} ).                                     
\end{equation}
We take advantage from the analogy with the properties of ordinary Gaussians, from the associated identity \ref{GHPol} $\partial _{x}^{n} e^{a\, x^{2} } =H_{n} (2\, a\, x,\, a)\, e^{a\, x^{2} } $ and from \ref{genfunctH}, we adapt eq. \ref{defGeg} to the pseudo-Gaussian case and we find

\begin{equation}\begin{split}\label{HGaussGeg}
 e_{n}^{(\nu )} (x^{2} )&=H_{n} (-2\, \hat{\gamma}\, x,\, -\hat{\gamma})\, e^{-\hat{\gamma}\, x^{2} } \nu_0= \\ 
& =(-1)^{n} n!\, \sum _{r=0}^{\left[\frac{n}{2} \right]}\frac{(-1)^{r} (2\, x)^{n-2\, r} }{(n-2\, r)!\, r!} (\hat{\gamma}^{n-r}  e^{-\hat{\gamma}\, x^{2} } )\, \nu_0.
\end{split}\end{equation}
On account of eq. \ref{ruleGeg} we note that

\begin{equation}
\left(\hat{\gamma}^{n-r} e^{-\hat{\gamma}\, x^{2} } \right)\, \nu_0=\left(\nu \right)_{n-r} e^{(\nu +n-r)} (x^{2} ).
\end{equation}
If we now introduce the two variable polynomials

\begin{equation} \label{GrindEQ__19_1_Geg} 
K_{n}^{(\nu )} (\xi ,\, \eta ):=n!\, \sum _{r=0}^{\lfloor\frac{n}{2} \rfloor}\frac{\left(\nu \right)_{n-r} \xi ^{n-2\, r} \eta^{r} }{(n-2\, r)!\, r!} ,  \quad \forall \xi,\eta\in\mathbb{R},\forall n\in\mathbb{N},Re(\nu)>0,
\end{equation} 
we can recast eq. \ref{HGaussGeg} in the non operatorial form

\begin{equation}\label{KGehex2}
e_{n}^{(\nu )} (x^{2} )=(-1)^{n} K_{n}^{(\nu )} \left(\frac{2x}{1+x^{2} } ,-\frac{1}{1+x^{2} } \right)\, \frac{1}{(1+x^{2} )^{\nu } } .      
\end{equation}
For $\xi =2 x,\, y=-1$ the polynomials \ref{GrindEQ__19_1_Geg} reduce to the ordinary \textbf{Gegenbauer polynomials}, namely\footnote{We remind the Definition of Gegenbauer polynomials given in \cite{Abramovitz} \begin{equation}\label{key}
C_{n}^{(\nu )} (x)=\sum_{k=0}^{\lfloor\frac{n}{2}\rfloor} 	(-1)^k\dfrac{\Gamma(n-k+\nu)}{\Gamma(\nu)k!(n-2k)!}(2x)^{n-2k}.
	\end{equation}}

\begin{equation}
K_{n}^{(\nu )} (2\, x,\, -1)=n!\, C_{n}^{(\nu )} (x)  .                     
\end{equation}
Furthermore, the identity

\begin{equation} \label{GrindEQ__22_1_Geg} 
\left(-1\right)^{n} e^{(-\nu )} (x^{2} )\, e_{n}^{(\nu )} (x^{2} )=K_{n}^{(\nu )} \left(\frac{2x}{1+x^{2} } ,-\frac{1}{1+x^{2} } \right) 
\end{equation} 
can be viewed as the associated Rodriguez formula \cite{L.C.Andrews}.\\

It is also worth stressing that the use of the relation \ref{vasteH} $\partial _{x}^{n} e^{a\, x^{2} +b\, x} =H_{n} (2\, a\, x+b,\, a)\, e^{a\, x^{2} +b\, x}$ and the discussed formalism yield the result

\begin{equation}\begin{split} \label{GrindEQ__24_1_Geg} 
& e_{n}^{(\nu )} (a\, x^{2} +b\, x)= \\ 
& =(-1)^{n} K_{n}^{(\nu )} \left(\frac{2\, a\, x+b}{1+a\, x^{2} +b\, x} ,-\frac{a}{1+a\, x^{2} +b\, x} \right)\, \frac{1}{(1+a\, x^{2} +b\, x)^{\nu } } .
\end{split}\end{equation} 
\end{exmp}
The results we have presented in this introduction disclose one of the advantages of the formalism which allows the \textit{derivation of the properties of Gegenbauer polynomials from those of Hermite}. Further consequences of this point of view will be discussed in the following sections.

\subsection{Gegenbauer Polynomials}

The study of the properties of the Gegenbauer polynomials is based on two variable Hermite polynomials of order $2$. It is natural to conclude that higher order \textit{HP} ar tailor suited to define generalized forms of Gegenbauer polynomials.

\begin{prop}
	Let, $\forall x,y\in\mathbb{R}, \forall n,m\in\mathbb{N}$,
	
	\begin{equation}\begin{split}
	& H_{n}^{(m)}(x,y)= n!\sum _{r=0}^{\lfloor\frac{n}{m}\rfloor }\dfrac{ x^{n-m\, r} y^{r} }{(n-m\, r)!\, r!}, \\
	& \sum _{n=0}^{\infty }\dfrac{t^{n} }{n!} H_{n}^{(m) }(x,y)=e^{xt+yt^{m}} , 
	\end{split}\end{equation} 
	the higher order two variable Hermite polynomials (see section \ref{higherHermite}) \cite{Appell},
	then, according to our formalism, we can identify 
	
	\begin{equation}\label{GrindEQ__1_2_Geg} 
	K _{n}^{(\nu,\, m )} (\xi ,\, \eta )=H_{n}^{(m)} (\hat{\gamma}\, \xi ,\, \hat{\gamma}\, \eta )\, \nu_{0}
	\end{equation}
	and obtain
\begin{equation} \label{GrindEQ__2_2_Geg} 
\sum _{n=0}^{\infty }\frac{t^{n} }{n!}  K _{n}^{(\nu,\, m )} (-\xi ,\, -\eta )=e^{(\nu )} (\xi \, t+\eta \, t^{\,m} )=\frac{1}{\left(1+\xi \, t+\eta \, t^{\,m} \right)^{\nu } },   \qquad \mid t \mid < 1.
\end{equation}

The repeated derivatives of functions like $e^{(\nu)}(ax^{m}+bx)$  can be expressed by using the properties of the higher order Hermite Kamp\'{e} de F\'{e}ri\'{e}t polynomials and of their generalized forms. The use of the following identity involving multivariable HP (see ref. \cite{Babusci})

\begin{equation}\begin{split} \label{GrindEQ__9_2_Geg} 
& \partial _{x}^{n} e^{P(x)} =H_{n}^{(m,m-1,\,...\,,2)} \left(P'(x),\, \frac{P''(x)}{2} ,\, \frac{P'''(x)}{3!} ,\, ...\,,\, \frac{P^{(m)} (x)}{m!} \right)\, e{}^{P(x)} , \\ 
& P(x)=a\, x^{m} +b\, x ,    
\end{split}\end{equation} 
can be exploited, along with eq. \ref{GrindEQ__9_2_Geg}, to get (see also eq. \ref{HGaussGeg})

\begin{equation}\begin{split} \label{GrindEQ__10_2_Geg}  & \partial _{x}^{n} e^{(\nu )} (a\, x^{m} +b\, x)=e_{n}^{(\nu )} (a\, x^{m} +b\, x)= \\ 
& =H_{n}^{(m,\,m-1,\,...\,,2)} \left( -\hat{\gamma}\, P'(x),\, -\hat{\gamma}\frac{P''(x)}{2} ,\, -\hat{\gamma}\frac{P'''(x)}{3!} ,...\,,\, -\hat{\gamma}\frac{P^{(m)} (x)}{m!} \right) \cdot\\
& \cdot e^{^{-\hat{\gamma}\, P(x)\, } } \nu _{0} = \\ 
& =K _{n}^{(\nu ,\, m,\,m-1,\,...\,,2)} \left( -\dfrac{P'(x)}{P(x)+1} ,\, -\dfrac{P''(x)}{2( P(x)+1)} ,\, -\dfrac{P'''(x)}{3!( P(x)+1)} ,\,...\,,\, \right. \\
& \left. - \dfrac{P^{(m)} (x)}{m!( P(x)+1)} \right) \, \frac{1}{\left( P(x)+1\right)^{\nu } } , 
\end{split}\end{equation} 
where

\begin{equation}\begin{split}\label{KmultiplyGeg}
& K _{n}^{(\nu ,\, m,\,m-1,\,...\,,2)}(x_{1},x_{2},x_{3}, ... \, , x_{m})=\\
& =\dfrac{1}{\Gamma(\nu)}\int_{0}^{\infty}e^{-s}s^{\nu-1}H_{n}^{(\,m,\,m-1,\, ... \, , 2)}(x_{1}s,x_{2}s,x_{3}s,... \, ,x_{m}s)ds,\\
& H_{n}^{(m,\,m-1,\, ... \, , 2)}(x_{1},x_{2},x_{3},... \, ,x_{m})=n!\sum_{r=0}^{ \lfloor\frac{n}{m}\rfloor }\dfrac{x_{m}^{r}H_{n-mr}^{(m-1,\,... \,,2)}(x_{1}, ... \, ,x_{m-1})}{(n-mr)!r!}\\
& \sum _{n=0}^{\infty }\dfrac{t^{n} }{n!} H_{n}^{ \left(  \left\lbrace m\right\rbrace  \right) } \left(  \left\lbrace x\right\rbrace  \right)= e^{\sum_{s=1}^{m}x_{s}t^{s}},\\
& \left\lbrace m\right\rbrace = m, m-1,\, ...\,, 2; \qquad \left\lbrace x\right\rbrace = x_{1},x_{2},\, ... \, , x_{m} .
\end{split}\end{equation}  

The same method allows some progress in the derivation of \textbf{Gegenbauer generating functions} and indeed, by exploiting eq. \ref{Hnlgf}, we find 

\begin{equation}\begin{split} \label{GrindEQ__11_2_Geg} 
\sum _{n=0}^{\infty }\frac{t^{n} }{n!}  K _{n+l}^{(\nu )} (-\xi ,\, -\eta )&=H_{l} (-\hat{\gamma}\, (\xi +2\, \eta \, t),\, -\hat{\gamma}\, \eta )\, e^{-\hat{\gamma}\, \left(\xi \, t+\eta \, t^{2} \right)} \nu_0= \\ 
& =\dfrac{(-1)^{l}K _{l}^{(\nu )} \left(\, \dfrac{(\xi +2\, \eta \, t)}{1+\xi \, t+\eta \, t^{2} } ,\, -\, \dfrac{\eta }{1+\xi \, t+\eta \, t^{2} } \right)}{\left(1+\xi \, t+\eta \, t^{2} \right)^{\nu } },  \\
 \mid t \mid < \left|  \dfrac{\xi -\sqrt{\xi^{2}-4\eta}}{2\eta}\right| &,
\end{split}\end{equation} 
which can be easily derived from the corresponding case of the Hermite polynomials $\sum _{n=0}^{\infty }\dfrac{t^{n} }{n!}  H _{n+l} (x,y)=H_{l}(x+2yt,y)e^{xt+yt^{2}}$ \cite{Babusci} .
\end{prop}

All the previous results can be obtained by the use of the Laplace transform method. The integral transforms are indeed not an alternative, but the rigorous support of the umbral methods we are developing.

\subsection{Jacobi Polynomials}\label{JacPol}

In the previous section we have exploited a, likely, powerful tool to deal with a plethora of problems concerning the theory of special functions, whose relevant technicalities can accordingly be reduced be to straightforward exercises in elementary calculus.\\

Let us now take a step further, by introducing the following  polynomials.

\begin{defn}
	We define a new family of two variable polynomials $\forall \xi,\eta,\alpha,\beta\in\mathbb{R}
	,\forall n\in\mathbb{N}$
	
\begin{equation} \label{GrindEQ__5_3_Geg} 
\frac{1}{n!} R_{n}^{(\,\alpha ,\, \beta )} (\xi ,\, \eta ):=\, \hat{c}_{1}^{\,\alpha } \hat{c}_{2}^{\,\beta } \left[\hat{c}_{1} \xi +\hat{c}_{2} \eta \right]^{\, n} \, \varphi _{\,1, 0}\, \varphi _{\,2, 0},  
\end{equation} 
where the operators $\hat{c}$ labelled by two different index act on two different vacua as

\begin{equation}
 \hat{c}_{1}^{\,\nu } \hat{c}_{2}^{\,\mu } \varphi _{\,1, 0}\, \varphi _{\,2,0} =\left(\hat{c}_{1}^{\,\nu } \varphi _{\,1, 0} \right)\left(\hat{c}_{2}^{\,\mu } \varphi _{\,2,0} \right)=\frac{1}{\Gamma \left(\,\nu +1\right)} .\frac{1}{\Gamma \left(\,\mu +1\right)} .
\end{equation}
\end{defn}

\begin{cor}
According to the previous definition we obtain the following explicit form for the polynomials defined in eq. \ref{GrindEQ__5_3_Geg}

\begin{equation}
R_{n}^{(\alpha ,\, \beta )} (\xi ,\, \eta )=(n!)^{2} \sum _{s=0}^{n}\frac{\xi ^{\,n-s} \eta ^{\,s} }{\left[(n-s)!\, \right]\, s!\,\Gamma (n-s+\alpha +1)\, \Gamma (s+\beta +1)}.
\end{equation}
\end{cor}
The relevant properties can easily be derived by the use of elementary algebraic manipulations. It is indeed easily checked that

\begin{propert}
\begin{equation}\begin{split}\label{ricorrGeg}
 \dfrac{1}{(n+1)!} R_{n+1}^{(\alpha ,\, \beta )} (\xi ,\, \eta )&=\left[\hat{c}_{1} \xi +\hat{c}_{2} \eta \right]\, \hat{c}_{1}^{\alpha } \hat{c}_{2}^{\beta } \left[\hat{c}_{1} \xi +\hat{c}_{2} \eta \right]^{\, n} \, \varphi _{1,\, 0} \varphi _{2,\, 0} = \\ 
& =\dfrac{1}{n!}\, \left(\xi \, R_{n}^{(\alpha +1,\, \beta )} (\xi ,\, \eta )+\eta \, R_{n}^{(\alpha ,\, \beta +1)} (\xi ,\, \eta )\right)
\end{split}\end{equation}
and that

\begin{equation}\begin{split}\label{derivGeg}
& \partial _{\xi } R_{n}^{(\alpha ,\, \beta )} (\xi ,\, \eta )=n^{2} \, R_{n-1}^{(\alpha +1,\, \beta )} (\xi ,\, \eta ), \\ 
& \partial _{\eta } R_{n}^{(\alpha ,\, \beta )} (\xi ,\, \eta )=n^{2} \, R_{n-1}^{(\alpha ,\, \beta +1)} (\xi ,\, \eta ).
\end{split}\end{equation}
\end{propert}
Furthermore we can determine its generating functions by the use analogous elementary procedures. We obtain for example

\begin{propert}
\begin{equation} \label{GrindEQ__9_3_Geg} 
\sum _{n=0}^{\infty }\frac{t^{n} }{\left(n !\right)^{2} }  R_{n}^{(\alpha ,\, \beta )} (\xi ,\, \eta )=\hat{c}_{1}^{\,\alpha } \hat{c}_{2}^{\,\beta } e^{\,t\, \left(\hat{c}_{1} \xi +\hat{c}_{2} \eta \right)} \varphi _{1,\, 0} \varphi _{2,\, 0} , \quad \forall t\in\mathbb{R} 
\end{equation} 
and, if we  note that, by generalizing Tricomi-Bessel function \ref{TrBn} $\forall \nu\in\mathbb{R}: Re(\nu)>0$,  we get

\begin{equation}\begin{split}\label{TrBnu}
& \hat{c}^{\nu } e^{-\hat{c}\, x} \varphi_0=\sum _{r=0}^{\infty }\frac{(-x)^{r} \hat{c}^{r+\nu } }{r!\, }  \varphi_0=C_{\nu } (x),\\ 
& C_{\nu } (x)=\sum _{r=0}^{\infty }\frac{(-1)^{r} x^{r} }{r!\, \Gamma (r+\nu +1)} , 
\end{split}\end{equation}
which are linked to the cylindrical Bessel functions by (see eq. \ref{C0J0} in generalized form)
\begin{equation} \label{GrindEQ__11_3_Geg} 
C_{\nu } (x)=\left(\dfrac{1}{x} \right)^{\frac{\nu }{2} } J_{\nu } (2 \sqrt{x} ), 
\end{equation} 
we can write the generating function \ref{GrindEQ__9_3_Geg} in terms of a product of Bessel functions\footnote{Where $I_{\nu}(x)=(-i)^{\nu}J_{\nu}(ix)$ is the first kind modified Bessel function (see Chapter \ref{Chapter3}).} $\forall  t\in\mathbb{R}$

\begin{equation}\begin{split}
\sum _{n=0}^{\infty }\frac{t^{n} }{\left(n!\right)^{2} }  R_{n}^{(\alpha ,\, \beta )} (\xi ,\, \eta )&=C_{\alpha } (-\xi \, t)\, C_{\beta } (-\eta \, t)= \\ 
& =\dfrac{1}{\sqrt{(\xi ^{\,\alpha } \eta ^{\,\beta } )\, t^{\,\alpha +\beta } } } \, \; I_{\alpha } \left( 2 \sqrt{\xi \, t} \right) \, I_{\beta } \left( 2 \sqrt{\eta \, t} \right). 
\end{split}\end{equation}
\end{propert}

\begin{defn}
The polynomials $R_{n}^{(\alpha ,\, \beta )} (\xi ,\, \eta )$, $\forall \xi,\eta,\alpha,\beta\in\mathbb{R},\forall n\in\mathbb{N}$, can be used to define the ordinary \textbf{Jacobi polynomials} (JP) $\forall x\in\mathbb{R}$ \cite{L.C.Andrews}  through the identity

\begin{equation}\begin{split}\label{JacobiGeg}
 P_{n}^{(\alpha ,\, \beta )} (x)&=\frac{\Gamma (n+\alpha +1)\, \Gamma (n+\beta +1)}{\left(n!\right)^{2} } \, R_{n}^{(\alpha ,\, \beta )} \left(\xi (x),\, \eta (x)\right), \\ 
 \xi (x)&:=\frac{x-1}{2} ,\\
 \eta (x)&:=\frac{1+x}{2}. 
\end{split}\end{equation}
\end{defn}

\begin{propert}
The relevant recurrences are obtained from eqs. \ref{ricorrGeg}, \ref{derivGeg} and write

\begin{equation}\begin{split}\label{npunoGeg}
(n+1)P_{n+1}^{(\alpha ,\, \beta )} (x)=& \dfrac{1}{2}    x\, \left[\left(n+\beta +1\right)P_{n}^{(\alpha +1,\, \beta )} (x)+\left(n+\alpha +1\right)P_{n}^{(\alpha ,\, \beta +1)} (x)\right]\, +\\
& -\dfrac{1}{2}\left[\left(n+\beta +1\right)P_{n}^{(\alpha +1,\, \beta )} (x)-\left(n+\alpha+1 \right)P_{n}^{(\alpha ,\, \beta +1)} (x)\right]
\end{split}\end{equation}
and

\begin{equation}\begin{split}\label{depnGeg} \dfrac{d}{dx} P_{n}^{(\alpha ,\, \beta )} (x)&= \frac{\Gamma (n+\alpha +1)\, \Gamma (n+\beta +1)}{\, 2\, (n-1)!^{2} } \cdot\\
& \cdot\left(R_{n-1}^{(\alpha +1,\, \beta )} (\xi (x),\, \eta (x))+R_{n-1}^{(\alpha ,\, \beta +1)} (\xi (x),\, \eta (x))\right)= \\ 
& =\dfrac{1}{ 2 } \, \left[(n+\beta)P_{n-1}^{(\alpha +1,\, \beta )}(x) +(n+\alpha)P_{n-1}^{(\alpha ,\, \beta +1)}(x) \right]=\\
&=\frac{n+\alpha +\beta +1}{2} P_{n-1}^{(\alpha +1,\, \beta +1)} (x).
\end{split}\end{equation}
\end{propert}

\begin{propert}
The relevant generating function can be written as

\begin{equation}\begin{split}
\sum _{n=0}^{\infty }\frac{t^{n} }{\, \Gamma (n+\alpha +1)\, \Gamma (n+\beta +1)}  P_{n}^{(\alpha ,\, \beta )} (x)=& \left(\dfrac{2}{ \sqrt{2(x-1 )t}} \right)^{\alpha } \left(\dfrac{2}{ \sqrt{2(x+1 )t}} \right)^{\beta }\cdotp  \\ 
& \cdotp I_{\alpha } (\sqrt{2\, (x-1)\, t} )\, I_{\beta } \left(\sqrt{2\, (x+1)\, t} \right).
\end{split}\end{equation}
\end{propert}
We can now deduce further consequences from the previous umbral restyling of the Jacobi polynomials, e.g.\\

\begin{prop}
The index doubling "theorem" can derived by noting that

\begin{equation}\begin{split} \label{GrindEQ__15_3_Geg} 
 R_{2\, n}^{(\alpha ,\, \beta )} (\xi ,\, \eta )&=(2n)!\, \hat{c}_{1}^{\alpha } \hat{c}_{2}^{\beta } \left[\hat{c}_{1} \xi +\hat{c}_{2} \eta \right]^{\, n} \left[\hat{c}_{1} \xi +\hat{c}_{2} \eta \right]^{\, n} \, \varphi _{1,\, 0} \varphi _{2,0} = \\ 
& =\frac{(2\, n)!}{n!} \sum _{s=0}^{n}\binom {n}  {s} \,  \, \xi ^{n-s} \eta ^{s} R_{n}^{(n-s+\alpha ,\, s+\beta )} (\xi ,\, \eta ) 
\end{split}\end{equation} 
which, on account of eq. \ref{JacobiGeg}, yields 

\begin{equation}\begin{split}
P_{2\, n}^{(\alpha ,\, \beta )} (x)=& \dfrac{n!}{(2n)!}\Gamma (2\, n+\alpha +1)\, \Gamma (2\, n+\beta +1)\sum _{s=0}^{n}\binom {n}  {s} \,  \, \left( \xi (x)\right) ^{n-s} \left( \eta (x)\right) ^{s} \cdotp \\ 
& \cdotp\dfrac{P_{n}^{(n-s+\alpha ,\, s+\beta )} (x)}{\Gamma (2\, n-s+\alpha +1)\, \Gamma (n+s+\beta +1)}.
\end{split}\end{equation}
\end{prop}
Furthermore, an analogous procedure yields the following

\begin{prop}
The argument duplication formula is 

\begin{equation}\begin{split} \label{GrindEQ__17_3_Geg} 
P_{\, n}^{(\alpha ,\, \beta )} (2\, x)=& \dfrac{\Gamma (n+\alpha +1)\, \Gamma (n+\beta +1)}{n!}\, \sum _{s=0}^{n}\binom{n}{s} \left(\dfrac{x}{2} \right)^{s} \sum _{r=0}^{s}\binom{s}{r}\cdotp\\
& \cdotp\dfrac{(n-s)!P_{n-s}^{(s-r+\alpha ,\, r+\beta)} (x)}{\Gamma (\, n-r+\alpha +1)\, \Gamma (n-s+r+\beta +1)} .
\end{split}\end{equation} 
\end{prop}
It is evident that the method is so straightforward that all the previous identities can easily be generalized, as touched on in the following.

\begin{prop}
The associated Laguerre polynomials \ref{diffAHP}-\ref{assAHP} \\$\left( L_{n}^{(\alpha )}(x,\, y)-\Lambda _{n}^{(\alpha )} (x,\, y)\right) $ further confirm the mutual link with the Jacobi family and we find indeed that 

\begin{equation}\begin{split} \label{GrindEQ__19_3_Geg} 
 R_{n}^{(\alpha,\beta)}(\xi,\eta)&=(n!)^{2}\sum_{s=0}^{n}\dfrac{(-1)^{s}L_{n-s}^{(\alpha)}(-\xi,\xi+\eta)\;L_{s}^{(\beta)}(\eta,\xi+\eta)}{\Gamma(n-s+\alpha+1)\Gamma(s+\beta+1)},\\
 P_{n}^{(\alpha ,\, \beta )} (x)&=\frac{\Gamma (n+\alpha +1)\, \Gamma (n+\beta +1)}{n!} \, \, \hat{c}_{1}^{\alpha } \hat{c}_{2}^{\beta }\cdot\\
& \cdot \left[\left(y+\hat{c}_{1} \frac{x-1}{2} \right)-\left(y-\hat{c}_{2} \frac{x+1}{2} \right)\right]^{\, n} \varphi _{1,\, 0} \varphi _{2,\, 0} = \\ 
 =\Gamma (n+\alpha +1)& \Gamma (n+\beta +1)\, \sum _{s=0}^{n}\dfrac{(-1)^{s} L_{n-s}^{(\alpha )} \left(\dfrac{1-x}{2} ,\, y\right)\, L_{s}^{(\beta )} \left(\dfrac{x+1}{2} ,\, y\right)}{\Gamma (n-s+\alpha +1)\, \Gamma (s+\beta +1)} .
\end{split}\end{equation} 
\end{prop}

We have covered some of the properties of Jacobi polynomials by employing a minimal computational effort, we have fixed the formalism we are going to use and have provided an idea of the consequences which can be drawn by means of these methods.

\subsection{Legendre Polynomials}

The Legendre polynomials are a particular of Jacobi \cite{L.C.Andrews} and can be identified as

\begin{cor}
	By according to the positions in \ref{JacobiGeg} $\forall x\in\mathbb{R}, \forall n\in\mathbb{N}$, we get
\begin{equation} \label{GrindEQ__1_4_Geg} 
P_{n} (x)=P_{n}^{(0,\, 0)} (x)=R_{n}^{(0,0)}(\xi,\eta). 
\end{equation} 
\end{cor}
Their properties can be therefore derived as a consequence of those of the $R_n$ polynomials in the particular case of $\alpha =\beta =0$. Let us therefore go back to eq. \ref{GrindEQ__19_3_Geg} and note that 

\begin{propert}
\begin{equation}\begin{split}\label{legendreGeg}
 \dfrac{1}{n!} P_{n} (x)&=\, \left[\hat{c}_{1} \frac{x-1}{2} +\hat{c}_{2} \frac{x+1}{2} \right]^{\, n} \, \varphi _{1,\, 0} \varphi _{2,0} , \\ 
 \dfrac{1}{n!} P_{n} (0)&=  \left( -\dfrac{\hat{c}_{1}}{2} +\dfrac{\hat{c}_{2}}{2} \right) ^{n} \, \varphi _{1,\, 0} \varphi _{2,0} =\dfrac{1}{n!}  R_{n} \left(-\frac{1}{2} ,\, \frac{1}{2} \right)= \\ 
& =\dfrac{(-1)^{n} }{2^{n}n!} \sum _{s=0}^{n}(-1)^{s}\left(\frac{n!}{s!\, (n-s)!} \right)^{2}  ,\\
& P_{n} (1)=R_{n} (0,\, 1)=1,\\
& P_{n} (-1)=R_{n}(-1,\, 0)=(-1)^{n} .
\end{split}\end{equation}

The use of the auxiliary polynomials  $R_{n} $ is a fairly important tool to state further identities, as e.g.

\begin{equation}\begin{split} \label{GrindEQ__3_4_Geg} 
 P_{n} (\lambda \, x)&=n!\, \left[\hat{c}_{1} \frac{\lambda \, x-1}{2} +\hat{c}_{2} \frac{\lambda \, x+1}{2} \right]^{\, n} \, \varphi _{1,\, 0} \varphi _{2,0} = \\ 
& =n!\, \left[\lambda \, \left(\hat{c}_{1} \frac{\, x-1}{2} +\hat{c}_{2} \frac{\, x+1}{2} \right)+\hat{c}_{1} \frac{\, \lambda -1}{2} +\hat{c}_{2} \frac{\, -\lambda +1}{2} \right]^{\, n} \, \varphi _{1,\, 0} \varphi _{2,0} = \\ 
& =(n!)^{2} \sum _{s=0}^{n}\,\lambda^{n-s} \sum _{r=0}^{s}\dfrac{\xi (\lambda )^{s-r}  \eta (-\lambda )^{r}P_{n-s}^{(s-r,\, r)}(x)}{(s-r)!r!(n-r)!(n-s+r)!}.
\end{split}\end{equation} 
Furthermore we obtain

\begin{equation}\begin{split} \label{GrindEQ__4_4_Geg} 
 P_{n+m} (\, x)&=(n+m)!\, \left[\hat{c}_{1} \frac{\, x-1}{2} +\hat{c}_{2} \frac{\, x+1}{2} \right]^{\, n+m} \, \varphi _{1,\, 0} \varphi _{2,0} = \\ 
& =n!m! \sum _{s=0}^{m}\binom{n+m}{s} \dfrac{\xi(x)^{m-s}\eta(x)^{s}P_{n}^{(m-s,\, s)}(x)}{(m-s)!(n+s)!} 
\end{split}\end{equation} 
and

\begin{equation} \label{GrindEQ__5_4_Geg} 
P_{n} (\, x+y)=(n!)^{2} \sum _{s=0}^{n}\,\left(\frac{y}{2} \right)^{s} \sum _{r=0}^{s}\dfrac{P_{n-s}^{(s-r,\, r)} \left(x \right)}{(s-r)!r!(n-r)!(n-s+r)!}.
\end{equation} 
The previous identity cannot be considered an ``addition theorem'' in the strict sense, but rather a Taylor series expansion.
\end{propert}

The next step is the derivation of the differential equation satisfied by the Legendre polynomials.

\begin{prop}
By the use of eqs. \ref{npunoGeg}, \ref{depnGeg} we get

\begin{equation}\begin{split}
 n\, P_{n-1} (x)&=\left[(1-x^{2} )\, \frac{d}{dx} +n\, x\right]\, P_{n} (x) \\ 
 (n+1)P_{n+1} (x)&=\left\{\left(2\, n+1\right)\, x-\left[(1-x^{2} )\, \frac{d}{dx} +n\, x\right]\right\}\, P_{n} (x).
\end{split}\end{equation} 
By combining the previous recurrences, we can introduce the following operators

\begin{equation}\begin{split} \label{GrindEQ__8_4_Geg} 
&\hat{N}_{-} =(1-x^{2} )\, \frac{d}{dx} +\hat{n}\, x, \\ 
& \hat{N}_{+} =-(1-x^{2} )\, \frac{d}{dx} +(\hat{n}+1\, )\, x ,
\end{split}\end{equation} 
defined in such a way that

\begin{equation}\begin{split}
& \hat{N}_{-} P_{n} (x)=n\, P_{n-1} (x), \\ 
&\hat{N}_{+} P_{n} (x)=(n+1)\, P_{n+1} (x),
\end{split}\end{equation}
where $\hat{n}\, $ is a kind of number operator ``counting'' the index  of the Legendre polynomial, namely

\begin{equation}
\hat{n}\, P_{m+k} (x)=(m+k)\, P_{m+k} (x).
\end{equation}                                              
According to the previous definitions we find

\begin{equation} \begin{split}\label{GrindEQ__10_4_Geg} 
 \hat{N}_{+} \hat{N}_{-} P_{n} (x)&= \left[ -(1-x^{2})\dfrac{d}{dx}+(\hat{n}+1)x\right]\left[ (1-x^{2})\dfrac{d}{dx}+\hat{n}x\right]P_{n}(x)=   \\ 
& =\left[-(1-x^{2} )\, \frac{d}{dx} +n\, x\right]\, \left[(1-x^{2} )\, \frac{d}{dx} +n\, x\right]\, P_{n} (x)=n^{2} P_{n} (x),
\end{split}\end{equation}
which explicitly yields the following second order equation satisfied by the Legendre polynomials written in the form

\begin{equation}
\left(\, \frac{d}{dx} (1-x^{2} )\, \frac{d}{dx} \right)\, P_{n} (x)+n\, (n+1)\, P_{n} (x)=0 .                       
\end{equation}
\end{prop}

In the forthcoming section we extend the umbral formalism to make further progress by including the properties of the associated Legendre polynomials and the theory of Spherical Harmonics.

\subsection{Generalized Forms}

In this section we explore generalized forms of the presented polynomials families.

\begin{prop}
 Let us consider the evaluation of the following repeated derivatives $\forall m\in\mathbb{N}, \forall \alpha,\beta,a,b\in\mathbb{R},\in Re(\nu)>0, \forall x\in\mathbb{R}:Q(x)>-1$ 

\begin{equation}\begin{split}
& F_{m}^{(\nu)}(x)=\left( \dfrac{d}{dx}\right)^{m}\left(  \dfrac{e^{P(x)}}{(1+Q(x))^{\nu}}\right),  \\
& P(x)=\alpha x^{2}+\beta x,\\
& Q(x)=ax^{2}+bx .
\end{split}\end{equation}
The use of the umbral procedure \ref{defGeg} allows a significant simplification of the relevant algebra. By setting indeed

\begin{equation}
\dfrac{e^{P(x)}}{  (1+Q(x))^{\nu}}=e^{P(x)}e^{(\nu)}(Q(x))=e^{P(x)-\hat{\gamma}Q(x)}\nu_{0}=e^{(\alpha -\hat{\gamma}a)x^{2}+(\beta-\hat{\gamma}b)x}\nu_{0},
\end{equation}
we find, applying eq. \ref{vasteH} ,

\begin{equation}
F_{m}^{(\nu)}(x)=H_{m}\left( 2(\alpha -\hat{\gamma}a)x+(\beta-\hat{\gamma}b),\alpha -\hat{\gamma}a\right) e^{P(x)-\hat{\gamma}Q(x)}\nu_{0}.
\end{equation}
The use of the so far developed rules yields

\begin{equation}\begin{split}
 F_{m}^{(\nu)}(x)&=m!\sum_{r=0}^{ \lfloor\frac{m}{2}\rfloor }\dfrac{1}{(m-2r)!r!}\sum_{s=0}^{m-2r}\binom{m-2r}{s}A^{m-2r-s}B^{s}\cdot\\
 & \cdot \sum_{q=0}^{r}\binom{r}{q}(-1)^{s+q}\alpha^{r-q}a^{q}(\nu)_{s+q}\; e^{(\nu+s+q)}(Q(x))e^{P(x)}\\
 A&=2\alpha x+\beta;\\
 B&=2 a x+b
\end{split}\end{equation}
or, in a more compact form

\begin{equation}\begin{split}
& F_{m}^{(\nu)}(x)=\Omega_{m}^{(\nu)}\left( P^{'}(x),\dfrac{P^{''}(x)}{2};\dfrac{Q^{'}(x)}{(1+Q(x))},\dfrac{Q^{''}(x)}{2(1+Q(x))} \right)\dfrac{e^{P(x)}}{\left(1+Q(x) \right)^{\nu} } \\
& \Omega_{m}^{(\nu)}(x,y;u,z)=\sum_{s=0}^{m}\binom{m}{s}(-1)^{s}H_{m-s}(x,y)K_{s}^{(\nu)}(u,-z).
\end{split}\end{equation}
\end{prop}
It is now worth to explore more accurately the role of the $K_{n}^{(\nu)}(.,.)$
polynomials introducted in section \ref{UTGeg}. To this aim we consider the following particular case 

\begin{exmp}
Let  $\nu=\frac{1}{2}$, by using eq. \ref{KGehex2}, we get

\begin{equation}
e_{n}^{(\frac{1}{2})}(x^{2})=(-1)^{n}K_{n}^{(\frac{1}{2})}\left( \dfrac{2x}{1+x^{2}},-\dfrac{1}{1+x^{2}}\right) \dfrac{1}{(1+x^{2})^{\frac{1}{2}}},
\end{equation}
furthermore, by recalling the identity \ref{KmultiplyGeg}

\begin{equation}
K_{n}^{(\frac{1}{2})}(a,b)=\dfrac{1}{\sqrt{\pi}}\int_{0}^{\infty}e^{-s}s^{-\frac{1}{2}}H_{n}(as,bs)ds
\end{equation}
and, reminding eq. \ref{idGeg} $H_{n}(x,y)=y^{\frac{n}{2}}H_{n}\left( \dfrac{x}{\sqrt{y}},1\right)
$, we can easily infer that, $\forall t\in\mathbb{R}:1-at-bt^{2}>0$,

\begin{equation}
\sum_{n=0}^{\infty}\dfrac{t^{n}}{n!}K_{n}^{(\frac{1}{2})}(a,b)=\dfrac{1}{\sqrt{1-at-bt^{2}}},
\end{equation}
which is the \textbf{generating function of Legendre polynomials} for $a=2x$, $b=-1$. Moreover the use of the identity \ref{idGeg} yields

\begin{equation}
(1+x^{2})^{\frac{n+1}{2}}e_{n}^{(\frac{1}{2})}(x^{2})=(-1)^{n}n!P_{n}\left(\dfrac{x}{\sqrt{1+x^{2}}} \right) ,
\end{equation}
which can be extended to the cases involving the generalized
Legendre forms.
\end{exmp}

The same procedure can be applied to derive the following generating function for ordinary Legendre. 

\begin{exmp}
	$\forall n,l\in\mathbb{N},\forall x\in\mathbb{R},\forall t\in\mathbb{R}:(1-2xt+t^{2})^{{l+1}}>0$, we find
\begin{equation} \label{GrindEQ__6_4_Geg} 
\sum _{n=0}^{\infty }\binom{n+l}{l}\, t^{n}  P_{n+l} (x)=\frac{P_{l}\left(\dfrac{x-t}{\sqrt{1-2xt+t^{2}}} \right)  }{(1-2xt+t^{2})^{\frac{l+1}{2}} } . 
\end{equation}
It is also a particular case of eq. \ref{GrindEQ__11_2_Geg}.
\end{exmp}

\begin{cor}
According to the above
point of view, the $m^{th}$ derivative of the $P_{n}(x)$ can therefore be
easily calculated, thus finding

\begin{equation}\begin{split}
\left(\frac{d}{dx} \right)^{m} P_{n} (x)&=\dfrac{2^{m}}{\sqrt{\pi}(n-m)!}\int_{0}^{\infty}e^{-s}s^{m-\frac{1}{2}}H_{n-m}(2xs,-s)ds=\\[1.1ex]
& =\dfrac{1}{\sqrt{\pi}}\sum_{r=0}^{\lfloor\frac{n-m}{2}\rfloor  }\dfrac{(-1)^{r}2^{n-2r}x^{n-m-2r}\Gamma(n-r+\frac{1}{2})}{(n-m-2r)!r!}.
\end{split}\end{equation}

On the other side the successive derivatives of the Legendre
polynomials can be obtained from \ref{legendreGeg}  and yields the
following link with the Jacobi polynomials

\begin{equation} \begin{split}
 \left(\frac{d}{dx} \right)^{m} P_{n} (x)&=\frac{1}{2^{m} } \frac{\left(n!\right)^{2} }{(n-m)!} \left(\hat{c}_{1} +\hat{c}_{2} \right)^{m} \left[\hat{c}_{1} \frac{x-1}{2} +\hat{c}_{2} \frac{x+1}{2} \right]^{\, n-m} \, \varphi _{1,\, 0} \varphi _{2,0} = \\[1.1ex] 
& = \dfrac{\left(n!\right)^{2} }{2^{m} } \sum _{s=0}^{m}\binom{m}{s}\,  \dfrac{P_{n-m}^{(m-s,\, s)} (x)}{(n-s)!(n-m+s)!}.
\end{split}\end{equation}
\end{cor}

We now discuss a further alternative formulation of the theory of Legendre polynomials, using a formalism touched in \cite{DGMR}, which
will be embedded with the technique developed until now.

\begin{Oss}
We notice that we can obtain a similar result at \ref{eq2HermLagbis} by introducing a new family of polynomials $\forall x,y\in\mathbb{R}, \forall n\in\mathbb{N}$ 

\begin{equation} \label{GrindEQ__8_5_Geg} 
\Pi _{n} (x,\, y)=\left( (x+{}_y\hat{h }_{c})^{n} \theta_0 \right) \varphi _{0}  ,
\end{equation} 
where (by using the $\hat{c}$-operator \ref{Opc})

\begin{equation}\label{GrindEQ__9_5_Geg} 
 \left( {}_y\hat{h }_{c}^{r}\; \theta_0\right)  \varphi _{0} =\frac{y^{\frac{r}{2} } \, r!}{\Gamma \left(\frac{r}{2} +1\right)} \left|\cos \left(r\frac{\pi }{2} \right)\right|\, \hat{c}^{\frac{r}{2} } \varphi _{0}  =\frac{y^{\frac{r}{2} } \, r!}{\Gamma \left(\frac{r}{2} +1\right)^{2} } \left|\cos \left(r\frac{\pi }{2} \right)\right|.
\end{equation} 
According to the above definition (by following the proof of Proposition \ref{propHpol}), we obtain the explicit expression for the $\Pi $ polynomials  as

\begin{equation} \label{GrindEQ__10_5_Geg} 
\Pi _{n} (x,\, y)=H_{n} (x,\, \hat{c}\, y)\, \varphi _{0} =n!\, \sum _{k=0}^{\lfloor\frac{n}{2}\rfloor }\frac{x^{n-2\, k} y^{k} }{\left(k!\right)^{2} \, (n-2\, k)!}  .  
\end{equation} 
They are essentially\textbf{ hybrid Laguerre-Hermite polynomials} (including the operator ${}_y\hat{h}$ di Laguerre and the operator $\hat{c}$ di Hermite) satisfying the "heat'' equation 

\begin{equation} \label{GrindEQ__11_5_Geg} 
\left\lbrace  \begin{array}{l}
 {}_{l} \partial_{y} G(x,\, y)=-\partial _{x}^{2} G\left(x,y\right) \\[1.6ex] 
 G(x,\, 0)=x^{n} ,
\end{array}\right.\end{equation} 
with ${}_{l} \partial_{y}$ l-derivative \ref{prove}. We note that the polynomials \ref{GrindEQ__10_5_Geg} can be defined through the operational rule (by using \ref{TrBzero})

\begin{equation} \label{GrindEQ__12_5_Geg} 
\Pi _{n} (x,\, y)=C_{0} (-y\, \partial _{x}^{2} )\, x^{n}  
\end{equation} 
or also (by applying \ref{C0J0})
\begin{equation} \label{eq21HermLag} 
\Pi _n(x,y)=J_0\left( 2 i \sqrt{y}\,\partial_x \right) x^n
\end{equation}
and the Legendre polynomials can be identified with the particular case

\begin{equation}
P_{n} (x)=\Pi _{n} \left( x,\, -\frac{1-x^{2} }{4} \right).            
\end{equation}
The procedure suggests that the \textbf{negative order hybrid functions} can be written as 

\begin{equation} \label{eq2HermLag2} 
\Pi_{-\nu} (x,y)=\left( (x+{}_y\hat{h }_{c} )^{-\nu} \theta_0 \right) \varphi _{0} =\dfrac{1}{\Gamma (\nu )}\int_{0 }^\infty s^{\nu -1}e^{-sx}J_0\left( 2is\sqrt{y}\right)\; ds.
\end{equation} 
Infinite integrals containing e.g. products of Bessel and exponential functions, like \ref{eq17HermLag}, can therefore be expressed in terms of  negative order $\Pi$  functions. We find for example that

\begin{equation} \label{eq23HermLag} 
S(a,b)=\int_{0 }^\infty e^{-ax}J_0(bx)dx= \Pi _{-1}\left( a,-\frac{b^2}{4}\right) .
\end{equation}  
Finally, from the previous identities we find the \textbf{Legendre generating function}, $\forall t\in\mathbb{R}$,

\begin{equation}\label{functgenPnGeg}
\sum _{n=0}^{\infty }\frac{t^{n} }{n!} P_{n} (x)=e^{x\, t-\hat{c}\frac{1-x^{2} }{4} t^{2} } \varphi _{0} =e^{x\, t}  J_{0} \left[t\sqrt{1-x^{2} } \right].               
\end{equation}

We can now derive a further consequence from the above equation
and from the umbral definition of the Legendre polynomials \ref{legendreGeg}, according to which we find

\begin{equation}\begin{split}  \label{GrindEQ__17_5_Geg} 
 \sum _{n=0}^{\infty }\frac{t^{n} }{n!} P_{n} (x)&=\sum _{n=0}^{\infty }t^{n}  \left[\hat{c}_{1} \frac{x-1}{2} +\hat{c_{2}}\frac{x+1}{2} \right] ^{n}  \varphi _{1,0} \; \varphi _{2,0} = \\ 
& =\frac{1}{1-t\, \left[\hat{c}_{1} \frac{x-1}{2} +\hat{c}_{2}\frac{x+1}{2} \right]}  \varphi _{1,0}  \;\varphi _{2,0}. 
\end{split}\end{equation} 
The use of standard Laplace transform identities yields

\begin{equation}\begin{split} 
& \frac{1}{1-t\, \left[\hat{c}_{1} \frac{x-1}{2} +\hat{c}_{2}\frac{x+1}{2} \right]} \, \left[\varphi _{1,0}\;  \varphi _{2,0} \right]=\int _{0}^{\infty }e^{-s}  e^{s\, t\, \left(\hat{c}_{1} \frac{x-1}{2} +\hat{c}_{2}\frac{x+1}{2} \right)} ds\;\varphi _{1,0}\;  \varphi _{2,0}= \\[1.1ex] 
& =\int _{0}^{\infty }e^{-s}  C_{0} \left( \frac{1-x}{2} s\, t\right) \, C_{0} \left( -\frac{1+x}{2} s\, t\right) \, ds,
\end{split}\end{equation}
which once confronted with \ref{functgenPnGeg} yields

\begin{equation}\begin{split} \label{GrindEQ__16_5_Geg} 
&\int _{0}^{\infty }e^{-s}  C_{0} \left( \frac{1-x}{2} s\, t\right) \, C_{0} \left( -\frac{1+x}{2} s\, t\right) \, ds= \\ 
& =e^{xt} J_{0} \left( t\sqrt{1-x^{2} } \right).   
\end{split}\end{equation} 
\end{Oss}

Before closing this part we underline as, for the proposed umbral definition of the Gegenbauer
polynomials, we can adopt an
analogous point of view by noting that the use of the
operational identity

\begin{equation} \label{GrindEQ__20_5_Geg} 
M(a\, x,\, b\, y)=a^{x\, \partial _{x} } b^{y\, \partial _{y} }M(x,\, y), 
\end{equation} 
based on the Euler dilation operator \cite{Mansour}, allows the following reshuffling of eq. \ref{KmultiplyGeg}

\begin{equation}
K_{n}^{(\nu )} (\xi ,\, \eta )=\frac{1}{\Gamma (\nu )} \int _{0}^{\infty }e^{-s}  s^{\nu -1} s^{\xi \, \partial _{\xi }+ \eta \, \partial _{\eta } } \, ds\, H_{n} (\, \xi ,\, \eta )
\end{equation} 
and the use of the properties of the Gamma function eventually leads to the following operational definition of the Gegenbauer polynomials

\begin{equation}\begin{split} 
& K_{n}^{(\nu )} (\xi ,\, \eta )=\hat{\Gamma }_{\nu } \, H_{n} (\xi ,\, \eta ),\\[1.1ex] 
& \hat{\Gamma }_{\nu } = \dfrac{\Gamma (\nu +\xi \, \partial _{\xi } +\eta \, \partial _{\eta } )}{\Gamma (\nu )} .
\end{split}\end{equation}
The operator $\hat{\Gamma }_{\nu } $ is therefore a differential realization of its umbral counterpart $\hat{\gamma}$.

\section{Voigt Functions}

The Hermite and Laguerre functions are the extension of the corresponding polynomials to negative and/or real indices. In this paragraph we apply the proposed method for the study of the \textit{\textbf{Voigt functions}} which find several applications in spectroscopy \cite{Thorne} and are defined by the convolution of Gaussian and Lorentzian distributions. Apart from its interest in Physics, they have raised a certain interest in Mathematics for their relation with a number of special functions.\\

 In this section, we explore the relevant link with  the Hermite functions and study the associated consequences within the context of the formalism so far developed.

\begin{defn}
We denote the \textbf{ Voigt funtions} (VF) by $K(x,y,z),\;L(x,y,z)$ and define them in terms of the integral representations \cite{Pagnini} $\forall x,z\in\mathbb{R}, \forall y\in\mathbb{R}^+$

\begin{equation} \begin{split}\label{eq24HermLag} 
& K(x,y,z)=\frac{1}{\sqrt{\pi}}\int_{0 }^\infty e^{-x\xi -y\xi ^2}\cos \left( z\xi\right) \;d\xi\;,\\
& L(x,y,z)=\frac{1}{\sqrt{\pi}}\int_{0 }^\infty e^{-x\xi -y\xi ^2}\sin \left( z\xi\right) \;d\xi
\end{split}\end{equation}   
(the definition in \cite{Pagnini} includes two variables only $(x,z)$ and assumes $y=\frac{1}{4}$).\\
If we introduce the complex VF 

\begin{equation} \label{eq26HermLag} 
E(x,y,z)=\frac{1}{\sqrt{\pi}}\int_{0 }^\infty e^{-(x-iz)\xi -y\xi ^2}\;d\xi
\end{equation}  
and define $VF$ in eq. \ref{eq24HermLag} as the relevant real and imaginary parts, we can easily conclude that it is expressible in terms of "erfc" function. We can indeed exploit eqs. \ref{NOHpol}-\ref{NOHf} to end up with the identity 

\begin{equation} \label{eq27HermLag} 
\begin{split}  
E(x,y,z)&=\dfrac{1}{\sqrt{\pi}}\int_{0}^\infty e^{-(x-iz)\xi -y\xi ^2}\;d\xi =
\dfrac{1}{\sqrt{\pi}}H_{-1}(x-iz,y)=\\
& =\dfrac{1}{\sqrt{\pi}}\left( x-iz+{}_y\hat{h}\right) ^{-1}\theta_0= \dfrac{1}{2\sqrt{y}}e^{\frac{(x-iz)^2}{4y}}erfc\left( \dfrac{x-iz}{2\sqrt{y}}\right) .
\end{split}  
\end{equation}  
\end{defn}
We obtain the relevant derivatives by the use of the well known properties of the Hermite functions and find that 

\begin{propert}
	$\forall m\in\mathbb{N}$
\begin{equation} \label{eq28HermLag}  
\partial _z^m E(x,y,z)=\dfrac{i^m\; m!}{\sqrt{\pi}} H_{-(m+1)}(x-iz,y).   
\end{equation}  
\end{propert}

The procedure we have envisaged allows to unify many of the previous analyses \cite{M.A.Pathan} aimed at getting different way  of expressing the \textit{VF} in forms suitable for various specific applications.\\

Furthermore, we note that, by using the following
\begin{defn}
We define  \textbf{Voigt} (V-) \textbf{transform} of a function f(z) 

\begin{equation} \label{eq29HermLag} 
{}_V \hat{f}(x,y;z)=\int_{0 }^\infty e^{-xt-yt^2}f(zt)\;dt
\end{equation}  
(which can be viewed as a generalized form of transform).
 \end{defn}     
 \noindent we can get 
 
 \begin{prop}
 
 \begin{equation} \label{eq32HermLag} 
 \begin{split}  
 & {}_V \hat{f}(x,y;z)= \sum _{n=0}^\infty a_n H_{-(n+1)}(x,-y)z^n,\\
 & f(z)=\sum _{n=0}^\infty a_n\dfrac{z^n}{n!}.
 \end{split}
 \end{equation}
 	\end{prop}
 	
 	\begin{proof}[\textbf{Proof.}]
  If we use of the identity \ref{tderf} $t^{z\partial _z }f(z)=f(tz)$, Laplace transform
and the technique of Example \ref{HermCalcImportant}, we can write

\begin{equation} \label{eq31HermLag} 
{}_V \hat{f}(x,y;z)=\displaystyle \int_{0 }^\infty e^{-xt-yt^2}t^{z\partial _z }\;dt\;f(z) 
=\Gamma \left( z\partial _z +1\right) H_{-(z\partial _z +1)}(x,-y)f(z).
\end{equation}               
By expanding the function $f(z)$ in series and by using the property $f \left( z\partial_z \right)z^n=f(n)z^n $ \cite{Babusci}, we can finally write 

\begin{equation*}  
\begin{split}  
& {}_V \hat{f}(x,y;z)= \sum _{n=0}^\infty a_n H_{-(n+1)}(x,-y)z^n,\\
& f(z)=\sum _{n=0}^\infty a_n\dfrac{z^n}{n!}.
\end{split}
\end{equation*}               
\end{proof}

\begin{cor}
According to the present formalism the VF \ref{eq24HermLag} is the V-transform \ref{eq29HermLag} of the \textbf{circular functions}, thus reading 

\begin{equation} \label{eq33HermLag} 
\begin{split} 
& K(x,y,z)= \sum _{n=0}^\infty (-1)^n H_{-(2n+1)}(x,-y)z^{2n},\\
& L(x,y,z)= z\sum _{n=0 }^\infty (-1)^n H_{-(2n+2)}(x,-y)z^{2n}.
\end{split} 
\end{equation}  
\end{cor}              

%

\begin{rem}
\textit{Let us furthermore note that the V-transform of a $0$-order cylindrical Bessel function is}

\begin{equation} \label{eq35HermLag} 
{}_V \hat{f}(x,y;z)=\displaystyle \sum _{n=0 }^\infty \dfrac{(2n)!}{n!^2}(-1)^nH_{-(2n+1)}(x,-y)\left( \frac{z}{2}\right)^{2n}.
\end{equation}  
\textit{The generalization of the $V$-functions proposed in ref. \cite{Srivastava} can be viewed as $V$-transforms of different families of functions.}
\end{rem}

In Chapter \ref{Chapter2} we have noted that Hermite functions of negative order can be defined by means of infinite integrals yielding the relevant integral representation, however the use of the formalism we are proposing may be useful in a wider context as e.g. for evaluation of definite integrals as shown below 

\begin{equation} \label{eq36HermLag} 
\begin{split} 
 \int _{0}^x \xi^{\nu -1}e^{-a\xi -b\xi^2}d\xi &= 
\int _{0 }^x \xi^{\nu -1}e^{-(a+{}_{-\mid b\mid} \hat{h})\xi}d\xi\;\theta_{0}=\\
& = \dfrac{1}{\left( a+{}_{-\mid b\mid} \hat{h}\right) ^\nu}\gamma \left( \nu ,\left( a+{}_{-\mid b\mid} \hat{h}\right) x\right)\theta_{0} =\\
& =\sum_{n=0}^\infty \dfrac{(-1)^n H_n(a,-b)x^{\nu+n}}{n!(\nu+n)},\\
 \gamma (\nu , x)&=\int _{0 }^x \xi^{\nu -1}e^{-\xi }d\xi 
\end{split}
\end{equation}  
with  $\gamma (\nu , x)$ being the incomplete gamma function \cite{L.C.Andrews}.\\

The discussion we have developed yields a fairly interesting mean employing symbolic methods for computational purposes. The technique we have proposes is, in some sense, self generating, namely it can be mounted as a kind of chinese box tool yielding as illustrated below.\\

\begin{exmp} 
Let us therefore consider the following V-transform

\begin{equation}\label{Vfint}
{}_V \hat{f}_\mu (x,y,z)=\int_0^\infty t^{\mu-1}e^{-xt-yt^2}J_0(zt)\;dt,
\end{equation}
which, according to the previous discussion, can formally be written as

\begin{equation}\begin{split}\label{key}
 {}_V \hat{f}_\mu (x,y,z)&=\int_0^\infty t^{\mu-1+z\partial_z} e^{-xt-yt^2}\;dt\;J_0(z)=\\
 & =\sum_{n=0}^\infty \dfrac{\Gamma\left(2n+\mu \right) }{n!^2}(-1)^n H_{-\left(2n+\mu \right)}(x,-y)\left( \dfrac{z}{2}\right)^{2n}=
e^{-\hat{g}\;\left( \frac{z}{2}\right)^2 }\Psi_0,\\
 \hat{g}^n\;\Psi_0&= \dfrac{\Gamma\left(2n+\mu \right) }{n!}H_{-\left(2n+\mu \right)}(x,-y).
\end{split}\end{equation}
This umbral form can be exploited, e.g., to derive the integral 

\begin{equation}\label{key}
\int_{-\infty}^\infty {}_V \hat{f}_\mu (x,y,z)\; dz=
2\sqrt{\pi}\hat{g}^{-\frac{1}{2}}\;\Psi_0=
2\sqrt{\pi}\dfrac{\Gamma\left( \mu-1\right)}{\Gamma\left( \frac{1}{2}\right)}H_{-\left( \mu-1\right)}(x,-y)
\end{equation}
as also checked by a direct integration of eq. \ref{Vfint}, namely

\begin{equation}\begin{split}\label{key}
\int_{-\infty}^\infty {}_V \hat{f}_\mu (x,y,z)\; dz&=
\int_{0}^\infty  t^{\;\mu-1}e^{-xt-yt^2}
\left(\int_{-\infty}^\infty J_0(zt)\;dz \right)dt =\\
& =2\int_{0}^\infty  t^{\;\mu-2}e^{-xt-yt^2}\;dt.
\end{split}\end{equation}
\end{exmp}

 The methods we have envisaged are, accordingly to the previous examples, eligible for an applicative tool in a wide range of applications, e.g. to treat the Free Electron Laser high gain equation \cite{FelHigh}. To  further stress this point of view, we consider the generalization of the functions by noting that the formalism offers noticeable degrees of freedom yielding a significant amount of new directions along it can be developed and applied.

\section{Chebyshev, Lacunary Legendre and \\ Legendre-type Polynomials}

 In this section, we go back to the integral representation method \cite{L.C.Andrews} which allows the framing of different special polynomials and functions, as well, in terms of multivariable Hermite polynomials \cite{Babusci} and of other polynomials belonging to standard and generalized forms of Legendre and Legendre like type. We deal in particular with Chebyshev, Legendre, Jacobi{\dots} \cite{DattLoren}.\\ 

An example of interplay between two variable \textbf{ Chebyshev polynomials of the second kind}  \textit{(CP)} $U_{n} (x,y)$  \cite{L.C.Andrews} and Hermite polynomials is provided by the following

\begin{prop}
Let 

\begin{equation} \label{GrindEQ__1_LacLeg} 
U_{n} (x,y)=\frac{1}{n!} \int _{0}^{\infty }e^{-s}  H_{n} (-s\, x,-sy)ds, \quad \forall x,y\in\mathbb{R}, \forall n\in\mathbb{N} ,
\end{equation} 
the integral representation \cite{L.C.Andrews} of the \textbf{ Chebyshev polynomials of the second kind} $U_{n} (x,y)$ then, by applying eqs. \ref{classHerm}-\ref{genfunctH}-\ref{FunzGamma}, we can recast the explicit definition of $U_{n} (x,y)$ as

\begin{equation}\label{key}
U_{n} (x,y)=(-1)^{n} \sum _{r=0}^{\lfloor\frac{n}{2} \rfloor}\frac{(n-r)!\, x^{n-2r} (-y)^{r} }{(n-2\, r)!\, r!} . 
\end{equation}
\end{prop}

\begin{cor}
Multiplying both sides of eq. \ref{GrindEQ__1_LacLeg} by $t^{n} $ and then by summing up over the index $n$, we obtain the well known \textbf{generating function} of second kind Chebyshev polynomials, namely

\begin{equation}\begin{split}\label{GrindEQ__5_LacLeg} 
\sum _{n=0}^{\infty }t^{n}  U_{n} (x,y)&=\sum _{n=0}^{\infty }\frac{t^{n} }{n!} \int _{0}^{\infty }e^{-s}  H_{n}  (-s\, x,-s\, y)ds=\int _{0}^{\infty }e^{-s\, \left(1+xt+yt^{2} \right)}  ds=\\
& =\frac{1}{1+x\, t+y\, t^{2} } ,\\
 Re(1+x\, t+y\, t^{2})&>0. 
\end{split}\end{equation} 
Using furthermore the generating function \cite{Babusci} \ref{Hnlgf} and employing the same procedure as before we also easily find that 

\begin{equation} \label{GrindEQ__7_LacLeg} 
\sum _{n=0}^{\infty }t^{n}  \frac{(n+l)!}{n!} U_{n+l} (x,y)=l!\dfrac{U_{l} \left(\dfrac{x+2\, y\, t}{1+x\, t+y\, t^{2} } ,\, \dfrac{y}{1+x\, t+y\, t^{2} } \right)}{\left(1+x\, t+y\, t^{2} \right)}  
\end{equation} 
and indeed we get from eq. \ref{GrindEQ__1_LacLeg}
\begin{equation} \label{GrindEQ__8_LacLeg} 
\sum _{n=0}^{\infty }t^{n}  \frac{(n+l)!}{n!} U_{n+l} (x,y)=\int _{0}^{\infty }e^{-s\, (1+xt+yt^{2} )}  H_{l} (-(x+2y\, t)s,-sy)\, ds ,
\end{equation} 
which after redefining the integration variable as $s\, (1+x\, t+y\, t^{2} )=\sigma $ and using again \ref{GrindEQ__1_LacLeg}, yields the result reported in eq. \ref{GrindEQ__5_LacLeg}.
\end{cor} 

It is evident that the procedure we have outlined can easily be extended to generalized forms of Chebyshev polynomials $U_{n}^{(m)} (x,y)$ indeed, through the use of\textit{ lacunary Hermite polynomials}  (see Appendix \ref{higherHermite}), we can state 
\begin{cor}
	We introduce the lacunary Legendre polynomials \cite{PHumbert,Boas} 

\begin{equation}\begin{split} \label{GrindEQ__10_LacLeg} 
 U_{n}^{(m)} (x,y)&=\frac{1}{n!} \int _{0}^{\infty }e^{-s}  H_{n}^{(m)} (-s\; x,-s\;y)ds=\\
 & =
(-1)^{n} \sum _{r=0}^{\lfloor\frac{n}{m} \rfloor}\frac{(-1)^{(m-1)r} (n-(m-1)\; r)!\; x^{n-mr} y^{r} }{(n-m\; r)!\; r!}  , \\ 
 \sum _{n=0}^{\infty }t^{n}  U_{n}^{(m)} (x,y)&=\frac{1}{1+x\; t+y\; t^{m} },\\
 \forall m\geq2,\;\;  &Re(1+x\, t+y\, t^{m})>0.
\end{split} \end{equation} 
\end{cor}
For this more general case too, the use of the properties of \textit{lacunary HP} can be usefully exploited to explore those of \textit{lacunary CP}. To this aim we note e.g. that 

\begin{prop}
Let, $\forall m,l\in\mathbb{N}, t\in\mathbb{R}$, 

\begin{equation} \begin{split}\label{GrindEQ__11_LacLeg} 
& \sum _{n=0}^{\infty }\frac{t^{n} }{n!}  H_{n+l}^{(m)} (x,y)=H_{l}^{(m,m-1,...,1)} \left(\left\{\frac{p_{m}^{(n)} (x,y;t)}{n!} \right\}_{n=1,...,m} \right)e^{p_{m} (x,y;\, t)} ,\\ 
& H_{n}^{(p,p-1,...,1)} (x_{1} ,...,x_{p} )=n!\sum _{r=0}^{\lfloor\frac{n}{p}\rfloor}\frac{H_{n-p\, r}^{(p-1,p-2,...,1)} (x_{1} ,...,x_{p-1} )\, x_{p}^{r} }{(n-p\, r)!r!}, \\ 
& p_{m} (x,y;t)=x\, t+y\, t^{m} , \\ 
& p_{m}^{(n)} (x,y; t)=\partial _{t}^{n} p_{m} (x,y;t),\quad n\le m  
\end{split} \end{equation} 
the Rainville generating function \cite{DattLoren}, where $H_{n}^{(p,p-1,...,1)} (x_{1} ,...,x_{p} )$ are p-variable complete (non lacunary) Hermite polynomials with generating function \cite{Babusci}

\begin{equation} \label{GrindEQ__12_LacLeg} 
\sum _{n=0}^{\infty }\frac{t^{n} }{n!}  H_{n}^{(p,p-1,...,1)} (x_{1} ,...,x_{p} )=e^{\sum _{s=1}^{p}x_{s} t^{s}  }.  
\end{equation} 
Then,

\begin{equation}\label{GrindEQ__13_LacLeg} 
\sum _{n=0}^{\infty }t^{n}  \frac{(n+l)!}{n!} U_{n+l}^{(m)} (x,y)=l!
\frac{U_{l}^{(m,m-1,...)} \left(\frac{p_{m}^{(1)} (x,y;t)}{1+p_m(x,y;t)} ,\, \frac{1}{2} \frac{p_{m}^{(2)} (x,y;\, t)}{1+p_m(x,y;t)} ,...,\frac{1}{m!} \frac{p_{m}^{(m)} (x,y;t)}{1+p_m(x,y;t)} \right)}{1+p_{m} (x,y;t)} 
\end{equation} 
where the complete $p$-variable Chebyshev polynomials are specified, by means of the Laplace transform

\begin{equation}\label{key}
U_{n}^{(p,p-1,...,1)} (x_{1} ,...,x_{p} )=\frac{1}{n!} \int _{0}^{\infty }e^{-s}  H_{n}^{(p,p-1,...,1)} (-x_{1} s,...,-x_{p} s)ds, 
\end{equation}                    
straightforwardly yielding the generating function

\begin{equation} \label{GrindEQ__15_LacLeg} 
\sum _{n=0}^{\infty }t^{n}  U_{n}^{(p,p-1,...,1)} (x_{1} ,...,x_{p} )=\frac{1}{1+\sum _{s=1}^{p}x_{s} t^{s}  }, \;\; Re\left( 1+\sum _{s=1}^{p}x_{s} t^{s}\right) >0 . 
\end{equation} 
\end{prop}

The formalism associated with generalized Chebyshev polynomials is fairly flexible, therefore if we are interested in the successive derivatives of a rational function, we find

\begin{cor}
\begin{equation}\begin{split} \label{GrindEQ__16_LacLeg} 
 &\partial _{t}^{m} \left(\frac{1}{1+p_{2} (x,y;t)} \right)=\partial _{t}^{m} \int _{0}^{\infty }e^{-s} e^{-s\, x\, t-s\, y\, t^{2} }  ds=\\
& =\int _{0}^{\infty }e^{-s}  H_{m} \left( -sp_{2}^{(1)} (x,y;t)\, ,\, -\frac{s}{2} p_{2}^{(2)} (x,y;t)\right)e^{-s\, x\, t-s\, y\, t^{2} } \, ds =\\
& =m!\frac{U_{m} \left(\frac{p_{2}^{(1)} (x,y;t)}{1+p_{2} (x,y;t)} ,\frac{1}{2} \frac{p_{2}^{(2)} (x,y;t)}{1+p_{2} (x,y;t)} \right)}{1+p_{2} (x,y;\, t)} ,
\end{split} \end{equation} 
which yields a kind of Rodriguez formula for Chebyshev type polynomials, as reported below

\begin{equation} \label{GrindEQ__17_LacLeg} 
\left( 1+p_{2} (x,y;\, t)\right) \partial _{t}^{m} \left(\frac{1}{1+p_{2} (x,y;t)} \right)=m! U_{m} \left(\frac{p_{2}^{(1)} (x,y;t)}{1+p_{2} (x,y;t)} ,\frac{1}{2} \frac{p_{2}^{(2)} (x,y;t)}{\left( 1+p_{2} (x,y;t)\right) } \right),
\end{equation} 
which should be confronted with an analogous expression valid for the Higher order Hermite polynomials

\begin{equation} \label{GrindEQ__18_LacLeg} 
\partial _{t}^{m} \left(e^{p_{n} (x,y;t)} \right)=H_{m}^{(n,n-1,...,1)} \left(\left\{\frac{p_{n}^{(s)} (x,y;t)}{s!} \right\}_{s=1,...,p} \right)e^{p_{n} (x,y;t)}.  
\end{equation} 
\end{cor}

The formalism we have just discussed shows how the properties of Hermite polynomials and of its generalized forms is a powerful tool to deal with other families of polynomials. In the forthcoming section we apply the obtained results to Chebyshev polynomials family.

\subsection{ Umbral Methods and Chebyshev Polynomials }

According to the bynomial umbral formalism used in section \ref{NewtonBinHP} about \textsl{HP}, the definition of the two variable Chebyshev polynomials can be given as

\begin{prop}
	We recast the two variable Chebyshev polynomials in integral umbral bynomial terms as
\begin{equation} \label{GrindEQ__12b_LacLeg} 
U_{n} (x,y)=\frac{1}{n!} \int _{0}^{\infty }e^{-s\, }  (-x\, s+{}_{(-y\, s)} \hat{h}\, )^{n} \theta_0\;ds ,
\end{equation} 
or in umbral form as

\begin{equation} \label{GrindEQ__13b_LacLeg} 
U_{n} (x,y)=\frac{1}{n!} (-x\, \hat{f}+{}_{(-y\, \hat{f})} \hat{h})^{n} \beta_0\theta_0 ,
\end{equation} 
where $\hat{f}^{r}\beta_0 :=\beta_r=r!$.
\end{prop}

\begin{proof}[\textbf{Proof.}]
	To proof eq. \ref{GrindEQ__12b_LacLeg} is enough to substitute HP with bynomial umbral expression \ref{eq2HermLagbis}. About eq. \ref{GrindEQ__13b_LacLeg},
	by using eqs. \ref{eq2HermLagbis}, Newton bynomial and $\hat{f}$-operator, we get
\begin{equation*}\begin{split} \label{GrindEQ__14b_LacLeg} 
 U_{n} (x,y)&=\frac{1}{n!} \sum _{r=0}^{n}\left(\begin{array}{c} {n} \\ {r} \end{array}\right) (-x\, \hat{f})^{n-r} \left[{}_{(-y\, \hat{f})} \hat{h}\right]^{\; r}\beta_0\theta_0 =\\
 & =\frac{1}{n!} \sum _{r=0}^{n}\left(\begin{array}{c} {n} \\ {r} \end{array}\right) (-x)^{n-r} \, \hat{f}^{n-r} (-y\, \hat{f})^{\frac{r}{2} } \left|\cos \left(r\, \frac{\pi }{2} \right)\right|\beta_0=\\ 
& =\frac{1}{n!} \sum _{r=0}^{n}\left(\begin{array}{c} {n} \\ {r} \end{array}\right) (-x)^{n-r} \, \hat{f}^{n-\frac{r}{2} } (-y\, )^{\frac{r}{2} } \left|\cos \left(r\, \frac{\pi }{2} \right)\right|\beta_0=\\
& =(-1)^{n} \sum _{r=0}^{\lfloor\frac{n}{2} \rfloor}\frac{(n-r)!\, x^{n-2r} (-y)^{r} }{(n-2\, r)!\, r!}.
\end{split} \end{equation*} 
\end{proof}

\begin{cor}
According to such a definition we find

\begin{equation} \label{GrindEQ__15b_LacLeg} 
\partial _{x} U_{n} (x,y)=-\frac{\, \hat{f}}{(n-1)!} (-x\, \hat{f}+{}_{(-y\, \hat{f})} \hat{h})^{n-1} \beta_0.
\end{equation} 
which yields 

\begin{equation}\begin{split}\label{key}
& \partial _{x}^m U_{n} (x,y)=(-1)^m\frac{\, \hat{f}^m}{(n-m)!} (-x\, \hat{f}+{}_{(-y\, \hat{f})} \hat{h})^{n-m} \beta_0, \quad m<n,\\
&  \partial _{x}^n U_{n} (x,y)=(-1)^n n!.
\end{split}\end{equation}

\end{cor}

\subsection{ Legendre and Legendre-like Polynomials}

It is almost natural to note that the polynomials defined through the integral representation, $\forall x,y\in\mathbb{R},\forall n\in\mathbb{N}$, \cite{L.C.Andrews} 

\begin{equation}\begin{split}\label{key}
P_{n} (x,y)&=\dfrac{(-1)^{n}}{\sqrt{\pi}} \sum _{r=0}^{\lfloor\frac{n}{2}\rfloor }\frac{\Gamma \left(n-r+\frac{1}{2} \right)x^{n-2r} (-y)^{r} }{(n-2\, r)!\, r!} =\\
& = \frac{1}{n!\Gamma \left(\frac{1}{2} \right)} \int _{0}^{\infty }e^{-s}  s^{-\frac{1}{2} } H_{n} (-s\, x,-sy)ds
\end{split}\end{equation}
are generated by

\begin{equation}\label{key}
\sum _{n=0}^{\infty }t^{n}  P_{n} (x,y)=\frac{1}{\sqrt{1+x\, t+y\, t^{2} } } 
\end{equation}
and, upon replacing $x\to -2\, x,\, y=1$, are recognized as \textbf{Legendre polynomials} \cite{EWW}.\\

\noindent All the results of the previous section can be naturally transposed to this family of polynomials thus finding e.g.

\begin{lem}
\begin{equation} \label{GrindEQ__13c_LacLeg} 
\partial _{t}^{m} \frac{1}{\sqrt{1+p_{2} (x,y;t)} } =m!\dfrac{P_{m} \left(\dfrac{p_{2}^{(1)} (x,y;t)}{1+p_{2} (x,y;t)} ,\dfrac{1}{2!} \dfrac{p_{2}^{(2)} (x,y;t)}{1+p_{2} (x,y;t)} \right)}{\sqrt{1+p_{2} (x,y;t)} }, \quad \forall m\in\mathbb{N}  
\end{equation} 
or by its extension to the \textbf{lacunary} forms, namely

\begin{equation}\label{key}
\partial _{t}^{n} \frac{1}{\sqrt{1+p_{m} (x,y;t)} } =n!\dfrac{P_{n} \left(\left\{\dfrac{1}{s!} \dfrac{p_{m}^{(s)} (x,y;t)}{1+p_{m} (x,y;t)} \right\}_{s=1,...,n} \right)}{\sqrt{1+p_{m} (x,y;t)} } .
\end{equation}  
\end{lem}

\chapter{Umbral Trigonometries}\label{ChapterTR}
\numberwithin{equation}{section}
\markboth{\textsc{\chaptername~\thechapter. Umbral Trigonometries}}{}

In this Chapter we develop a new point of view to introduce families of functions, which can be identified as \textit{generalization of the ordinary trigonometric or hyperbolic functions}. They are defined using a procedure based on umbral methods, inspired to the Bessel Calculus of  Cholewinsky and Haimo \cite{Cholewinski}. We propose further extensions of the method and of the relevant concepts as well and obtain new families of integral transform allowing the framing of the previous concepts within the context of generalized Borel tranforms.\\

The original parts of the Chapter, containing their adequate bibliography, are based on the following original papers.\\

\cite{FromCircular} \textit{G. Dattoli, S. Licciardi, E. Di Palma, E. Sabia; “From circular to Bessel functions: a transition through the umbral method”, Fractal Fract, 1(1), 9, 2017}.\\

\cite{GenTrFun} \textit{G. Dattoli, S. Licciardi, E. Sabia; "Generalized Trigonometric Functions and Matrix Parameterization”; Int. J. Appl. Comput. Math, pp. 1-14, 2017}.\\

\cite{ExpMatrices} \textit{G. Dattoli, S. Licciardi, F. Nguyen, E. Sabia; “Evolution equations involving Matrices raised to non-integer exponents”; Modeling in Mathematics, Atlantis Transactions in Geometry, vol 2. pp. 31-41, J. Gielis, P. Ricci, I. Tavkhelidze (eds), Atlantis Press, Paris, Springer 2017}.\\

\cite{Airy} \textit{G. Dattoli, S. Licciardi, R.M. Pidatella; “Theory of Generalized Trigonometric functions: From Laguerre to Airy forms“; arXiv: 1702.08520, 2017, submitted for publication to Electronic Journal of Differential Equations (EJDE) 2017}.\\

$\star$ \textit{G. Dattoli, S. Licciardi, E. Sabia; "New Trigonometries", work in progress.}

\section{From Circular to Bessel Function}
The formalism we have outlined so far suggests the existence of a thread, linking different families of
special functions and polynomials. Albeit the same results can be obtained using a conventional (non-umbral) procedure,
 we stress that the protocol we have proposed is flexible, direct, straightforward and naturally suited
 for this type of problems.
In addition the method offers new possibility for the introduction of auxiliary polynomials \cite{Borel} which are all framed, by the use of purely algebraic manipulations, in a context which can be understood as that of a generalized trigonometry.

\subsection{The Umbral Version of the Trigonometric Functions}

The first step in this direction is a redefinition of the ordinary circular functions in an umbral fashion. We indeed prove that they are a manifestation of the Gauss function if we take the freedom of writing as
follows.

\begin{prop}
We consider the umbral operator \ref{Opd} of Definition \ref{defndop}, ${}_{\alpha,\;\beta}\hat{d}^{\;\kappa}\psi_0=\dfrac{\Gamma(\kappa+1)}{\Gamma(\alpha\kappa+\beta)}$, for $\alpha=2, \beta=1$, then we can recast the ordinary cosine, $\forall x\in\mathbb{R}$, as 

\begin{equation} 
\cos (x):=e^{-{}_{2,1}\hat{d}\, x^{2} } \psi_{0}.
\end{equation}
To improve the writing we indicate ${}_{2,1}\hat{d}=\hat{C}$
\begin{equation} \label{GrindEQ__11_FROM}
\cos (x)=e^{-\hat{C}\, x^{2} } \psi_{0}.
\end{equation}
\end{prop}

\begin{proof}[\textbf{Proof.}]
By expanding eq. \ref{GrindEQ__11_FROM} we recover the Taylor series expansion of the cos-function

\begin{equation} \label{GrindEQ__13_FROM}
e^{-\hat{C}\, x^{2} } \psi_{0} =\sum _{r=0}^{\infty }\dfrac{(-1)^{r} x^{2\, r} }{r!}  \hat{C}^{\;r} \psi _{0} =\sum _{r=0}^{\infty }\dfrac{(-1)^{r} x^{2\, r} }{(2\, r)!}.
\end{equation}
\end{proof}
It is easy to check the consistency of the definition \ref{GrindEQ__11_FROM} with the elementary properties of the trigonometric functions, by keeping indeed the derivative with respect to $x$, we find

\begin{cor}
\begin{equation}\begin{split}
\partial_{x} e^{-\hat{C}\, x^{2} } \psi _{0} &=-2x\hat{C}e^{-\hat{C}\, x^{2} } \psi _{0}  =-2x\sum _{r=0}^{\infty }(-1)^{r} \dfrac{(r+1)!}{(2r+2)!}  \dfrac{x^{2r} }{r!} = \\
& =-\sum _{r=0}^{\infty }(-1)^{r} \dfrac{x^{2r+1} }{(2r+1)!}  =-\sin (x).
\end{split}\end{equation}
\end{cor}
It is interesting to recover the cyclical law of the successive derivatives of the circular functions by using the present formalism. To this aim, we remind the identity \ref{GHPol} $\partial_{x}^n e^{-ax^2}=H_n (-2ax,-a)e^{-ax^2}=(-1)^n H_n (2ax,-a)e^{-ax^2}$ and, by keeping successive derivatives of both sides of eq. \ref{GrindEQ__11_FROM}, we find

\begin{cor}
\begin{equation}\begin{split}\label{GrindEQ__19_FROM}
\partial_{x}^n e^{-\hat{C} x^{2} } \psi _{0} &=(-1)^{n} H_{n} (2\hat{C}x, -\hat{C}) e^{-\hat{C}x^{2} } \psi _{0} = \\
& =(-1)^{n}n! \sum _{r=0}^{\left[\frac{n}{2} \right]}(-1)^{r} \dfrac{(2x)^{n-2r} }{(n-2r)!r!}  \cos \left(x;n-r\right)= \\
& =(-1)^{n} n!\sum _{r=0}^{\left[\frac{n}{2} \right]}(-1)^{r} \dfrac{x }{(n-2\, r)!\, r!}  \dfrac{j_{n-r-1} (x)}{(2x)^{r} }=  \\
& =\cos \left(x+n\, \dfrac{\pi }{2} \right)
\end{split}\end{equation}
\end{cor}
Within the present context, $\cos$ and $\sin$ functions are the \textit{0-th} and \textit{1-st} order cases of a more general class of functions, defined as

\begin{defn}
	We define the class of function $\forall x\in\mathbb{R}\forall n\in\mathbb{N}$
\begin{equation} \label{GrindEQ__15_FROM}
\cos (x;n):=\hat{C}^{\;n} e^{-\hat{C}\, x^{2} } \psi _{0} =\sum _{r=0}^{\infty }\frac{(-1)^{r} }{r!}  \frac{(n+r)!}{\left[2\, \left(n+r\right)\right]\, !} x^{2\, r}.
\end{equation}
They can be identified with the spherical Bessel functions \cite{S.Khan} according to the identity

\begin{equation}
\cos (x;n+1):=\frac{j_{n} (x)}{2^{n+1} x^{n} }.
\end{equation}
\end{defn}
This last result is an interesting yet unexpected outcome of our formalism, indicating how the umbral procedure, we have developed, offers a natural way of connecting \textit{\textbf{circular and Bessel type functions}}, through the use of the exponential function.

\begin{Oss}
The differential equation satisfied by the functions \ref{GrindEQ__15_FROM} can be derived from those of circular Bessel \cite{Abramovitz} according to the identities 

\begin{equation}\begin{split}
& Z_{n} (x)=\cos (x;n), \\
& j_{n-1} (x)=2^{n} x^{n-1} Z_{n} (x) ,\\
& x\, Z_{n} ''(x)+2 n\, Z_{n} '(x)+x\, Z_{n}(x) =0.
\end{split}\end{equation}
\end{Oss}
Regarding the integrals of functions \ref{GrindEQ__15_FROM}, we find

\begin{cor}
	By the use of GWI \ref{GWi} we obtain
\begin{equation}\begin{split} \label{GrindEQ__20_FROM}
\int _{-\infty }^{+\infty }\cos (x;n) \, dx&=\int _{-\infty }^{+\infty }\left[\hat{C}^{n} e^{-\hat{C} x^{2} } \psi _{0} \right] \, dx =\hat{C}^{n} \int _{-\infty }^{+\infty }e^{-\hat{C} x^{2} }  \, dx\, \psi _{0} =\\
& =\sqrt{\dfrac{\pi }{\hat{C}} } \, \hat{C}^{\,n} \psi _{0} =\sqrt{\pi }\, \hat{C}^{\,n-\frac{1}{2} }\psi _{0}=\sqrt{\pi } \dfrac{\Gamma \left(n+\dfrac{1}{2} \right)}{\Gamma (2\, n)}.
\end{split}\end{equation}
\end{cor}

Further insight into the "genesis" of the trigonometric functions can be obtained by applying again the Gauss transform method \ref{GWi} as follows.

\begin{defn}
$\forall x\in\mathbb{R}$ we define the function 
\begin{equation}\begin{split} \label{GrindEQ__22_FROM}
c_{0}^{\left(\frac{1}{2} \right)} (x)&=e^{-\hat{C}^{\frac{1}{2} } x\, } \psi _{0} =\sum _{r=0}^{\infty }\dfrac{(-x)^{r} }{r!}  \hat{C}^{\frac{r}{2} } \psi _{0} = \\
&=\sum _{r=0}^{\infty }\dfrac{\Gamma \left(\dfrac{r}{2} +1\right)}{\left(r!\right)^{2} } (-x)^{r}
\end{split}\end{equation}
a \textbf{Bessel trigonometric function} .
\end{defn}	
Then
\begin{prop}
	$\forall x\in\mathbb{R}$
\begin{equation} \begin{split}\label{GrindEQ__21_FROM}
e^{-\hat{C}\, x^{2} } \psi _{0} &=\dfrac{1}{\sqrt{\pi } } \int _{-\infty }^{+\infty }e^{-\xi ^{2} }  \left[e^{-2i\, \hat{C}^{\frac{1}{2} }  \, x\;\xi } \psi _{0} \right]d\, \xi = \\
&=\dfrac{1}{\sqrt{\pi } } \int _{-\infty }^{+\infty }e^{-\xi ^{2} }  c_{0}^{\left(\frac{1}{2} \right)} \left(2\, i\, x\, \xi \right)d\, \xi.
\end{split}\end{equation}
\end{prop}

We note that  keeping the successive derivatives of the function defined in eq. \ref{GrindEQ__22_FROM} we find

\begin{prop}
	$\forall p\in\mathbb{N}$
\begin{equation}\begin{split} \label{GrindEQ__23_FROM}
\partial_x^{p} c_{0,0}^{\left(\frac{1}{2} \right)} (x):&= \partial_x^{p} c_{0}^{\left(\frac{1}{2} \right)} (x)=\partial_x^{p} \left[e^{-\hat{C}^{\frac{1}{2} } x\, } \psi _{0} \right]=(-1)^{p} \left[\hat{C}^{\frac{p}{2} }e^{-\hat{C}^{\frac{1}{2}}x} \psi _{0} \right]= \\
&=(-1)^{p} \sum _{r=0}^{\infty }\frac{\Gamma \left(\frac{r+p}{2} +1\right)}{r!\, (r+p)!} (-x)^{r},
\end{split}\end{equation}
which can be associated with the special function 

\begin{equation} \label{GrindEQ__24_FROM}
c_{\mu ,\alpha }^{(\nu )} (x)=\sum _{r=0}^{\infty }\frac{\Gamma \left(\nu \, r+\alpha +1\right)}{r!\, \Gamma (r+\mu +1)} (-x)^{r}
\end{equation}
and therefore

\begin{equation}
\partial_x^{p} c_{0,\, 0}^{\left(\frac{1}{2} \right)} (x)=(-1)^{p}c_{p,\, \frac{p}{2} }^{\left(\frac{1}{2} \right)} (x).
\end{equation}
\end{prop}

\begin{Oss}
The origin of the functions \ref{GrindEQ__24_FROM} can easily be traced back to the Tricomi-Bessel functions \ref{TrBalfa1} $C_{\beta}(x)=\sum_{r=0}^{\infty}\dfrac{(-x)^{r}}{r!\Gamma(r+\beta+1)}$ and are recognized to be associated with an extension of the Borel transform of the functions \ref{29FROM} (see section \ref{SecBorel}), namely \cite{Borel} 

\begin{equation}\label{29FROM}
c_{\mu ,\, \alpha }^{(\nu )} (x)=\int _{0}^{\infty }e^{-s}  s^{\alpha } C_{\mu } (s^{\nu } x)\, ds.
\end{equation}
\end{Oss}

We derive, as straightforaward exercise, the associated infinite integrals, by considering two paradigmatic examples. \\

The first is rather artificial and concerns the evaluation of the following integral

\begin{exmp}
\begin{equation}\begin{split} \label{GrindEQ__27_FROM}
\int _{-\infty}^{\infty }e^{-a\, x^{2} } c_{\mu ,\alpha }^{(\nu )} (bx) dx&=\int _{-\infty}^{\infty }\left[e^{-a \, x^{2} -\hat{\chi }_{\mu ,\, \alpha }^{(\nu )} \, b\, x} \zeta _{0} \right]dx=  \\
&=\sqrt{\dfrac{\pi }{a } } e^{\frac{b^{2} }{4a} \left(\hat{\chi }_{\mu ,\, \alpha }^{(\nu )} \right)^{2} } \zeta_{0} , \\
\left(\hat{\chi }_{\mu ,\, \alpha }^{(\nu )} \right)^{r} \zeta_{0} &=\dfrac{\Gamma \left(\nu \, r+\alpha +1\right)}{\, \Gamma (r+\mu +1)}.
\end{split}\end{equation}
Accordingly, we eventually get

\begin{equation}
\int _{-\infty}^{\infty }e^{-a\, x^{2} } c_{\mu ,\alpha }^{(\nu )} (b\, x) \, dx=\sqrt{\frac{\pi }{a} } \sum _{r=0}^{\infty }\dfrac{b^{2r}}{(4a)^{r}r!}\dfrac{\Gamma \left(2\,r\,\nu \, +\alpha +1\right)}{\Gamma \left(2\, r+\mu +1 \right)}.
\end{equation}
\end{exmp}
A further and more familiar example, na\"{\i}ve consequence of this procedure, is the evaluation of the \textit{Fresnel} integral
\cite{Abramovitz}

\begin{equation} \label{GrindEQ__29_FROM}
C(x)=\int _{x}^{+\infty }\cos \left(\xi ^{2} \right) \, d\, \xi, \quad \forall x\in\mathbb{R},
\end{equation}
for $x=0$.  The use of the identities \ref{GrindEQ__11_FROM} yields

\begin{exmp}
	By applying a variable change, 
\begin{equation}\begin{split}
 C(0)&=\int _{0}^{+\infty }\left[e^{-\hat{C}\, x^{4} } \psi _{0} \right]\,  dx=\left(\frac{1}{4} \int _{0}^{\infty }e^{-y}  y^{\frac{1}{4} -1} dy\right)\, \hat{C}^{-\frac{1}{4} } \psi _{0} = \\
& =\frac{1}{4} \frac{\Gamma \left(\dfrac{1}{4} \right)\, \Gamma \left(\dfrac{3}{4} \right)}{\Gamma \left(\dfrac{1}{2} \right)} =\dfrac{1}{2} \sqrt{\dfrac{\pi }{2} }.
\end{split}\end{equation}
\end{exmp}
The previous results have emerged in quite a natural fashion from our formalism. Other means, of more conventional nature, can be applied, albeit with more computational effort.

\subsection{Laguerre Polynomials and Trigonometric Function}

In this section we provide a link betweeen trigonometric-like functions and Laguerre polynomials. Before proceeding in this direction we need the definition of further auxiliary tools. The strategy we follow is that of introducing an appropriate family of polynomials providing a bridge between Laguerre and trigonometric functions. To this aim, we replace in eq. \ref{LagDefPhi}  $\hat{c}$ with $\hat{C}$, thus defining a new family of polynomials suited for our purposes. 

\begin{defn}
	$\forall x,y\in\mathbb{R}, \forall n\in\mathbb{N}$, we introduce $\lambda _{n} (x,\, y)$ polynomials as
	
\begin{equation} \begin{split}\label{GrindEQ__32_FROM}
\lambda _{n} (x,\, y)&:=\left[(y-\hat{C}\, x)^{n} \psi _{0} \right]=\sum _{s=0}^{n}  \binom{n}{s}(-1)^{s}\; y^{n-s} x^{s} \hat{C}^{\,s} \psi _{0} =\\
& =\sum _{s=0}^{n}\dfrac{(-1)^{s} s!}{(2\, s)!}  \binom{n}{s}\ y^{n-s} x^{s} =n!\sum _{s=0}^{n}\dfrac{(-1)^{s} }{(2\, s)!(n-s)!}  \, y^{n-s} x^{s}.
\end{split}\end{equation}
They are umbrally equivalent to $L_{n} (x,\, y)$.
\end{defn}

 A straightforward application of our procedure yields 
 
 \begin{prop}
The $\lambda _{n} (x,\, y)$ generating functions, $\forall t\in\mathbb{R}:\mid t\mid <\frac{1}{y}$, 
\begin{equation} \begin{split}\label{36FROM}
\sum _{n=0}^{\infty }t^{n}  \lambda _{n} (x,\, y)&=\sum _{n=0}^{\infty }t^{n}  \left[(y-\hat{C}\, x)^{n} \psi _{0} \right]=\dfrac{1}{(1-y\, t)\, \left[1+\dfrac{\hat{C}\, x\, t}{1-yt} \right]} \psi _{0} = \\
&=\frac{1}{1-y\, t} \left[\sum _{r=0}^{\infty }\left(-\dfrac{\hat{C} x t}{1-yt} \right)^{r} \psi _{0}  \right]=\dfrac{1}{1-yt} e_{0} \left(\dfrac{x\, t}{1-y\, t} \right), \\
 e_{0} (x)&=\sum _{r=0}^{\infty }(-1)^{r} \dfrac{r!}{(2\, r)!}  x^{r}
\end{split}\end{equation}
and\footnote{The generating function \ref{cospiFROM}  indicates that the  $\lambda _{n} (x,\, y)$  belong to the family of App\'el polynomials in the $y$ variable \cite{Appell}, this is a characteristic shared with the   $L_{n} (x,\, y)$.} $\forall t\in\mathbb{R} $

\begin{equation}\label{cospiFROM}
\sum _{n=0}^{\infty }\frac{t^{n} }{n!}  \lambda _{n} (x,\, y)=e^{y\, t} \left[e^{-x\, \hat{C}\, t} \psi _{0} \right]=e^{yt} \cos (\sqrt{xt} ).
\end{equation}
\end{prop}

Furthermore, the  $\lambda_{n}$ polynomials \ref{GrindEQ__32_FROM} are easily shown to satisfy the recurrences \cite{FromCircular}

\begin{propert}
\begin{equation} \begin{split}
&\partial _{y} \lambda _{n} (x,\, y)=n\, \lambda _{n-1} (x,y), \\
&\hat{\Delta }\, \lambda _{n} (x,\, y)=n\, \lambda _{n-1} (x,\, y), \\
&\hat{\Delta }=-4 x^{\frac{1}{2} } \partial _{x} x^{\frac{1}{2} } \partial _{x} =-2 \left(1 +2\,x\, \partial _{x} \right)\, \partial _{x} \end{split}\end{equation}
which, once combined, yield the differential equation

\begin{equation} \label{39FROM}
\left\lbrace  \begin{array}{l}
\partial _{y} \lambda _{n} (x,\, y)=\hat{\Delta }\, \lambda _{n} (x,\, y) \\[1.6ex]
\lambda _{n} (x,\, 0)=(-1)^{n} \dfrac{n!}{(2\, n)!} x^{n}.
\end{array}\right.
\end{equation}
It, accordingly, allows the following operational definition

\begin{equation}
\lambda _{n} (x,\, y)=e^{y\, \hat{\Delta }} \lambda _{n} (x,\, 0),
\end{equation}
which can be further handled to get

\begin{equation}\label{41FROM}
e^{y\, \hat{\Delta }} e_{0} (x)=\frac{1}{1-y} e_{0} \left(\frac{x}{1-y} \right).
\end{equation}
\end{propert}

Eqs. \ref{36FROM} - \ref{41FROM} are closely symilar to analogous identities satisfied by the Laguerre polynomials \cite{L.C.Andrews,K.A.Penson}. In particular eq. \ref{39FROM} is a kind of heat equation involving the differential operator $ \hat{\Delta}$. To complete the analogy with Laguerre polynomials family \ref{diffAHP} we introduce the \textbf{associated $\lambda$-polynomials} specified by

\begin{defn}
	We define  the associated $\lambda$-polynomials as
\begin{equation} \begin{split}\label{GrindEQ__36_FROM}
\lambda _{n}^{(\nu )} (x,\, y)&:=\hat{C}^{\nu } \, (y-\hat{C}\, x)^{n} \psi _{0} = \\
& =\sum _{s=0}^{n}(-1)^{s}  \binom{n}{s}\, y^{n-s} x^{s} \hat{C}^{\nu +s} \psi _{0} =\\
&=n!\sum _{s=0}^{n}(-1)^{s}  \dfrac{\Gamma (\nu +s+1)}{s!\, (n-s)!\Gamma (2\, (\nu +s)+1)} \, y^{n-s} x^{s}.
\end{split}\end{equation}
\end{defn}

\begin{cor}
The relevant generating function, by applying eq. \ref{GrindEQ__15_FROM}, writes

\begin{equation}
\sum _{n=0}^{\infty }\frac{t^{n} }{n!}  \lambda _{n}^{(\nu )} (x,\, y)=e^{y\, t} \cos\left(\sqrt{x\, t} ;\nu \right).
\end{equation}
\end{cor}

Before proceeding further we note that the $\lambda$-polynomials are linked to other family of polynomials playing an important role in analysis

\begin{Oss} 
	The polynomials
\begin{equation} \label{GrindEQ__38_FROM}
h_{n} (x)=\sum _{k=0}^{n}\left(\begin{array}{c} {n+k} \\ {n-k} \end{array}\right)\, \left(-x\right)^{k}, \quad \forall x\in\mathbb{R}, \forall n\in\mathbb{N},
\end{equation}
are orthogonal polynomials with weight function

\begin{equation} \label{GrindEQ__39_FROM}
\rho (x)=\frac{1}{2\, \pi } \sqrt{\frac{4-x}{x} },
\end{equation}
and are involved other in the theory of \textbf{Catalan numbers} and of the solution of the \emph{Hausdorff moment problem} \cite{T.S.Chihara}.\\

The  $h_{n} (x)$ can be readily written in terms of the integral transform of  polynomials $\lambda _{n} (x,\, y)$ according to the identity

\begin{equation} \label{GrindEQ__40_FROM}
h_{n} (x)=\frac{1}{n!} \int _{0}^{\infty }e^{-\xi }\xi^{\,n}  \lambda _{n} (x\,\xi ,\, 1)\,d\xi.
\end{equation}
\end{Oss}

Before concluding this section, we introduce the following \textbf{$\lambda$ based Bessel functions}, defined by the generating function
\begin{equation}
\sum _{n=-\infty }^{+\infty }t^{n}  {}_{\lambda } J_{n} (x,\, y)=e^{\frac{(y-\hat{C}\, \, x)}{2} \left(t-\frac{1}{t} \right)} \psi _{0} =e^{\frac{y}{2} \left(t-\frac{1}{t} \right)} \cos \left(\sqrt{\frac{x}{2} \left(t-\frac{1}{t} \right)}\right),
\end{equation}
which are characterized by the Anger type identity \cite{L.C.Andrews}

\begin{equation}\label{key}
\sum _{n=-\infty }^{+\infty }e^{in\theta}\left( {}_{\lambda } J_{n} (x,\, y)\right) =e^{iy\sin(\theta)}\cos\left(\sqrt{ix\sin(\theta)} \right) ,
\end{equation}
which can be exploited in problems involving non linear oscillations, ruled by differential equations of the type

\begin{equation}\label{key}
Z\;Z^{''}+{Z^{'}}^2+\dfrac{Z}{2}=0.
\end{equation}

\begin{Oss}
By keeping $y=0$, we find ${}_{\lambda } J_{n} (x,\, 0)$, which can be cast in the following series form

\begin{equation}
{}_{\lambda } J_{n} (x,\, 0)=\sum _{r=0 }^{+\infty }\dfrac{(-1)^{3r+n}(2r+n)!}{r!(r+n)![2(2r+n)]!}\left(\dfrac{x}{2} \right)^{2r+n}.
\end{equation}

In the case of $n=0$, abusing our umbral notation and by recalling the rule of the Gaussian successive derivatives, write 

\begin{equation} \label{GrindEQ__49_FROM}
\partial _{x}^{n} {}_{\lambda } J_{0} (x,\, 0)=(-1)^{n} H_{n} \left(\hat{c}\,\hat{C}^{2}\,  \frac{x}{2} ,\, -\frac{\hat{c}\, \hat{C}^{2} }{4} \right)\, e^{-\hat{c}\, \left(-\hat{C} \frac{x}{2} \right)^{2} } \psi _{0} \varphi_{0}.
\end{equation}
The Bessel and circular umbral operators $(\hat{c}$ and $\hat{C})$ act on the "vacua" $(\psi_{0}-\varphi_{0})$.
\end{Oss}

The last examples, regarding artificial construction of Bessel type functions, have been aimed at further stressing that, even though complicated in their explicit representation in terms of series, the operational method greatly simplifies the study of the relevant properties.

\section{From Laguerre to Airy Forms}\label{LagAiry}

The cylindrical Bessel are generalizations of the trigonometric functions, while the associated modified forms are an extension of the relevant hyperbolic counterparts \cite{L.C.Andrews}. Such an academic identification is non-particularly deep and might be useful for pedagogical reasons or as a guiding element to study their properties as e.g. those relevant to the asymptotic forms. We must however underline that Bessel and trigonometric/hyperbolic functions share some resemblances only, but they do not display any  full correspondence. The search for functions which are "true" generalizations of the trigonometric \textit{(t-)} or hyperbolic \textit{(h-)} forms is however recurrent in the mathematical literature.  The attempts in this direction can be ascribed to different strategies, roughly speaking the geometrical \cite{Ferrari} and the analytical \cite{D.E.Edmunds} point of views. The first is based on definitions extending to higher powers the  Pythagorean identity of ordinary trigonometric functions, such a program identifies new trigonometries, with their own geometrical interpretation on elliptic curves and with different numbers playing the role of $\pi$ \cite{Ferrari}. The second invokes the analogy with series expansions, differential equations and the theory of special functions. The generalized \textit{t-h} functions, defined within these two contexts, are different. In particular those belonging to the geometric strategy can be recognized as elliptic functions, including Jacobi and Weierstrass forms. \\

 In this section we develop a systematic procedure within the framework of the analytical point of view.\\

We look for "true" generalizations, in the sense that the functions we define allow a one to one mapping onto the properties of the elementary \textit{t-h} functions, like addition or duplication theorems. To this aim we exploit the methods provided so far about the  understanding of Bessel functions as umbral manifestation of Gauss or of exponential functions \cite{Babusci}. These conceptual tools, as well as the ideas developed by Cholewinsky and Reneke in ref. \cite{FMCholewinski}, provide the elements underlying the formalism of this section, aimed at exploring in depth the identification of trigonometric functions associated with  Bessel functions, by getting the proper algebraic environment to establish the relevant properties.\\

Our starting point is the particular following partial differential equation in which we use the results discussed in Chapter \ref{Chapter2}.

\begin{exmp}
	We consider, $\forall x,y\in\mathbb{R}$, the initial conditions problem

\begin{equation}\label{GrindEQ__1_Airy} 
\left\lbrace  \begin{array}{l}
 {}_{l} \partial _{x} F(x,y)={}_{l}  \partial _{y} F(x,y) \\[1.6ex] 
 F(x,\, 0)=x^{\;n}  \\[1.6ex] 
 F(0,y)=y^{\;n} , 
 \end{array}\right.\end{equation} 
where ${}_{l} \partial _{\xi } $ is the  $l$-derivative \ref{prove}.\\

\noindent It is easily checked that the solution of eq. \ref{GrindEQ__1_Airy}, when $x,y>0$, can be cast in the form \cite{Airy}

\begin{equation} \label{GrindEQ__2_Airy} 
\varLambda_{n} (x,y)=\sum _{r=0}^{n}\binom{n}{r}^{2}  x^{\;n-r} y^{\;r} , 
\end{equation} 
where $\varLambda_{n} (x,y)$ is an example of hybrid polynomial, introduced in ref. \cite{Lorenzutta}.
\end{exmp}

 For reasons which will be clear in the following, we introduce the following notation, borrowed from ref. \cite{D.E.Edmunds}.
 
 \begin{defn}
 	We define the composition rule $\forall x,y\in\mathbb{R}, \forall n\in\mathbb{N}$
\begin{equation} \label{GrindEQ__3_Airy} 
\varLambda_{n} (x,y)=\sum _{r=0}^{n}\binom{n}{r}^{2}  x^{\;n-r} y^{\;r} :=(x\oplus _{l} y)^{\;n}  
\end{equation} 
the \textbf{Laguerre bynomial sum} (lbs).\\

\noindent  It is evident that such a notion is an extension of the Newton Bynomial, which can be generated by the action of the shifting exponential operator on an ordinary monomial, namely

\begin{equation}
e^{\;y\, \partial _{x} } \, x^{\;n} =\sum _{r=0}^{\infty }\dfrac{y^{\;r} }{r!} \partial _{x}^{\;r}  x^{\;n} =\sum _{r=0}^{n}\dfrac{y^{\;r} }{r!} \dfrac{n!}{(n-r)!}  x^{\;n-r} =(x+y)^{\;n}.
\end{equation}
\end{defn}            

\begin{cor}
An analogous rule for the generation of lbs can be achieved by replacing the exponential function with the $l$-exponential \ref{GrindEQ__5_Airy} $_{l} e(\eta )=\sum _{r=0}^{\infty }\frac{\eta ^{\;r} }{\left(r!\right)^{2} }$ (it is a $0$-order Bessel-Tricomi function and satisfies the l-eigenvalue equation \ref{GrindEQ__9_Airy} ${}_{l} \partial _{x} \left({}_{l} e(\lambda \, x)\right)=\lambda \left({}_{l} e(\lambda \, x)\right) $) and the ordinary derivative with the $l$-derivative, satisfying the identity \ref{LagDers} ${}_{l} \partial _{\eta }^{\;n} =\partial _{\eta }^{\;n} \eta ^{\;n} \, \partial _{\eta }^{\;n} $. Accordingly we find

\begin{equation}\begin{split} \label{GrindEQ__7_Airy} 
 {}_{l} e(y\, {}_{l} \partial _{x} )\, x^{n} &=\sum _{r=0}^{\infty }\frac{y^{r} }{\left(r!\right)^{2} } \partial _{x}^{r}  x^{r} \partial _{x}^{r} \, x^{n}  =\sum _{r=0}^{n}\frac{y^{r} }{\left(r!\right)^{2} } \frac{\left(n!\right)^{2} }{\left[(n-r)!\right]^{2} }  x^{n-r} =\\
& = \sum _{r=0}^{n}\binom{n}{r}^2 x^{n-r}y^r=(x\oplus _{l} y)^{n} .
\end{split}\end{equation} 
\end{cor}

 According to the previous identities we can also state that

\begin{propert}
\begin{equation}\begin{split} \label{GrindEQ__9b_Airy} 
& {}_{l}e(y\, {}_{l} \partial _{x} )\, {}_{l}e(x)={}_{l}e(y)\, {}_{l}e(x), \\[1.1ex] 
& {}_{l}e(y\, {}_{l} \partial _{x} )\, {}_{l}e(x)={}_{l}e(x\oplus _{l} y).
\end{split}\end{equation} 
The above identities allow the derivation of the following \textbf{"semi-group" property} of the $l$-exponential

\begin{equation}\label{lExpProp}
{}_{l} e(y)\, {}_{l} e(x)={}_{l} e(x\oplus _{l} y) .                     
\end{equation} 
\end{propert}

\begin{defn}
In full analogy with the ordinary Euler formulae we introduce the l-trigonometric (l-t) functions through the identity

\begin{equation} \label{GrindEQ__11_Airy} 
\, {}_{l} e(i\, x)={}_{l} c(x)+i\, {}_{l} s(x), 
\end{equation} 
where l-t \textit{cosine} and \textit{sine} functions are specified by the series\footnote{ The \textit{l-h} functions are defined by the corresponding series expansion   $\begin{array}{l} {{}_{l} ch(x)=\sum _{r=0}^{\infty }\frac{x^{2r} }{\left[(2\, r)!\right]^{2} }  ,} \;\; {{}_{l} sh(x)=\sum _{r=0}^{\infty }\frac{x^{2r+1} }{\left[(2\, r+1)!\right]^{2} }  } \end{array}$. }

\begin{equation}\begin{split}\label{GrindEQ__12_Airy}
& {}_{l} c(x)=\sum _{r=0}^{\infty }\frac{(-1)^{r} x^{2r} }{(2\, r)!^{2} }  , \\ 
& {}_{l} s(x)=\sum _{r=0}^{\infty }\frac{(-1)^{r} x^{2r+1} }{(2\, r+1)!^{2} } . 
\end{split}\end{equation}
\end{defn}          

 Fig. \ref{fig1Airy} provides the plot of ${}_{l} s(x)$ vs ${}_{l} c(x)$ in a Lissajous-like diagram.\\ 
     
  \begin{figure}[htp]
  	\centering
  	\includegraphics[width=.5\textwidth]{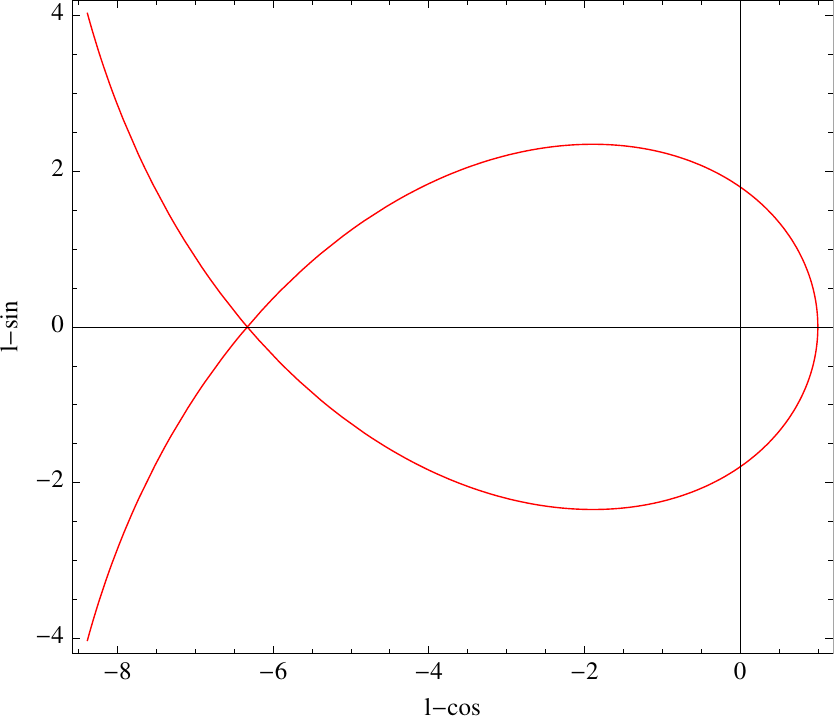}
  	\caption{Fish-like Lissajous diagram of \textit{l-t} functions, $ {}_{l} s(x) $ vs    $ {}_{l} c(x) $.}\label{fig1Airy}
  \end{figure}        
          
It is easily checked that \cite{Airy}

\begin{propert}
	$\forall \alpha\in\mathbb{R}$, l-t \textit{cosine} and \textit{sine} satisfy the identities
\begin{equation}\begin{split} \label{GrindEQ__13_Airy} 
& {}_{l} \partial _{x} \left[{}_{l} c(\alpha \, x)\right]=-\alpha \, {}_{l} s(\alpha \, x), \\[1.1ex] 
& {}_{l} \partial _{x} \left[{}_{l} s(\alpha \, x)\right]=\alpha \, {}_{l} c(\alpha \, x)
\end{split} \end{equation} 
and therefore the  \textbf{"harmonic" equation}

\begin{equation}\begin{split}
& \left({}_{l} \partial _{x} \right)^{2} \left[{}_{l} c(\alpha \, x)\right]=-\alpha ^{2} {}_{l} c(\alpha \, x), \\[1.1ex]
& \left({}_{l} \partial _{x} \right)^{2} \left[{}_{l} s(\alpha \, x)\right]=-\alpha ^{2} {}_{l} s(\alpha \, x).
\end{split}\end{equation}
\end{propert}

It is worth noting that 

\begin{prop}
Laguerre  and ordinary trigonometric functions are linked by the Borel type transforms \ref{fBxInt} \cite{Airy}

\begin{equation}\begin{split} \label{GrindEQ__70b_Airy} 
& \int _{0}^{\infty }e^{-t}  {}_{l} c(xt)dt=\cos (x), \\ 
& \int _{0}^{\infty }e^{-t}  {}_{l} s(xt)dt=\sin (x).
\end{split} \end{equation} 
The use of the dilatation operator identity \ref{tderf}-\ref{Brecast}  yields the following identifications

\begin{equation}\begin{split} \label{GrindEQ__14_Airy} 
& {}_{l} c(x)=\left(\Gamma (x\, \partial _{x} +1)\right)^{-1} \cos (x), \\[1.1ex] 
& {}_{l} s(x)=\left(\Gamma (x\, \partial _{x} +1)\right)^{-1} \sin (x).
\end{split} \end{equation} 
We note that

\begin{equation}\begin{split}
 {}_{l} e(i\, x)&={}_{l}c(x)+i\, {}_{l}s(x)= \\ 
& =\left(\Gamma (x\, \partial _{x} +1)\right)^{-1} e^{i\, x}.
\end{split} \end{equation} 	
\end{prop}

The use of the properties \ref{lExpProp} allows the  derivation of the following addition theorems for the functions in eqs. \ref{GrindEQ__12_Airy} .

\begin{prop}
	$\forall x,y\in\mathbb{R}$, the composition rules of l-t cosine and sine are stated as 
\begin{equation}\begin{split} \label{GrindEQ__15_Airy} 
& {}_{l} c(x\oplus_{l} y)={}_{l} c(x){}_{l} c(y)-{}_{l} s(x)\, {}_{l} s(y), \\ 
& {}_{l} s(x\oplus_{l} y)={}_{l} c(x){}_{l} s(y)+{}_{l} s(x)\, {}_{l} c(y).
\end{split}\end{equation} 
\end{prop}

\begin{proof}[\textbf{Proof.}]
The proof of the second identity is given by noting that

\begin{equation}\begin{split} \label{GrindEQ__16_Airy} 
& {}_{l} e(i\, x)\, e(i\, y)=\left[{}_{l} c(x)+i\, {}_{l} s(x)\right]\, \left[{}_{l} c(y)+i\, {}_{l} s(y)\right]= \\ 
& =\left[{}_{l} c(x){}_{l} c(y)-{}_{l} s(x)\, {}_{l} s(y)\right]+i\, \left[{}_{l} c(x){}_{l} s(y)+{}_{l} s(x)\, {}_{l} c(y)\right]
\end{split}\end{equation} 
and since

\begin{equation} \label{GrindEQ__17_Airy} 
{}_{l} e(i\, x)\, {}_{l} e(i\, y)={}_{l} e(i\left(x\oplus_{l} y\right))={}_{l} c(x\oplus_{l} y)+i\, {}_{l} s(x\oplus_{l} y), 
\end{equation} 
we can equate real and imaginary parts to infer the identities \ref{GrindEQ__15_Airy}. \\
The first equation is analogous.
\end{proof}

It is evident that, according to the procedure we have proposed, the properties of ordinary trigonometric functions are extended to their $l$-counterparts, provided that the ordinary sum is replaced by the composition rule specified in eq. \ref{GrindEQ__3_Airy}.

\begin{prop}
  The formalism allows the derivation of the corresponding \textbf{duplication formulae}, which can be stated by defining the following product rule 

\begin{equation}
\left( x\oplus_{l} x\right) ^{n}=\dfrac{(2n)!}{(n!)^{2}}x^{n}:=\left(2\otimes_{l} x \right) ^{n},
\end{equation}                      
which, along with eq. \ref{GrindEQ__15_Airy}, yields 

\begin{equation}\begin{split}
& {}_{l} c(2\otimes_{l} x)=\left({}_{l} c(x)\right)^{2} -\left({}_{l} s(x)\right)^{2} , \\[1.1ex]
& {}_{l} s(2\otimes_{l} x)=2\;{}_{l} c(x){}_{l} s(x).
\end{split}\end{equation}
Furthermore the sum can be iterated as

\begin{equation}\begin{split}
& \left( x\oplus_{l} (x\oplus_{l} x)\right) ^{n}= \left( 3\otimes_{l} x\right) ^{n}, \\
& \left( x\oplus_{l} (..._{l}\; (x\oplus_{l} x))\right) ^{k}=\left( n\otimes_{l} x\right) ^{k} \end{split}\end{equation}  
and, accordingly, we can state the following extension of the De Moivre formulae

\begin{equation}
\left[{}_{l} c(x)+i\, {}_{l} s(x)\right]^{n} ={}_{l} c\left(n\otimes_{l} x\right)+i\, {}_{l} s\left(n\otimes_{l} x\right).
\end{equation}
\end{prop}

 It is worth stressing some of the properties of the composition rule \ref{GrindEQ__3_Airy} .

\begin{propert}
\begin{equation}\label{trivialAiry}
\begin{split}
& a) \;\; (x\oplus _{l} y)^{n} =(y\oplus _{l} x)^{n}, \\                  
& b) \;\; (1\oplus _{l} 1)^{n} =\dfrac{(2\, n)!}{\left(n!\right)^{2} } ,\\                  
& c) \;\; (1\oplus _{l} (-1))^{n}=\dfrac{i^{\;n}\;n! }{\left( \left(\frac{n}{2}   \right)!\right) ^{2} }\dfrac{\left( 1+(-1)^{n}\right)}{2}= \left\lbrace \begin{array}{ll} 0,&  n =2k+1,\; k\in \mathbb{N},\\[1.1ex]  \dfrac{i^{\;n}\;n! }{\left( \left(\frac{n}{2}   \right)!\right) ^{2} }, & n=2k,\;\;\; k\in \mathbb{N} , \end{array}\right. \\
& d) \;\; (i\oplus _{l} (-i))^{n}=\dfrac{(-1)^{n}\;n! }{\left( \left(\frac{n}{2}   \right)!\right) ^{2} }\dfrac{\left( 1+(-1)^{n}\right)}{2}= \left\lbrace \begin{array}{ll} =0,&  n =2k+1, k\in \mathbb{N},\\[1.1ex] = \dfrac{(-1)^{n}\;n! }{\left( \left(\frac{n}{2}   \right)!\right) ^{2} }, & n=2k,\;\;\; k\in \mathbb{N}.  \end{array}\right.
\end{split}
\end{equation}
It is furthermore worth noting that 

\begin{equation}\label{NepAiry}
e) \;\;\;\; {}_{l}e( x)=\lim_{n\rightarrow\infty}\left(1\oplus _{l} \left( \dfrac{x}{n^{2}}\right)\right)^{n} ,  
\end{equation}
which provides the quantity 

\begin{equation}\label{key}
{}_l e:= {}_l e(1)=\left(1\oplus _{l} \left( \dfrac{1}{n^{2}}\right)\right)^{n}=2.279585302336067\dots
\end{equation}
which is a kind of \textbf{Laguerre-Napier number} ${}_l e$, presumably trascendent.
\end{propert}

\begin{Oss}
The previous identities, albeit trivial, are important to appreciate the structural differences with respect to their ordinary counterparts, for example eq. \ref{trivialAiry} implies that

\begin{equation} \label{GrindEQ__21_Airy} 
{}_{l}e(i\, x)\, {}_{l} e(-i\, x)\ne 1 
\end{equation} 
and therefore that

\begin{equation}
\left[{}_{l} c(x)\right]^{2} +\left[{}_{l} s(x)\right]^{2} \ne 1  .
\end{equation} 
\end{Oss}

\begin{Oss}
We observe that correspondent results can be obtained by starting from different special functions. If we consider e.g. the Mittag-Leffler function \ref{ML} for $\beta=1$, $E_{\alpha,1}=\sum_{r=0}^\infty \dfrac{x^r}{\Gamma(\alpha r+1)}$, we notice that 

\begin{equation} \label{eq8CML} 
E_{\alpha,1 } (x+y)\ne E_{\alpha ,1} (x)\, E_{\alpha,1 } (y) 
\end{equation} 	
and therefore, to realize the semigroup properties, we extend the Newton bynomial as
	
	\begin{equation} \begin{split}\label{eq9CML} 
	& (x\oplus _{ml_{\alpha}} y)^{n} :=\sum _{r=0}^{n}\binom{n}{r} _{\alpha }\! x^{n-r} y^{r} , \\[1.2ex]
	& \binom{n}{r} _{\alpha } :=\dfrac{\Gamma (\alpha \, n+1)}{\Gamma \left(\alpha \, (n-r)+1\right)\Gamma (\alpha r+1)} ,
	\end{split} \end{equation} 
	which allows the conclusion
	
	\begin{equation} \label{eq10CML} 
	E_{\alpha,1 } (x\oplus _{ml_{\alpha}} y)=E_{\alpha,1 } (x)\, E_{\alpha,1 } (y).
	\end{equation} 
	The associated sin and cos-like functions defined by
	
	\begin{equation}\begin{split} \label{eq11CML} 
	& C_{\alpha,1 } (x)=\frac{E_{\alpha,1 } (ix)+E_{\alpha,1 } (-ix)}{2} ,\\ 
	& S_{\alpha,1 } (x)=\frac{E_{\alpha,1 } (ix)-E_{\alpha,1 } (-ix)}{2\, i} 
	\end{split} \end{equation} 
	also implying that
	
	\begin{equation}
	E_{\alpha,1 }(ix)=C_{\alpha,1 } (x)+i\;S_{\alpha,1 } (x)
	\end{equation} 
	and they are characterized by the addition formulae
	
	\begin{equation}\begin{split} \label{eq12CML} 
	& C_{\alpha,1 } (x\oplus _{ml_{\alpha}} y)=C_{\alpha,1 } (x)\, C_{\alpha,1 } (y)-S_{\alpha,1 } (x)\, S_{\alpha,1 } (y), \\[1.2ex]
	& S_{\alpha,1 } (x\oplus _{ml_{\alpha}} y)=S_{\alpha,1 } (x)\, C_{\alpha,1 } (y)+C_{\alpha ,1} (x)\, S_{\alpha,1 } (y),
	\end{split}\end{equation} 
	resembling those of their circular counterpart. It is furthermore worth noting that, as we noted yet in Lemma \ref{soleig}, if $\alpha=n \in \mathbb{N}$, the ML function satisfies the eigenvalue equation
	
	\begin{equation} \label{eq13CML} 
	n ^{n } \left(x^{\frac{n -1}{n } } \frac{d}{dx} \right)^{n }  E_{n,1 } (\lambda x)=\lambda E_{n,1 } (\lambda x). 
	\end{equation} 
	It is therefore evident that by introducing the ML derivative operator
	
	\begin{equation} \label{eq14CML} 
	{}_{ml} \hat{D}_{x} =n ^{n } \left(x^{1-\frac{1}{n } } \frac{d}{dx} \right)^{n  }  ,
	\end{equation} 
	we find
	
	\begin{equation} \label{eq15CML} 
	E_{n,1 } \left(y\; {}_{ml} \hat{D}_{x} \right)E_{n,1 } (x)=E_{n,1 } \left(x\oplus _{ml_{\alpha}} y\right) .
	\end{equation} 
	Accordingly, the operator $E_{n,1 } \left(y\; {}_{ml} \hat{D}_{x} \right)$ is a shift operator in the sense that it provides a shift of the argument of the ML function according to the composition rule established in eq. \ref{eq9CML}.
\end{Oss}

In the forthcoming sections we will go deeper into the theory of these families of functions and we will be able to better appreciate the similitudes and the differences with the ordinary forms.

\subsection{Generalized Trigonometric Functions, Ordinary and Higher Order Bessel Functions}

The Bessel functions are characterized by a continuous variable and by a real or complex index. A fairly natural extension of the function defined by eqs. \ref{GrindEQ__12_Airy} is therefore provided by the \textit{l-t} functions, associated with the $\alpha$-order Bessel like function \ref{GrindEQ__21_OpOrd}, for $m=1$, then  

\begin{cor}
We get the function
\begin{equation} \label{GrindEQ__23_Airy} 
_{l} e_{\alpha } (\eta )=\sum _{r=0}^{\infty }\frac{\eta ^{\; r} }{r!\, \Gamma (r+\alpha +1)} ,  
\end{equation} 
which is an eigenfunction of the operator \cite{Airy}

\begin{equation}
 {}_{(\alpha ,\, l)} \partial _{x} =\partial _{x} x\, \partial _{x} +\alpha \, \partial _{x} = x^{-\alpha } \partial _{x} x^{\;\alpha +1} \partial _{x} .
\end{equation}    

We can now proceed as in the previous section, by noting that the polynomials
\begin{equation}\begin{split} \label{GrindEQ__25_Airy} 
\varLambda_{n} (x,y;\alpha )&:=\sum _{r=0}^{n}\binom{n}{r}\; \dfrac{\Gamma (n+\alpha +1)}{\Gamma (n-r+\alpha +1)\, \Gamma (r+\alpha +1)\, }  x^{\;n-r} y^{\;r} = \\ 
&=(x\oplus _{(\alpha ,\, l)} y)^{\;n} 
\end{split} \end{equation} 
are solutions of the equation

\begin{equation}
{}_{(\alpha ,\, l)} \partial _{x} l_{n} (x,y;\alpha )={}_{(\alpha ,\, l)} \partial _{y} l_{n} (x,y;\alpha ) .
\end{equation}                  
We further define the composition rule

\begin{equation}\label{compruleAiry}
{}_{l} e_{\alpha } (y\, {}_{\left(\alpha ,\, l\right)} \partial _{x} )\, x^{n} =(x\oplus _{(\alpha ,\;l)} y)^{n}
\end{equation}                
and prove that

\begin{equation}\begin{split}\label{GrindEQ__30b_Airy} 
&{}_{l} e_{\alpha } (y\, {}_{\left(\alpha ,\, l\right)} \partial _{x} )\, {}_{l} e_{\alpha } (x)=\, {}_{l} e_{\alpha } (x\oplus _{(\alpha ,\;l)} y), \\[1.1ex] 
& {}_{l} e_{\alpha } (y\, )\, {}_{l} e_{\alpha } (x)=\, {}_{l} e_{\alpha } (x\oplus _{(\alpha ,\;l)} y).
\end{split}\end{equation}             
Finally, by a slight extension of the discussion of the introductory section, we define the \textit{l-t} functions of $\alpha$-order as

\begin{equation}\begin{split}\label{GrindEQ__29_Airy} 
&{}_{\left(\alpha ,\, l\right)} c(x)=\sum _{r=0}^{\infty }\dfrac{(-1)^{r} x^{2r} }{\left[(2\, r)!\right]\Gamma (\alpha +2r+1)}  , \\ 
& {}_{(\alpha ,\, l)} s(x)=\sum _{r=0}^{\infty }\dfrac{(-1)^{r} x^{2r+1} }{\left[(2\, r+1)!\right]\Gamma (\alpha +2r+2)} ,
\end{split}\end{equation}
which are shown to satisfy the addition theorems \ref{GrindEQ__15_Airy} for the composition rule \ref{compruleAiry} .
\end{cor}

\begin{cor}
The procedure can be further extended by the use of Humbert Bessel like functions, which are defined by the series \cite{Cocolicchio}

\begin{equation} \label{GrindEQ__32b_Airy} 
_{l} e_{\alpha ,\;\beta } (\eta )=\sum _{r=0}^{\infty }\frac{\eta ^{r} }{r!\, \Gamma (r+\alpha +1)\, \Gamma (r+\beta +1)}   .
\end{equation} 
They satisfy the differential equation

\begin{equation} \label{GrindEQ__31_Airy} 
\left[ \partial _{\eta } \left(\alpha +\eta \, \partial _{\eta } \right)\left(\beta +\eta \, \partial _{\eta } \right)\right]  {}_{l} e_{\alpha ,\beta } (\eta )= {}_{l} e_{\alpha ,\;\beta } (\eta ) 
\end{equation} 
and are therefore eigenfunctions of the operator

\begin{equation} \label{GrindEQ__32_Airy} 
{}_{\left(\alpha ,\;\beta ,\, l\right)} \partial _{\eta } =\partial _{\eta } \eta \, \partial _{\eta } \eta \, \partial _{\eta } +\left(\alpha +\beta \right)\, {}_{l} \partial _{\eta } +\alpha \, \beta \, \partial _{\eta },  
\end{equation} 
whose  $ r$-order  derivative is

\begin{equation}\label{rderAiry}
{}_{\left(\alpha ,\;\beta ,\, l\right)} \partial _{\eta }^{r}x^{n}=\dfrac{n!(n+\alpha)!(n+\beta)!}{(n-r)!(n-r+\alpha)!(n-r+\beta)!}x^{n-r}.
\end{equation}
By following the same procedure as before, we introduce the composition rule

\begin{equation}\begin{split}\label{GrindEQ__33_Airy} 
& \left( x\oplus _{(\alpha ,\, \beta ,\;l)} y\right) ^{n} = \\ 
& =\sum _{r=0}^{n}\binom{n}{r}\, \dfrac{\Gamma (n+\alpha +1)\, \Gamma (n+\beta +1)\;\;x^{\;n-r} y^{\;r} }{\Gamma (n-r+\alpha +1)\, \Gamma (r+\alpha +1)\, \Gamma (n-r+\beta +1)\, \Gamma (r+\beta +1)\, } , 
\end{split} \end{equation}  
so that the associated \textit{l-t} functions, defined as,

\begin{equation}\begin{split} \label{GrindEQ__34_Airy} 
& {}_{\left(\alpha ,\;\beta ,\, l\right)} c(x)=\sum _{r=0}^{\infty }\dfrac{(-1)^{\;r} x^{\;2r} }{(2\, r)!\Gamma (\alpha +2r+1)\, \Gamma (\beta +2\, r+1)}  , \\ 
& {}_{(\alpha ,\;\beta ,\, l)} s(x)=\sum _{r=0}^{\infty }\dfrac{(-1)^{\;r} x^{\;2r+1} }{(2\, r+1)!\Gamma (\alpha +2r+2)\, \Gamma (\beta +2\, r+2)},
\end{split}\end{equation} 
are straightforwardly shown to satisfy the differential equations

\begin{equation}\begin{split} \label{GrindEQ__35_Airy} 
& {}_{\left(\alpha ,\;\beta ,\, l\right)} \partial _{\eta } \left[{}_{\left(\alpha ,\;\beta ,\, l\right)} c(\lambda \, x)\right]=-\lambda \left[{}_{\left(\alpha ,\;\beta ,\, l\right)} s(\lambda \, x)\right]\, , \\[1.1ex] 
& {}_{\left(\alpha ,\;\beta ,\, l\right)} \partial _{\eta } \left[{}_{\left(\alpha ,\;\beta ,\, l\right)} s(\lambda \, x)\right]=\lambda \left[{}_{\left(\alpha ,\;\beta ,\, l\right)} c(\lambda \, x)\right]
\end{split}\end{equation} 
and the addition theorems based on the extension of the definition of sum specified in eq. \ref{GrindEQ__9b_Airy} and \ref{GrindEQ__30b_Airy}.
\end{cor}

We have shown that the concept of \textit{l-t} function is a fairly natural consequence of the  notion of Laguerre derivative, of its extensions and of the associated eigenfunctions, which belong to Bessel like forms. In the following we will show how to frame the Cholewinsky-Reneke \textit{l-h} functions within the present framework. Before entering this specific aspect of the problem we introduce some consequences of the previous formalism on the theory of diffusion equation associated to the Laguerre derivative and to its generalization.

\subsection{ Bessel Diffusion Equations}

This short section, in which we discuss some evolutive equations based on the operators introduced in the previous sections, is an apparent detour from the main stream of the section.\\

Laguerre type diffusive equations like \cite{D.Babusci}

\begin{equation}\label{GrindEQ__36_Airy} 
\left\lbrace  \begin{array}{l}
\partial _{\tau } F(x,\tau )={}_{l} \partial _{x} F(x,y) \\[1.1ex] 
F(x,\, 0)=f(x),
\end{array}\right.
 \end{equation} 
can be formally solved as

\begin{equation} \label{GrindEQ__37_Airy} 
F(x,\, \tau )=e^{\tau \, {}_{l} \partial _{x} } f(x) .
\end{equation} 
To make the above solution meaningful it is necessary to specify how to calculate the action of the exponential operator containing the Laguerre derivative on the function $f(x)$. We discuss therefore, as introductory example, the case in which $f(x)=e^{x} $, and proceed as follows:

\begin{enumerate}
	\item  We note that the exponential can be written as an integral transform of the Tricomi function
	
	\begin{equation} \label{GrindEQ__40b_Airy} 
	e^{x} =\int _{0}^{\infty }e^{-t}  {}_{l} e(x\, t)\, dt .
	\end{equation} 
	
	\item  We use the properties \ref{GrindEQ__9b_Airy} to end up with
\end{enumerate}

\begin{equation} \label{GrindEQ__41b_Airy} 
e^{\tau \, {}_{l} \partial _{x} } e^{x} =\int _{0}^{\infty }e^{-t\;(1-\tau )}  {}_{l}e(x\, t)\, dt=\dfrac{1}{1-\tau } e^{\frac{x\, }{1-\tau } }  .
\end{equation} 
More in general, whenever

\begin{equation} \label{GrindEQ__40_Airy} 
f(x)=\int _{0}^{\infty }\tilde{f}(t) {}_{l} e(x\, t)\, dt ,
\end{equation} 
the solution of the problem \ref{GrindEQ__36_Airy} can be cast in the form

\begin{equation} \label{GrindEQ__43b_Airy} 
F(x,\, \tau )=\int _{0}^{\infty }\tilde{f}(t)\, e^{t\;\tau }  {}_{l} e(x\, t)\, dt 
\end{equation} 
and  $\tilde{f}(t)$ is the $l$-transform of the function $f(x)$.\\

Before going further with the above formalism, we note that the equation

\begin{equation}\label{GrindEQ__44b_Airy} 
\left\lbrace  \begin{array}{l}
\partial _{\tau } F(x,\tau )={}_{(\alpha,\;\beta,\;l)} \partial _{x} F(x,y) \\[1.1ex] 
 F(x,\, 0)=f(x)
\end{array}\right.
 \end{equation} 
can be solved in an analogous way provided that we replace ${}_{l}e(x)$ with ${}_{l}e_{\alpha ,\, \beta } (x\, )$ in eq. \ref{GrindEQ__43b_Airy}.\\

 In the case in which we consider equations of the type

\begin{equation}\label{GrindEQ__43_Airy} 
\left\lbrace  \begin{array}{l}
 {}_{l} \partial _{\tau } F(x,\tau )={}_{l} \partial _{x} F(x,y) \\[1.1ex] 
 F(x,\, 0)=f(x),
\end{array}\right.  \end{equation} 	
the solution of the problem can be obtained as

\begin{equation}
F(x,\, \tau )= {}_{l}e(\tau \, {}_{l} \partial _{x} )f(x)=f(x\oplus_{l} \tau ).
\end{equation}  

Further comments on the previous statements will be provided in the following section.

\section{ Pseudo-Hyperbolic Functions and\\ Generalized Airy Diffusion Equations}

\noindent As already stressed the study of generalized forms of trigonometric and of hyperbolic functions is an old ``leit motiv'' in the mathematical literature. On the eve of the seventies of the last century Ricci introduced \cite{Ricci} a family of \textit{\textbf{pseudo hyperbolic functions} (PHF)}, which will be proven of noticeable importance for the topics we are discussing. \\

\begin{defn}
 According to ref. \cite{Ricci} the PHF of order $ 3 $ are defined by the series

\begin{equation}\begin{split}
& {}_{[k,3]} e\left(x\right)=\sum _{r=0}^{\infty }\dfrac{x^{3\, r+k} }{(3\, r+k)!}  , \\ 
& k=0,1, 2 .
\end{split}\end{equation}    
\end{defn}                      
On account of the properties of the cubic roots of the unit \cite{Abramovitz}

\begin{equation}\begin{split} \label{GrindEQ__46_Airy} 
& \hat{\omega }_{p} =e^{\frac{2\, i\, p\;\pi }{3} \, } ,\;\;\;\;\;\;\; p=0,1,2, \\ 
& \hat{\omega }_{p}^{3} =1,\; \;\;\;\;\;\;\;\;\;\;\;\;\;p=0, 1, 2,\\
& \hat{\omega }_{p}^{2} +\hat{\omega }_{p}=-1, \;\;\;p=1,2,
\end{split}\end{equation} 
we can state \cite{DMR} the following 

\begin{defn}
We define the \textbf{Euler-like exponential formulae} $\forall x\in\mathbb{R}$

\begin{equation}\begin{split} 
& e^{\hat{\omega}\, x} =\sum _{k=0}^{2}\hat{\omega}^{k}\; {}_{[k,3]} e (x), \\ 
& {}_{[k,3]}e(x)=\dfrac{1}{3} \sum _{p=0}^{2}\hat{\omega }_{p}^{k}\; e^{\hat{\omega }_{p} x}.
\end{split}\end{equation}  
\end{defn}

\begin{prop}
The PHF of order 3 are eingenfunctions of the cubic operator

\begin{equation} \label{GrindEQ__48_Airy} 
\left(\partial {}_{x} \right)^{3} {}_{[k,3]} e(\lambda \, x)=\lambda ^{3} {}_{[k,3]} e(\lambda \, x) 
\end{equation} 
and can be exploited to generalize the \textbf{exponential translation operator} as

\begin{equation}\label{51bAiry}
{}_{[k,3]} \hat{T}(y)={}_{[k,3]} e(y\, \partial _{x} ) .
\end{equation}     
\end{prop}

\begin{cor}             
In the case of  $k=0$ the action of this operator on an ordinary monomial is given by 

\begin{equation}\begin{split}\label{52bAiry}
 {}_{[0,3]} \hat{T}(y)\, x^{3n} &=\frac{1}{3} \left[e^{\hat{\omega }_{0} \, y\, \partial _{x} } +e^{\hat{\omega }_{1} \, y\, \partial _{x} } +e^{\hat{\omega }_{2} \, y\, \partial _{x} } \right]\, x^{3n} = \\ 
& =\frac{1}{3} \sum _{\alpha =0}^{2}(x+\hat{\omega }_{\alpha }  y)^{3n} =(x\oplus_{\left[0,3\right]} y)^{3n} .
\end{split}\end{equation}   
\end{cor}

\begin{cor}
By direct use of the series expansion definition of the function $_{[0,3]} e\left(x\right)$ , we end up with\footnote{A straightforward consequence of eq. \ref{52bAiry} is that (see also ref. \cite{FMCholewinski})\\
	$(1\oplus_{[0,3]} 1)^{3n} =\dfrac{1}{3} \left(2^{3n}+\left(1+e^{\frac{2i\pi}{3}} \right)^{3n}+\left(1+e^{\frac{4i\pi}{3}} \right)^{3n}  \right) =\dfrac{1}{3} \left(2^{3\, n} +(-1)^{n} 2\right)$ }

\begin{equation}\begin{split}\label{53bAiry}
 {}_{[0,3]} e\left(y\, \partial _{x} \right)x^{3\, n} &=\sum _{r=0}^{\infty }\frac{y^{3r} }{(3r)!} \,  \partial _{x}^{3r} x^{3n} = \\ 
& =\sum _{r=0}^{n}\binom{3n}{3r} \; y^{3r}  x^{3\, \left(n-r\right)} =(x\oplus_{\left[0,3\right]} y)^{3n} .
\end{split}\end{equation}                                        
It is therefore evident that the following further  identities can be stated
\begin{equation}\begin{split} \label{GrindEQ__52_Airy} 
& {}_{[0,3]} e\left(y\, \partial _{x} \right)\, {}_{[0,3]} e(x)={}_{[0,3]} e(x\oplus_{0} y), \\[1.1ex] 
& {}_{[0,3]} e\left(y\, \partial _{x} \right)\, {}_{[0,3]} e(x)={}_{[0,3]} e(y)\, {}_{[0,3]} e(x)  
\end{split}\end{equation} 
which, once merged, yield 
\begin{equation} \label{GrindEQ__53_Airy} 
{}_{[0,3]} e\left(x\right)\, {}_{[0,3]} e(y)={}_{[0,3]} e(x\oplus_{[0,3]} y) .
\end{equation} 
In this way we have obtained a result allowing the introduction of \textbf{t-h} like functions according to the paradigm developed so far.
\end{cor}

By a straightforward generalization of the discussion developed in these last sections, we introduce the \textbf{generalized  $h$-functions} defined as

\begin{prop}
	We define the generalized  $h$-functions
	
\begin{equation}\begin{split} \label{GrindEQ__54_Airy} 
& {}_{[0,3]}ch\left(x\right)=\frac{1}{2} \left( {}_{[0,3]}e(x)+ {}_{[0,3]}e\left(-x\right)\right) =\sum _{r=0}^{\infty }\dfrac{x^{\;6\, r} }{(6\, r)!} ,  \\ 
& {}_{[0,3]}sh\left(x\right)=\frac{1}{2} \left( {}_{[0,3]} e(x)- {}_{[0,3]}e\left(-x\right)\right) =\sum _{r=0}^{\infty }\dfrac{x^{\;6\, r+3} }{(6\, r+3)!}  
\end{split}\end{equation} 
and easily state that the relevant addition theorems read

\begin{equation}\begin{split} \label{GrindEQ__55_Airy} 
& {}_{[0,3]}ch\left( \alpha \oplus_{[0,3]} \beta \right) ={}_{[0,3]}ch(\alpha ){}_{[0,3]}ch(\beta )+ {}_{[0,3]} sh(\alpha ) {}_{[0,3]}sh(\beta ), \\
& {}_{[0,3]} sh\left( \alpha \oplus_{[0,3]} \beta \right) = {}_{[0,3]}ch(\alpha ) {}_{[0,3]}sh(\beta )+ {}_{[0,3]}sh(\alpha ) {}_{[0,3]}ch(\beta )  .
\end{split}\end{equation}
\end{prop}
Analogous  conclusions can be reached for ${}_{[k,m]} e(x),\,\; k=0,\dots, m-1$.\\

To obtain the link with the topics discussed so far, we consider a particular case of the function defined in eq. \ref{GrindEQ__32b_Airy}, namely \cite{Airy}

\begin{exmp}
	We consider the function 
\begin{equation} \label{GrindEQ__56_Airy} 
_{l} e_{\alpha -\frac{1}{3} ,\, -\frac{2}{3} } \left(\left(\frac{\eta }{3} \right)^{3} \right)=\sum _{r=0}^{\infty }\frac{\eta ^{3r} }{3^{3r} r!\, \Gamma \left(r+\alpha +\frac{2}{3} \, \right)\Gamma \left(r+\frac{1}{3} \right)}   .
\end{equation} 
The \textbf{associated Laguerre derivative} can be written in terms of the operator

\begin{equation} \label{GrindEQ__57_Airy} 
{}_{\alpha } \hat{\vartheta }={}_{\left(\alpha -\frac{1}{3} ,-\frac{2}{3} ,\, l\right)} \partial _{\left(\frac{\eta }{3} \right)^{3} } =\partial _{\eta } \eta ^{-3\, \alpha } \partial _{\eta } \eta ^{3\, \alpha } \partial _{\eta } , 
\end{equation} 
appearing in the \textbf{generalized Airy equation}

\begin{equation} \label{GrindEQ__58_Airy} 
\partial _{t} F(x,t)={}_{\alpha } \hat{\vartheta }\, F(x,t), 
\end{equation} 
studied in ref. \cite{FMCholewinski}.\\

The function in eq. \ref{GrindEQ__56_Airy} is equivalent, apart from an unessential normalizing factor, to the function exploited by Cholewinsky and Reneke \cite{FMCholewinski} to study the solution of eq. \ref{GrindEQ__58_Airy} which is linked, in particular, to the following expression

\begin{equation}\label{61Airy}
G_{\alpha } (\eta )=\Gamma 
\left(\frac{1}{3} \right)\, \Gamma \left(\alpha +\frac{2}{3} \right) {}_{l}e_{\alpha -\frac{1}{3} ,\, -\frac{2}{3} } \left(\left(\frac{\eta }{3} \right)^{3} \right) .
\end{equation}               
We can make the previous results more transparent by using the identity

\begin{equation} \label{GrindEQ__60_Airy} 
\frac{1}{(3\, n)!} \Gamma \left(n+\frac{2}{3} \right)=\frac{\Gamma \left(\frac{1}{3} \right)\, \Gamma \left(\frac{2}{3} \right)}{3^{3n} n!\, \Gamma \left(n+\frac{1}{3} \right)} , 
\end{equation} 
which allows to recast eq. \ref{61Airy} in the form

\begin{equation} \label{GrindEQ__61_Airy} 
 G_{\alpha } (\eta )=\frac{1}{B\left(\frac{2}{3} ,\alpha \right)} \sum _{r=0}^{\infty }\frac{B\left(r+\frac{2}{3} ,\, \alpha \right)\, \eta ^{3r} }{\left(3r\right)\, !\, }  , 
 \end{equation} 
where $B(x,y)$ is the Beta-function \ref{BetaF}, and derive the identity

\begin{equation} \label{GrindEQ__66_Airy} 
{}_{\alpha } \hat{\vartheta }\, G_{\alpha } (\lambda \;\eta )=\lambda ^{3} G_{\alpha } (\lambda \;\eta ) .
\end{equation}
Finally, the use of the Euler dilatation operator 
yields the following integral transform defining the function \ref{61Airy} in terms of the pseudo hyperbolic function of order 3

\begin{equation}\label{65Airy}
G_{\alpha } (\eta )=\frac{1}{B\left(\frac{2}{3} ,\alpha \right)} \int _{0}^{1}t^{-\frac{1}{3}} (1-t)^{\alpha -1 } {}_{[0,3]} e(\eta \, t^{\frac{1}{3} } )dt .
\end{equation}  

According to the paradigm developed so far we use the associated translation operator to define the composition rule 

\begin{equation}\begin{split} \label{GrindEQ__64_Airy} 
 &G_{\alpha } (y\; {}_{\alpha }\hat{\vartheta }^{\frac{1}{3}} )\, x^{3\, n}:= \\ 
& =\sum _{r=0}^{n}\binom{n}{r}\, \frac{\Gamma \left(\alpha +\frac{2}{3} \right)\, \Gamma \left(\frac{1}{3} \right)\, \Gamma \left(n+\alpha +\frac{2}{3} \right)\Gamma \left(n+\frac{1}{3} \right) \;y^{3\, r} x^{3n-3 r}}{\Gamma \left(r+\frac{1}{3} \right)\, \Gamma \left(n-r+\frac{1}{3} \right)\, \Gamma \left(n-r+\alpha +\frac{2}{3} \right)\, \Gamma \left(r+\alpha +\frac{2}{3} \right)}   =\\
& =\left(x\oplus_{{}_{\alpha } [0|3]} y\right)^{3n}
\end{split} \end{equation} 
forming the \textbf{generalized lbs for the addition theorem of the relevant l-h functions.}
\end{exmp}

A natural extension of the previously developed formalism allows the introduction of the l/h-functions. Their definitions and properties, apart from some computational complications, does not produce any significant conceptual progress.\\

In this section we have seen how a "wise" combination of different methods borrowed from special function theory, operational and umbral calculus and integral transforms, opens new interesting possibilities for the introduction and systematic study of new families of trigonometric functions.\\
In addition, the method yields, as byproduct, the opportunity of getting natural solutions of a large family of PDE belonging to the family of generalized heat equation, whose links with the radial heat equation have been touched on in ref. \cite{D.E.Edmunds}.\\

We stress two points just touched on in the previous part of the Chapter.

\begin{rem}
\textit{We have noted in eq. \ref{NepAiry} that the Laguerre exponential can be obtained through a limit procedure analogous to that involving the Napier number. The same method can e.g. be used to state that}

\begin{equation}
J_{0}(x)=\lim_{n\rightarrow\infty}\left(1\oplus_{l} \left(-\left( \dfrac{x}{2n} \right)^{2}\right)   \right)^{n} ,
\end{equation}
\textit{which is recognized as the asymptotic limit of Laguerre polynomials \cite{L.C.Andrews}.}\\

\noindent \textit{The second point we want to emphasize is the geometrical interpretation of the $ l\!-\!t $ functions from the geometrical point of view.\\
Such an interpretation is provided in Fig. \ref{fig1Airy}, where we have considered the Lissajous curves plotting $l$-sin vs $l$-cos.}\\

\textit{In Fig. \ref{fig2Airy} it is evident that the curves are open since no periodic behaviour is envisaged. However a kind of self-similarity can be noted when the amplitude of the oscillations increase with increasing $x$. }  \\

\begin{figure}[htp]
	\centering
	\includegraphics[width=.45\textwidth]{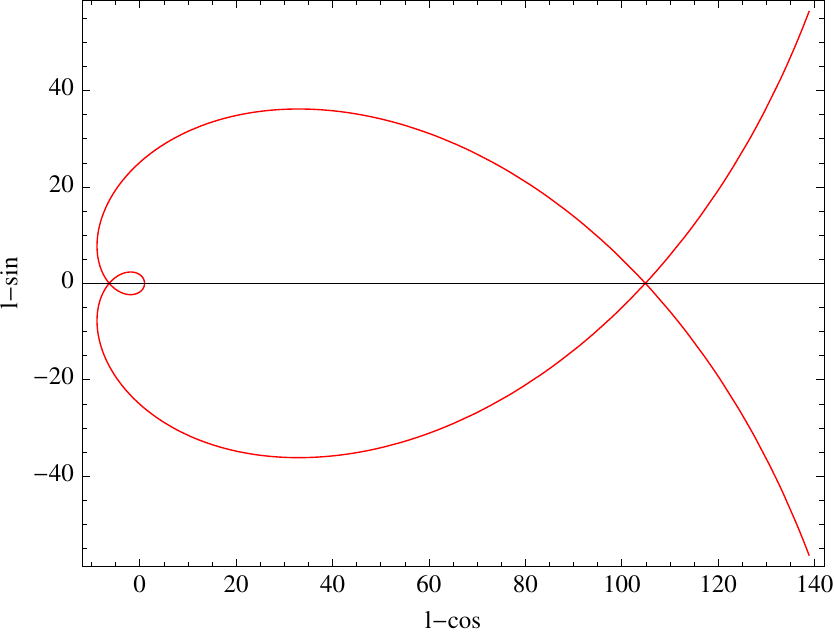}
	\caption{ Fish-like Lissajous diagram of \textit{l-t} functions, $ {}_{l} s(x) $ vs    $ {}_{l} c(x) $ for larger $ x $-range.}\label{fig2Airy}
\end{figure}

\textit{The last figure provides the explicit correspondence of the $l$-sinus and $l$-cosinus along with the relevant "$l$-angle", intended as the area intercepted by the $l$-curve and the segment forming the angle in the positive abscissa direction.}

\begin{figure}[htp]
	\centering
	\includegraphics[width=.45\textwidth]{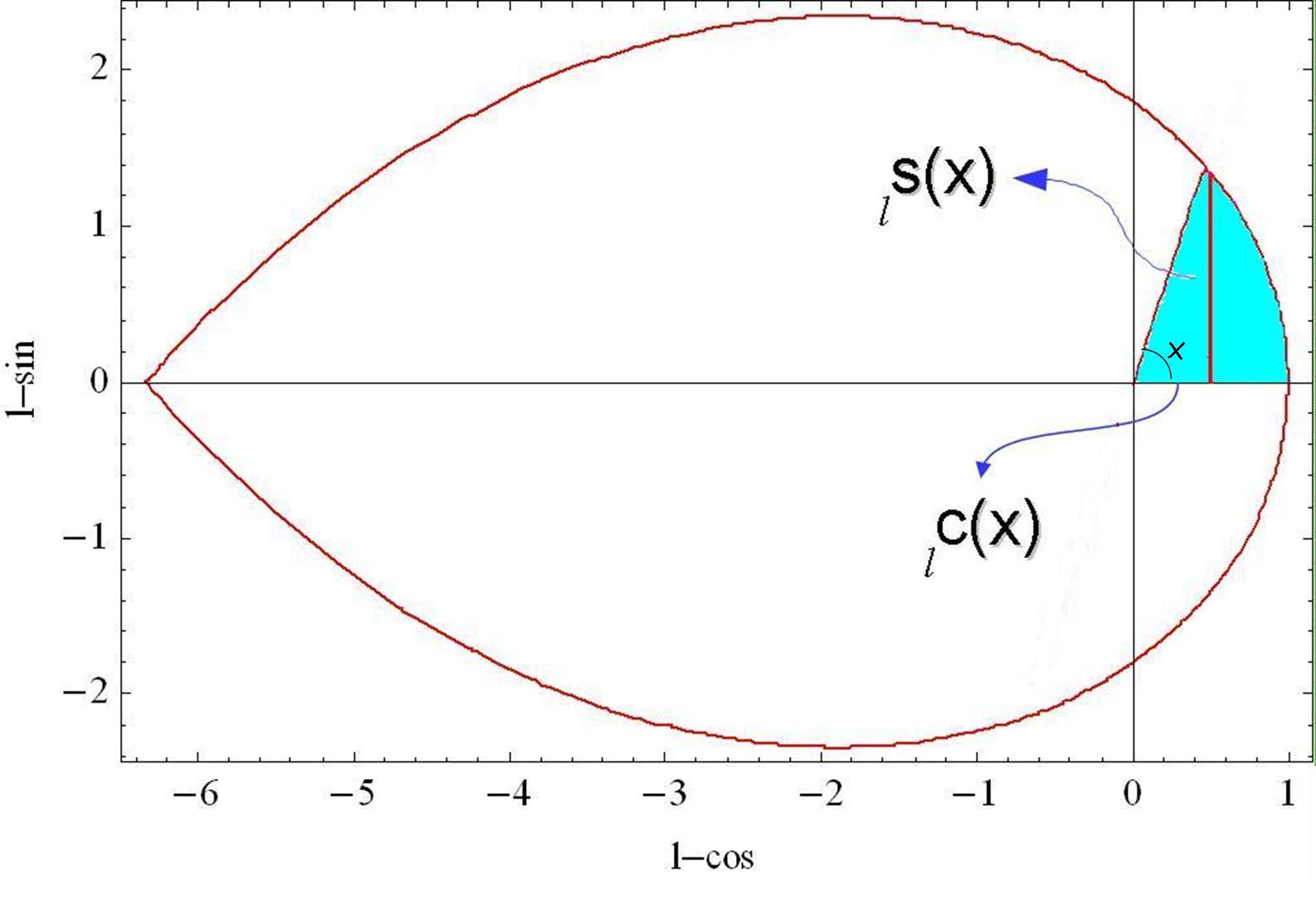}
	\caption{$l$-sinus and $l$-cosinus along with parameter $ x $ which is understood as the blue dashed area.}\label{fig3Airy}
\end{figure} 
\end{rem}

These paragraphs have provided the formalism for a wider understanding of concepts associated with the Laguerre derivative, the underlying algebraic structures and  possible extensions to new forms of Trigonometry. \\

In the next section we extend the concepts to cases involving Matrix Parameterization. 

\section{Generalized Trigonometric Functions and Matrix Parameterization}

Further forms of \textbf{\textit{Generalized Trigonometric Functions}} \textit{(GTF)} can be introduced by means of an extension of the Euler  exponential identity, involving matrices instead of the imaginary unit. The point of view we develop is based on the assumption that matrices are by themselves complex numbers in a broader sense. We see how such an approach yields new families of trigonometry, whose usefulness in application (as charged beam transport) is also discussed. The technique we develop largely benefit from umbral and operational methods which are the leitmotiv of the treatment we propose.

\begin{defn}
According to refs. \cite{Fjelstad}-\cite{DiPalma}, we introduce the Generalized Trigonometric Functions of order 2, $C(t)-S(t)$, by means of the identity

\begin{equation} \label{GrindEQ__1_GenTrFun} 
e^{t\, \hat{M}} =C(t)\, \hat{1}+S(t)\, \hat{M} , \quad \forall t\in\mathbb{R},
\end{equation} 
where $\hat{M},\, \hat{1}$ are respectively a $2\times 2$ non-singular matrix and the unit, namely

\begin{equation}
\hat{M}=\left(\begin{array}{cc} {a} & {b} \\ {c} & {d} \end{array}\right),\, \, \hat{1}=\left(\begin{array}{cc} {1} & {0} \\ {0} & {1} \end{array}\right) .
\end{equation}    
\end{defn}    

\begin{cor}
From eq. \ref{GrindEQ__1_GenTrFun} also follows that

\begin{equation}\begin{split} \label{GrindEQ__3_GenTrFun} 
&e^{t\, \lambda _{+} } =C(t)\, +S(t)\, \lambda _{+} , \\ 
& e^{t\, \lambda _{-} } =C(t)\, +S(t)\, \lambda _{-}, 
\end{split} \end{equation} 
with $\lambda \, _{\pm } $ being the eigenvalues of $\hat{M}$, assumed to be non-singular, thus getting the explicit form of the second order \textit{GTF}, namely

\begin{equation}\begin{split} \label{GrindEQ__4_GenTrFun} 
& C(t)=\frac{\lambda _{-} e^{\lambda _{+} t} -\lambda _{+} e^{\lambda _{-} t} }{\lambda _{-} -\lambda _{+} } , \\ 
& S(t)=\frac{e^{\lambda _{+} t} -e^{\lambda _{-} t} }{\lambda _{+} -\lambda _{-} }.
\end{split} \end{equation} 
The structure of eq. \ref{GrindEQ__1_GenTrFun} is that of the \textit{Euler-De Moivre} identity, with $\hat{M}$ playing the role of imaginary unit, on the other side eq. \ref{GrindEQ__3_GenTrFun} represents the scalar counterpart of \ref{GrindEQ__1_GenTrFun} and, accordingly, $\lambda \, _{\pm } $ are understood as conjugated imaginary units.
\end{cor}

\begin{propert}
The properties of the $\cos $ and  $\sin $ like  functions $C(t),\, S(t)$ can be inferred from either eqs. \ref{GrindEQ__1_GenTrFun}-\ref{GrindEQ__3_GenTrFun}, which yield for example (see also ref. \cite{DiPalma})

\begin{equation}\begin{split} \label{GrindEQ__5a_GenTrFun}
& C^{2} +\Delta _{\hat{M}} S^{2} +Tr(\hat{M})\, CS=e^{Tr(\hat{M})\, t} ,\\ 
& Tr(\hat{M})=a+d, \\ 
& \Delta _{\hat{M}} =a\, d-b\, c, \quad \forall a,b,c,d\in\mathbb{R}:\Delta _{\hat{M}}\neq 0,
\end{split} \end{equation} 
recognized as the \textbf{fundamental trigonometric identity} and

\begin{equation}\begin{split} \label{GrindEQ__5b_GenTrFun}
& C(2\, t)=C^{2} -\Delta _{\hat{M}} S^{2} , \\ 
& S(2\, t)=2\, C(t)\, S(t)+Tr(\hat{M})\, S^{2} ,
\end{split} \end{equation} 
recognized as the \textbf{duplication formulae}.
\end{propert}

\begin{propert}
 By keeping the derivative of both sides of eq. \ref{GrindEQ__1_GenTrFun} with respect to the variable $ t $, we find

\begin{equation}\label{derGenTrFun}
\frac{d}{dt} e^{t\, \hat{M}} =\left(\frac{d}{dt} C(t)\right)\, \hat{1}+\left(\frac{d}{dt} S(t)\right)\, \hat{M}.
\end{equation}
Being also

\begin{equation}
\frac{d}{dt} e^{t\, \hat{M}} =\hat{M}e^{t\, \hat{M}} =C(t)\hat{M}+S(t)\, \hat{M}^{2} 
\end{equation}
and since

\begin{equation}\label{der2GenTrFun}
\hat{M}^{2} =-\Delta _{\hat{M}} \hat{1}+Tr(\hat{M})\, \hat{M},
\end{equation}
we end up, after combining eqs. \ref{derGenTrFun}-\ref{der2GenTrFun} and equating ``real'' and ``imaginary'' parts, the following identities, specifying the properties under derivatives of the \textit{GTF}

\begin{equation}\begin{split} \label{GrindEQ__6d_GenTrFun}
& \frac{d}{dt} C(\, t)=-\Delta _{\hat{M}} S(t), \\ 
& \frac{d}{dt} S(\, t)=Tr(\hat{M})\, S(t)+C(t).
\end{split} \end{equation}
\end{propert}

\begin{cor}
We can infer directly from eq. \ref{GrindEQ__4_GenTrFun} that the second order \textit{GTF's} exhibit, under variable reflection, the identities

\begin{equation}\begin{split} \label{GrindEQ__7new_GenTrFun}
& C(-\, t)=e^{-Tr(\hat{M})\, t} \left(-Tr(\hat{M})\, S(t)+C(t)\right)=e^{-Tr(\hat{M})\, t} \left(\frac{d}{dt} S(t)\right), \\ 
& S(-t)=-e^{-Tr(\hat{M})\, t} S(t) , 
\end{split} \end{equation}
which underscore the significant difference with the ordinary \textit{TF} (be they circular or hyperbolic) with definite even or odd parities.
\end{cor}

 Further properties can be argued by the use of other means. By keeping e.g. the freedom of treating $\hat{M}$ as an \underline{\textit{ordinary algebraic quantity}} we can formally derive integrals involving \textit{GTF} thus finding e.g.

\begin{propert}
\begin{equation}\begin{split} \label{GrindEQ__8_GenTrFun} 
 \int _{}^{t} e^{t'\, \hat{M}}dt' &={}_{I} C(t)\, \hat{1}+{}_{I} S(t)\, \hat{M}, \\ 
 \int _{}^{t} e^{t'\, \hat{M}}dt' &=\frac{1}{\hat{M}} e^{t\, \hat{M}} =C(t)\, \hat{M}^{-1} +S(t)\, \hat{1}, \\ 
 {}_{I} C(t)&=\int _{}^{t} C(t')dt',\\
 {}_{I} S(t)&=\int _{}^{t} S(t')dt' .
\end{split} \end{equation} 
Moreover, since the following identity holds 

\begin{equation}\begin{split} \label{GrindEQ__9_GenTrFun} 
& \hat{M}^{-1} =c_{-1} \hat{1}+s_{-1} \hat{M}, \\ 
& c_{-1} =\frac{\lambda _{-} \lambda _{+}^{-1} -\lambda _{+} \lambda _{-}^{-1} }{\lambda _{-} -\lambda _{+} } =\frac{Tr(\hat{M})}{\Delta _{\hat{M}} } , \\ 
& s_{-1} =\frac{\lambda _{+}^{-1} -\lambda _{-}^{-1} }{\lambda _{+} -\lambda _{-} } =-\frac{1}{\Delta _{\hat{M}} } ,
\end{split} \end{equation} 
we obtain the  ``primitives'' of the \textit{GTF's}

\begin{equation}\begin{split} \label{GrindEQ__10_GenTrFun} 
& {}_{I} C(t)=\frac{Tr(\hat{M})}{\Delta _{\hat{M}} } C(t)+S(t), \\ 
& {}_{I} S(t)=-\frac{1}{\Delta _{\hat{M}} } C(t).
\end{split} \end{equation}
\end{propert}
 
A straightforward consequence of the previous relationships is

\begin{exmp}
\begin{equation}\begin{split} \label{GrindEQ__11_GenTrFun} 
& \int _{0}^{\infty }C(-t')dt'=\frac{Tr(\hat{M})}{\Delta _{\hat{M}} } , \\ 
& \int _{0}^{\infty } S(-t')dt'=-\frac{1}{\Delta _{\hat{M}} },
\end{split} \end{equation} 
which hold true only if the integrals are convergent, namely if $Re(\lambda _{\pm } )$ are both positive.
\end{exmp}

 A further slightly more intriguing example is provided by the following \textit{Gaussian} integral.
 
 \begin{exmp}
\begin{equation}\begin{split}
 \int _{-\infty }^{+\infty } e^{-t^{2} \, \hat{M}}dt &=\sqrt{\frac{\pi }{\hat{M}} } =\sqrt{\pi } \left(c_{-\frac{1}{2} } \hat{1}+s_{-\frac{1}{2} } \hat{M}\right), \\
c_{-1/2} &=\frac{\lambda _{-} \lambda _{+}^{-1/2} -\lambda _{+} \lambda _{-}^{-1/2} }{\lambda _{-} -\lambda _{+} } , \\ 
 s_{-1/2} &=\frac{\lambda _{+}^{-1/2} -\lambda _{-}^{-1/2} }{\lambda _{+} -\lambda _{-} } ,
\end{split} \end{equation} 
which yields the following generalizations of the \textit{Fresnel} integrals, obtained by other means in ref. \cite{DiPalma},

\begin{equation}\begin{split}\label{solGenTrFun}
& \int _{-\infty }^{+\infty } C(-t^{2} )dt=\sqrt{\pi } c_{-\frac{1}{2} } , \\ 
& \int _{-\infty }^{+\infty }S(-t^{2} )dt=\sqrt{\pi } s_{-\frac{1}{2} } .
\end{split} \end{equation}
The convergence of these integrals depends on the eigenvalues $\lambda _{\pm } $, if convergence is ensured, eq. \ref{solGenTrFun} provides the most general form of solution.
\end{exmp} 

This result should be understood in the spirit of our umbral treatment of exponential and the fact that we treat $\hat{M}$ as an ordinary algebraic quantity.

\begin{cor}
 Iterating the procedure, leading to eqs. \ref{GrindEQ__6d_GenTrFun}, namely by keeping successive derivatives with respect to $ t $ of both sides of \ref{GrindEQ__1_GenTrFun} and by noting that $\forall n\in\mathbb{N}$

\begin{equation} \label{GrindEQ__14_GenTrFun} 
\hat{M}^{n} =c_{n} \hat{1}+s_{n} \hat{M} ,
\end{equation} 
we end up with 

\begin{equation}\begin{split} \label{GrindEQ__15_GenTrFun} 
& \left(\frac{d}{dt} \right)^{n} C(t)=c_{n} C(t)+c_{n+1} S(t), \\ 
& \left(\frac{d}{dt} \right)^{n} S(t)=s_{n} C(t)+s_{n+1} S(t).
\end{split} \end{equation} 
It is evident that the coefficients $c_{\nu } ,\, s_{\nu } $ are essentially \textit{GTF} in which $e^{\lambda _{\pm } t} $ are replaced by $\lambda _{\pm }^{\nu } $. The relevant properties will be discussed later in this Chapter.
\end{cor}

\begin{prop}
 The \textbf{addition formulae} too can be derived in terms of the $c_{n} ,\, s_{n} $  coefficients as 

\begin{equation}\begin{split}\label{coeffGenTrFun}
 C(t+t')&=C(t)\, C(t')+c_{2} S(t)\, S(t'), \\[1.1ex] 
 S(t+t')&=\left(C(t)\, +s_{2} S(t)\right)\, S(t')+S(t)\, C(t'), \\[1.1ex] 
 s_{2} &=Tr(\hat{M}),\\
  c_{2} &=-\Delta _{\hat{M}} .
\end{split} \end{equation} 
In absence of the simple reflection properties of the ordinary circular functions, we can establish the \textbf{subtraction formulae} according to the expressions given below

\begin{equation}\begin{split} \label{GrindEQ__17_GenTrFun} 
& C(t-t')=e^{-s_{2} \, t'} \left[s_{2} S(t')C(t)+C(t')C(t)-c_{2} S(t)\, S(t')\right], \\[1.1ex] 
& S(t-t')=-e^{-\, s_{2} t'} \left[C(t)S(t')-C(t')S(t)\right]
\end{split} \end{equation} 
which, once combined with eq, \ref{coeffGenTrFun}, yield the following \textbf{prosthaphaeresis like identities}

\begin{equation}\begin{split} \label{GrindEQ__18_GenTrFun} 
C(p)-e^{s_{2} \frac{(p-q)}{2} } C(q)=&-s_{2} S\left(\frac{p-q}{2} \right)C\left(\frac{p+q}{2} \right)+ \\ 
& +2c_{2} S\left(\frac{p-q}{2} \right)S\left(\frac{p+q}{2} \right).
\end{split} \end{equation}
\end{prop}
 
In the forthcoming section we provide some examples aimed at providing the usefulness of this family of functions in applications.

\subsection{ GTF, Matrix Parameterization and Generalized Complex Forms}\label{GTFMP}

To proceed further, we remind that eq. \ref{GrindEQ__1_GenTrFun} follows from the \textit{Cayley-Hamilton} Theorem \cite{Birkhoff},  which allows to write a given function (usually an exponential) of a matrix $\hat{\Sigma }$ in terms of its characteristic polynomial. We now use the \textit{GTF} to provide the reverse procedure, namely we write a given matrix $\hat{\Sigma }$ in exponential form, namely

\begin{equation}\begin{split} \label{GrindEQ__19_GenTrFun} 
& \hat{\Sigma }=e^{\hat{T}} , \\[1.1ex]
 & \hat{\Sigma }=\left(\begin{array}{cc} {l} & {m} \\ {n} & {p} \end{array}\right)\, ,\; \hat{T}=\left(\begin{array}{cc} {\alpha } & {\beta } \\ {\gamma } & {\delta } \end{array}\right).
\end{split} \end{equation}
The problem we are interested in is therefore that of finding the elements of the exponentiated matrix $\hat{T}$ , once those of $\hat{\Sigma }$ are known. The use of eq. \ref{GrindEQ__1_GenTrFun} yields

\begin{equation} \label{GrindEQ__20_GenTrFun} 
e^{\, \hat{T}} =C(1)\, \hat{1}+S(1)\, \hat{T} ,
\end{equation} 
where the \textit{GTF} are expressed in terms of the eigenvalues of $\hat{T}$. It is therefore worth to remind that both     $\hat{\Sigma },\, \hat{T}$  are diagonalised through the same matrix $\hat{D}$ and therefore

\begin{equation}\begin{split} \label{GrindEQ__21_GenTrFun} 
& \hat{D}^{-1} \hat{\Sigma }\, \hat{D}=e^{\hat{D}^{-1} \hat{T}\, \hat{D}} , \\ 
& \hat{D}^{-1} \hat{\Sigma }\, \hat{D}=\left(\begin{array}{cc} {\sigma _{+} } & {0} \\ {0} & {\sigma _{-} } \end{array}\right)=\left(\begin{array}{cc} {e^{\tau _{+} } } & {0} \\ {0} & {e^{\tau _{-} } } \end{array}\right),
\end{split} \end{equation}
where $\sigma _{\pm } ,\, \tau _{\pm } $ denote the eigenvalues of the  $\hat{\Sigma }$  and $\hat{T}$ matrices respectively. It is furthermore evident that

\begin{equation} \label{GrindEQ__22_GenTrFun} 
\tau _{\pm } =\ln (\sigma _{\pm } ) .
\end{equation} 
We can therefore write

\begin{prop}
\begin{equation}\begin{split} \label{GrindEQ__23_GenTrFun} 
 \hat{\Sigma }&=C(1)\, \hat{1}+S(1)\, \hat{T}, \\[1.1ex] 
 C(1)&=\frac{\ln (\sigma _{-} )\, \sigma _{+} -\ln (\sigma _{+} )\, \sigma _{-} }{\ln (\sigma _{-} )-\ln (\sigma _{+} )} , \\[1.1ex] 
 S(1)&=\frac{\sigma _{+} \, -\sigma _{-} \, }{\ln (\sigma _{+} )-\ln (\sigma _{-} )}
\end{split} \end{equation}
and 

\begin{equation} \label{GrindEQ__24_GenTrFun} 
\hat{T}=\left(\begin{array}{cc} {\dfrac{l-C(1)}{S(1)} } & {\dfrac{m}{S(1)} } \\[2.5ex] 
{\dfrac{n}{S(1)} } & {\dfrac{p-C(1)}{S(1)} } \end{array}\right). 
\end{equation} 
\end{prop}

\begin{Oss}
It is now worth noting that

\begin{equation} \label{GrindEQ__25_GenTrFun} 
\hat{\Sigma }^{n} =e^{n\, \hat{T}} =C(n)\, \hat{1}+S(n)\, \hat{T} 
\end{equation} 
and it should be stressed that the arguments of the \textit{GTF} in the elements of the matrix $\hat{T}$ in eq. \ref{GrindEQ__20_GenTrFun} remains the unity.
\end{Oss}

The parameterization we have proposed is a generalized form of what is known in the Physics of \textit{charged beam transport} as the \textit{Courant-Snyder} parameterization, which is exploited to adapt the beam sizes to the characteristics of the transport device or in laser optics to transport an optical beam through ordinary lens systems \cite{Steffen}.\\

In the following  we will extend the method to matrices with larger dimensions, using higher order \textit{GTF}. Before doing this, we take advantage from the present point of view to extend the notion of \textbf{complex number}, which is defined as  

\begin{defn}
	$\forall x,y\in\mathbb{R}$, we define the complex number through the identities
	
\begin{equation}\begin{split} \label{GrindEQ__26_GenTrFun} 
& \zeta _{+} =x+\lambda _{+} \, y, \\ 
& \zeta _{-} =x+\lambda _{-} \, y,
\end{split} \end{equation} 
with "modulus''

\begin{equation} \label{GrindEQ__27_GenTrFun} 
\zeta _{+} \zeta _{-} =x^{2} +Tr(\hat{M})\, x\, y+\Delta _{\hat{M}} y^{2}  .
\end{equation} 
The relevant trigonometric form can be written as ($\lambda $ may be either $\lambda _{+} $ or its conjugate form $\lambda _{-} $)

\begin{equation}\begin{split} \label{GrindEQ__28_GenTrFun} 
& \zeta =\left|A\right|e^{\lambda \, \vartheta } , \\[1.1ex] 
& \left|A\right|=\sqrt{\zeta _{+} \zeta _{-} } e^{-Tr(\hat{M})\, \frac{\vartheta }{2} } , \\[1.1ex] 
& \vartheta =\frac{1}{\lambda _{+} -\lambda _{-} } \ln \left[\frac{1+\frac{y}{x} \lambda _{+} }{1+\frac{y}{x} \lambda _{-} } \right].
\end{split} \end{equation} 
\end{defn}

The conclusion, we may draw from this last result, is that \textit{the concept of imaginary number is more subtle than it might be thought, it is not necessarily associated with the roots of a negative number but can be constructed with any pair of numbers, solutions of a second degree algebraic equation }\cite{DiPalma}.\\

We have tried to keep our treatment of \textit{GTF} following in a close parallel with the ordinary circular trigonometry, it is therefore important to note that the geometrical image of the condition \ref{GrindEQ__5a_GenTrFun} is no more a circle but a more complicated curve not necessarily closed. Notwithstanding a ``cos'' and ``sin'' like interpretation of the \textit{GTF} is still possible (see Figs. \ref{fig1GenTrFun} - \ref{fig2GenTrFun}). It is however worth noting that \textit{GTF} may be circular or  hyperbolic like, according to whether  $Im(\lambda )$ be $\ne 0 \, {\rm or\; }=0$. The argument of the \textit{GTF} cannot be simply regarded as angles, notwithstanding, it is natural to ask whether there is any quantity playing the role of $\pi $, even though if e.g. $Tr(\hat{M})\ne 0$ we are not dealing with periodic functions.\\

 To clarify this point we give the following
 
 \begin{defn}
We try to keep advantage from the \textit{Euler}-formula "$e^{\;i\, \frac{\pi }{2} } =i$" to define two distinct quantities $\pi _{\pm } $ such that 

\begin{equation}\begin{split} \label{GrindEQ__29_GenTrFun} 
& e^{\lambda _{-} \frac{\pi _{-} }{2} } =\lambda _{-} , \\ 
& e^{\lambda _{+} \frac{\pi _{+} }{2} } =\lambda _{+} ,
\end{split} \end{equation} 
yielding

\begin{equation}
\pi _{\pm } =\frac{2\, \ln (\lambda _{\pm } )}{\lambda _{\pm } }.
\end{equation} 
\end{defn}
It is furthermore worth noting the "funny'' identities

\begin{propert}
\begin{equation}\begin{split} \label{GrindEQ__30_GenTrFun} 
 e^{\lambda _{\pm } \pi _{\pm } } &=Tr(\hat{M})\lambda _{\pm } -\Delta _{\hat{M}} , \\ 
 \lambda _{-}^{\lambda _{+} } &=e^{\frac{\Delta _{\hat{M}} }{2} \pi _{-} } , \\ 
 \lambda _{+}^{\lambda _{-} } &=e^{\frac{\Delta _{\hat{M}} }{2} \pi _{+} } , \\ 
 e^{\lambda _{-}^{2} \frac{\pi _{-} }{2} } &=e^{(Tr(\hat{M})\lambda _{-} \frac{\pi _{-} }{2} } e^{-\, \frac{\Delta _{\hat{M}} }{2} \pi _{-} } =\lambda _{-}^{\lambda _{-} } .
\end{split} \end{equation} 
The last of which can also be reinterpreted as

\begin{equation}
\lambda _{-}^{\lambda _{-} } =\lambda _{-}^{Tr(\hat{M})-\lambda _{+}} .
\end{equation} 
It is also evident that if $Im(\lambda _{\pm } )\ne 0$, the \textit{GTF} functions exhibit infinite zeros on the real axis, which for $ C $ and $ S $ are, respectively, given by 

\begin{equation}\begin{split} \label{GrindEQ__31_GenTrFun} 
& {}_{C} t_{n}^{*} =\frac{1}{2(\lambda _{-} -\lambda _{+} )} \left[\lambda _{-} \pi _{-} -\lambda _{+} \pi _{+} -4\, i\, n\, \pi \right], \\ 
& {}_{S} t_{n}^{*} =\frac{2\, in\, \pi }{(\lambda _{-} -\lambda _{+} )} .
\end{split} \end{equation} 
\end{propert}

\begin{figure}[htp]
	\centering
	\begin{subfigure}[c]{0.48\textwidth}
		\includegraphics[width=0.9\linewidth]{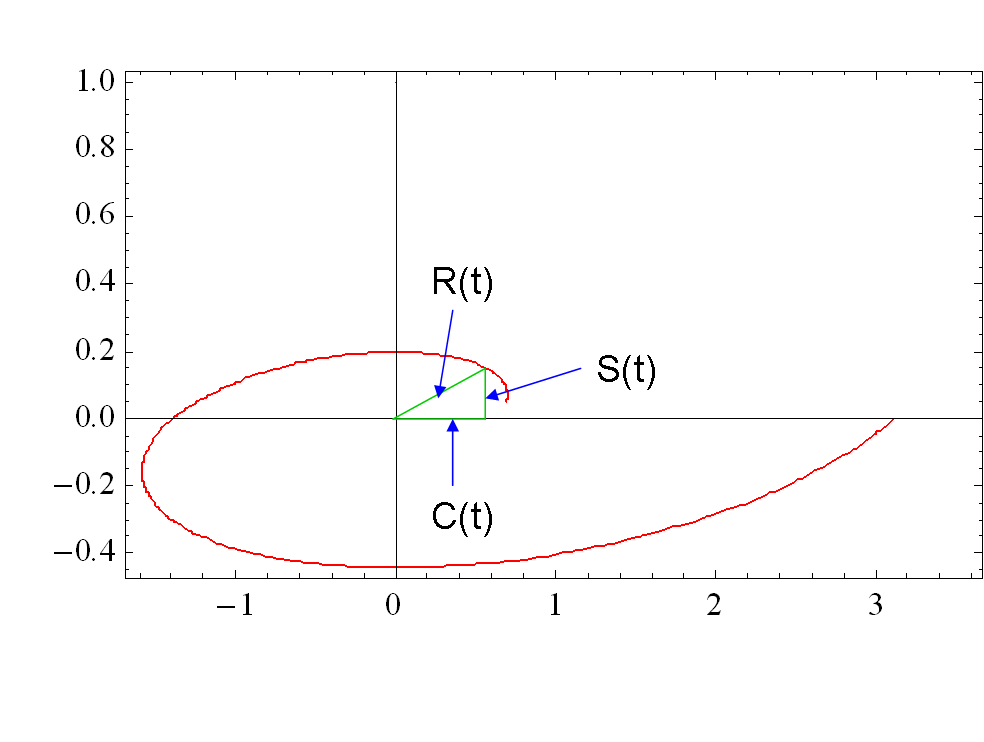}
		\caption{Geometrical Images.}
		\label{FigCirc}
	\end{subfigure}
	\begin{subfigure}[c]{0.48\textwidth}
		\includegraphics[width=0.9\linewidth]{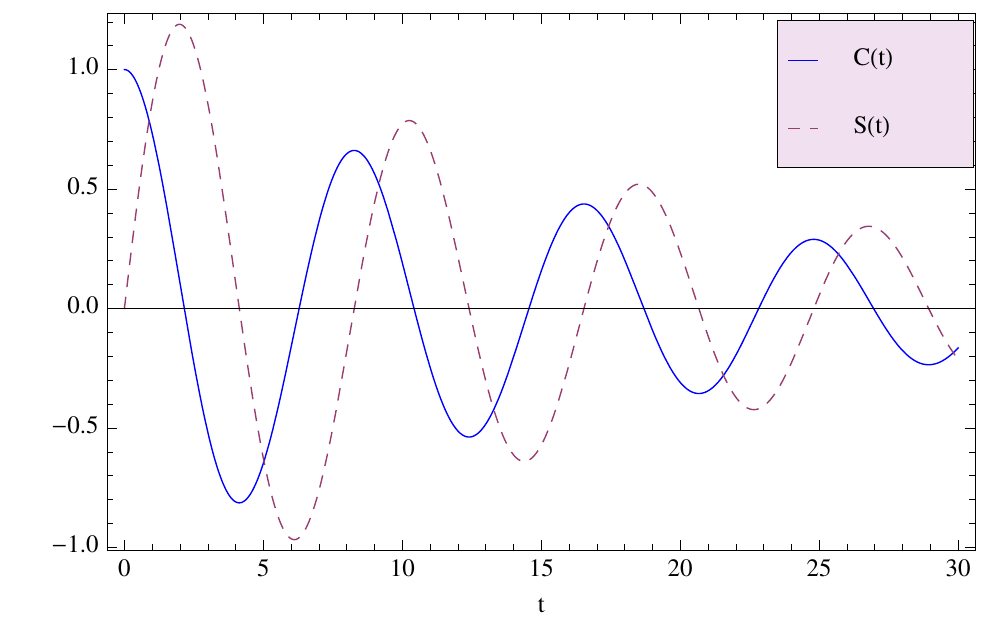}
		\caption{Behavior vs. Argument.}
		\label{Fig1CosSiGenTrFun}
	\end{subfigure}
	\\[7mm]
	\begin{subfigure}[c]{0.48\columnwidth}
		\includegraphics[width=0.9\linewidth]{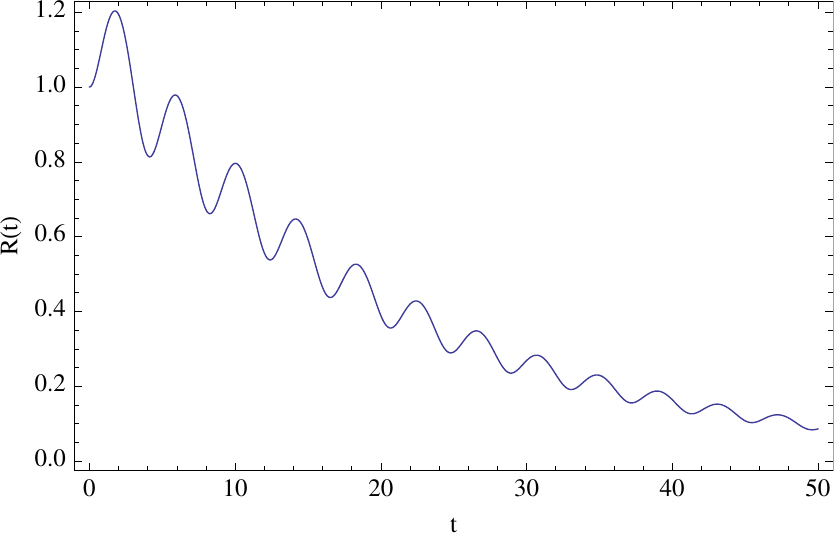}
		\caption{$R(t)=\sqrt{C(t)^{2} +S(t)^{2} } $}
		\label{Fig2RgGenTrFun}
	\end{subfigure}
	\\[3mm]
	\caption{Generalized Trigonometric Functions for $Im(\lambda )\ne 0$.}\label{fig1GenTrFun} 
\end{figure}

\begin{figure}[htp]
	\centering
	\begin{subfigure}[c]{0.48\textwidth}
		\includegraphics[width=0.92\linewidth]{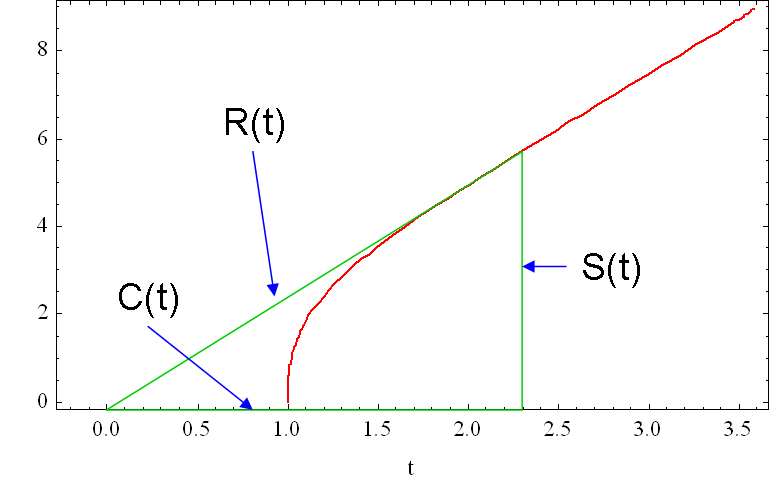}
		\caption{Geometrical Images.}
		\label{fig:fig1cossinGenTrFun}
	\end{subfigure}
	\begin{subfigure}[c]{0.48\textwidth}
		\includegraphics[width=0.9\linewidth]{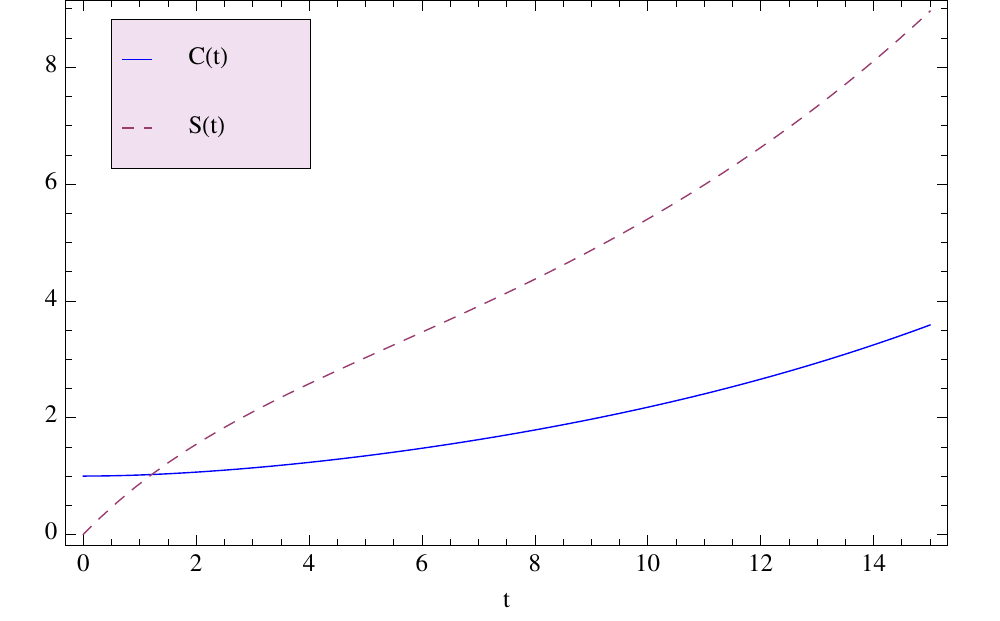}
		\caption{Behavior vs. Argument.}
		\label{fig:fig1cossinGenTrFun}
	\end{subfigure}
	\\[3mm]
	\caption{Generalized Trigonometric Functions for $Im(\lambda )= 0$.}\label{fig2GenTrFun} 
\end{figure}
\newpage
To appreciate the analogies and the differences as well, we have reported in Figs. \ref{fig3GenTrFun} the function $C(t)$ and its counterpart $C(-t)$.

\begin{figure}[htp]
	\centering
	\begin{subfigure}[c]{0.48\textwidth}
		\includegraphics[width=0.9\linewidth]{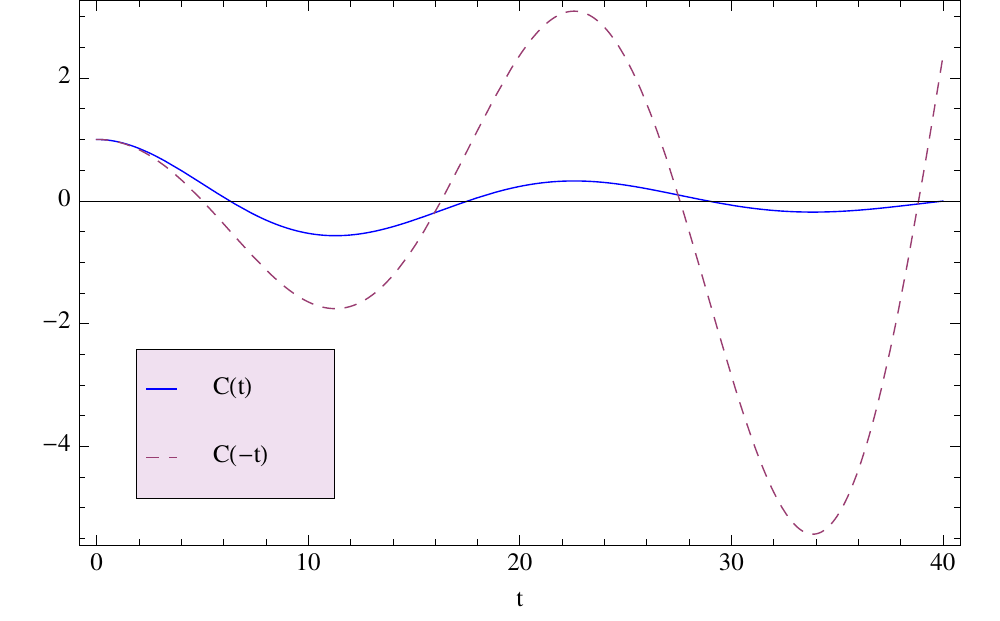}
		\caption{$Im(\lambda )\ne 0$.}
		\label{fig:fig1cossinGenTrFun}
	\end{subfigure}
	\begin{subfigure}[c]{0.48\textwidth}
		\includegraphics[width=0.9\linewidth]{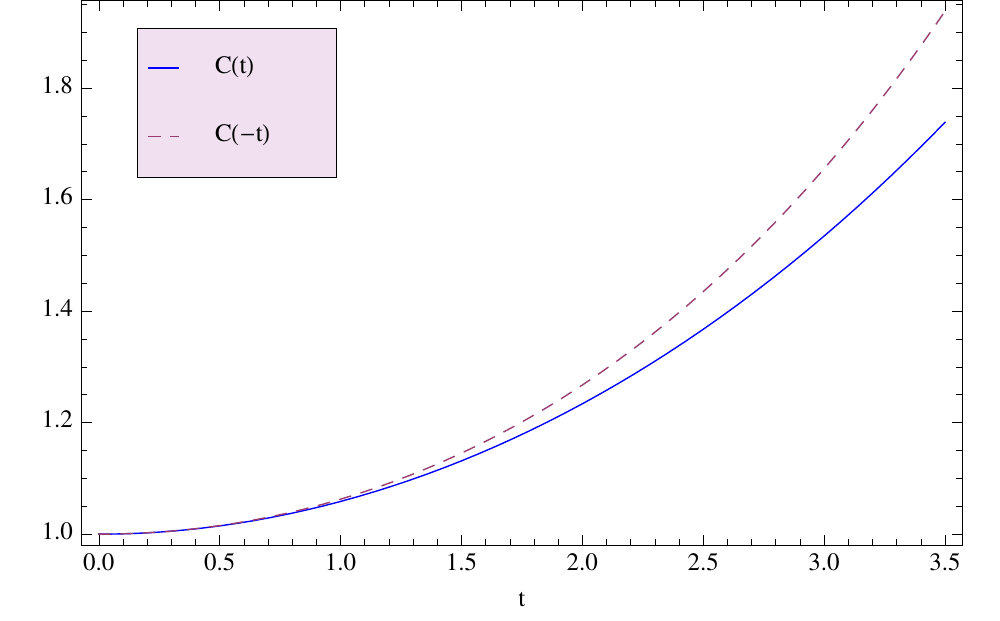}
		\caption{$Im(\lambda )=0$.}
		\label{fig:fig1cossinGenTrFun}
	\end{subfigure}
	\\[3mm]
	\caption{Behavior of $ C(t) $ and $ C(-t) $ vs Argument.}\label{fig3GenTrFun} 
\end{figure}
\newpage

\subsection{ Third and Higher Order GTF }

According to terminology of ref. \cite{DiPalma}, the order of the \textit{GTF} is associated with that of the corresponding generating matrix then, e.g., 

\begin{exmp}
If $\hat{M}$ is a $3\times 3$ non-singular matrix with three distinct eigenvalues we have

\begin{equation} \label{GrindEQ__32_GenTrFun} 
e^{t\, \hat{M}} =C_{0} (t)\, \hat{1}+C_{1} (t)\, \hat{M}+C_{2} (t)\, \hat{M}^{2} . 
\end{equation} 
We can introduce the third order \textit{GTF}, $C_{0,\, 1,\, 2} (t)$ 

\begin{equation}\begin{split} \label{GrindEQ__33_GenTrFun} 
& \left(\begin{array}{c} {C_{0} (t)} \\ {C_{1} (t)} \\ {C_{2} (t)} \end{array}\right)=\left[\hat{V}(\lambda _{1} ,\lambda _{2} ,\lambda _{3} )\right]^{-1} \, \left(\begin{array}{c} {e^{\lambda _{1} t} } \\ {e^{\lambda _{2} t} } \\ {e^{\lambda _{3} t} } \end{array}\right), \\ 
& \hat{V}(\lambda _{1} ,\lambda _{2} ,\lambda _{3} )=\left(\begin{array}{ccc} {1} & {\lambda _{1} } & {\lambda _{1}^{2} } \\ {1} & {\lambda _{2} } & {\lambda _{2}^{2} } \\ {1} & {\lambda _{3} } & {\lambda _{3}^{2} } \end{array}\right),
\end{split} \end{equation} 
where $\hat{V}(\lambda _{1} ,\lambda _{2} ,\lambda _{3} )$ is the \textbf{Vandermonde} matrix, constructed with the eigenvalues of $\hat{M}$. The inverse of $\hat{V}$ can be written as \cite{Horn}\\

\begin{equation}\begin{split} 
& \left(\begin{array}{ccc} {1} & {\lambda _{1} } & {\lambda _{1}^{2} } \\[3ex] {1} & {\lambda _{2} } & {\lambda _{2}^{2} } \\[3ex] {1} & {\lambda _{3} } & {\lambda _{3}^{2} } \end{array}\right)^{-1}=\\
& =\left(\begin{array}{ccc} {\dfrac{\lambda _{2} \lambda _{3} }{(\lambda _{1} -\lambda _{2} )\, \left(\lambda _{1} -\lambda _{3} \right)} } & {\dfrac{\lambda _{1} \lambda _{3} }{(\lambda _{2} -\lambda _{1} )\, \left(\lambda _{2} -\lambda _{3} \right)} } & {\dfrac{\lambda _{1} \lambda _{2} }{(\lambda _{3} -\lambda _{1} )\, \left(\lambda _{3} -\lambda _{2} \right)} } \\[2ex] {-\dfrac{\lambda _{2} +\lambda _{3} }{(\lambda _{1} -\lambda _{2} )\, \left(\lambda _{1} -\lambda _{3} \right)} } & {-\dfrac{\lambda _{1} +\lambda _{3} }{(\lambda _{2} -\lambda _{1} )\, \left(\lambda _{2} -\lambda _{3} \right)} } & {-\dfrac{\lambda _{1} +\lambda _{2} }{(\lambda _{3} -\lambda _{1} )\, \left(\lambda _{3} -\lambda _{2} \right)} } \\[2ex] {\dfrac{1}{(\lambda _{1} -\lambda _{2} )\, \left(\lambda _{1} -\lambda _{3} \right)} } & {\dfrac{1}{(\lambda _{2} -\lambda _{1} )\, \left(\lambda _{2} -\lambda _{3} \right)} } & {\dfrac{1}{(\lambda _{3} -\lambda _{1} )\, \left(\lambda _{3} -\lambda _{2} \right)} } \end{array}\right).
\end{split}\end{equation}
We can therefore write the third order \textit{GTF} as 

\begin{equation}\begin{split} \label{GrindEQ__35_GenTrFun} 
& C_{0} (t)=\frac{1}{\Delta (\lambda _{1} ,\lambda _{2} ,\lambda _{3} )} \sum _{i,j,k=1}^{3}\frac{\varepsilon _{i,j,k} }{2}  \lambda _{i} \lambda _{j} (\lambda _{j} -\lambda _{i} )\, e^{\lambda _{k} t} , \\ 
& C_{1} (t)=\frac{1}{\Delta (\lambda _{1} ,\lambda _{2} ,\lambda _{3} )} \sum _{i,j,k=1}^{3}\frac{\varepsilon _{i,j,k} }{2}  (\lambda _{i}^{2} -\lambda _{j}^{2} )\, e^{\lambda _{k} t} , \\
& C_{2} (t)=-\frac{1}{\Delta (\lambda _{1} ,\lambda _{2} ,\lambda _{3} )} \sum _{i,j,k=1}^{3}\frac{\varepsilon _{i,j,k} }{2}  (\lambda _{i} -\lambda _{j} )\, e^{\lambda _{k} t},
\end{split} \end{equation} 
where

\begin{equation} \label{GrindEQ__36_GenTrFun} 
\Delta (\lambda _{1} ,\, \lambda _{2} ,\, \lambda _{3} )=(\lambda _{2} -\lambda _{1} )\, \left(\lambda _{3} -\lambda _{1} \right)\left(\lambda _{3} -\lambda _{2} \right) 
\end{equation} 
is the Vandermonde determinant and $\varepsilon _{i,j,k} $ is the \textit{Levi-Civita} tensor.
\end{exmp}

\begin{cor}
 By following the same procedure of previous paragraph, we can extend to the third order the properties of the second order  case. It is easily argued that they satisfy third order differential equations and that the relevant addition formulae read

\begin{equation}\begin{split} \label{GrindEQ__37_GenTrFun} 
 C_{0} (t+t')&=C_{0} (t)\, C_{0} (t')+{}_{0} c_{3} \left[C_{1} (t)\, C_{2} (t')+C_{1} (t')\, C_{2} (t)\right]+{}_{0} c_{4} C_{2} (t)\, C_{2} (t'), \\ 
 C_{1} (t+t')&=\left[C_{0} (t)\, C_{1} (t')+C_{1} (t)\, C_{0} (t')\right]+{}_{1} c_{3} \left[C_{1} (t)\, C_{2} (t')+C_{1} (t')\, C_{2} (t)\right]+\\
 & +{}_{1} c_{4} C_{2} (t)\, C_{2} (t'), \\ 
 C_{2} (t+t')&=\left[C_{0} (t)\, C_{2} (t')+C_{1} (t)\, C_{1} (t')+C_{2} (t)\, C_{0} (t')\right]+{}_{2} c_{3} \left[C_{1} (t)\, C_{2} (t')+\right. \\
 & \left. +   C_{1} (t')\, C_{2} (t)\right]+  {}_{2} c_{4} C_{2} (t)\, C_{2} (t'),
\end{split} \end{equation}
where ${}_{\alpha } c_{n} ,\, \alpha =0,\, 1,\, 2$ are the third order GTF with $e^{\lambda _{\alpha } t} $ replaced by $\lambda _{\alpha }^{n} $.\\

 More in general we also find that

\begin{equation} \label{GrindEQ__38_GenTrFun} 
C_{\alpha } (n\, t)=\sum _{\substack{ n_{1} ,n_{2} ,\, n_{3} =0 \\  n_{1} +n_{2} +n_{3} =n} }^{n}\binom{n}{n_{1}\;n_{2}\;n_{3}} \, {}_{\alpha } c_{n-n_{1} } C_{0}^{n_{1} } C_{1}^{n_{2} } C_{2}^{n_{3} } , 
\end{equation} 
with $\binom{n}{n_{1}\;n_{2}\;n_{3}} $ being the multinomial coefficient. 
\end{cor}

\begin{cor}
 It is also easily understood that the analogous of eqs. \ref{GrindEQ__9_GenTrFun}, \ref{GrindEQ__11_GenTrFun}  for the third order \textit{GTF}, read

\begin{equation}\begin{split} \label{GrindEQ__39_GenTrFun} 
& {}_{I} C_{0} (t)={}_{0} c_{-1} C_{0} (t)+C_{1} (t), \\ 
& {}_{I} C_{1} (t)=C_{2} (t)+{}_{1} c_{-1} C_{2}  \\ 
& {}_{I} C_{2} (t)={}_{2} c_{-1} C_{0} (t), \\ 
& \int _{0}^{\infty }dt C_{\alpha } (-t)={}_{\alpha } c_{-1} , \\ 
& \int _{-\infty }^{\infty }dt C_{\alpha } (-t^{2} )=\sqrt{\pi } {}_{\alpha } c_{-\frac{1}{2} } , \\ 
& \alpha =0,\, 1,\, 2,
\end{split} \end{equation}
where

\begin{equation}\begin{split} \label{GrindEQ__40_GenTrFun} 
& {}_{0} c_{\nu } =\frac{1}{\Delta (\lambda _{1} ,\lambda _{2} ,\lambda _{3} )} \sum _{i,j,k=1}^{3}\frac{\varepsilon _{i,j,k} }{2}  \lambda _{i} \lambda _{j} (\lambda _{j} -\lambda _{i} )\, \lambda _{k}^{\nu } , \\ 
& {}_{1} c_{\nu } =\frac{1}{\Delta (\lambda _{1} ,\lambda _{2} ,\lambda _{3} )} \sum _{i,j,k=1}^{3}\frac{\varepsilon _{i,j,k} }{2}  (\lambda _{i}^{2} -\lambda _{j}^{2} )\, \lambda _{k}^{\nu } , \\ 
& {}_{2} c_{\nu } =-\frac{1}{\Delta (\lambda _{1} ,\lambda _{2} ,\lambda _{3} )} \sum _{i,j,k=1}^{3}\frac{\varepsilon _{i,j,k} }{2}  (\lambda _{i} -\lambda _{j} )\, \lambda _{k}^{\nu } .
\end{split} \end{equation}	
\end{cor}

\begin{cor}
It is now worth stressing that the following identities hold true in the case of third order matrices expressed in terms of \textit{GTF}, namely 

\begin{equation} \label{GrindEQ__41_GenTrFun} 
\hat{M}^{n} ={}_{0} c_{n} \hat{1}+{}_{1} c_{n} \hat{M}+{}_{2} c_{n} \hat{M}^{2} . 
\end{equation} 

Let us now consider the possibility of extending the \textit{Courant-Snyder} parameterization to third order matrices. To this aim we set

\begin{equation} \label{GrindEQ__42_GenTrFun} 
\hat{\Sigma }=e^{\hat{T}} . 
\end{equation} 
The explicit form of the matrix $\hat{T}$ can be obtained by setting

\begin{equation} \label{GrindEQ__43_GenTrFun} 
\hat{\Sigma }=C_{0} (1)\, \hat{1}+C_{1} (1)\, \hat{T}+C_{2} (1)\, \hat{T}^{2} , 
\end{equation} 
where $ C_\alpha(1)$  are written in terms of the eigenvalues of the matrix $ \hat{T} $ according to the prescription discussed in sec. \ref{GTFMP}. Furthermore, since

\begin{equation} \label{GrindEQ__44_GenTrFun} 
\hat{\Sigma }^{-1} =C_{0} (-1)\, \hat{1}+C_{1} (-1)\, \hat{T}+C_{2} (-1)\, \hat{T}^{2} , 
\end{equation} 
the matrix $\hat{T}$ can be obtained as

\begin{equation} \label{GrindEQ__45_GenTrFun} 
\hat{T}=\frac{C_{2} (-1)\, \hat{\Sigma }-C_{2} (1)\, \hat{\Sigma }^{-1} +\left[C_{2} (-1)\, C_{0} (1)-C_{2} (1)\, C_{0} (-1)\right]\, \hat{1}}{C_{2} (-1)\, C_{1} (1)-C_{1} (-1)\, C_{2} (1)}  
\end{equation} 
\end{cor}

It is evident that the results we have obtained so far can be extended to an arbitrary  $n\times n$ matrix, it is however instructive to consider more specific examples involving particular cases as e.g. a $5\times 5$ anti-symmetric matrix $\hat{F}$, which can be exponentiated as it follows \cite{Jansen}

\begin{exmp}
	We consider
\begin{equation} \label{GrindEQ__46_GenTrFun} 
e^{t\, \hat{F}} =\hat{1}+\frac{1}{\sqrt{\Gamma } } \left[f_{1} (t)\, \hat{F}+f_{2} (t)\, \hat{F}^{2} +f_{3} (t)\hat{F}^{3} +f_{4} (t)\, \hat{F}^{4} \, \right] ,
\end{equation} 
where

\begin{equation}\begin{split} \label{GrindEQ__47_GenTrFun} 
 \Gamma &=Tr(\hat{F}^{4} )-\frac{1}{4} \left[Tr(\hat{F}^{2} )\right]^{2} , \\ 
 \theta _{\pm }^{2} &=-\frac{1}{4} Tr(\hat{F}^{2} )\pm \frac{1}{2} \sqrt{\Gamma } , \\ 
 f_{1} (t)&=\left(\frac{\sin (\theta _{-} t)}{\theta _{-} } \theta _{+}^{2} -\frac{\sin (\theta _{+} t)}{\theta _{+} } \theta _{-}^{2} \right), \\ 
 f_{2} (t)&=\left(\frac{1-\cos (\theta _{-} t)}{\theta _{-}^{2} } \theta _{+}^{2} -\frac{1-\cos (\theta _{+} t)}{\theta _{+}^{2} } \theta _{-}^{2} \right), \\ 
 f_{3} (t)&=\left(\frac{\sin (\theta _{-} t)}{\theta _{-} } -\frac{\sin (\theta _{+} t)}{\theta _{+} } \right), \\ 
 f_{4} (t)&=\left(\frac{1-\cos (\theta _{-} t)}{\theta _{-}^{2} } -\frac{1-\cos (\theta _{+} t)}{\theta _{+}^{2} } \right).
\end{split} \end{equation} 
We can provide the identification of the $f$ functions with the $ fifth $ order \textit{GTF}

\begin{equation}\begin{split} \label{GrindEQ__48_GenTrFun} 
& C_{0} (t)=1,\\
& C_{\alpha } (t)=\frac{1}{\sqrt{\Gamma } } f_{\alpha } (t), \\ 
& \alpha =1,...,4.
\end{split} \end{equation}	
It is also worth noting that, from the previous identities, the following relationships are easily inferred

\begin{equation}\begin{split} \label{GrindEQ__49_GenTrFun} 
& \hat{F}^{2\, n+1} =\frac{1}{\sqrt{\Gamma } } \left[{}_{1} f_{2\, n+1} \, \hat{F}+{}_{3} f_{2\, n+1} \hat{F}^{3} \, \right], \\ 
& \hat{F}^{2\, n} =\hat{1}+\frac{1}{\sqrt{\Gamma } } \left[{}_{2} f_{2\, n} \, \hat{F}^{2} +{}_{4} f_{2\, n} \hat{F}^{4} \, \right],
\end{split} \end{equation}	
where the coefficients

\begin{equation}\begin{split} \label{GrindEQ__50_GenTrFun} 
& {}_{1} f_{2\, n+1} =\left(\frac{\theta _{-}^{2\, n+1} }{\theta _{-} } \theta _{-}^{2} -\frac{\theta _{+}^{2\, n+1} }{\theta _{+} } \theta _{+}^{2} \right), \\ 
& {}_{3} f_{2\, n+1} =\left(\frac{\theta _{-}^{2\, n+1} }{\theta _{-} } -\frac{\theta _{+}^{2\, n+1} }{\theta _{+} } \right), \\ 
& {}_{2} f_{n} =\left(\frac{\theta _{-}^{2\, n} }{\theta _{-}^{2} } \theta _{+}^{2} -\frac{\theta _{+}^{2\, n} }{\theta _{+}^{2} } \theta _{-}^{2} \right), \\ 
& {}_{4} f_{n} =\left(\frac{\theta _{-}^{2\, n} }{\theta _{-}^{2} } -\frac{\theta _{+}^{2\, n} }{\theta _{+}^{2} } \right)
\end{split} \end{equation}
play a role analogous to that of ${}_{\alpha } c_{n} $ introduced in the previous sections.
\end{exmp}

\subsection{Miscellaneous  Considerations on the GTF}

In the previous sections  we have introduced the properties of the auxiliary coefficients $c_{n} $ and $ s_{n} $, their role is fairly important within the present context and warrants further analysis.\\

\noindent To this aim we note that they satisfy the following recurrences

\begin{propert}
	The $c_{n} $ and $ s_{n} $ properties
\begin{equation}\begin{split} \label{GrindEQ__51_GenTrFun} 
& \left(\begin{array}{c} {c_{n+1} } \\ {s_{n+1} } \end{array}\right)=\left(\begin{array}{cc} {0} & {-\Delta _{\hat{M}} } \\[1.5ex] {1} & {Tr(\hat{M})} \end{array}\right)\, \left(\begin{array}{c} {c_{n} } \\ {s_{n} } \end{array}\right), \\ 
& \left(\begin{array}{c} {c_{0} } \\ {s_{0} } \end{array}\right)=\left(\begin{array}{c} {1} \\ {0} \end{array}\right),
\end{split} \end{equation} 
 follow from the identities                         

\begin{equation}\begin{split} & \hat{M}^{n+1} =c_{n+1} \hat{1}+s_{n+1} \hat{M}, \\[1.1ex] 
& \hat{M}^{n+1} =-\Delta _{M} s_{n} \hat{1}+\left[c_{n} +Tr(\hat{M})\, s_{n} \right]\, \hat{M}.
\end{split} \end{equation}  
The above recurrences can be cast in the decoupled form

\begin{equation}\label{recGenTrFun}
c_{n+2} -Tr(\hat{M})\, c_{n+1} +\Delta _{\hat{M}} \, c_{n} =0.
\end{equation}
For $s_{n} $ we find an analogous expression. 
\end{propert}              

\begin{cor}
 The solution of the difference equation \ref{recGenTrFun} can be obtained by the use of the \textit{Binet} method \cite{Vajda}, after setting $c_{n} =r^{n} $ we find indeed

\begin{equation} \label{GrindEQ__54_GenTrFun} 
c_{n} =\alpha _{1} r_{+}^{n} +\alpha _{2} r_{-}^{n} , 
\end{equation} 
with $r_{\pm } $ being solutions of the auxiliary equation\\

\begin{equation}
r^{2} -Tr(\hat{M})r+\Delta _{\hat{M}} =0
\end{equation}
and $\alpha _{1,2} $ being defined through the "initial constants'' $c_{0,1} $.\\

Accordingly we obtain

\begin{equation} 
c_{n} =\frac{1}{r_{-} -r_{+} } \left[c_{0} (r_{-} r_{+}^{n} -r_{+} r_{-}^{n} )+c_{1} \left(r_{-}^{n} -r_{+}^{n} \right)\right].
\end{equation} 
It is also interesting to note that, by rescaling $ n = m-2 $, eq. \ref{recGenTrFun} writes

\begin{equation}\label{cmGenTrFun}
c_{m} =-\Delta _{\hat{M}} \, c_{m-2} +Tr(\hat{M})\, c_{m-1} .
\end{equation}
Eq. \ref{cmGenTrFun}, for  $Tr(\hat{M})=1$  and    $\Delta _{\hat{M}} =-1$  (e.g. the eigenvalues of $\hat{M}$ are the golden ratio and the opposite of the golden ratio conjugate), reduces to the \textit{Fibonacci} sequence. 
\end{cor}

These coefficients play a more general role when extended to the case of higher order matrices and the systematic study of their properties may simplify the calculations of problems where exponentiation of matrices are involved.\\

\noindent In the past, different generalizations of the trigonometric functions have been proposed, in addition to those quoted in this paragraph, different avenues have been explored along this direction. The tool exploited within such a framework can be comprised into three different branches:\\

$a)$ Use of matrix methods and generalization of the Euler exponential rule.\\

$b)$ Extension of the trigonometric fundamental identity, providing a thread with elliptic functions \cite{Ferrari}.\\

$c)$ Generalized forms of the series expansion, providing a link with integer order \textit{Mittag-Leffler} function \cite{Ricci,Cheick,Ansari}.\\

\noindent This last point of view provides a significant step forward in the theory of special functions, yielding a tool for applications in the field of classical and quantum optics \cite{Nieto,Sun,Renieri}.\\

\noindent Preliminary attempts to merge the points of view $a)$ and $c)$ have been put forward in refs. \cite{DMR,DSDF}.\\

\noindent Even though the matter presented in this section may sound abstract there are important applications in beam transport optics as illustrated below.

\subsubsection{Application}

The use of $4\times 4$ matrices is currently employed to deal with transverse coupling in charged beam transport \cite{Qin}. \textit{Baumgarten} \cite{Baumgarten} has proposed the use of real \textit{Dirac} matrices \cite{Majorana} to construct a generalization of the one dimensional \textit{Courant-Snyder} theory of beam transport.\\

 Within such a context the beam transport through a solenoid can be written as

\begin{equation} \label{GrindEQ__57_GenTrFun} 
\frac{d}{ds} \left(\begin{array}{c} {x} \\[1.2ex] {\dfrac{x'}{K} } \\[1.5ex] {y} \\[1.2ex] {\dfrac{y'}{K} } \end{array}\right)=K\left(\begin{array}{cccc} {0} & {1} & {1} & {0} \\[2ex] {-1} & {0} & {0} & {1} \\[2ex] {-1} & {0} & {0} & {1} \\[2ex] {0} & {-1} & {-1} & {0} \end{array}\right)\left(\begin{array}{c} {x} \\[1.2ex] {\dfrac{x'}{K} } \\[1.5ex] {y} \\[1.2ex] {\dfrac{y'}{K} } \end{array}\right) ,
\end{equation} 
where $K$ is the solenoid strength and the column vector is represented by the position an velocity for the transverse coordinates $(x,y)$, finally $s$ is the propagation coordinate, playing the role of time.\\

\noindent The solution of the previous system of differential equation can be written as

\begin{equation}\begin{split} \label{GrindEQ__58_GenTrFun} 
& \underline{Z}=\hat{U}(s)\underline{Z}_{0} , \\ & \hat{U}(s)=e^{Ks\hat{T}},  \\[1.1ex] 
& \hat{T}=\left(\begin{array}{cccc} {0} & {1} & {1} & {0} \\ {-1} & {0} & {0} & {1} \\ {-1} & {0} & {0} & {1} \\ {0} & {-1} & {-1} & {0} \end{array}\right).
\end{split} \end{equation} 
The use of the techniques outlined in the previous section yields for the evolution operator

\begin{equation} \label{GrindEQ__59_GenTrFun} 
\hat{U}(s)=\hat{1}+\frac{\sin (2Ks)}{2} \hat{T}+\frac{1-\cos (2Ks)}{4} \hat{T}^{2} -\frac{\sin (2Ks)}{16} \hat{T}^{3},  
\end{equation} 
however the above expression simplifies since  $\hat{T}^{3} =-4\, \hat{T}$.\\

\noindent In this case the  \textit{GTF}  are simple combinations of the ordinary circular functions.\\

\noindent The method proposed is however fairly important because the (sixteen) real \textit{Majorana} matrices provide a basis for the $4\times 4$ matrices and it could be interesting to develop a systematic study within the context of \textit{GTF} viewed as the associated auxiliary functions. The relevant applications might be interesting for a plethora of problems including e.g. four level systems interacting with external radiation.\\

We conclude this part by noting that, in terms of  the \textit{Majorana} matrices, the solenoid transport matrix reads

\begin{equation}\begin{split} \label{GrindEQ__60_GenTrFun} 
& \hat{T}=\gamma _{0} -\gamma _{9} , \\ 
& \gamma _{0} =\left(\begin{array}{cccc} {0} & {1} & {0} & {0} \\ {-1} & {0} & {0} & {0} \\ {0} & {0} & {0} & {1} \\ {0} & {0} & {-1} & {0} \end{array}\right),\,\;\; \gamma _{9} =\left(\begin{array}{cccc} {0} & {0} & {-1} & {0} \\ {0} & {0} & {0} & {-1} \\ {1} & {0} & {0} & {0} \\ {0} & {1} & {0} & {0} \end{array}\right).
\end{split} \end{equation} 
Since $\gamma _{0} ,\, \gamma _{9} $ are commuting quantities, we can also write 

\begin{equation}\begin{split} \label{GrindEQ__61_GenTrFun} 
& e^{\hat{T}\, \xi } =e^{\hat{\gamma }_{0} \xi } e^{-\hat{\gamma }_{9} \xi } , \\ 
& e^{\hat{\gamma }_{0,9} \xi } =\cos (\xi )\, \hat{1}+\sin (\xi )\, \hat{\gamma }_{0,9}
\end{split} \end{equation}  
and easily recover the result reported in eq. \ref{GrindEQ__59_GenTrFun}.

\section{Evolution Equations Involving Matrices Raised to Non-Integer Exponents}

The use of matrices evolution equations raised to non-integer exponents finds applications in problems involving the solution of two or three level systems ruled by Klein-Gordon type equations \cite{Baym}. We develop a fairly simple method exploiting the wealth of results obtained on fractional calculus and provide an example of application.\\

We discuss the solution of matrix evolution equations, using the formalism of fractional operators, namely of operators raised to a non-integer exponent.\\

We introduce the topics, their scope and the formalism we will employ, by discussing a fairly simple example regarding the solution of the \textbf{second order differential equation }.

\begin{exmp}
We consider the differential problem
\begin{equation} \label{GrindEQ__1_EvolMatr} 
\left\lbrace  \begin{array}{l} \dfrac{d^{2} }{dt^{2} } \; \underline{Y}=-\hat{A}\, \underline{\, Y} \\[1.1ex]
 \underline{Y}(0)=\underline{Y}_{0} \\[1.1ex]
 \dfrac{d}{dt} \; \underline{Y}|_{t=0}=\underline{\dot{Y}}_{0} ,
\end{array}\right. 
 \end{equation} 
where $\hat{A}$ is a non-singular $2\times 2$ matrix, with positive defined determinant

\begin{equation} \label{GrindEQ__2_EvolMatr} 
\hat{A}=\left(\begin{array}{cc} {a} & {b} \\ {c} & {d} \end{array}\right) ,
\end{equation} 
$\underline{Y}=\left(\begin{array}{c} {y_{1} } \\ {y_{2} } \end{array}\right)$ a two component column vector and  $\underline{Y}_{0} $,  $\underline{\dot{Y}}_{0} $ the initial conditions of the problem.\\

 The solution of eq. \ref{GrindEQ__1_EvolMatr} can be obtained by standard means, namely by introducing a further component $\underline{W}=\dfrac{d}{dt}\; \underline{Y}$, thus transforming it into a first order differential equation involving $4\times 4$ matrices. The procedure we follow foresees different means involving the use of square root matrices. Treating the eq. \ref{GrindEQ__1_EvolMatr} as a standard harmonic oscillator equation, we write the relevant solution as

\begin{equation} \label{GrindEQ__3_EvolMatr} 
\underline{Y}(t)=e^{i\, t\, \sqrt{\hat{A}} } \underline{c}_{1} +e^{-i\, t\, \sqrt{\hat{A}} } \underline{c}_{2} , 
\end{equation} 
where $\underline{c}_{1,2} $ are column vectors depending on the initial conditions and will be specified later.\\

The problem we are facing with is that of providing an operational meaning to an exponential operator containing the square root of a $2\times 2$ matrix. We proceed therefore as it follows:

\begin{itemize}
	
	\item [$a)$] Use standard matrix algebra \cite{Birkhoff} to write
	
	\begin{equation}\begin{split}\label{GrindEQ__4_EvolMatr} 
	& \sqrt{\hat{A}} =-\dfrac{\lambda _{-} \lambda _{+} }{\lambda _{-} -\lambda _{+} } \left(\dfrac{1}{\sqrt{\lambda _{-} } } -\dfrac{1}{\sqrt{\lambda _{+} } } \right)\, \hat{1}+\dfrac{\sqrt{\lambda _{-} } -\sqrt{\lambda _{+} } }{\lambda _{-} -\lambda _{+} } \hat{A}, \\ 
	& f_{0} (\lambda _{+} ,\, \lambda _{-} )=-\dfrac{\lambda _{-} \lambda _{+} }{\lambda _{-} -\lambda _{+} } \left(\dfrac{1}{\sqrt{\lambda _{-} } } -\dfrac{1}{\sqrt{\lambda _{+} } } \right), \\ 
	& f_{1} (\lambda _{+} ,\, \lambda _{-} )=\dfrac{\sqrt{\lambda _{-} } -\sqrt{\lambda _{+} } }{\lambda _{-} -\lambda _{+} } ,
	\end{split} \end{equation} 
	with $\hat{1}$, $\lambda _{\pm } $ being the unit matrix and the eigenvalues of the matrix $\hat{A}$.
	\item [$b)$] Write then
	
	\begin{equation} \label{GrindEQ__5_EvolMatr} 
	e^{\tau \, \sqrt{\hat{A}} } =e^{f_{0} (\lambda _{+} ,\lambda _{-} )\tau \, \hat{1}} e^{f_{1} (\lambda _{+} ,\lambda _{-} )\, \tau \, \hat{A}} . 
	\end{equation} 
	
	\item [$c)$] Use the Cayley-Hamilton theorem to write the explicit form of the matrix exponential as (see refs. \cite{Birkhoff}, \cite{Babusci})
	
	\begin{equation}\begin{split}\label{GrindEQ__6_EvolMatr} 
	 e^{\tau \, \sqrt{\hat{A}} } &=e^{f_{0} (\lambda _{+} ,\lambda _{-} )\tau } \left(\begin{array}{cc} {1} & {0} \\ {0} & {1} \end{array}\right)\left(\begin{array}{cc} {e_{1,1} } & {e_{1,2} } \\ {e_{2,1} } & {e_{2,2} } \end{array}\right), \\ 
	 e_{1,1} &=\left[\dfrac{(a-d)\, \tau }{\sqrt{\Delta } } f_{1} (\lambda _{+} ,\lambda _{-} )\sinh \left(\dfrac{\sqrt{\Delta } }{2} \right)+\cosh \left(\dfrac{\sqrt{\Delta } }{2} \right)\right]\cdot\\
	& \cdot e^{\frac{1}{2} f_{1} (\lambda _{+} ,\lambda _{-} )\, \left(a+d\right)\tau } , \\ 
	 e_{2,2} &=\left[-\dfrac{(a-d)\, \tau }{\sqrt{\Delta } } f_{1} (\lambda _{+} ,\lambda _{-} )\sinh \left(\dfrac{\sqrt{\Delta } }{2} \right)+\cosh \left(\dfrac{\sqrt{\Delta } }{2} \right)\right]\cdot\\
	 & \cdot e^{\frac{1}{2} f_{1} (\lambda _{+} ,\lambda _{-} )\, \left(a+d\right)\tau } , \\ 
	 e_{1,2} &=f_{1} (\lambda _{+} ,\lambda _{-} )\dfrac{2\, b\tau }{\sqrt{\Delta } } \sinh \left(\dfrac{\sqrt{\Delta } }{2} \right)e^{\frac{1}{2} f_{1} (\lambda _{+} ,\lambda _{-} )\, \left(a+d\right)\tau } ,\\ 
	 e_{2,1} &=f_{1} (\lambda _{+} ,\lambda _{-} )\dfrac{2\, c\tau }{\sqrt{\Delta } } \sinh \left(\dfrac{\sqrt{\Delta } }{2} \right)e^{\frac{1}{2} f_{1} (\lambda _{+} ,\lambda _{-} )\, \left(a+d\right)\tau } ,\\ 
	 \Delta &=\left(f_{1} (\lambda _{+} ,\lambda _{-} )\, \tau \right)^{2} \left((a-d)^{2} +4bc\right).
	\end{split} \end{equation}
	
\end{itemize}

\noindent The solution of eq. \ref{GrindEQ__1_EvolMatr} can accordingly be written as

\begin{equation}\begin{split}\label{GrindEQ__7_EvolMatr} 
& \underline{Y}(t)=C\left(t\, \sqrt{\hat{A}} \right)\, \underline{Y}_{0} +\frac{1}{\sqrt{\hat{A}} } S\left(t\, \sqrt{\hat{A}} \right)\underline{\dot{Y}}_{0} , \\ 
& C\left(t\, \sqrt{\hat{A}} \right)=\frac{e^{i\, t\, \sqrt{\hat{A}} } +e^{-i\, t\, \sqrt{\hat{A}} } }{2} , \\ 
& S\left(t\, \sqrt{\hat{A}} \right)=\frac{e^{i\, t\, \sqrt{\hat{A}} } -e^{-i\, t\, \sqrt{\hat{A}} } }{2i} ,
\end{split} \end{equation} 
where $C\left(t\, \sqrt{\hat{A}} \right),\, S\left(t\, \sqrt{\hat{A}} \right)$ are pseudo oscillating cos and sin-like solutions.\\

A graphical example of such solution is shown in Fig. \ref{fig1EvolMatr}, where we have reported the time behavior of the two components $y_{1,2} $ for different values of the matrix entries: a) $\hat{A}_{1}=\left(\begin{array}{cc} {1} & {0.5} \\ {0.5} & {1} \end{array}\right)$;
\; b) $\hat{A}_{2}=\left(\begin{array}{cc} {1.7} & {0.5} \\ {0.5} & {1.8} \end{array}\right)$.\\

\begin{figure}[htp]
	\centering
	\begin{subfigure}[b]{0.47\columnwidth}
		\includegraphics[width=.9\textwidth]{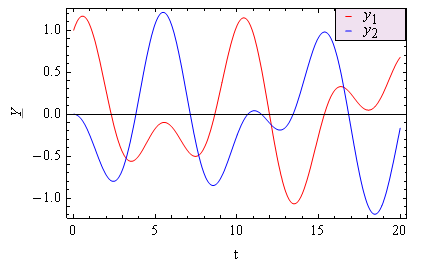}
		\caption{$\hat{A}_{1}$}
		\label{A1EvolMatr}
	\end{subfigure}
	\begin{subfigure}[b]{0.47\columnwidth}
		\includegraphics[width=.9\textwidth]{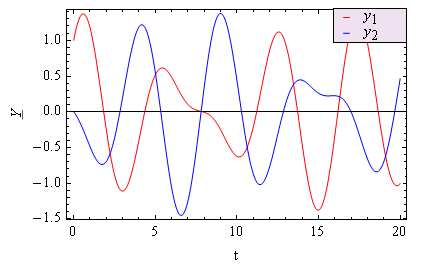}
		\caption{$\hat{A}_{2}$}
		\label{A2EvolMatr}
	\end{subfigure}
	\caption{$\underline{Y}$  vs. time; in red and blue the components $y_{1} $ , $y_{2} $ respectively, with initial conditions
		$\underline{Y}_{0}=\binom{1}{0}$ and
		$\underline{\dot{Y}}_{0}=\binom{-0.1}{0.2}$.}\label{fig1EvolMatr}  
\end{figure}  
\begin{figure}[htp]
	\centering
	\begin{subfigure}[b]{0.47\columnwidth}
		\includegraphics[width=.7\textwidth]{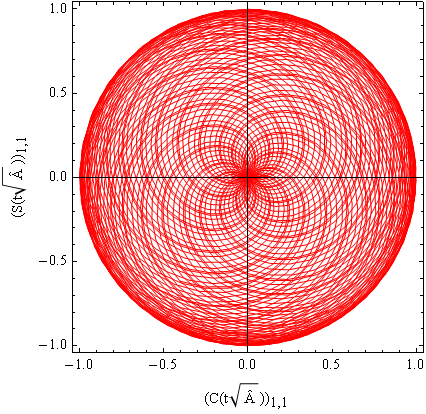}
		\caption{$\hat{A}_{1}$}
		\label{A1EvolMatr}
	\end{subfigure}
	\begin{subfigure}[b]{0.47\columnwidth}
		\includegraphics[width=.7\textwidth]{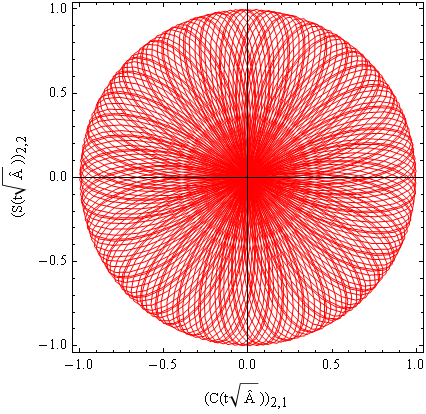}
		\caption{$\hat{A}_{1}$}
		\label{A1bisEvolMatr}
	\end{subfigure}
	\begin{subfigure}[b]{0.47\columnwidth}
		\includegraphics[width=.7\textwidth]{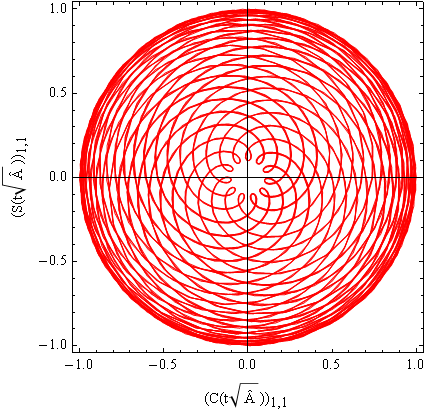}
		\caption{$\hat{A}_{2}$}
		\label{A2EvolMatr}
	\end{subfigure}
	\begin{subfigure}[b]{0.47\columnwidth}
		\includegraphics[width=.7\textwidth]{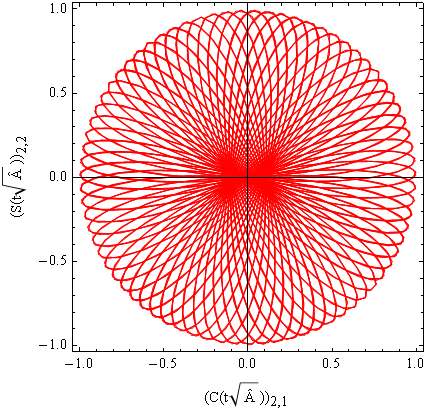}
		\caption{$\hat{A}_{2}$}
		\label{A2bisEvolMatr}
	\end{subfigure}
	\caption{$S\left(t\, \sqrt{\hat{A}} \right)$ vs $C\left(t\, \sqrt{\hat{A}} \right)$, with different components combinations for $\hat{A}_{1}$ (a-b) and $\hat{A}_{2}$ (c-d).}
	\label{figSilviaEvolMatr}  
\end{figure} 
For completness sake we have reported in Figs. \ref{figSilviaEvolMatr}  the so-called Lissajous curves arising from plotting $S\left(t\, \sqrt{\hat{A}} \right)$ vs $C\left(t\, \sqrt{\hat{A}} \right)$.
\end{exmp} 

We will discuss the application of the method to a physical problem later in this Chapter and concentrate on further refinement of the relevant mathematical details. In the following sections,  we will extend the method including the integral transform technique and the square root of to higher order matrices, we will discuss a relevant application  and will comment further possible developments.

\subsection{Fractional Matrix Exponentiation}
\markboth{\textsc{Fractional Matrix exponentiation}}{}

The results so far obtained will be further elaborated, including the extension to higher dimensions but, before addressing this specific aspect of the discussion, we complement the previous treatment by using techniques developed within the context of fractional calculus \cite{R.Hermann} and more specifically with the context of evolution problems regarding the solution of evolution equations, like \textit{Bethe-Salpeter }or other relativistic forms, involving square roots of differential operators \cite{Quattromini}.\\

\noindent The theory of differ-integral calculus, namely of derivatives and integrals of non-integer order, has received a significant support from methods associated with the Laplace transform  and, within such a framework, the use of Lévy transform \cite{Penson} has been proved to be an effective tool to deal with exponentials with arguments consisting of fractional operators. To this aim we remind that \cite{Borel}

\begin{equation}\begin{split} \label{GrindEQ__8_EvolMatr} 
& e^{-p^{\frac{1}{2} } } =\int _{0}^{\infty }g_{\frac{1}{2} }  (\eta )\, e^{-\eta \, p} d\eta , \\ 
& g_{\frac{1}{2} } (\eta )=\frac{1}{2\, \sqrt{\pi \eta ^{3} } } e^{-\; \frac{1}{4\, \eta } } ,
\end{split} \end{equation} 
with $g_{\frac{1}{2} } (\eta )$ being the Lévy-Smirnov distribution (used e.g. in \cite{Fisher} and later in Chapter \ref{ChapterNumberTh}).

\begin{defn}
  According to eq. \ref{GrindEQ__8_EvolMatr} we can express the \textbf{exponential of a square root matrix} in terms of the integral transform

\begin{equation}\label{GrindEQ__9_EvolMatr}
e^{-\tau \, \sqrt{\hat{A}} } =\int _{0}^{\infty }g_{\frac{1}{2} }  (\eta )\, e^{-\eta \, \tau ^{2} \hat{A}} d\eta.  
\end{equation} 
\end{defn}
The advantage offered by eq. \ref{GrindEQ__9_EvolMatr} is that the exponential inside the integral \ref{GrindEQ__9_EvolMatr} depends on the matrix $\hat{A}$ without any further exponentiation, this allows a direct use of the Cayley-Hamilton theorem to solve the problem of getting an explicit expression for the l.h.s. of eq. \ref{GrindEQ__9_EvolMatr}. Without considering the general case, reported in section \ref{FCEv}, we discuss a few interesting examples.

\begin{exmp}
	We consider the matrix
\begin{equation} \label{GrindEQ__10_EvolMatr} 
\hat{A}=\hat{i}=\left(\begin{array}{cc} {0} & {-1} \\ {1} & {0} \end{array}\right) ,
\end{equation} 
namely the "unit circular'' matrix satisfying the identity

\begin{equation} \label{GrindEQ__11_EvolMatr} 
\hat{i}^{2} =-\hat{1} 
\end{equation} 
and generalizing the Euler identity

\begin{equation}
e^{\hat{i}\vartheta \, } =\cos (\vartheta )\, \hat{1}+\sin (\vartheta )\, \hat{i}.
\end{equation}
The use of eq. \ref{GrindEQ__6_EvolMatr} yields

\begin{equation}
e^{-\tau \sqrt{\hat{i}} } =e^{-\, \frac{\sqrt{2} }{2} \tau } \left(\begin{array}{cc} {\cos \left(\frac{\sqrt{2} }{2} \tau \right)} & {-\sin \left(\frac{\sqrt{2} }{2} \tau \right)} \\ {\sin \left(\frac{\sqrt{2} }{2} \tau \right)} & {\cos \left(\frac{\sqrt{2} }{2} \tau \right)} \end{array}\right),
\end{equation}
representing a kind of \textbf{damped matrix rotation}.
\end{exmp}

\begin{exmp}
Using an almost similar argument we find that  the square root of the ``unit hyperbolic matrix''

\begin{equation} \label{GrindEQ__13_EvolMatr} 
\hat{h}=\left(\begin{array}{cc} {0} & {1} \\ {1} & {0} \end{array}\right) 
\end{equation} 
can be written as

\begin{equation}
\sqrt{\hat{h}} =\frac{1}{2} \left[\left(1+i\right)\, \hat{1}+\left(1-i\right)\, \hat{h}\right]  .
\end{equation}              
The solution of equations  like

\begin{equation}
\frac{d^{2} }{d\, \tau ^{2} } \underline{Y}=\hat{h}\, \underline{Y}
\end{equation}
can accordingly be written as

\begin{equation}\begin{split} \label{GrindEQ__15_EvolMatr} 
&\underline{Y}(\tau )=\left(\begin{array}{cc} {\cosh (\tau _{+} )\cosh (\tau _{-} )} & {\sinh (\tau _{+} )\sinh (\tau _{-} )} \\ {\sinh (\tau _{+} )\sinh (\tau _{-} )} & {\cosh (\tau _{+} )\cosh (\tau _{-} )} \end{array}\right)\, \underline{Y}_{0} + \\ 
& \;\;\;+\frac{1}{\sqrt{\hat{h}} } \left(\begin{array}{cc} {\sinh (\tau _{+} )\cosh (\tau _{-} )} & {\cosh (\tau _{+} )\sinh (\tau _{-} )} \\ {\cosh (\tau _{+} )\sinh (\tau _{-} )} & {\sinh (\tau _{+} )\cosh (\tau _{-} )} \end{array}\right)\, \underline{\dot{Y}}_{0} , \\[1.1ex]
& \tau _{\pm } =\frac{1}{2} \left(1\pm i\right)\, \tau .
\end{split} \end{equation} 
\end{exmp}

It is evident that the technique we have envisaged can be extended to higher order matrices, the only problem is that the procedure becomes slightly more cumbersome from the analytical point of view, but it is easily implemented with Mathematica\texttrademark.\\

\begin{exmp}
In the case in which $\hat{A}$ is a $3\times 3$ matrix we find

\begin{equation} \label{GrindEQ__16_EvolMatr} 
e^{\tau \, \sqrt{\hat{A}} } =e^{f_{0} (\lambda _{1} ,\lambda _{2} ,\, \lambda _{3} )\tau \, \hat{1}} e^{f_{1} (\lambda _{1} ,\lambda _{2} ,\lambda _{3} )\, \tau \, \hat{A}} e^{f_{2} (\lambda _{1} ,\lambda _{2} ,\lambda _{3} )\, \tau \, \hat{A}^{2} },  
\end{equation} 
with $\lambda _{1,2,3} $ being the associated eigenvalues and

\begin{equation}
\left(\begin{array}{c} {f_{0} } \\ {f_{1} } \\ {f_{2} } \end{array}\right)=\left(\begin{array}{ccc} {1} & {\lambda _{1} } & {\lambda _{1}^{2} } \\ {1} & {\lambda _{2} } & {\lambda _{2}^{2} } \\ {1} & {\lambda _{3} } & {\lambda _{3}^{2} } \end{array}\right)^{-1} \left(\begin{array}{c} {\sqrt{\lambda _{1} } } \\ {\sqrt{\lambda _{2} } } \\ {\sqrt{\lambda _{3} } } \end{array}\right) .
\end{equation}   

The explicit form of the matrix can be written in terms of the exponential of the matrix $\hat{A}$ if we use the identity

\begin{equation} \label{GrindEQ__18_EvolMatr} 
e^{f_{2}\;\tau  \hat{A}^{2} } =\frac{1}{\sqrt{\pi } } \int _{-\infty }^{+\infty }e^{-\xi ^{2} +2\xi \sqrt{f_{2} \tau } \hat{A}}  d\xi  ,
\end{equation} 
which yields

\begin{equation}
e^{\tau \, \sqrt{\hat{A}} } =\frac{e^{\tau \, f_{0} } }{\sqrt{\pi } } \left(\begin{array}{ccc} {1} & {0} & {0} \\ {0} & {1} & {0} \\ {0} & {0} & {1} \end{array}\right)\int _{-\infty }^{+\infty }e^{-\xi ^{2} +(f_{1} \tau +2\xi \sqrt{f_{2} \tau } )\hat{A}}  d\xi  .
\end{equation}

A fairly simple example showing the effectiveness of the method is provided by the matrix

\begin{equation} \label{GrindEQ__20_EvolMatr} 
\hat{A}=\left(\begin{array}{ccc} {0} & {-\omega _{3} } & {\omega _{2} } \\ {\omega _{3} } & {0} & {-\omega _{1} } \\ {-\omega _{2} } & {\omega _{1} } & {0} \end{array}\right) 
\end{equation} 
whose exponentiation yields a \textbf{Rodrigues  matrix}  $\hat{R}$ \cite{Birkhoff}, \cite{Murray} 

\begin{equation}\begin{split}\label{GrindEQ__21_EvolMatr} 
& e^{\vartheta \, \hat{A}} =\hat{R}(\vartheta )=\hat{1}+\, \vartheta \, {\rm sinc}\left(\Omega \, \vartheta \right)\hat{A}+\frac{1}{2} \vartheta ^{2} \left[{\rm sinc}\left(\frac{\Omega \, \vartheta }{2} \right)\right]^{2} \hat{A}^{2} , \\ 
&\Omega =\sqrt{\omega _{1}^{2} +\omega _{2}^{2} +\omega _{3}^{2} } .
\end{split} \end{equation} 

The use of the previous identity and a Gaussian integration, finally yields the following factorization

\begin{equation}
e^{\tau \, \sqrt{\hat{A}} } =e^{\tau \, f_{0} }   \, \left[ \hat{1}+\left( \dfrac{1-e^{-\Omega ^{2} f_{2} \tau }}{\Omega^{2}} \right) \hat{A}^{2} \right] \, \hat{R}(\tau \, f_{1})  .
\end{equation}    
\end{exmp}                                
   
The extension to higher order matrices will be discussed in the final section containing further comments on the technique we have proposed.

\subsection{A Physical Application}
\markboth{\textsc{Physical Applications}}{} 

As previously noted, the use of pseudo differential operators for the solution of \textbf{Klein-Gordon or Bethe-Salpeter equations} has gained considerable interest in the last few years , we use the methods discussed in the previous sections to study the solution of the stationary Klein-Gordon equation describing the axion photon coupling in a transverse
magnetic field. According to the analysis developed in ref. \cite{Rafelt} photon $\gamma$ and axion $a$ fields can be viewed as two state polarization, which, propagating in an intense magnetic field, undergo a kind of Cotton-Mouton rotation. Within such a framework the photon-axion interaction can be viewed as a kind of Primakoff process (see Fig. \ref{fig2EvolMatr}) in which the vertex coupling occur between a real external photon, the virtual photon associated with the static magnetic field and the axion.\\

\begin{figure}[htp]
	\centering
	\includegraphics[width=.4\textwidth]{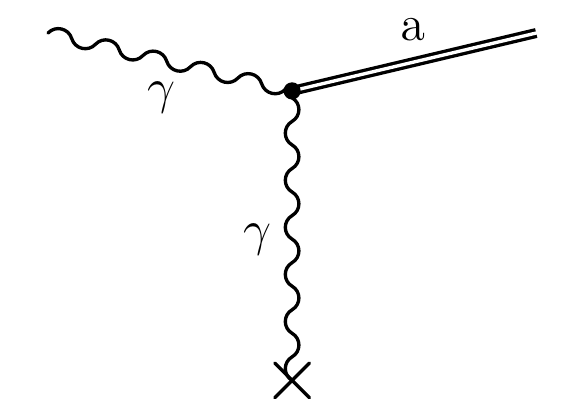}
	\caption{Primakoff process for photon ($\gamma$) - photon ($\gamma$) - axion ($a$) production .}\label{fig2EvolMatr}
\end{figure} 

The equation we consider is \cite{Rafelt} 

\begin{equation} \label{GrindEQ__23_EvolMatr} 
\left(\dfrac{\partial }{\partial \, t} \right)^{2} \left(\begin{array}{c} {\gamma} \\ {a} \end{array}\right)=-\omega ^{2} \left(\begin{array}{cc} {2n-1} & {\dfrac{g_{a \gamma } B}{\, \omega } } \\ {\dfrac{g_{a \gamma } B}{\, \omega } } & {1-\dfrac{m_{a}^{2} }{\omega ^{2} } } \end{array}\right)\left(\begin{array}{c} {\gamma} \\ {a} \end{array}\right) ,
\end{equation} 
where $n$ is the refractive photon index associated to the magnetic field $B$, $g_{a \gamma } $ the axion photon coupling constant,  $m_{a} $ is the axion mass and $\omega $ is the external photon frequency. Without further considering the specific details, we find that the eigenvalues of the matrix are

\begin{equation}\begin{split}
\label{GrindEQ__24_EvolMatr} 
& \lambda _{\pm } =\frac{N+M\pm \sqrt{\left[N-M\right]^{2} +4G^{2} } }{2} , \\ 
& N=2\, n-1,\\
& M=1-\left(\frac{m_{a} }{\omega } \right)^{2} ,\\
& G=\frac{g_{a \gamma } B}{\omega } .
\end{split} \end{equation} 
If we make the assumption that $\dfrac{m_{a} }{\omega } <<1$ (which means that the energy of the external photon is much larger than the axion mass \cite{Olive}) and $N\cong 1$, we are left with $\lambda _{\pm } =1\pm G$. Furthermore, by noting also that $G<<1$, we find

\begin{equation}\begin{split} \label{GrindEQ__25_EvolMatr} 
& f_{0} (\lambda _{+} ,\, \lambda _{-} )=\frac{1-G^{2} }{2G} \left(\frac{1}{\sqrt{1-G} } -\frac{1}{\sqrt{1+G} } \right)\cong \frac{1}{2}-\frac{3}{16} G^{2} ,  \\[1.1ex] 
& f_{1} (\lambda _{+} ,\, \lambda _{-} )=\frac{\sqrt{1-G} -\sqrt{1+G} }{-2G} \cong \frac{1}{2} +\frac{G^{2} }{16} 
\end{split} \end{equation} 
and

\begin{equation} \label{GrindEQ__26_EvolMatr} 
e^{i\omega \, t\, \sqrt{\hat{A}} } =e^{-i\omega t\left[(1-\frac{1}{8} G^{2} )\right]} \left(\begin{array}{cc} {\cos (f_{1} G\, \omega \, t)} & {-\sin (f_{1} G\omega \, t)} \\ {\sin (f_{1} G\omega \, t)} & {\cos (f_{1} G\omega \, t)} \end{array}\right), 
\end{equation} 
which is essentially a rotation matrix induced by the axion-photon coupling constant. The validity of our solution is limited to the case in which both $\lambda _{\pm } $ are non-negative, therefore the following conditions are to be satisfied

\begin{equation}
m_{a} \le \omega \, \sqrt{1-\dfrac{1}{2n-1}\left( \dfrac{g_{a\gamma } B}{ \omega }\right) ^{2} } .
\end{equation}                  
It is interesting to look at the behavior vs. time of the probability of creating an axion during the interaction for different values of the parameter. In Fig. \ref{fig4EvolMatr} we have shown $\left|a\right|^{2} $ vs $\omega \, t$ for the cases

\begin{enumerate}
	\item [1)] $\dfrac{g_{\alpha \gamma } B}{\, \omega } \cong 2\cdot 10^{-9} ,\qquad \left(\dfrac{m_{a} }{\omega } \right)\cong 0.3$;
	
	\item [2)] $\dfrac{g_{\alpha \gamma } B}{\, \omega } \cong 10^{-8} ,\qquad \left(\dfrac{m_{a} }{\omega } \right)\cong 0.71$.
\end{enumerate}

\begin{figure}[htp]
	\centering
	\begin{subfigure}[b]{0.47\columnwidth}
		\includegraphics[width=.9\textwidth]{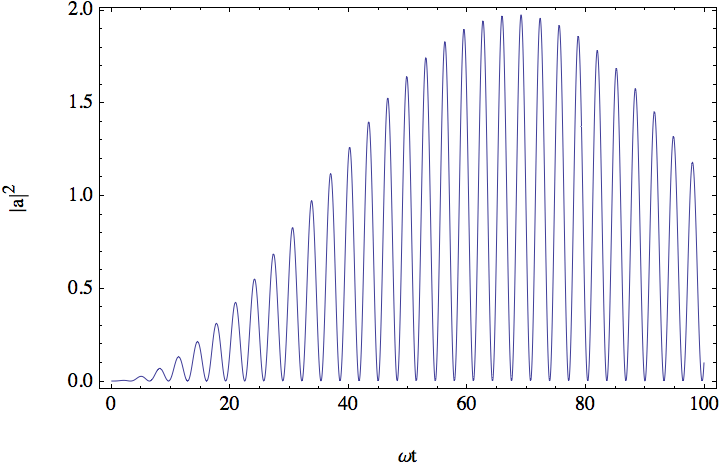}
		\caption{Case 1).}
		\label{FigAss1EvolMatr}
	\end{subfigure}
	\begin{subfigure}[b]{0.47\columnwidth}
		\includegraphics[width=.9\textwidth]{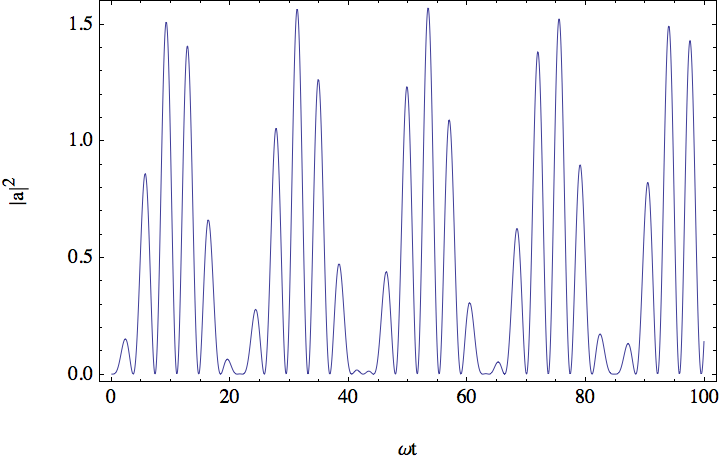}
		\caption{Case 2).}
		\label{FigAssC2EvolMatr}
	\end{subfigure}
	\caption{Axion generation probability ($\mid a\mid^{2}\cdot10^{15}$) vs. $\omega  t$ for different parameters.}\label{fig4EvolMatr}
\end{figure} 

The probability of converting an axion into a photon displays an analogous behaviour vs. time. 
The combined processes of photon-axion-photon conversion are the pivoting tools of shining through the wall experiments \cite{Rafelt,Olive}, which are of noticeable importance for the search of dark matter constituents.

\subsection{Higher Order Matrices}\label{FCEv}

In these concluding remarks we want to emphasize that the mathematical tool we have developed is amenable for further extensions and contains seed elements for further studies.

\begin{exmp}
Regarding the extension of the method to matrices with dimensionality higher than $3\times 3$ we note that in the case of a $4\times 4$ matrix we have

\begin{equation} \label{GrindEQ__28_EvolMatr} 
e^{\tau \, \sqrt{\hat{A}} } =\prod _{r=0}^{3}e^{f_{r} \tau \, \hat{A}^{r} } .  
\end{equation} 
The use of the Airy transform \cite{Babusci} allows to write

\begin{equation}\begin{split} \label{GrindEQ__29_EvolMatr} 
& e^{\alpha \, \hat{A}^{3} } =\int _{-\infty}^{\infty }Ai(t)\, e^{\sqrt[{3}]{\alpha } t\, \hat{A}}  dt, \\ 
& Ai(t)=\dfrac{1}{2\pi}\int_{-\infty}^{\infty}e^{(\frac{i}{3}\xi^{3}+it\xi)}d\xi.
\end{split} \end{equation} 
Thus finally getting, for the square root of a $4\times 4$ matrix, the following expression in terms of Airy and Gauss transforms:

\begin{equation}\begin{split} \label{GrindEQ__30_EvolMatr} 
& e^{\tau \, \sqrt{\hat{A}} } =\dfrac{e^{f_{0} (\tau )}}{\sqrt{\pi}} \int _{-\infty }^{+\infty }\, \int _{0}^{\infty } \,  e^{-\xi ^{2} } Ai(\lambda ) \hat{U}(f_{1} \tau +2\sqrt{f_{2} \tau } \xi +\sqrt[{3}]{f_{3} \tau } \lambda )d\xi d\lambda, \\[1.1ex] 
& \hat{U}(\tau )=e^{\tau \, \hat{A}} .
\end{split} \end{equation} 
\end{exmp}
The case of matrices with higher order dimensionalities ($5\times 5, \dots$), requires only higher order integral transforms of the type discussed in ref.  \cite{Sacchetti}.\\

The use of the Lévy-Smirnov distribution allows to write (see eq. \ref{GrindEQ__21_EvolMatr})

\begin{equation} \label{GrindEQ__31_EvolMatr} 
e^{-\tau \, \sqrt{\left(\begin{array}{ccc} {0} & {-\omega _{3} } & {\omega _{2} } \\ {\omega _{3} } & {0} & {-\omega _{1} } \\ {-\omega _{2} } & {\omega _{1} } & {0} \end{array}\right)} } =\int _{0}^{\infty }g_{\frac{1}{2} }  (\eta )\, \hat{R}(-\eta \, \tau ^{2} )d\eta .
\end{equation} 
Taking into account that

\begin{equation}\begin{split}
\label{GrindEQ__32_EvolMatr} 
& \int _{0}^{\infty }g_{\frac{1}{2} }  (\eta )\, d\eta =1, \\ 
& \int _{0}^{\infty }g_{\frac{1}{2} }  (\eta )e^{-i\, \eta \, x} \, d\eta =e^{-\, \frac{\sqrt{2\, x} }{2} } \left(\cos \left(\frac{\sqrt{2\, x} }{2} \right)-i\sin \left(\frac{\sqrt{2\, x} }{2} \right)\right)
\end{split} \end{equation} 
and, by using the explicit form of the Rodrigues matrix (see eqs. \ref{GrindEQ__20_EvolMatr} and \ref{GrindEQ__21_EvolMatr}) we end up with

\begin{equation} \begin{split}
& e^{-\tau \, \sqrt{\hat{A}} } = \\ 
& =\hat{1}-\frac{e^{-\frac{\sqrt{2\, \Omega } }{2} \tau } }{\Omega } \sin \left(\frac{\sqrt{2\, \Omega } }{2} \tau \right)\hat{A}+\frac{1 }{\Omega ^{2} } \left(1-e^{-\, \frac{\sqrt{2\, \Omega } }{2} \tau }\cos \left(\frac{\sqrt{2\, \Omega } }{2} \tau \right)\right)\hat{A}^{2}.
\end{split}\end{equation} 

The possibility of extending the method to higher order roots follows the same procedure we have envisaged and for example

\begin{equation} \label{GrindEQ__34_EvolMatr} 
e^{-\tau \sqrt[{n}]{\hat{A}} } =\int _{0}^{\infty }g_{\frac{1}{n} }  (\eta )\, e^{-\eta \tau ^{n} \hat{A}} d\eta,  
\end{equation} 
where the Lévy stable function $g_{\frac{1}{n} } (\eta )$ is in general expressible in terms of Meijer G-functions (see ref. \cite{Penson}).

\chapter{Bessel Functions and Umbral Calculus}\label{Chapter3}
\numberwithin{equation}{section}
\markboth{\textsc{\chaptername~\thechapter. Bessel Functions and Umbral Calculus}}{}

In this Chapter we describe a sistematic reformulation of the theory of special functions, and in particular of Bessel functions, in terms of the umbral conception developed so far. \\

The original parts of the Chapter, containing their adequate bibliography, are based on the following original papers.\\

\cite{ProdB} \textit{Giuseppe Dattoli, Elio Sabia, Emanuele Di Palma, Silvia Licciardi; “Products of Bessel functions and associated polynomials”; Applied Mathematics and Computation, Vol 266 Issue C, September 2015, pages 507-514, Elsevier Science Inc. New York, NY, USA}.\\

\cite{Babusci} \textit{D. Babusci, G. Dattoli, M. Del Franco, S. Licciardi; “Lectures on Mathematical Methods for Physics”, invited Monograph by World Scientific, Singapore, 2017, in press}.\\

$\star$ \textit{G. Dattoli, S. Licciardi; “Book on Bessel Functions and Umbral Calculus”, work in progress}.\\

We begin from Bessel functions and see how their study can be afforded by the use of elementary analytical means.\\

Such a point of view and the joint use of other tools as e.g. the Ramanujan Master Theorem (RMT) \ref{AppARMT}, allow the
evaluation of infinite integrals of \textit{Bessel functions in terms of ordinary Gaussian integrals }\cite{D.Babusci}. Further computational technicalities, as e.g. those involving the repeated derivatives  of Bessel functions
with respect to their variable or to their index, are indeed greatly simplified, like we saw e.g in \ref{derivsucc} and similar. In addition the method suggests new possibility for the
introduction of auxiliary polynomials \cite{ProdB}, allowing significant progresses  in the study of the properties of Bessel functions and of their link to
other forms belonging to the Bessel like family.\\

\section{Bessel Functions}
We remind the \textbf{cylindrical Bessel function of $0$-order} \ref{cBf} defined by the power series

\begin{equation} \label{cilindricBessel}
J_{0}(x)=\sum_{r=0}^{\infty}\dfrac{(-1)^{r}\left(\frac{x}{2} \right)^{2r} }{(r!)^{2}}, \quad \forall x\in\mathbb{R},
\end{equation}
which graphic is

\begin{figure}[h]
	\centering
	\includegraphics[width=.6\textwidth]{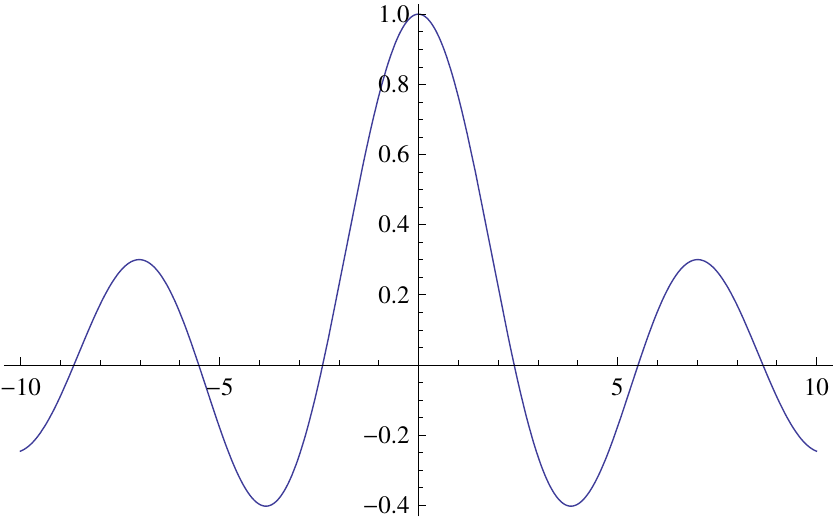}
	\caption{$ 0$-order of Bessel function, $J_0(x)$.  }\label{funzGaussiana.pdf}
\end{figure}
%

According to the notation outlined in Chapter \ref{Chapter1}, the series \ref{cilindricBessel} can be written in the form of a Gaussian \ref{J0op} by the use of the $\hat{c}$-operator \ref{Opc}

\begin{equation}\label{BesselGaussian}
J_{0}(x)=e^{-\hat{c}\left(\frac{x}{2} \right)^{2}}\varphi_{0},
\end{equation}
yielding the minimal umbral image of the Bessel function or, e.g., in rational terms \ref{J0opb} by the use of the $\hat{b}$-operator \ref{bOp}

\begin{equation}
J_0 (x)=\dfrac{1}{1+\hat{b}\left( \frac{x}{2}\right)^2 }\Phi_0 
\end{equation}
 and so on.\\
 
\noindent We focus on the Gaussian umbral image for the wealth of properties that Gaussian functions possess, so allowing to derive all the properties of Bessel in a very direct way, which significantly simplifies the ordinary procedure.\\

The infinite integrals involving \textit{0-th} order Bessel functions can be directly inferred from those of elementary Gaussian integrals. We have seen indeed \ref{IntJ0cop}

\begin{equation}
\int_{0}^{\infty}J_{0}(x)dx=\left(  \int_{0}^{\infty}e^{-\hat{c}\left(\frac{x}{2} \right)^{2}}dx\right)  \varphi_{0}=\sqrt{\frac{\pi}{\hat{c}}}\varphi_{0}=\sqrt{\pi}\dfrac{1}{\Gamma\left( \frac{1}{2}\right) }=1.
\end{equation}
Other properties can be derived using comparable simple means.

\begin{prop}
By keeping the first derivative of both sides of eq. \ref{BesselGaussian} we obtain, with $\hat{D}_{x}=\dfrac{d}{dx} $,
 
 \begin{equation}\label{derivBessel}
\hat{D}_{x}J_{0}(x)=-\left( \hat{c}\dfrac{x}{2}\right)e^{-\hat{c}\left( \frac{x}{2}\right)^{2} }\varphi_{0}=-\sum_{r=0}^{\infty}\dfrac{(-1)^{r}\left( \frac{x}{2}\right)^{2r+1} }{r!(r+1)!} . 
 \end{equation}
 The power series expansion in eq. \ref{derivBessel} represents the \textbf{first order cylindrical Bessel function}, namely

\begin{equation}
J_{1}(x)=\sum_{r=0}^{\infty}\dfrac{(-1)^{r}\left( \frac{x}{2}\right)^{2r+1} }{r!(r+1)!}.
\end{equation} 
\end{prop}

\begin{Oss}
We have therefore shown that

\begin{equation}
\hat{D}_{x}J_{0}(x)=-J_{1}(x).
\end{equation}
\end{Oss}

\begin{cor}
By keeping a further derivative of both sides of eq. \ref{derivBessel} we obtain
 
 \begin{equation}\begin{split}
 \hat{D}_{x}^{2}J_{0}(x)&=-\left(  \left( \hat{c}\dfrac{1}{2}\right)e^{-\hat{c}\left( \frac{x}{2}\right)^{2} }-\left( \hat{c}\dfrac{x}{2}\right)^{2}e^{-\hat{c}\left( \frac{x}{2}\right)^{2}} \right) \varphi_{0}=\\
& =-\dfrac{J_{1}(x)}{x}+\sum_{r=0}^{\infty}\dfrac{(-1)^{r}\left( \frac{x}{2}\right)^{2r+2} }{r!(r+2)!} . 
 \end{split}\end{equation}
The power series in the last equation provides the \textbf{second order Bessel equation}, namely

\begin{equation}
J_{2}(x)=\sum_{r=0}^{\infty}\dfrac{(-1)^{r}\left( \frac{x}{2}\right)^{2r+2} }{r!(r+2)!} . 
\end{equation} 
More in general, by keeping successive derivatives (see below for a more complete treatment), we can argue that the \textbf{\textit{n-th} order} of Bessel functions is
 
 \begin{equation}\label{derivsucc}
J_{n}(x)=\left( \hat{c}\dfrac{x}{2}\right)^{n}e^{-\hat{c}\left( \frac{x}{2}\right)^{2}}\varphi_{0}=\sum_{r=0}^{\infty}\dfrac{(-1)^{r}\left( \frac{x}{2}\right)^{2r+n} }{r!(r+n)!} .
 \end{equation}
\end{cor}

The above relation can be further exploited to get a general recurrence relation involving contiguous index Bessel functions and their derivatives. 

\begin{propert}
 From eq. \ref{derivsucc} we find
 
\begin{equation}\begin{split}\label{contiguous}
i)\; \hat{D}_{x}J_{n}(x)&=\left(  \dfrac{n}{x}\left( \hat{c}\dfrac{x}{2}\right)^{n}-\left( \hat{c}\dfrac{x}{2}\right)^{n+1}\right) e^{-\hat{c}\left( \frac{x}{2}\right)^{2} }\varphi_{0}=\\
& =\dfrac{n}{x}J_{n}(x)-J_{n+1}(x).  
 \end{split}\end{equation}
 The use of the series expansion yields
 
\begin{equation}\begin{split}
ii)\; nJ_{n}(x)&=n\sum_{r=0}^{\infty}\dfrac{(-1)^{r}\left( \frac{x}{2}\right)^{2r+n} }{r!(r+n)!}=\\
& =\sum_{r=0}^{\infty}\dfrac{(-1)^{r}(n+r)\left( \frac{x}{2}\right)^{2r+n} }{r!(r+n)!} - \sum_{r=0}^{\infty}\dfrac{(-1)^{r}r\left( \frac{x}{2}\right)^{2r+n} }{r!(r+n)!}=\\
& =\dfrac{x}{2}\left(  J_{n-1}(x)+J_{n+1}(x)\right)   
 \end{split}\end{equation} 
 which, for convenience, is rewritten in the form of the recurrence

\begin{equation}\label{rewritten}
\dfrac{n}{x}J_{n}(x)=\dfrac{1}{2}\left( J_{n-1}(x)+J_{n+1}(x)\right),
\end{equation} 
 exploited, along with eq. \ref{contiguous}, to get the identity

\begin{equation}\label{Jn-1}
\hat{D}_{x}J_{n}(x)=\dfrac{1}{2}\left( J_{n-1}(x)-J_{n+1}(x)\right).
\end{equation}
It is also straightforwardly proved that

\begin{equation}
iii) \;\hat{D}_{x}\left(  x^{n}J_{n}(x)\right) =\hat{D}_{x}\left( \left(\dfrac{\hat{c}}{2} \right)^{n}x^{2n}e^{-\hat{c}\left( \frac{x}{2}\right)^{2} }\varphi_{0} \right) =x^{n}J_{n-1} .
\end{equation}
\end{propert}

 \begin{cor}
The differential equation satisfied by the Bessel functions can be obtained from the previous recurrences \ref{rewritten}-\ref{Jn-1}
 by noting that, once combined, yield
 
\begin{equation}\begin{split}\label{opD}
& J_{n-1}(x)=\dfrac{n}{x} J_{n}(x)+\hat{D}_{x}J_{n}(x), \\
& J_{n+1}(x) =\dfrac{n}{x} J_{n}(x)-\hat{D}_{x}J_{n}(x) . 
 \end{split}\end{equation} 
\end{cor}

\begin{prop}
We introduce the \textbf{index shifting operators}
 
 \begin{equation}
\hat{E}_{\pm}=\dfrac{\hat{N}}{x}\mp\hat{D}_{x} ,
 \end{equation}
 with $\hat{N}$ being a "index" operator defined
 
 \begin{equation}
 \hat{N}J_{n}(x)=nJ_{n}(x).
 \end{equation}
Then, we can express the recurrence relations \ref{opD} in the form
 
 \begin{equation}\label{rrform}
\hat{E}_{\pm} J_{n}(x)=J_{n\pm 1}(x).
 \end{equation}

 It is also easily understood that the action of the product of the shift operators yield
 
 \begin{equation}
\hat{E}_{+}\hat{E}_{-}\left( J_{n}(x)\right)=J_{n}(x) , 
 \end{equation}
 which in differential forms reads

\begin{equation}\begin{split}\label{E+E-}
 \hat{E}_{+}\hat{E}_{-}\left( J_{n}(x)\right)&=\left( \dfrac{\hat{N}}{x}-\hat{D}_{x}\right)J_{n-1}(x)=\left( \dfrac{n-1}{x}-\hat{D}_{x}\right)J_{n-1}(x)=   \\
& =\left( \dfrac{n-1}{x}-\hat{D}_{x}\right)\left( \dfrac{n}{x}+\hat{D}_{x}\right)J_{n}(x)   
 \end{split}\end{equation}  
 and, after a few manipulations, we end up with
 
 \begin{equation}\label{EqDiff}
 \left\lbrace x^{2}\hat{D}_{x}^{2}+x\hat{D}_{x}+\left(  x^{2}-n^{2}\right)  \right\rbrace J_{n}(x)=0  
 \end{equation}
 or, in a more compact form,

\begin{equation}\label{compactf}
\left( x\hat{D}_{x}\right)^{2}J_{n}(x)=(n^{2}-x^{2})J_{n}(x). 
\end{equation} 
\end{prop}

The products of two Bessel functions (which will be detailed successively) provide new families of functions with interesting properties. Identities like the recurrence relations can be extended to  e.g. the \textbf{square of Bessel functions}. 

\begin{prop}
\begin{equation}\begin{split}\label{new23}
& a)\; \left(  \dfrac{1}{2}\left( \hat{D}_{x}^{2}+\dfrac{1}{x}\hat{D}_{x}\right)+1 \right) J_{n}^{2}(x)= \dfrac{1}{2}\left(  J_{n-1}^{2}(x)+J_{n+1}^{2}(x)\right) ,\\[1.1ex]
& b)\;\; \dfrac{n}{x} \hat{D}_{x}J_{n}^{2}(x)= \dfrac{1}{2}\left(  J_{n-1}^{2}(x)-J_{n+1}^{2}(x)\right). 
\end{split}\end{equation}
\end{prop}

\begin{proof}[\textbf{Proof.}]
 The eq. $ b) $ follows by noting that

\begin{equation}
\dfrac{n}{x} \hat{D}_{x}J_{n}^{2}(x)=\left( \dfrac{2n}{x}J_{n}(x)\right)\left(\hat{D}J_{n}(x) \right),  
\end{equation}
which, on account of the recurrences \ref{rewritten} e \ref{Jn-1},  yields

\begin{equation}
\dfrac{n}{x} \hat{D}_{x}J_{n}^{2}(x)=\dfrac{1}{2}\left(J_{n-1}(x)+J_{n+1}(x) \right) \left(J_{n-1}(x)-J_{n+1}(x) \right).
\end{equation}
The derivation of the \ref{new23}$ -a)) $ is more elaborated as shown below:\\

$i)$ Note that from eqs. \ref{rewritten} e \ref{Jn-1} it follows that  

\begin{equation}\begin{split}\label{propJni}
& \left( \dfrac{n}{x}\right)^{2} J_{n}^{2}(x)=\dfrac{1}{4}\left(  J_{n+1}^{2}(x)+J_{n-1}^{2}(x)+2J_{n+1}(x)J_{n-1}(x)\right) ,\\
& \left( \hat{D}_{x}J_{n}(x)\right)^{2}=\dfrac{1}{4}\left(  J_{n+1}^{2}(x)+J_{n-1}^{2}(x)-2J_{n+1}(x)J_{n-1}(x)\right). 
\end{split}\end{equation}

$ii)$ Apply standard operatorial rules to state the identity

\begin{equation}\begin{split}
 \left( \hat{D}_{x}J_{n}(x)\right)^{2}&=\dfrac{1}{2}\hat{D}_{x}^{2}J_{n}^{2}(x)-J_{n}(x)\hat{D}^{2}_{x}J_{n}(x)=\\
& =\dfrac{1}{2}\hat{D}_{x}^{2}J_{n}^{2}(x)-\dfrac{J_{n}(x)}{x^{2}}\left(  x^{2}\hat{D}^{2}_{x}J_{n}(x)\right) .            
\end{split}\end{equation}

$iii)$ Use the Bessel differential equation to get

\begin{equation}\begin{split}
& \dfrac{1}{2}\hat{D}_{x}^{2}J_{n}^{2}(x)+\dfrac{1}{2x}\hat{D}_{x}J_{n}^{2}(x)+\left( 1-\dfrac{n^{2}}{x^{2}}\right)J_{n}^{2}(x)= \\
& =\dfrac{1}{4}\left[ J_{n+1}^{2}(x)+J_{n-1}^{2}(x)-2J_{n+1}(x)J_{n-1}(x)\right],
\end{split}\end{equation}
which combined with the first in \ref{propJni} yields the first of the recurrences \ref{new23}.
\end{proof}

\begin{cor}
The use of the shift operator formalism can also be exploited to derive the differential equation satisfied by $J_{n}^{2}(x)$.\\

\noindent It is indeed fairly natural, in analogy with \ref{rrform} to set

\begin{equation}\begin{split}
& {}_{2}\hat{E}_{-}J_{n}^{2}(x)=J_{n-1}^{2}(x),\\
& {}_{2}\hat{E}_{+}J_{n}^{2}(x)=J_{n+1}^{2}(x),
\end{split}\end{equation}
where the shift operators have been defined as

\begin{equation}\begin{split}
& {}_{2}\hat{E}_{-}=\left[ \dfrac{1}{2}\left( \hat{D}_{x}^{2}+\dfrac{1}{x}\hat{D}_{x}\right)+1 \right]+\dfrac{\hat{N}}{x}\hat{D}_{x} ,\\
& {}_{2}\hat{E}_{+}=\left[ \dfrac{1}{2}\left( \hat{D}_{x}^{2}+\dfrac{1}{x}\hat{D}_{x}\right)+1 \right]-\dfrac{\hat{N}}{x}\hat{D}_{x}.
\end{split}\end{equation}
The differential equation satisfied by
 the function $y(x)=J_{n}^{2}(x)$ is,
  thereby,  provided by

\begin{equation}
{}_{2}\hat{E}_{-}\left( {}_{2}\hat{E}_{+}y\right)=y ,
\end{equation}
namely

\begin{equation}\begin{split}
& \left( \hat{\Theta}+\dfrac{n+1}{x}\hat{D}_{x}\right) \left( \hat{\Theta}-\dfrac{n}{x}\hat{D}_{x}\right)y=y,\\
& \hat{\Theta}=\dfrac{1}{2}\left( \hat{D}_{x}^{2}+\dfrac{1}{x}\hat{D}_{x}\right)+1 ,
\end{split}\end{equation}
which is a \textbf{fourth order differential equation} which is worth to be expanded in the form\footnote{The bracket $[\hat{A},\hat{B}]=\hat{A}\hat{B}-\hat{B}\hat{A}$ denotes the commutation parenthesis and is different from zero whenever $\hat{A},\hat{B}$ are not commutaing quantities.}

\begin{equation}\label{fourth}
\left( \hat{\Theta}^{2}+\dfrac{1}{x}\hat{D}_{x}\hat{\Theta}+\dfrac{n(n+1)}{x^{2}}\left( \dfrac{1}{x}-\hat{D}_{x}\right)\hat{D}_{x} -n\left[ \hat{\Theta},\dfrac{1}{x}\hat{D}_{x}\right] \right)y=y. 
\end{equation}
The explicit evaluation of the commutator yields

\begin{equation}
\left[ \hat{\Theta},\dfrac{1}{x}\hat{D}_{x}\right]=\hat{\Theta}\dfrac{1}{x}\hat{D}_{x}-\dfrac{1}{x}\hat{D}_{x}\hat{\Theta}=\dfrac{1}{x^{2}}\left( \dfrac{1}{x}-\hat{D}_{x}\right)\hat{D}_{x}  
\end{equation}
and the eq. \ref{fourth} becomes

\begin{equation}\label{myfourth}
\left( \hat{\Theta}^{2}+\dfrac{1}{x}\hat{D}_{x}\hat{\Theta}+n^{2}\left[ \hat{\Theta},\dfrac{1}{x}\hat{D}_{x}\right] \right)y=y ,
\end{equation}
thus eventually getting

\begin{equation}\label{myfx}
\left( x^{3}\hat{\Theta}^{2}+x^{2}\hat{D}_{x}\hat{\Theta}+n^{2}(1-x\hat{D}_{x})\hat{D}_{x}-x^{3} \right)y=0 .
\end{equation}
\end{cor}

The form we have just derived is different from the analogous expression given in Abramowitz and Stegun\footnote{M. Abramowitz and I. Stegun “Handbook of Mathematical Functions” Dover (1970) Chapter 9  p. 362 eq. (9.1.57).}, where the relevant differential equation is fourth order in terms of the operator $\hat{\vartheta}=x\hat{D}_{x}$. It is possible to compare eqs. \ref{myfourth}-\ref{myfx} with the form given in the Abramowitz and Stegun and extend the previous results to products of Bessel functions with different indices.\\

Further examples of application of the just outlined method are given below. 

\begin{exmp}
Given the function $\phi_{n}(x)=x^{p}J_{n}(\lambda x^{q})$ with $\forall n\in\mathbb{N}$ and $\lambda,p,q$ independent on the index $n$, we can find the relevant recurrences and differential equation.\\

We note that

\begin{equation}
\dfrac{2n}{\lambda x^{q}}J_{n}(\lambda x^{q})=J_{n-1}(\lambda x^{q})+J_{n+1}(\lambda x^{q}).
\end{equation}
We multiply both sides by $x^{p}$ and deduce that

\begin{equation}
\dfrac{2n}{\lambda x^{q}}\phi_{n}(x)=\phi_{n-1}(x)+\phi_{n+1}(x).
\end{equation}
By noting that

\begin{equation}
J_{n}(\lambda x^{q})=x^{-p}\phi_{n}(x)
\end{equation}
and by using the fact

\begin{equation}\begin{split}
& 2\hat{D}_{Z}J_{n}(Z)=J_{n-1}(Z)-J_{n+1}(Z),\\
& Z=\lambda x^{q}
\end{split}\end{equation}
and, after establishing the identity

\begin{equation}
\hat{D}_{Z}=\dfrac{1}{q\lambda x^{q-1}}\hat{D}_{x},
\end{equation}
we end up with

\begin{equation}
(p+nq)\phi_{n-1}(x)+(p-nq)\phi_{n+1}(x)=\dfrac{2n}{\lambda}x^{1-q}\phi_{n}^{'}(x).
\end{equation}

The relevant differential equations are obtained either by using a appropriate generalization of the shift operator method, or by noting that

\begin{equation}
\left[ Z^{2}\hat{D}_{Z}^{2}+Z\hat{D}_{Z}+(Z^{2}-n^{2})\right](x^{-p}\phi_{n}(x))=0 .
\end{equation}
\end{exmp}

Other properties which can be easily checked using the present formalism. 

\begin{propert}
	We have:\\
	
 $a)$ \textbf{Argument Reflection Condition} 
 
 \begin{equation}
 J_{n}(-x)=\left( \hat{c}\dfrac{(-x)}{2}\right)^{n}e^{-\hat{c}\left( \frac{x}{2}\right)^{2} }\varphi_{0}=(-1)^{n}J_{n}(x) .
 \end{equation}
 
 $b)$ \textbf{Index Reflection Condition}
 
 \begin{equation}\begin{split}\label{menn}
 J_{-n}(x)&=\left( \hat{c}\dfrac{x}{2}\right)^{-n}e^{-\hat{c}\left( \frac{x}{2}\right)^{2} }\varphi_{0}=\sum_{r=\;n}^{\infty}\dfrac{(-1)^{r}\left( \dfrac{x}{2}\right)^{2r-n} }{r!(r-n)!}=\\
& =\sum_{s=0}^{\infty}\dfrac{(-1)^{s+n}\left( \dfrac{x}{2}\right)^{2s+n} }{s!(s+n)!}=(-1)^{n}J_{n}(x).
 \end{split}\end{equation}
\end{propert}
 In this section we have derived the main properties of Bessel functions in a quite straightforward way and most of the simplicity of the results are associated with the formalism we used.

\section{Bessel Functions of Second Kind}

In the previous section we have considered integer order Bessel functions, but there is no argument against their extension to any \textit{Real }order, namely

\begin{defn}
	We introduce the Gaussian umbral form of \textbf{real order Bessel function}
\begin{equation}\label{gammaNu}
J_{\nu}(x)=\left( \hat{c}\dfrac{x}{2}\right)^{\nu}e^{-\hat{c}\left( \frac{x}{2}\right)^{2} }\varphi_{0}=\sum_{r=0}^{\infty}\dfrac{(-1)^{r}\left( \dfrac{x}{2}\right)^{2r+\nu} }{r!\Gamma(\nu+r+1)}, \quad \forall \nu\in\mathbb{R}.
\end{equation}
\end{defn}

It is easily checked that the function \ref{gammaNu} satisfies the same recurrences of the integer order case and, therefore, the same differential equation. 

\begin{cor}
From eq. \ref{gammaNu} we also find

\begin{equation}
J_{\nu}(-x)=\left(- \hat{c}\dfrac{x}{2}\right)^{\nu}e^{-\hat{c}\left( \frac{x}{2}\right)^{2} }\varphi_{0}=e^{i\pi\nu}J_{\nu}(x).
\end{equation}
\end{cor}

It is now worth noting that being \ref{EqDiff} a second order differential equation, it admits two independent solutions. In the case of integer index $J_{n}(x),J_{-n}(x)$ cannot be considered independent but, since the non-integer case $J_{\nu}(x)$ is not expressible in terms of $J_{-\nu}(x)$ and viceversa,  we can consider them as independent solutions of the Bessel equation.\\

\noindent \textit{Any linear combination of the two solutions represents a solution of the Bessel equation}, therefore we set \cite{L.C.Andrews}

\begin{defn}
	The identity
\begin{equation} \label{NeumannBessel}
Y_{\nu}(x)=\dfrac{J_{\nu}(x)\cos(\nu\pi)-J_{-\nu}(x)}{\sin(\nu\pi)}
\end{equation}
 is called the \textbf{Neumann-Bessel} (N-B) function.
 \end{defn}
 and we show  the following Theorem:\\

\begin{thm}
Let, $\forall x\in \mathbb{R}$, $\forall \nu\in \mathbb{R}$, $Y_{\nu}(x)$ the N-B function then,  $\forall n\in\mathbb{Z}$, 

 \begin{equation}\label{key}
  \lim_{\nu\rightarrow n}Y_{\nu}(x) =Y_{n}(x),
 \end{equation}
 where $Y_{n}(x)$ is the integer index counterpart of \ref{NeumannBessel}.
\end{thm}
\begin{proof}[\textbf{Proof.}]
  According to our definition \ref{gammaNu}, we  obtain 
      \begin{equation}\label{Ynu}
        Y_{\nu}(x)=\dfrac{\left( \dfrac{\hat{c}x}{2}\right)^{\nu}\cos(\nu\pi)-\left(\dfrac{ \hat{c}x}{2}\right)^{-\nu} }{\sin(\nu\pi)}e^{-\hat{c}\left( \frac{x}{2}\right)^{2} }\varphi_{0}
      \end{equation}
   so, for \ref{derivsucc} and \ref{menn},
      \begin{equation*}\begin{split}
         \lim_{\nu\rightarrow n}Y_{\nu}(x)&=\lim_{\nu\rightarrow n}\frac{\left( \frac{\hat{c}x}{2}\right)^{\nu}\cos(\nu\pi)-\left(\frac{ \hat{c}x}{2}\right)^{-\nu} }{\sin(\nu\pi)}e^{-\hat{c}\left( \frac{x}{2}\right)^{2} }\varphi_{0}=\\
       & =\frac{\pm J_{n}(x)-J_{-n}(x)}{0}= \dfrac{0}{0}.
      \end{split}\end{equation*}
   Applying the de L’Hopital rule, the derivation with respect to $\nu$ yields 
   
       \begin{align*}
       &\lim_{\nu\rightarrow n}Y_{\nu}(x)=\\
       & =\lim_{\nu\rightarrow n}\frac{\left( \hat{c}\frac{x}{2}\right)^{\nu}\ln\left( \hat{c}\frac{x}{2}\right)\cos(\nu\pi)-\pi\left( \hat{c}\frac{x}{2}\right)^{\nu}\sin(\nu\pi)+\left( \hat{c}\frac{x}{2}\right)^{-\nu}\ln\left( \hat{c}\frac{x}{2}\right)}{\pi\cos(\nu\pi)}  e^{-\hat{c}\left( \frac{x}{2}\right)^{2} }\varphi_{0} =\\
      & =\dfrac{\ln\left( \frac{x}{2}\right)\left((-1)^{n} \left( \hat{c}\frac{x}{2}\right)^{n}+\left( \hat{c}\frac{x}{2}\right)^{-n}\right) }{(-1)^{n}\pi}e^{-\hat{c}\left( \frac{x}{2}\right)^{2} }\varphi_{0}
      +\\
      & +\dfrac{1}{(-1)^{n}\pi}\lim_{\nu\rightarrow n} \ln(\hat{c})\left( (-1)^{\nu}  \left( \hat{c}\frac{x}{2}\right)^{\nu}+\left( \hat{c}\frac{x}{2}\right)^{-\nu}\right)e^{-\hat{c}\left( \frac{x}{2}\right)^{2} }\varphi_{0}=\\
      & =\frac{1}{\pi}\ln\left( \frac{x}{2}\right)J_{n}(x)+\frac{1}{(-1)^{n}\pi}\ln\left( \frac{x}{2}\right)J_{-n}(x)+\\
      & +\dfrac{1}{(-1)^{n}\pi}\lim_{\nu\rightarrow n} \ln(\hat{c})\left( (-1)^{\nu}  \left( \hat{c}\frac{x}{2}\right)^{\nu}+\left( \hat{c}\frac{x}{2}\right)^{-\nu}\right)e^{-\hat{c}\left( \frac{x}{2}\right)^{2} }\varphi_{0}=\\
      & =\frac{2}{\pi}\ln\left( \frac{x}{2}\right)J_{n}(x) +\lim_{\nu\rightarrow n}\frac{1}{\pi}\sum_{r=0}^{\infty}\dfrac{(-1)^{r}\left( \frac{x}{2}\right)^{2r+\nu}\hat{c}^{r+\nu}\ln(\hat{c}) }{r!}\varphi_{0}+\\
      & +\frac{1}{(-1)^{n}\pi}\lim_{\nu\rightarrow n}\sum_{r=0}^{\infty}\dfrac{(-1)^{r}\left( \frac{x}{2}\right)^{2r-\nu}\hat{c}^{r-\nu}\ln(\hat{c}) }{r!}\varphi_{0}.
     \end{align*}
  Now, to evaluate the logarithm of the operator $\hat{c}$ on $\varphi_{0}$, we use \ref{Opc} and set 
     \begin{align}
       \ln(\hat{c})\varphi_{0}&=\lim_{\alpha\rightarrow 0}\left( \dfrac{\hat{c}^{\alpha}-1}{\alpha}\right)\varphi_{0}=\lim_{\alpha\rightarrow 0}\dfrac{\varphi_{\alpha}-\varphi_{0}}{\alpha}= \notag\\
      & =\lim_{\alpha\rightarrow 0}\left[ \dfrac{1}{\alpha}\left( \dfrac{1}{\Gamma(\alpha+1)}-1\right) \right] =-\lim_{\alpha\rightarrow 0}\dfrac{\Gamma^{'}(\alpha+1)}{(\Gamma(\alpha+1))^{2}}=_{de\; L'H.}\notag\\
     & =-\lim_{\alpha\rightarrow 0}\dfrac{\Psi(\alpha+1)}{\Gamma(\alpha+1)},\\[1.2ex]
     \Psi(\beta)&=\dfrac{\Gamma^{'}(\beta)}{\Gamma(\beta)}, \;\; \beta\in\mathbb{R},
   \end{align}
  with $\Psi(\beta)$ 
  \textit{Digamma} function. Then, 
     \begin{equation*}\begin{split}
      \lim_{\nu\rightarrow n}Y_{\nu}(x)=\frac{2}{\pi}\ln&\left( \frac{x}{2}\right)J_{n}(x)+\frac{1}{\pi}\sum_{r=0}^{\infty}\frac{(-1)^{r}\left( \frac{x}{2}\right)^{2r+n} }{r!}\lim_{\alpha\rightarrow 0}\left( \frac{\hat{c}^{\alpha+r+n}-\hat{c}^{r+n}}{\alpha}\right)\varphi_{0}+\\
      & +\frac{1}{(-1)^{n}\pi}\sum_{r=0}^{\infty}\frac{(-1)^{r}\left( \frac{x}{2}\right)^{2r-n}}{r!}\lim_{\nu\rightarrow n}\lim_{\alpha\rightarrow 0}\left( \frac{\hat{c}^{\alpha+r-\nu}-\hat{c}^{r-\nu}}{\alpha}\right)\varphi_{0}  =\\
       =\frac{2}{\pi}\ln\left( \frac{x}{2}\right)J_{n}(x) &-\dfrac{1}{\pi}\sum_{r=0}^{\infty}\dfrac{(-1)^{r}\left( \frac{x}{2}\right)^{2r+n}}{r!}\lim_{\alpha\rightarrow 0}\dfrac{\Psi(\alpha+r+n+1)}{\Gamma(\alpha+r+n+1)}+\\
      & +\frac{1}{(-1)^{n}\pi}\left[-  \sum_{r=0}^{n-1}\dfrac{(-1)^{r}\left( \frac{x}{2}\right)^{2r-n}}{r!}\lim_{\nu\rightarrow n}\dfrac{\Psi(r-\nu+1)}{\Gamma(r-\nu+1)}+\right. \\
      & \left. +\sum_{r=n}^{\infty}\dfrac{(-1)^{r}\left( \frac{x}{2}\right)^{2r-n}}{r!}\lim_{\nu\rightarrow n}\lim_{\alpha\rightarrow 0}\left( \frac{\hat{c}^{\alpha+r-\nu}-\hat{c}^{r-\nu}}{\alpha}\right)\right] \varphi_{0}=\\
       =\frac{2}{\pi}\ln\left( \frac{x}{2}\right)J_{n}(x) &-\dfrac{1}{\pi}\sum_{r=0}^{\infty}\dfrac{(-1)^{r}\left( \frac{x}{2}\right)^{2r+n}}{r!}\dfrac{\Psi(r+n+1)}{\Gamma(r+n+1)}+\\
      & +\frac{1}{(-1)^{n}\pi}\left[ - \sum_{r=0}^{n-1}\dfrac{(-1)^{r}\left( \frac{x}{2}\right)^{2r-n}}{r!}\lim_{\nu\rightarrow n}\dfrac{\Psi(r-\nu+1)}{\Gamma(r-\nu+1)}+\right. \\
      & \left. +\sum_{s=0}^{\infty}\dfrac{(-1)^{s}\left( \frac{x}{2}\right)^{2s+n} }{(n+s)!}\lim_{\alpha\rightarrow 0}\left( \frac{\hat{c}^{\alpha+s}-\hat{c}^{s}}{\alpha}\right)\varphi_{0}\right] =\\
       =\frac{2}{\pi}\ln\left( \frac{x}{2}\right)J_{n}(x) &-\dfrac{1}{\pi}\sum_{r=0}^{\infty}\dfrac{(-1)^{r}\left( \frac{x}{2}\right)^{2r+n}}{r!}\dfrac{\Psi(r+n+1)}{\Gamma(r+n+1)}+\\
      & +\frac{1}{(-1)^{n}\pi}\left[ - \sum_{r=0}^{n-1}\dfrac{(-1)^{r}\left( \frac{x}{2}\right)^{2r-n}}{r!}\lim_{\nu\rightarrow n}\dfrac{\Psi(r-\nu+1)}{\Gamma(r-\nu+1)}+\right. \\
      & -\left. \sum_{s=0}^{\infty}\dfrac{(-1)^{n+s}\left( \frac{x}{2}\right)^{2s+n}}{(n+s)!}\dfrac{\Psi(s+1)}{\Gamma(s+1)}\right] 
     \end{split}\end{equation*}
   but, using the identities
     \begin{align}
     &\Gamma(x)\Gamma(1-x)=\dfrac{\pi}{\sin(x\pi)}, \;\;\forall x \notin \mathbb{Z} , \\   
    & \Psi(1-x)-\Psi(x)=\pi\cot(x\pi), \;\;\forall x\notin \mathbb{Z} ,  
     \end{align}
     we can write
    \begin{equation}
    \lim_{\nu\rightarrow n}\dfrac{\Psi(1-(\nu-r))}{\Gamma(1-(\nu-r))}=\left\lbrace \begin{array}{l l}
    (-1)^{n-r}(n-r-1)!&0\leq r<n ,\\-\infty&r\geq n.
    \end{array} \right.
    \end{equation}
    We eventually find
       
    \begin{equation*}\begin{split}
    \lim_{\nu\rightarrow n}Y_{\nu}(x)&=\frac{2}{\pi}\ln\left( \frac{x}{2}\right)J_{n}(x) -\dfrac{1}{\pi}\sum_{r=0}^{\infty}\dfrac{(-1)^{r}\left( \frac{x}{2}\right)^{2r+n}}{r!}\dfrac{\Psi(r+n+1)}{\Gamma(r+n+1)}+\\
    & -\frac{1}{\pi}\left[  \sum_{r=0}^{n-1}\dfrac{(n-r-1)!}{r!}\left( \frac{x}{2}\right)^{2r-n}+\sum_{s=0}^{\infty}\dfrac{(-1)^{s}}{(n+s)!}\dfrac{\Psi(s+1)}{s!}\left( \frac{x}{2}\right)^{2s+n}\right]=\\[1.1ex]
    & =Y_{n}(x).     
    \end{split}\end{equation*}
    \end{proof}

\begin{cor}
    Now, we recast in operators terms and obtain
    
   \begin{equation}\begin{split}
  & \Delta_{n}(x)=\left( \hat{\chi}\hat{c}\dfrac{x}{2}\right)^{n}e^{-\hat{\chi}\hat{c}\left( \frac{x}{2}\right)^{2} }\varphi_{0}\vartheta_{0}=\sum_{k=0}^{\infty}\dfrac{(-1)^{k}\Psi(k+n+1)}{k!\Gamma(k+n+1)}\left( \dfrac{x}{2}\right)^{2k+n},\\
  & \hat{\chi}^{\alpha}\vartheta_{0}=\dfrac{\Gamma^{'}(\alpha+1)}{\Gamma(\alpha+1)},\\
  & \delta_{n}(x)=\left(\hat{c}\dfrac{x}{2}\right)^{n}e^{-\hat{\chi}\hat{c}\left( \frac{x}{2}\right)^{2} }\varphi_{0}\vartheta_{0}=\sum_{k=0}^{\infty}\dfrac{(-1)^{k}\Psi(k+1)}{k!\Gamma(k+n+1)}\left( \dfrac{x}{2}\right)^{2k+n} 
  \end{split}\end{equation} 
and these positions yield

\begin{equation}\label{YN}
Y_{n}(x)=\dfrac{2}{\pi}\ln\left( \dfrac{x}{2}\right)J_{n}(x)-\dfrac{1}{\pi}\sum_{k=0}^{n-1}\dfrac{(n-k-1)!}{k!}\left(\dfrac{x}{2}\right)^{2k-n}-\dfrac{1}{\pi}\left[ \Delta_{n}(x)+\delta_{n}(x)\right] ,
\end{equation}
that is the $n\!-\!th$ order of Neumann-Bessel function.\\

In particular we have

\begin{equation}\label{Y0}
Y_{0}(x)=\dfrac{2}{\pi}\ln\left( \dfrac{x}{2}\right)J_{0}(x)-\dfrac{2}{\pi}\Delta_{0}(x).
\end{equation}

The function $Y_{0}(x)$ is the $0\!-\!th$ order N-B function, it has a logarithmic singularity at the origin and is the second solution of the Bessel equation with $n=0$.
\end{cor}

\begin{Oss}
The functions $\Delta_{n}(x),\delta_{n}(x)$ are not usually considered as independent functions. We call them the \textbf{first and second renormalized Bessel functions} respectively, for the reasons we will clarify in the following and we will discuss their properties later in this thesis. 
\end{Oss}

\begin{cor}
The theorem provides the derivative with respect to $\nu$ of the Bessel function

 \begin{equation}\begin{split}\label{ln}
 \hat{D}_{\nu}J_{\nu}(x)&=\left( \hat{c}\dfrac{x}{2}\right)^{\nu}e^{-\hat{c}\left( \frac{x}{2}\right)^{2} }\left[\ln\left( \dfrac{x}{2}\right)+\ln(\hat{c})  \right]\varphi_{0} =\\ 
& =\ln\left( \dfrac{x}{2}\right)J_{\nu}(x)-\sum_{r=0}^{\infty}\dfrac{(-1)^{r}\Psi(r+\nu+1)}{r!\;\Gamma(r+\nu+1)}\left( \dfrac{x}{2}\right)^{2r+\nu}.
 \end{split}\end{equation}
\end{cor}

\section{The Modified Bessel Functions of First Kind}

The \textbf{modified Bessel functions of the first kind} are a byproduct of the ordinary cylinder Bessel functions and can be defined as

\begin{defn}
	We introduce the umbral form of modified Bessel functions of the first kind
\begin{equation}\label{mBffk}
I_{\nu}(x)=\left(\hat{c}\dfrac{x}{2}\right)^{\nu}e^{\hat{c}\left( \frac{x}{2}\right)^{2} }\varphi_{0}=\sum_{r=0}^{\infty}\dfrac{\left( \dfrac{x}{2}\right)^{2r+\nu}}{r!\Gamma(\nu+r+1)}.
\end{equation}
\end{defn}

Their link with the $J_{\nu}(x)$ counterpart is provided by

\begin{propert}
\begin{equation}\label{mBffkJnu}
 I_{\nu}(ix)=i^{\nu}\left(\hat{c}\dfrac{x}{2}\right)^{\nu}e^{-\hat{c}\left( \frac{x}{2}\right)^{2} }\varphi_{0}=e^{i\nu\frac{\pi}{2}}J_{\nu}(x).
\end{equation}
\end{propert}
Therefore

\begin{prop}
The differential equation they satisfy can be inferred from \ref{EqDiff} by replacing $x$ with $ix$, therefore we find

\begin{equation}
\left\lbrace x^{2}\hat{D}_{x}^{2}+x\hat{D}_{x}-\left[ x^{2}+\nu^{2}\right] \right\rbrace I_{\nu}(x)=0 .
\end{equation}
The second solution is a Bessel like function too usually referred as \textbf{Macdonald function} or \textbf{Modified Bessel function of the second kind}, defined as

\begin{equation}
 K_{\nu}(x)=\dfrac{\pi}{2}\dfrac{I_{-\nu}(x)-I_{\nu}(x)}{\sin(\nu\pi)}=-\dfrac{\pi}{\sin(\nu\pi)}\left[ \sinh\left( \nu\ln\left( \dfrac{\hat{c}x}{2}\right) \right) \right]e^{\hat{c}\left( \frac{x}{2}\right)^{2} }\varphi_{0} .
\end{equation}

The integer order counterpart can be obtained by following a procedure analogous to that leading to the integer orders Macdonald functions.\\

\noindent We get indeed

 \begin{equation}\begin{split}
 K_{n}(x)&=-\lim_{\nu\rightarrow\pi}\dfrac{1}{\cosh(\nu\pi)}\left[ \cosh\left[ \nu\ln\left( \dfrac{\hat{c}x}{2}\right) \right]\ln\left( \dfrac{\hat{c}x}{2}\right) \right]e^{\hat{c}\left( \frac{x}{2}\right)^{2}}\varphi_{0}=  \\
& =(-1)^{n+1}\left[ \cosh\left[ n\ln\left( \dfrac{\hat{c}x}{2}\right) \right]\ln\left( \dfrac{\hat{c}x}{2}\right) \right]e^{\hat{c}\left( \frac{x}{2}\right)^{2}}\varphi_{0}.
 \end{split}\end{equation}
In the case $n=0$, we obtain

 \begin{equation}\begin{split}
 K_{0}(x)&=-\left[\ln\left( \dfrac{x}{2}\right)+\ln(\hat{c}) \right]e^{\hat{c}\left( \frac{x}{2}\right)^{2}}\varphi_{0}=  \\
& =-\ln\left( \dfrac{x}{2}\right)I_{0}(x)-\lim_{\alpha\rightarrow 0}\dfrac{\hat{c}^{\alpha}-1}{\alpha}e^{\hat{c}\left( \frac{x}{2}\right)^{2}}\varphi_{0}=\\
& =-\ln\left( \dfrac{x}{2}\right)I_{0}(x)+\sum_{r=0}^{\infty}\dfrac{(\frac{x}{2})^{2r}}{r!^{2}}\Psi(r+1) 
 \end{split}\end{equation}
 and therefore, by introducing the $\hat{h}$-operator,
 
  \begin{equation}\begin{split}  
 & K_{0}(x)=-\left[ \ln\left(  \frac{x}{2}\right)+\gamma\right]I_{0}(x)+ I_{0}(\hat{h}^{\frac{1}{2}}x) ,\\
 & I_{0}(\hat{h}x)=\sum_{r=0}^{\infty}\dfrac{\left( \dfrac{\hat{h}x}{2}\right)^{2r} }{(r!)^{2}},
\end{split}\end{equation} 
 with
\begin{equation}\begin{split} 
 & \hat{h}^{0}=0,\\
 & \hat{h}^{m}=h_{m}, \\
 & h_{m}=\sum_{k=1}^{m}\dfrac{1}{k},
   \end{split}\end{equation} 
   where $h_m$ are the \textbf{Harmonic Numbers}\footnote{They will be treated in the next Chapter.}.
 \end{prop}

\section{Products of Bessel Functions and Associated Polynomials}

Let us now consider the following

\begin{exmp}
Let 

\begin{equation}
f(x;a,b):= J_{0}(ax)J_{0}(bx)
\end{equation}
a product of 0-order Bessel functions, which can be formally written as the product of two Gaussians \ref{BesselGaussian}, namely \footnote{Even though not explicitly stated, it is evident that in the present formalism we have
	 \begin{equation}
	\left[J_{0}(x) \right]^{2}= e^{-\hat{c} \left(\frac{x}{2} \right)^{2} }e^{-\hat{c} \left(\frac{x}{2} \right)^{2} }\varphi_{0}.
	\end{equation}}

\begin{equation}
f(x;a,b)=e^{-(a^{2}\hat{c}_{1}+b^{2}\hat{c}_{2})\left(\frac{x}{2} \right)^{2} }\varphi_{0}^{(1)}\varphi_{0}^{(2)},
\end{equation}
where $\varphi_{0}^{(\alpha)}$ are the umbral vacua on which the operators $\hat{c}_{\alpha}$ act. The series expansion of the exponential and the use of the previously outlined rules yield

\begin{equation}\begin{split}
& f(x;a,b)=\sum_{r=0}^{\infty}\dfrac{(-1)^{r}}{r!}l_{r}(a^{2},b^{2})\left(\frac{x}{2} \right)^{2r},\\
& l_{r}(a,b)=r!\sum_{s=0}^{r}\dfrac{a^{(r-s)}b^{s}}{(s!)^{2}\left[(r-s)! \right]^{2} }.
\end{split}\end{equation}

Leaving for the moment unspecified the nature of the polynomials $l_{r}(a,b)$, we note that the function $f(x;a,b)$ can be cast in the umbral form

\begin{equation}\begin{split}\label{f(xab)ProdB}
& f(x;a,b)=e^{-\hat{l} \left(\frac{x}{2} \right)^{2} }\Phi_{0},\\[1.1ex]
& \hat{l}^{\nu}\Phi_{0}=l_{\nu}(a^{2},b^{2}).
\end{split}\end{equation}

The action of the operator $\hat{l}$ on the corresponding umbral vacuum holds for any real (positive/negative) or complex value of the  exponent $\nu$.
We have concluded that  the product of two cylindrical Bessel is theumbral equivalent of a $BF$ and thus the umbra of a Gaussian.
\end{exmp}

 Such a conclusion turns particularly useful if we are interested in the evaluation of the integrals of the function $f(x;a,b)$, a straightforward use of the so far developed procedure yields
 
 \begin{exmp}
\begin{equation}\begin{split}\label{RamanjProdB}
& \int_{-\infty}^{+\infty}f(x;a,b)dx=\int_{-\infty}^{+\infty}e^{-\hat{l} \left(\frac{x}{2} \right)^{2} }dx\;\Phi_{0}=2\sqrt{\pi}\;\hat{l}^{-\frac{1}{2}}\Phi_{0}=2\sqrt{\pi}\;l_{-\frac{1}{2}}(a^{2},b^{2}),\\[1.2ex]
& \mid a\mid>\mid b\mid , \\[1.2ex]
& l_{-\frac{1}{2}}(a^{2},b^{2})=\Gamma \left(\dfrac{1}{2} \right)\sum_{s=0}^{\infty}\dfrac{a^{-2\left(\frac{1}{2}+s \right) }b^{2s}}{(s!)^{2}\Gamma \left(\dfrac{1}{2}-s \right)^{2}}=\dfrac{1}{\sqrt{\pi} \mid a\mid}K\left(\dfrac{b}{a} \right), \\[1.2ex]
& K(k)={}_{2}F_{1}\left(\dfrac{1}{2},\dfrac{1}{2};1;k^{2} \right)= \sum_{s=0}^{\infty}\left[\dfrac{(2s)!}{2^{2s}(s!)^{2}} \right]^{2}k^{2s},
\end{split}\end{equation}
where ${}_{2}F_{1}\left(a,b;c;z \right)$ is the confluent hypergeometric function \cite{Abramovitz}.
\end{exmp}

\begin{rem}
\textit{We have left open the question on the nature of the polynomials $l_{r}(a,b)$, although we will discuss more deeply this point in the forthcoming sections, here we note that they can be viewed as a particular case of the Jacoby polynomials (seen in section \ref{JacPol}), as it can be inferred from the identity \cite{Babusci}:}

\begin{equation}\begin{split}
& l_{r}\left(\sqrt{\dfrac{x-1}{2}},\sqrt{\dfrac{x+1}{2}} \right)=\dfrac{1}{r!}P_{r}^{(0,0)}(x), \\[1.1ex]
& P_{r}^{(\alpha ,\beta)}(x)=\sum_{s=0}^{n}\binom {n+\alpha}{s}\binom {n+\beta}{n-s}\left( \dfrac {x-1}{2}\right) ^{n-s}\left( \dfrac {x+1}{2}\right) ^{s}.
\end{split}\end{equation}
\end{rem}

Furthermore,

\begin{cor}
 By the use of Gaussian umbral form of real order Bessel function \ref{gammaNu}, we obtain the following general expression for the product of two cylindrical Bessel functions of order $\nu ,\mu$ respectively,

\begin{equation}\begin{split}
 f_{\nu ,\mu}(x;a,b)&=J_{\nu}(ax)J_{\mu}(bx)=\left(\dfrac{x}{2} \right)^{\nu+\mu}(a\hat{c}_{1})^{\nu}(b\hat{c}_{2})^{\mu}e^{-(a^{2}\hat{c}_{1}+b^{2}\hat{c}_{2})\left(\frac{x}{2} \right)^{2} }\varphi_{0}^{(1)}\varphi_{0}^{(2)}=  \\
& =\left(\dfrac{x}{2} \right)^{\nu+\mu}\sum_{r=0}^{\infty}\dfrac{(-1)^{r}}{r!}l_{r}^{(\nu ,\mu)}(a^{2},b^{2})\left( \dfrac{x}{2}\right)^{2r}, \\
 l_{r}^{(\nu ,\mu)}(a,b)&=\sum_{s=0}^{r}\dfrac{a^{(r-s)+\nu}b^{s+\mu}}{\Gamma (\nu+s+1)\Gamma(\mu+r-s+1)s!(r-s)!}.
\end{split}\end{equation}
\end{cor}

  In the forthcoming section we take advantage from these results to extend the method to arbitrary products.

\subsection{Products of Bessel functions}

According to the tools outlined in the previous section we get

\begin{cor}
  The product of three \textit{0-th} order Bessel functions can be written as
\begin{equation}
f(x;a_{1},a_{2},a_{3})=e^{-(a_{1}^{2}\hat{c}_{1}+a_{2}^{2}\hat{c}_{2}+a_{3}^{2}\hat{c}_{3})\left(\frac{x}{2} \right)^{2} }\varphi_{0}^{(1)}\varphi_{0}^{(2)}\varphi_{0}^{(3)}
\end{equation}
or, in explicit form,

\begin{equation}\begin{split}\label{1.2ProdB}
& f(x;a_{1},a_{2},a_{3})=\sum_{r=0}^{\infty}\dfrac{(-1)^{r}}{r!}l_{r}(a_{1}^{2},a_{2}^{2},a_{3}^{2})\left(\dfrac{x}{2} \right)^{2r},\\[1.1ex]
& l_{r}(x_{1},x_{2},x_{3})=r!\sum_{s=0}^{r}\dfrac{x_{3}^{(r-s)}}{(s!)[(r-s)!]^{2}}l_{s}(x_{1},x_{2}).
\end{split}\end{equation}
It is evident that the extension to the case of $n$ Bessel functions writes as in the first of eqs. \ref{1.2ProdB} with
\begin{equation}
l_{r}(x_{1},\dots ,x_{n})=r!\sum_{s=0}^{r}\dfrac{x_{n}^{(r-s)}}{(s!)[(r-s)!]^{2}}l_{s}(x_{1},\dots,x_{n-1}).
\end{equation}
In ref. \cite{Vignat} all the $\alpha$ parameters  (actually the variables of the $l_r$ polynomials) are assumed to be $1$.
\end{cor}

\begin{lem}
From a formal point of view,  the use of the multinomial expansion allows to define the previous family of polynomials as

\begin{equation}
l_{r}(x_{1},\dots,x_{n})=(x_{1}\hat{c}_{1}+\dots +x_{n}\hat{c}_{n})^{r}x_{1}\hat{c}_{1}\dots x_{n}\hat{c}_{n}\varphi_{0}^{(1)}\dots \varphi_{0}^{(n)}
\end{equation}
and, the use of the multinomial expansion yields

\begin{equation}
l_{r}(x_{1},\dots ,x_{n})=\sum_{k_{1}+\dots +k_{n}=r}\binom {r}{k_{1} \;\dots\; \dots \;\dots\; k_{r}}\dfrac{x_{1}^{k_{1}}}{(k_{1}!)^{2}} \dots \dfrac{x_{n}^{k_{n}}}{(k_{n}!)^{2}}.
\end{equation}
\end{lem}

Going back to the two variable case it is easy to check that they satisfy the differential equation

\begin{Oss}
\begin{equation}
\partial _{x_{1}}x_{1}\partial _{x_{1}}l_{r}(x_{1},x_{2})=\partial _{x_{2}}x_{2}\partial _{x_{2}}l_{r}(x_{1},x_{2})=rl_{r-1}(x_{1},x_{2}),
\end{equation}
with $\partial _{x}x\partial _{x}$ l-derivative \ref{prove}, reminding eq. \ref{LagDefPhi}.
\end{Oss}

\begin{cor}
To obtain the extension to the product of arbitrary cylindrical Bessel, it will be sufficient to replace in the previous equations the function $l_{r}(a_{1}^{2}, \dots , a_{n}^{2})$ with $l_{r}^{(\nu_{1}, \dots ,\nu_{n})}(a_{1}^{2}, \dots ,a_{n}^{2})$

\begin{equation}\begin{split}\label{lnurpol}
& l_{r}^{(\nu_{1}, \dots ,\nu_{n})}(x_{1}, \dots ,x_{n})=\hat{c}_{1}^{\nu_{1}} \dots \hat{c}_{n}^{\nu_{n}}(x_{1}\hat{c}_{1}+\dots +x_{n}\hat{c}_{n})^{r}\varphi_{0}^{(1)}\dots \varphi_{0}^{(n)}= \\
& =\sum_{k_{1}+\dots +k_{n}=r}\binom {r}{k_{1} \;\dots \;\dots \;\dots\; k_{r}}\dfrac{x_{1}^{k_{1}}}{k_{1}!\varGamma (\nu_{1}+k_{1}+1)} \dots \dfrac{x_{n}^{k_{n}}}{k_{n}!\varGamma (\nu_{n}+k_{n}+1)}.
\end{split}\end{equation}
Thus getting an expression closely similar to that derived by Brychkov in \cite{Brychkov}
\begin{equation}
\prod_{s=1}^{n}J_{\nu_{s}}(a_{s}x)=\left(\dfrac{x}{2} \right)^{\sum_{s=1}^{n}\nu_{s}}\left(\prod_{k=1}^{n}a_{k}^{\nu_{k}} \right)\sum_{r=0}^{\infty}\dfrac{(-1)^{r}}{r!}l_{r}^{(\nu_{1},\dots , \nu_{n})}\left( a_{1}^{2},\dots ,a_{n}^{2}\right) \left(\dfrac{x}{2} \right)^{2r}.
\end{equation}
\end{cor}

In the case of modified Bessel function of first kind the procedure is the same. The function can be formally expressed as a quadratic exponential and we can recover the results of ref. \cite{Vignat}

\begin{prop}
	
\begin{equation}
\hat{l}_{\nu}(x)=\sum_{r=0}^{\infty}\dfrac{\Gamma (\nu +1)}{r!\Gamma (r+\nu +1)}\left(\dfrac{x}{2} \right)^{2r}=\Gamma (\nu+1)\hat{c}^{\nu}e^{\hat{c}\left(\frac{x}{2} \right)^{2}}\varphi_{0}.
\end{equation}
According to our formalism the relevant \textit{k-th} power reads
\begin{equation}
(\hat{l}_{\nu}(x))^{k}=\Gamma (\nu+1)^{k}\sum_{r=0}^{\infty}\dfrac{1}{r!}l_{r}^{(\nu,\dots , \nu)}(1,\dots , 1)\left(\dfrac{x}{2} \right)^{2r}.
\end{equation}
The polynomials defined in \cite{Vignat} are expressible in terms of our $l_{r}^{(\nu_{1},\dots , \nu_{n})}(x_{1},\dots , x_{n})$ as
\begin{equation}\begin{split}
& B_{r}^{(\nu)}(k)=\Gamma(\nu+1)^{k-1}\Gamma(r+\nu+1)l_{r}^{\{\nu\}}(k),\\[1.1ex]
& l_{r}^{(\nu,\dots , \nu)}(1,\dots , 1)=l_{r}^{\{\nu\}}(k), \\
& l_{r+1}^{\{\nu\}}(k)=\sum_{j=1}^{k}l_{r}^{\{\nu+1_{j}\}}(k),\\
& \{\nu+1_{j},k\}=(\nu,\dots , \nu+1, \dots ,\nu).
\end{split}\end{equation}
\end{prop}

In the forthcoming section we discuss the nature of the polynomials $l_{r}^{\{\nu\}}(k)$.

\subsection{$\mathbf{l_{r}^{\{\nu\}}(k)}$ Polynomials}

We observe that the relevant generating function of $l_{r}^{\{\nu\}}(k)$ polynomials is expressible in terms of product of Bessel like functions, namely

\begin{prop}
	$\forall t\in\mathbb{R}$ 
\begin{equation}
 \sum_{r=0}^{\infty}\dfrac{t^{r}}{r!}l_{r}^{(\nu_{1},\dots , \nu_{n})}(x_{1},\dots , x_{n})=\prod_{j=1}^{n}C_{\nu_{j}}(tx_{j}),
 \end{equation}
where $C_{\nu}(x)$ denotes the Tricomi-Bessel function of order $\nu$ \ref{TrBnu}\footnote{We remind the link between Tricomi-Bessel and Bessel functions \\ \ref{GrindEQ__11_3_Geg} $C_{\nu}(x)=\left(\dfrac{x}{2} \right)^{-\frac{\nu}{2}}J_{\nu}(2\sqrt{x})$.}.
\end{prop}

\begin{prop}
	By using eq. \ref{lnurpol} we find that
	
\begin{equation}\begin{split}
& l_{r}^{\{\nu\}}(k+1)=\hat{c}_{k+1}^{\nu}\hat{c}_{1}^{\nu} \dots \hat{c}_{k}^{\nu}(\hat{c}_{1}+\dots +\hat{c}_{k}+\hat{c}_{k+1})^{r}\varphi_{0}^{(1)}\dots \varphi_{0}^{(k)}\varphi_{0}^{(k+1)}=\\
& =\hat{c}_{k+1}^{\nu}\sum_{j=0}^{r}\binom {r}{j}\hat{c}_{k+1}^{r-j}l_{j}^{\nu}(k)\varphi_{0}^{(k+1)}=\sum_{j=0}^{r}\binom {r}{j}\dfrac{1}{\Gamma (r-j+\nu+1)}l_{j}^{\nu}(k).
\end{split}\end{equation}
The various identities reported in \cite{Vignat} follow from the above equation, which can be generalized in various ways, as e.g.

\begin{equation}\begin{split}
 l_{r}^{\{\nu\}}(k+s)&=\hat{c}_{k+1}^{\nu}\dots\hat{c}_{k+s}^{\nu}\hat{c}_{1}^{\nu}\dots \hat{c}_{k}^{\nu}(\hat{c}_{1}+\dots +\hat{c}_{k}+\hat{c}_{k+1}+\dots +\\
& +\dots+\hat{c}_{k+s})^{r}\varphi_{0}^{(1)}\dots \varphi_{0}^{(k)}\varphi_{0}^{(k+1)}\dots \varphi_{0}^{(k+s)}=\\
& =\sum_{j=0}^{r}\binom {r}{j}l_{r-j}^{\{\nu\}}(s)l_{j}^{\{\nu\}}(k).
\end{split}\end{equation}
\end{prop}

We have noted in eq. \ref{RamanjProdB} that the use of straightforward algebraic manipulations allows the derivation of an expression yielding the integral of the product of two cylindrical Bessel functions. We have checked that the extension to the products of three or more is anyway feasible.
Regarding the case of an integral of the product of three Bessel functions then we find

\begin{exmp}
\begin{equation}\begin{split}
& \int_{-\infty}^{+\infty}f(x;a_{1},a_{2},a_{3})dx=2\sqrt{\pi}l_{-\frac{1}{2}}(a_{1}^2,a_{2}^2,a_{3}^2), \\
& \mid a_{3}\mid > \mid a_{2}\mid > \mid a_{1}\mid ,\\
& l_{-\frac{1}{2}}(a_{1},a_{2},a_{3})=\Gamma\left( \dfrac{1}{2}\right) \sum_{s=0}^{\infty}\dfrac{a_{3}^{-\left(\frac{1}{2}+s \right) }}{(s!)\Gamma\left(\dfrac{1}{2}-s \right)^{2} }l_{s}(a_{1},a_{2}).
\end{split}\end{equation}
In eq. \ref{RamanjProdB} we have recognized that the series defining $l_{-\frac{1}{2}}(a,b)$ can be recognized as that defining a quarter period elliptic integral, in this case we obtain 

\begin{equation}\begin{split}
 l_{-\frac{1}{2}}(a_{1},a_{2},a_{3})&=\dfrac{1}{\sqrt{\pi \mid a_{3}\mid}} F(a_{1},a_{2},a_{3}),\\
 F(a_{1},a_{2},a_{3})&=\sum_{s=0}^{\infty}\left[\dfrac{(2s)!}{2^{2s}(s!)^{2}} \right]^{2}\dfrac{l_{s}(a_{1},a_{2})}{a_{3}^{s}}={}_{2}F_{1}\left(\dfrac{1}{2},\dfrac{1}{2};1;\dfrac{\tilde f}{a_{3}} \right)\chi_{0},\\
& \hat{f}^{r} \chi_{0}=s!l_{s}(a_{1},a_{2}),
\end{split}\end{equation}
namely, we have reduced the series at least formally to the same hypergeometric defining the elliptic integral period. This result can be easily generalized to the case of an arbitrary product.
\end{exmp}

A further element of interest concerns the fact that, since, as already remarked, by replacing $\hat{f}$ with $\hat{c}$ the functions defining the product of Bessel and the Bessel functions are umbral equivalent, we can take advantage from the formalism to establish e.g. the \textit{n-th} derivative of the $f(x;a,b)$ functions. 

\begin{prop}
By noting again that it is formally written as a Gaussian, by the use of property \ref{GHPol} $\hat{D}_{x}^{n}e^{ax^{2}}=H_{n}(2ax,a)e^{ax^{2}}$
we can write the \textit{n-th} derivative of the product of two Bessel functions in terms of the two variable Hermite polynomials $H_{n}(x,y)$ as

\begin{equation}\begin{split}
& \hat{D}_{x}^{n}f(x;a,b)=  \hat{D}_{x}^{n}e^{-\hat{l} \left(\frac{x}{2} \right)^{2} }\Phi_{0}=\\
& =H_{n}\left(-\hat{l} \frac{x}{2},-\dfrac{\hat{l}}{4}\right)e^{-\hat{l} \left(\frac{x}{2} \right)^{2} }\Phi_{0}=(-1)^{n}H_{n}\left(\hat{l} \frac{x}{2},-\dfrac{\hat{l}}{4}\right)e^{-\hat{l} \left(\frac{x}{2} \right)^{2} }\Phi_{0}.
\end{split}\end{equation}
The use of the properties of the $\hat{l}$-operator finally yields the explicit result as

\begin{equation}\begin{split}
& \hat{D}_{x}^{n}f(x;a,b)=\dfrac{(-1)^{n}}{2^{n}}n!\sum_{r=0}^{\lfloor\frac{n}{2}\rfloor}\dfrac{(-1)^{r}x^{n-2r}}{r!(n-2r)!}{}_{(n-r)}f(x;a,b),\\
& {}_{s}f(x;a,b)=\sum_{r=0}^{\infty}\dfrac{(-1)^{r}}{r!}l_{r+s}(a,b)\left(\dfrac{x}{2} \right)^{2r}.
\end{split}\end{equation}
\end{prop}

\begin{rem}
\textit{We just touch now the application of the formalism to the theory of multi-index Bessel functions. We remind that the \textbf{Humbert functions} \cite{BabDatPa} within the present formalism are defined as}

\begin{equation}
I_{m_{1},m_{2}}(x)=\hat{c}_{1}^{m_{1}}\hat{c}_{2}^{m_{2}}e^{\hat{c}_{1}\hat{c}_{2}x}\varphi_0^{(1)}(0)\varphi_0^{(2)}(0)=\sum_{s=0}^{\infty}\dfrac{x^{r}}{r!(m_{1}+r)!(m_{2}+r)!}.
\end{equation}
\textit{The relevant properties are easily deduced, for example we find}
\begin{equation}
\hat{D}_{x}I_{m_{1},m_{2}}(x)=\hat{c}_{1}^{m_{1}+1}\hat{c}_{2}^{m_{2}+1}e^{\hat{c}_{1}\hat{c}_{2}x}\varphi_0^{(1)}(0)\varphi_0^{(2)}(0)=I_{m_{1}+1,m_{2}+1}(x),
\end{equation}
\textit{or, by applying the same integration procedure as before, we obtain}

\begin{equation}\begin{split}\label{BWProdB}
& \int_{-\infty}^{+\infty}I_{0,0}(x)e^{-\beta x^{2}}dx=\sqrt{\dfrac{\pi}{\beta}}I_{0,0}\left(\dfrac{1}{4\beta}\mid 2 \right),\\
& I_{m_{1},m_{2}}(x\mid k)=\sum_{r=0}^{\infty}\dfrac{x^{r}}{r!\Gamma (kr+1+m_{1})\Gamma (kr+1+m_{2})}.
\end{split}\end{equation}
\textit{The second of eq. \ref{BWProdB} is a \textbf{two indexes Bessel-Wright function} \ref{BW} and the Gaussian integral in the first of eq. \ref{BWProdB} can be viewed as the integral transform adopted for their definition.}
\end{rem}

In this Chapter we have shown that our formalism of umbral nature can be exploited to simplify in a significant way the technicalities underlying the theory of Bessel functions and of their manipulations leading to combinations or to the introduction of new forms.

\chapter{Number Theory and Umbral Calculus}\label{ChapterNumberTh}
\numberwithin{equation}{section}
\markboth{\textsc{\chaptername~\thechapter. Number Theory and Umbral Calculus}}{}

In this Chapter we provide examples about the use of umbral calculus in Number Theory, in particular on the reformulation of \textit{Harmonic Numbers} and \textit{Motzkin Numbers} in operatorial form. \\

The original parts of the Chapter, containing their adequate bibliography, are based on the following original papers.\\

\cite{Motzkin} \textit{M. Artioli, G. Dattoli, S. Licciardi, S. Pagnutti; “Motzkin numbers: an operational point of view"; arXiv:1703.07262 2017, submitted for publication  to Online Electronic Integer Sequences, 2017}.\\

\cite{MotzkinWolfram} \textit{M. Artioli, G. Dattoli, S. Licciardi; “Motzkin Numbers and Their Geometrical Interpretation”; Wolfram Demonstrations Project, 2017}.\\

\cite{Harmonic} \textit{G. Dattoli, B. Germano,  S. Licciardi, M.R. Martinelli; “Umbral methods and Harmonic Numbers”, researchgate 2017, submitted for publication to Mediterranean Journal of Mathematics, 2017}.\\

$\star$ \textit{G. Dattoli, S. Licciardi, E. Sabia; “On the properties of Generalized Harmonic numbers” , work in progress}.

\section{Umbral Methods and Harmonic Numbers}  

The theory of \textit{\textbf{harmonic based function}} is discussed here within the framework of umbral operational methods. We derive a number of results based on elementary notions relying on the properties of Gaussian integrals.\\

Methods employing the concepts and the formalism of umbral calculus have been exploited in \cite{Srivastava} to guess the existence of generating functions involving Harmonic Numbers \cite{Sondow}. The conjectures put forward in \cite{Srivastava} have been proven in \cite{Coffee}-\cite{Cvijovic}, further elaborated in subsequent papers \cite{Mezo} and generalized to Hyper-Harmonic Numbers in \cite{Conway}.\\

\noindent In this section we use the same point of view of \cite{Srivastava}, by discussing the possibility of exploiting the formalism developed therein in a wider context.\\

\subsection{Harmonic Numbers and Generating Functions}

We remind that \textbf{\textit{Harmonic Numbers} }are defined as \cite{Sondow}
\begin{equation} \label{GrindEQ__1_HarmNumb} 
h_{n} :=\sum _{r=1}^{n}\frac{1}{r}, \quad \forall n\in\mathbb{N}_0  . 
\end{equation} 
It is furthermore evident that the integral representation for this family of numbers can be derived using a standard procedure, reported below.

\begin{prop} 
In terms of Laplace transform, we obtain 
	
	\begin{equation}\label{intHarmNumb}
	h_{n} =\sum _{r=1}^{n}\int _{0}^{\infty }e^{-s\, r}   ds, \quad \forall n\in\mathbb{N}_0,
	\end{equation}
	thereby getting $n^{th}$ harmonic number through  the Euler's integral \cite{Rochowicz,Lagarias}
	
	\begin{equation}\label{gauHarmNumb}
	h_{n} =\int _{0}^{1}\frac{1-x^{n} }{1-x} dx ,
	\end{equation}
	valid more in general $\forall n\in\mathbb{R}^+$.
\end{prop}

\begin{proof}[\textbf{Proof.}]
	$\forall n\in\mathbb{N}_0$, by applying the Laplace transform, the Theorem of uniform convergence and the sum of a geometric serie, we obtain
	\begin{equation*}\begin{split}
	h_n&=\sum_{r=1}^n\int_0^\infty e^{-sr}ds=
	\int_0^\infty \left[  \left( \sum_{r=0}^n e^{-sr}\right) -1\right] \;ds=\\
	& =\int_0^\infty \dfrac{1-\left( e^{-s}\right)^{n+1}}{1-e^{-s}} -1\;ds=\int_{-\infty}^0 \dfrac{1-\left( e^s\right)^{n+1}}{1-e^s} -1\;ds=\\ 
	& =\int_{-\infty}^{0}\frac{e^{(n+1) s}-e^s }{e^s -1} ds
	\end{split}\end{equation*}
	and by applying the change of variables $e^s\rightarrow x$ we obtain
	\begin{equation*}
	h_n=\int _{0}^{1}\frac{1-x^{n} }{1-x} dx .
	\end{equation*}
\end{proof}
\noindent According to \cite{Lagarias}, from this point onwards, the definition in eq. \ref{gauHarmNumb} can be so extended to \textbf{non-natural} values of \textbf{\textit{n}} and, therefore, it can be exploited as an alternative definition holding for $n$  a \textbf{positive real}.

\begin{defn}
	The function 
	\begin{equation}\label{vacuumHarmNumb}
	\varphi_h(z):=\varphi_{h_z}= \int_0^1 \dfrac{1-x^z}{1-x}dx, \quad \forall z\in\mathbb{R}^+,
	\end{equation}	
	is called \textbf{Harmonic Number Umbral Vacuum}, or simply the vacuum.
\end{defn}
\begin{defn}
	The operator
	\begin{equation}\label{vsopHarmNumb}
	\hat{h}:=e^{\partial_z}
	\end{equation}
	is the \textbf{vacuum shift operator}, being z the domain's variable of the function on which the operator acts.
\end{defn}	
\begin{thm}
	The umbral operator, $\mathbf{\hat{h}^n}$, $\forall n\in\mathbb{R}^+$ defines the harmonic numbers, $h_n$, as the action of the shift operator \ref{vsopHarmNumb} on the HNU-vacuum \ref{vacuumHarmNumb}:
	\begin{equation}
	\left. \hat{h}^n \varphi_{h_z}\right| _{z=0}=h_n ,
	\end{equation}
	or simply	 
	\begin{equation}\begin{split}\label{defOpHarmNumb}  
	& \hat{h}^{n} =h_{n} . \\
	&  h_{0} =0 .
	\end{split}\end{equation}
\end{thm} 
\begin{proof}[\textbf{Proof.}]
	$\forall n\in\mathbb{R}^+$, by applying the shift operator \ref{vsopHarmNumb} on the vacuum \ref{vacuumHarmNumb}, we obtain
	\begin{equation}\begin{split}\label{proofopshiftvacuum}
	 \hat{h}^n \varphi_{h_0}=\left. \hat{h}^n \varphi_{h_z}\right| _{z=0}&=
	\left. e^{n\partial_z}\varphi_{h_z}\right| _{z=0}=
	\left.\varphi_{h _{z+n}}\right| _{z=0}=
	\left. \int_0^1 \dfrac{1-x^{z+n}}{1-x}dx\right| _{z=0}=\\
	& =\int _{0}^{1}\frac{1-x^{n} }{1-x} dx =h_n.
	\end{split}\end{equation}	
\end{proof}
 
\begin{propert}
	$\forall n,m\in\mathbb{R}^+$, we have
	\begin{equation}\begin{split}  \label{GrindEQ__4_HarmNumb} 
	& i)\;\; \hat{h}^{n} \hat{h}^{m} =\hat{h}^{n+m} ;\\
	& ii) \; \left( \hat{h}^{n}\right)^m=\hat{h}^{n\;m}.
\end{split}	\end{equation} 
\end{propert}
\noindent The proof follows from eq. \ref{vsopHarmNumb}. 

\begin{defn} 
	We call \textbf{Harmonic Based Exponential Function} \textit{(HBEF)} the series	
	\begin{equation}\label{defHBEFHarmNumb}   
	{}_{h} e(x):=e^{\;\hat{h}\, x} =1+\sum _{n=1}^{\infty }\frac{h_{n} }{n!} \,  x^{n} ,\quad \forall x\in\mathbb{R} .
	\end{equation} 
\end{defn}
\noindent This function, as already discussed in \cite{Srivastava}, has quite remarkable properties.\\

\noindent The relevant derivatives, $\forall m\in\mathbb{R}^+$, can accordingly  be expressed as (see later the Corollary for further comments) 

\begin{equation} \begin{split}\label{GrindEQ__6_HarmNumb} 
& \left(\frac{d}{dx} \right)^{m} {}_{h} e(x):={}_{h} e(x,m)=\hat{h}^{m} e^{\hat{h}\, x} =h_{m} +\sum _{n=1}^{\infty }\frac{h_{n+m} }{n!} \,  x^{n} , \quad \forall x\in \mathbb{R},\forall m\in\mathbb{N}\\ 
& \left(\frac{d}{dx} \right)^{m} {}_{h} e(x,k)={}_{h} e(x,k+m) , \quad \forall k\in\mathbb{N},
\end{split}\end{equation} 
and, according to eq. \ref{defHBEFHarmNumb} we also find that

\begin{equation} \label{GrindEQ__7_HarmNumb} 
\int _{0}^{\infty }{}_{h} e (-\alpha \, x)\, e^{-x} dx=\int _{0}^{\infty }e^{-(\alpha \, \hat{h}+1)\, x}  dx=\frac{1}{\alpha \, \hat{h}+1} , \quad \mid \alpha\mid<1. 
\end{equation} 

\begin{cor}
	By expanding the umbral function on the r.h.s. of eq. \ref{GrindEQ__7_HarmNumb}, we obtain 
	
	\begin{equation} \label{GrindEQ__8_HarmNumb} 
	\frac{1}{\alpha \, \hat{h}+1} =1+\sum _{n=1}^{\infty }(-1)^{n}  \alpha ^{n} h_{n}, \quad \mid \alpha\mid<1 . 
	\end{equation} 
\end{cor}
\begin{proof}[\textbf{Proof.}]
	Using the Taylor expansion and the eq. \ref{defOpHarmNumb}, $\mid \alpha\mid<1$, we have	
	\begin{equation*}
	\frac{1}{\alpha \, \hat{h}+1} =
	\sum_{n=0}^{\infty}(-\alpha \hat{h})^n=
	1+\sum_{n=1}^{\infty} (-1)^n \alpha^n \hat{h}^n=
	1+\sum _{n=1}^{\infty }(-1)^n\alpha^n h_n ,
	\end{equation*}	
\end{proof}	
\noindent which is an expected conclusion, achievable by direct integration, underscored  here to stress the consistency of the procedure.\\

\noindent A further interesting example comes from the following $"Gaussian''$ integral.

\begin{exmp}
	\begin{equation}\label{krakHarmNumb}
	\int _{-\infty }^{\infty }{}_{h} e (-\alpha \, x)\, e^{-x^{2} } dx=\int _{-\infty}^{\infty }e^{-(\alpha \, \hat{h}\, x+x^{2} )\, }  dx=\sqrt{\pi } e^{\frac{\alpha ^{2} \, \hat{h}^{2} }{4} }, \quad \forall \alpha\in\mathbb{R}.
	\end{equation} 
\end{exmp}
\noindent The last term in eq. \ref{krakHarmNumb} has been obtained by treating $\hat{h}$ as an ordinary algebraic quantity and then by applying the standard rules of the Gaussian integration \ref{GWi}.

\begin{Oss}
We notice that, using the eq. \ref{defHBEFHarmNumb}, we obtain 
\begin{equation}
{}_{h^2}e\left( \dfrac{\alpha^2}{4}\right):=e^{\frac{\hat{h}^2 \alpha^2}{4}}=
 1+\sum _{r=1}^{\infty }\frac{h_{2r} }{r!}  \left(\frac{\alpha }{2} \right)^{2\, r}, \quad \forall \alpha\in\mathbb{R}.
\end{equation}
\end{Oss}
Let us now consider the following slightly more elaborated example, involving the integration of two "Gaussians'', namely the ordinary case and its $HBEF$ analogous.

\begin{exmp}
	\begin{equation} \label{GrindEQ__10_HarmNumb} 
	\int _{-\infty }^{\infty }{}_{h} e (-\, \alpha \, x^{2} )\, e^{-\, x^{2} } dx=\int _{-\infty }^{\infty }e^{-(\, \hat{h}\, \alpha +\, 1)x^{2} \, }  dx=\sqrt{\frac{\pi }{1+\alpha \, \hat{h}} } , \quad\left|\alpha \right|<1 .
	\end{equation}
	
 This last result, obtained after applying elementary rules, can be worded as it follows: the integral in eq. \ref{GrindEQ__10_HarmNumb} depends on the operator function on its r.h.s., for which we should provide a computational meaning. The use of the Newton bynomial yields\\

	\begin{equation}\begin{split} \label{GrindEQ__11_HarmNumb} 
	& \sqrt{\frac{\pi }{1+\alpha \hat{h} } } =\sqrt{\pi } \sum _{r=0}^{\infty }\binom{-\frac{1}{2}}{r} \, \left(\alpha \, \hat{h}\right)^{\, r} =\sqrt{\pi } \left(1+\sqrt{\pi } \sum _{r=1}^{\infty }\frac{\alpha ^{r} h_{r} }{\Gamma \left(\frac{1}{2} -r\right)\, r!}  \, \right), \\[1.1ex] 
	& \left|\alpha \right|<1 .
	\end{split} \end{equation} 
\end{exmp} 

 It is evident that the examples we have provided show that the use of concepts borrowed from umbral theory offers a fairly powerful tool to deal with the ``harmonic based'' functions.\\

\subsection{Harmonic Based Functions and Differential Equations}

In the following we further push the formalism to stress the associated flexibility.\\ 

We note indeed that 

\begin{prop}
The function

\begin{equation} \label{GrindEQ__12_HarmNumb} 
{}_{\sqrt{h} } e(x):=e^{\hat{h}^{\frac{1}{2} } \, x} =1+\sum _{n=1}^{\infty }\frac{\left(\sqrt{\hat{h}} \, x\right)^{n} }{n!} \,  =1+\sum _{n=1}^{\infty }\frac{h_{n/2} }{n!} \,  x^{n} , \quad \forall x\in\mathbb{R} ,
\end{equation} 
defines, $\forall \alpha\in\mathbb{R}$, a $HBEF$ through the following $Gauss$ transform

\begin{equation} \label{GrindEQ__13_HarmNumb} 
\int _{-\infty }^{+\infty }{}_{\sqrt{h} } e (\alpha \, x)\, e^{-x^{2} } dx=\int _{-\infty }^{+\infty }e^{\hat{h}^{\frac{1}{2} } \, \alpha \, x-x^{2} }  dx
=\sqrt{\pi } e^{\hat{h}\, \left(\frac{\alpha }{2} \right)^{2} \, } =\sqrt{\pi }\; {}_{h} e\left(\left(\frac{\alpha }{2} \right)^{2} \right) . 
\end{equation} 
\end{prop}
On the other side, the function \ref{GrindEQ__12_HarmNumb} can be expressed in terms of the $HBEF$, ${}_h e(x)$, using appropriate integral transform methods \cite{Doetsch}.

\begin{defn}
	Let 
	\begin{equation}
	g_{\frac{1}{2} } (\eta )\, =\frac{1}{2\, \sqrt{\pi \eta ^{3} } } e^{-\frac{1}{4\, \eta } }, \quad \forall \eta\in\mathbb{R}^+ 
	\end{equation}
	the Levy distribution of order $\dfrac{1}{2}$, then \cite{Doetsch} 
	
	\begin{equation} \label{GrindEQ__14_HarmNumb} 
	e^{-p^{\frac{1}{2} } \, x} =\int _{0}^{\infty }e^{-p\, \eta \, x^{2} }  \, g_{\frac{1}{2} } (\eta )\, d\eta , \quad \forall p\in\mathbb{R}^+ 
	\end{equation}
	is the associated  Levy integral transform. 
\end{defn}
The use of eq. \ref{GrindEQ__12_HarmNumb} allows to write the identity

\begin{cor}
	\begin{equation}\begin{split} \label{GrindEQ__15_HarmNumb} 
	& {}_{\sqrt{h} } e(-x)=\int _{0}^{\infty }{}_{h} e (-\eta \, x^{2} )\, g_{\frac{1}{2} } (\eta )\, d\eta , \\
	&  g_{\frac{1}{2} } (\eta )\, =\frac{1}{2\, \sqrt{\pi \eta ^{3} } } e^{-\frac{1}{4\, \eta } } .
	\end{split}\end{equation}
\end{cor}
\begin{proof}[\textbf{Proof.}]
	\begin{equation*}
	{}_{\sqrt{h}} e(-x)=e^{-\hat{h}^{\frac{1}{2}}x}=\int_{0}^{\infty}e^{-\hat{h}\eta x^2}g_{\frac{1}{2} } (\eta )\, d\eta =\int_{0}^{\infty}{}_{h} e (-\eta \, x^{2} )\, g_{\frac{1}{2} } (\eta )\, d\eta.
	\end{equation*}
\end{proof} 

\begin{thm}
	The function ${}_h e(x)$ satisfies the first order non homogeneous differential equation
	
	\begin{equation}\label{cauchyHarmNumb} 
	\left\lbrace  \begin{array}{l} {}_h e'(x)=\dfrac{d}{dx}{}_h e(x)={}_h e(x)+\dfrac{e^{x} -x-1}{x},\quad \forall x\in\mathbb{R} \\[1.6ex]
	{}_h e(0)=1 .\end{array}\right.
	\end{equation}
\end{thm}	
\begin{proof}[\textbf{Proof.}]
	Eq. \ref{GrindEQ__6_HarmNumb}, for $m=1$, yields 		
	\begin{equation} \label{GrindEQ__16_HarmNumb} 
	{}_h e'(x):= {}_{h} e(x,1)=1+\sum _{n=1}^{\infty }\frac{h_{n+1} }{n!} \,  x^{n}.  
	\end{equation}
	Being $h_{n+1} =h_{n} +\frac{1}{n+1} $, we find
	\begin{equation}
	{}_{h} e(x,1)=1+\sum _{n=1}^{\infty }\frac{h_{n+1} }{n!} \,  x^{n}  ={}_h e(x)+\dfrac{1}{x}\left(e^x-x-1\right) ,
	\end{equation} 
	hence eq. \ref{cauchyHarmNumb} follows.
\end{proof} 

\begin{cor}
	The solution of  eq. \ref{cauchyHarmNumb} yields for the HBEF the explicit expression in terms of ordinary special functions.
	\begin{equation}\begin{split} \label{GrindEQ__18_HarmNumb} 
	& {}_{h} e(x)=1+e^{z} \left(\ln (x)+E_{1} (x)+\gamma \right), \\ 
	& E_{1} (x)=\int _{x}^{\infty }\frac{e^{-t} }{t}  dt, \\ 
	& \left(\ln (x)+E_{1} (x)+\gamma \right)=-\sum _{n=1}^{\infty }\frac{(-x)^{n} }{n\, n!}  , \\ 
	& \gamma \equiv  Euler\!-\! Mascheroni\; constant  .
	\end{split}\end{equation} 
\end{cor}
\noindent The previous expression is the generating function of harmonic numbers originally derived by Gosper (see \cite{Gosper,Sondow}).\\

\noindent By iterating the previous procedure we find the following general recurrence

\begin{cor}\label{corehxmder}
	\begin{equation}\label{recHarmNumb}
	_{h} e(x,m)={}_{h} e(x)+\sum _{r=0}^{m-1}\left(\frac{d}{dx} \right)^{r}  \frac{e^{x} -1-x}{x}.
	\end{equation}     
\end{cor}

\begin{defn}
	The bynomial expansion 
	\begin{equation} \label{GrindEQ__20_HarmNumb} 
	h_{n} (x):=(x+\hat{h})^{n} =x^n+\sum _{s=1}^{n}\binom{n}{s}\;  x^{n-s}\;h_s  , \quad\forall x\in\mathbb{R},\forall n\in\mathbb{N}_0 
	\end{equation}
	specifies the \textbf{Harmonic Polynomials}.
\end{defn} 
\noindent They are easily shown to be linked to the $HBEF$ by means of the generating function.

\begin{cor}
	\begin{equation} \label{GrindEQ__21_HarmNumb} 
	\sum _{n=0}^{\infty }\frac{t^{n} }{n!}  h_{n} (x)=e^{x\, t} {}_{h} e(t) , \quad \forall x,t\in\mathbb{R}.
	\end{equation} 
\end{cor}
\begin{proof}[\textbf{Proof.}]
	\begin{equation*}
	\sum _{n=0}^{\infty }\frac{t^{n} }{n!}  h_{n} (x)=
	\sum _{n=0}^{\infty }\frac{t^{n} }{n!}(x+\hat{h})^{n}=e^{t(x+\hat{h})}=
	e^{x\, t} {}_{h} e(t) .
	\end{equation*}
\end{proof}

They belong to the family of App\'el polynomials and satisfy the recurrences:

\begin{propert}
	\begin{align}	
	& i) \;\;
	\frac{d}{dx} h_{n} (x)=n\, h_{n-1} (x),\quad \forall x\in\mathbb{R},\label{pr2aHarmNumb}\\[1.1ex] 
	\begin{split} & ii) \;\; h_{n+1} (x)=(x+1)\, h_{n} (x)+f_{n} (x),\label{pr2bHarmNumb} \\
	& \;\;\; \;\; f_{n} (x):=\sum _{s=1}^{n}\frac{n!}{ (n-s)!}  \frac{x^{n-s} }{(s+1)!} =\int _{0}^{1}(x+y)^{n}  dy-x^n,\quad \forall x\in\mathbb{R} .
	\end{split} 
	\end{align} 
\end{propert}
\begin{proof}[\textbf{Proof.}]
	The derivation of eq. \ref{pr2aHarmNumb} is trivial. By regarding eq. \ref{pr2bHarmNumb} we have:
	\begin{equation*}\begin{split}
	h_{n+1}(x)&=(x+\hat{h})(x+\hat{h})^n=(x+\hat{h})\left( x^n+\sum _{s=1}^{n}\binom{n}{s}\;  x^{n-s}\;\hat{h}^s\right)  =\\
	& =x\;h_n+1\cdot x^n+\sum _{s=1}^{n}\binom{n}{s}\;  x^{n-s}\;\hat{h}^{s+1}=\\
	& =x\;h_n (x)+\left( x^n+\sum _{s=1}^{n}\binom{n}{s}\;  x^{n-s}\;\hat{h}^s\right)+ \sum _{s=1}^{n}\dfrac{n!\;x^{n-s}}{(n-s)!(s+1)!} =\\
	& =(x+1)h_n(x)+\sum _{s=1}^{n}\dfrac{n!\;x^{n-s}}{(n-s)!(s+1)!}
	\end{split}	\end{equation*} 
	and
	\begin{equation*}\begin{split}
& 	\sum _{s=1}^{n}\frac{n!}{ (n-s)!}  \frac{x^{n-s} }{(s+1)!}=\left. \sum _{s=1}^{n}\frac{n!}{s!\, (n-s)!}  \frac{x^{n-s} }{s+1}y^{s+1}\right|_{y=1}=\\
	& =\sum _{s=1}^{n}\binom{n}{s}{x^{n-s}}\int_0^1 y^s dy  =\int_0^1 \sum_{s=0}^n \binom{n}{s}x^{n-s}y^s -x^n dy=\\
& = \int_0^1(x+y)^n dy -x^n.
	\end{split}\end{equation*}
\end{proof}

\begin{cor}
	The identity
	
	\begin{equation}
	h_{n} (-1)=(-1)^n \left( 1-\frac{1}{n}\right),\quad \forall n\in\mathbb{N} ,
	\end{equation}
	follows from the eq. \ref{pr2bHarmNumb}
	after setting $x=-1$. \\
	The identity
	\begin{equation} h_{n} =1+\sum _{s=1}^{n}\binom{n}{s}\, h_{s}  (-1),\quad \forall n\in\mathbb{N}_0,
	\end{equation}
	is a consequence of the fact that 
	$\hat{h}^n=((\hat{h}-1)+1)^n$.
\end{cor}                
The harmonic $Hermite$ polynomials (touched on in ref. \cite{Srivastava}-\cite{Coffee}-\cite{Zhukovsky}) can also be written as

\begin{defn} 
	\begin{equation}\begin{split} \label{GrindEQ__24_HarmNumb} 
	& \sum _{n=0}^{\infty }\frac{t^{n} }{n!}\;  {}_{h} H_{n} (x)=e^{x\, t} {}_{h} e(t^{2} ),\quad \forall x,t\in\mathbb{R}, \\ 
	& {}_{h} H_{n} (x):=H_n (x, \hat{h})= e^{\hat{h}\partial_x ^2}x^n =x^n+n!\, \sum _{r=1}^{\lfloor\frac{n}{2} \rfloor}\frac{x^{n-2\; r} \hat{h}^{r} }{(n-2\, r)!\, r!}.
	\end{split} \end{equation}
\end{defn} 

\begin{propert}
	The recurrences identity of the umbral Hermite polynomials
	\begin{equation}\begin{split} \label{GrindEQ__25_HarmNumb} 
	& i)\; \frac{d}{dx} {}_{h} H_{n} (x)=n\, {}_{h} H_{n-1} (x), \quad \forall x\in\mathbb{R}, \\ 
	& ii)\; {}_{h} H_{n+1} (x)=\left(x+2\, \hat{h}\, \frac{d}{dx} \right)\, {}_{h} H_{n} (x)=\left(x+2\frac{d}{dx} \right)\, {}_{h} H_{n} (x)+2\, \alpha '_{n} (x), \\ 
	& \alpha _{n} (x)=n!\sum _{s=1}^{\lfloor\frac{n}{2} \rfloor}\frac{x^{n-2\, s} }{(s+1)! (n-2s)!}, \\ 
	& \alpha '_{n} (x)=\frac{d}{dx} \alpha _{n} (x),
	\end{split} \end{equation}
	are a byproduct of the previous identities and a consequence of the monomiality principle in ref. \cite{Germano}.
\end{propert}

\begin{cor} 
	The umbral Hermite satisfy the \textbf{second order non homogeneous ODE}
	
	\begin{equation} \label{GrindEQ__26_HarmNumb} 
	\left(x\frac{d}{dx} +2\left(\frac{d}{dx} \right)^{2} \right)\, {}_{h} H_{n} (x)=n \;{}_{h}H_{n} (x)-2\, \alpha '_{n} (x) .
	\end{equation}
\end{cor}

\subsection{Truncated Exponential Numbers}

Now, we want to stress the possibility of extending the present procedure to the \textbf{\textit{Truncated Exponential Numbers}}, namely
\begin{equation} \label{GrindEQ__27_HarmNumb} 
e_{n} =\sum _{r=0}^{n}\frac{1}{r!} , \quad \forall n\in \mathbb{N}.  
\end{equation} 
The relevant integral representation writes \cite{PERicci}

\begin{equation} \label{GrindEQ__28_HarmNumb} 
e_{\alpha } :=\frac{1}{\Gamma (\alpha +1)} \int _{0}^{\infty }e^{-s}  (1+s)^{\alpha } ds ,
\end{equation} 
which holds for $\alpha\in\mathbb{R} $, too. For example we find

\begin{exmp}
	\begin{equation}\label{GammamezHarmNumb}
	e_{-\frac{1}{2} } =\frac{e}{\sqrt{\pi } } \Gamma \left( \frac{1}{2},1 \right) ,
	\end{equation}               
	with $\Gamma \left( 1,\frac{1}{2} \right) $ being the truncated Gamma function.
\end{exmp}
According to the previous  discussion and to eq. \ref{GammamezHarmNumb}, setting $\hat{e}^\alpha \leftrightarrow e_{\alpha}$, we also find that
\begin{equation}\begin{split} \label{GrindEQ__30_HarmNumb} 
& \int _{-\infty }^{+\infty }e^{-\hat{e}\, x^{2} }  dx=\sqrt{\pi } e_{-\frac{1}{2} } , \\ 
& e^{-\hat{e}\, x^{2} } =\sum _{r=0}^{\infty }(-1)^{r} \; \frac{e_{r} }{r!}\; x^{2\, r} .
\end{split} \end{equation} 
This last identity is a further proof that the implications offered by the topics treated in this thesis are fairly interesting and provide countless applications.

\section{On the Properties of Generalized Harmonic Numbers}

	In this section we introduce \textbf{\textit{higher order harmonic numbers}} and derive the relevant properties and generating functions by the use of an umbral type technique.\\

In \cite{Srivastava}-\cite{Harmonic} different problems concerning harmonic numbers ($HN$) and the relevant generating functions have been touched. The distinctive feature of these investigations is the use  of a fairly powerful technique, employing an umbral like formalism, which has allowed the framing of the theory of $HN$ within an algebraic context.\\

\noindent Some of the points raised in \cite{Srivastava}-\cite{Harmonic} have been reconsidered, made rigorous and generalized by means of different  technical frameworks in further researches \cite{Coffee}-\cite{Mezo}.\\

The present investigation concerns the application of the method foreseen in \cite{Srivastava}-\cite{Harmonic} to generalized forms of
harmonic numbers like

\begin{equation}\begin{split}\label{key}
& {}_m h_n=\sum_{r=1}^{n}\dfrac{1}{r^m},\quad n>0,\\
& {}_m h_0=0,
\end{split}\end{equation}
namely \textit{\textbf{Higher Order Harmonic Numbers}} ($HOHN$) satisfying the property

\begin{equation}\label{key}
{}_m h_{n+1}={}_m h_n+ \dfrac{1}{(n+1)^m},
\end{equation}
whose associated series provided by the limit $\lim\limits_{n\rightarrow	\infty} {}_m h_n,\; m>1$
is, unlike the ordinary $HN$ ($m=1$), not diverging.\\

In this section we derive a number of not previously known properties and the relevant consequences.\\

As introductory example we provide the following

\begin{exmp}
We consider the	second order HN ($m=2$) and write

\begin{equation}\label{key}
{}_2 h_n=\int_0^1 \dfrac{1-x^n}{x-1}\ln(x)\;dx, \quad \forall n\in\mathbb{N},
\end{equation}
which is obtained after setting

\begin{equation}\label{key}
\dfrac{1}{r^2}=\int_0^\infty e^{-sr}s\;ds,
\end{equation}
noting that

\begin{equation}\label{2HnIntOnTheProp}
{}_2 h_n=\int_0^\infty \dfrac{e^{-s(n+1)}-e^{-s}}{e^{-s}-1}\;s\;ds
\end{equation}
and then by changing variable of integration.\\

It is worth stressing that the integral representation allows the extension of HN to non-integer values of thei ndex. The second order HN stays between integer and real values of the index, as shown in the plot given in Fig. \ref{fig1OnTheProp}, where it is pointed out that the asymptotic limit of the second order harmonic numbers is $\dfrac{\pi^2}{6}$.

\begin{figure}[htp]
	\centering
	\includegraphics[width=0.55\linewidth]{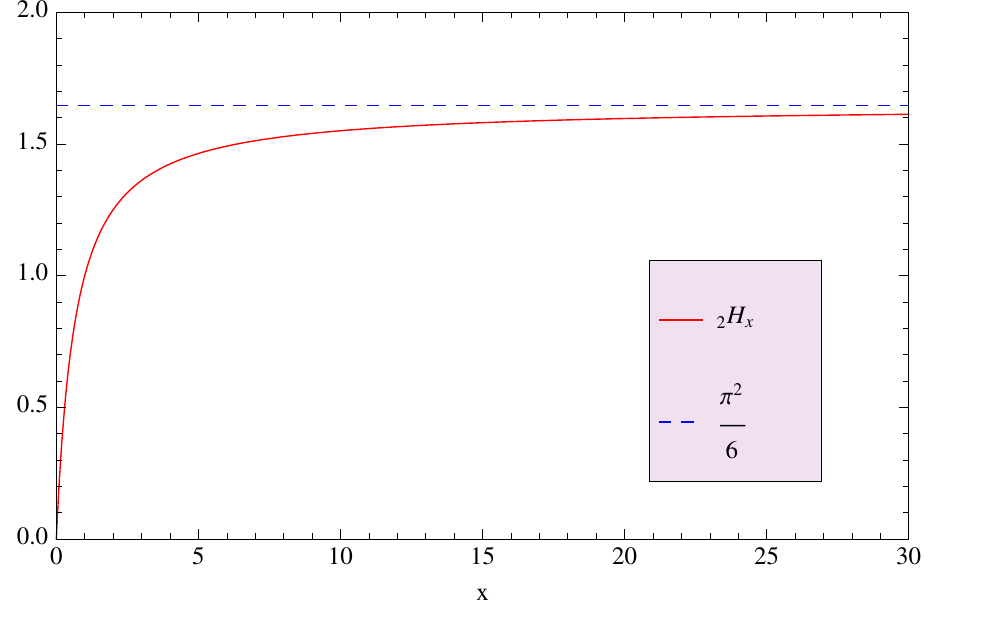}
	\caption{${}_2 h_x$ vs $x$ and $\lim\limits_{x\rightarrow	\infty} {}_2 h_x = \dfrac{\pi^2}{6}$.}\label{fig1OnTheProp} 
\end{figure}
\end{exmp}

The relevant extension to negative real indices will be considered later in the Chapter.\\

Let us first consider the generating function associated with the second order $HN$, which can be cast in the form of an umbral exponential series, namely

\begin{prop}
	We introduce
\begin{equation}\label{genfunOnTheProp}
{}_{{}_2}{}_h e(t):=1+\sum_{n=1}^{\infty}\dfrac{t^n}{n!}\left( {}_2 h_n\right)=e^{{}_2 \hat{h}t} \xi_0,
\end{equation}
where ${}_2 \hat{h}$ is the umbral operator such that

\begin{equation}\begin{split}\label{opOnTheProp}
& {}_2 \hat{h}^n \xi_0:=\xi_n= {}_2 h_n,\quad n>0,\\
& {}_2 \hat{h}^0 \xi_0={}_2 h_0 =1
\end{split}\end{equation}
and

\begin{equation}\label{key}
{}_2 \hat{h}^\nu {}_2 \hat{h}^\mu \xi_0={}_2 \hat{h}^{\nu+\mu}\xi_0, \quad \forall \nu, \mu\in\mathbb{R}.
\end{equation}
\end{prop}

\begin{cor}
From eqs. \ref{2HnIntOnTheProp}-\ref{opOnTheProp} follows that

\begin{equation}\label{deremOnTheProp}
{}_{{}_2}{}_h e(t,m)=\partial_t ^m {}_{{}_2}{}_h e(t)=
\partial_t ^m e^{{}_2 \hat{h}t}={}_2 \hat{h}^m e^{{}_2 \hat{h}t}={}_2 h_m + \sum_{n=1}^\infty \dfrac{t^n}{n!}\left( {}_2 h_{n+m}\right) .
\end{equation}

Limiting ourselves to the first derivative only, it appears evident that the generating function \ref{genfunOnTheProp} satisfies the identity

\begin{equation}\begin{split}\label{EinOnTheProp}
& \partial_t \; {}_{{}_2}{}_h e(t)= {}_{{}_2}{}_h e(t)+f_2 (t),\\
& f_2 (t)=\sum_{n=1}^\infty \dfrac{t^n}{(n+1)(n+1)!}=\dfrac{1}{t}\int_0^{\,t} \dfrac{e^\xi -\xi-1}{\xi}\;d\xi =-\dfrac{Ein(-t)+t}{t} , \\ 
& Ein(z)=\int _{0}^{\,z}\dfrac{1-e^{-\zeta } }{\zeta }  d\zeta, \\ 
& {}_{{}_2}{}_h e(0)=1.
\end{split}\end{equation}
\end{cor}

\begin{prop}
Accordingly to eq. \ref{EinOnTheProp}, the problem of specifying the generating function of second order $HN$ is  reduced to the solution of a first order differential equation, which can be written as

\begin{equation} \label{GrindEQ__9_OnTheProp} 
{}_{{}_2}{}_h e(t)=e^{t} \left( 1+\sum _{n=1}^{\infty }\frac{1}{(n+1)^{2} }  \left(1-e^{-t} e_{n} (t)\right)\right) , 
\end{equation} 
where

\begin{equation} \label{GrindEQ__10_OnTheProp} 
e_{n} (x)=\sum _{r=0}^{n}\frac{x^{r} }{r!}   
\end{equation} 
are the \textbf{truncated exponential polynomials} \cite{PERicci}. They belong to the family of App\'{e}ll type polynomials \cite{Appell} and are defined through the operational identity \cite{Gosper}

\begin{equation}\label{key}
e_{n} (x)=\frac{1}{1-\partial _{x} } \frac{x^{n} }{n!} .
\end{equation}   
\end{prop}

We can further elaborate on the previous identities and set 

\begin{cor}
\begin{equation}\begin{split} \label{GrindEQ__11_OnTheProp} 
& \sum _{n=1}^{\infty }\frac{e_{n} (t)}{(n+1)^{2} }  =Q_2(t), \\ 
& Q_2(t)=\frac{1}{1-\partial _{t} } f_2(t).
\end{split} \end{equation} 
Furthermore, since

\begin{equation} \label{GrindEQ__11_OnTheProp} 
\sum _{n=1}^{\infty }\frac{1}{(n+1)^{2} }  =\frac{\pi ^{2} }{6} -1 ,
\end{equation} 
we end up with 

\begin{equation}\begin{split} \label{GrindEQ__12_OnTheProp} 
& {}_{{}_2}{}_h e(t)=e^{t} \;\Sigma _{2} (t),\\ 
& \Sigma _{2} (t)=\frac{\pi ^{2} }{6} -Q_2(t)\;e^{-t} .
\end{split} \end{equation} 
This new result can be viewed as an extension of the generating function for the first order HN derived by Gosper (see below) \cite{Schmidt}.
\end{cor}

It is furthermore evident that the formalism allows the straightforward derivation of other identities like  

\begin{lem}
\begin{equation} \label{GrindEQ__13_OnTheProp} 
\sum _{n=1}^{\infty }\frac{t^{n} }{n!} \left({}_{2} h_{n+m} \right) =e^{t} \sum _{s=0}^{m}\left(\begin{array}{c} {m} \\ {s} \end{array}\right)\, \Sigma _{2}^{(s)} (t) -{}_{2} h_{m} , 
\end{equation} 
where the upper index ($s$) denotes $s$-order derivative and is a direct consequence of the identity in eq. \ref{deremOnTheProp}.
\end{lem}

The extension to $HOHN$ with $m>2$, follows the same logical steps, namely derivation of the associated Cauchy problem

\begin{cor}
\begin{equation}\label{GrindEQ__14_OnTheProp} 
\left\lbrace  \begin{array}{l} \partial _{t} \left({}_{{}_p}{}_h e(t)\right)= {}_{{}_p}{}_h e(t)+f_{p} (t), \\[1.2ex]
  f_{p} (t)=\sum _{n=1}^{\infty }\frac{t^{n} }{(n+1)^{p-1} \, (n+1)!}  ,\\[1.2ex]
{}_{{}_p}{}_h e(0)=1,\\[1.2ex]
 \forall t\in \mathbb{R}, \forall p\in\mathbb{N}:p>1. \end{array}\right.
\end{equation} 
The Cauchy problem solution writes

\begin{equation} \label{GrindEQ__15_OnTheProp} 
{}_{{}_p}{}_h e(t)=e^{t} \left( 1+\sum _{n=1}^{\infty }\frac{1}{(n+1)^{p} }  \left(1-e^{-t} e_{n} (t)\right)\right)  
\end{equation} 
or

\begin{equation}\begin{split}\label{GrindEQ__16_OnTheProp} 
& {}_{{}_p}{}_h e(t)=e^{t} \;\Sigma _{p} (t), \\[1.1ex] 
& \Sigma _{p} (t)=\zeta \left(p\right)-Q_{p} (t)e^{-t} ,\\ 
& \zeta \left(p\right)=\sum _{n=1}^{\infty }\frac{1}{n^{p} },\\
& Q_{p} (t)=\sum _{n=1}^{\infty }\frac{1}{(n+1)^{p} } e_{n} (t)  =\frac{1}{1-\partial _{t} } f_{p} (t),\\ 
& f_{p} (t)=\sum _{n=1}^{\infty }\frac{t^{n} }{(n+1)^{p-1} \, (n+1)!}  .
\end{split} \end{equation} 
The case $p=1$ should be treated separately, because the sum on the r.h.s. of eq. \ref{GrindEQ__15_OnTheProp} is apparently diverging.
\end{cor}

It is accordingly worth noting that, since

\begin{Oss}
\begin{equation} \label{GrindEQ__17_OnTheProp} 
f_{1} (t)=\sum _{n=1}^{\infty }\frac{t^{n} }{\, (n+1)!}  =\frac{1}{t} (e^{t} -t-1) ,
\end{equation} 
we find

\begin{equation}\begin{split}\label{GrindEQ__18_OnTheProp} 
& {}_{{}_1}{}_h e(t)=e^{t} \left(1+\int _{0}^{t}\frac{1-(\tau +1)\, e^{-\tau } }{\tau }  d\tau \right) =e^{t}\; \Sigma _{1} (t), \\ 
& \Sigma _{1} (t)=e^{-t} +Ein(t).
\end{split} \end{equation} 
Eq. \ref{GrindEQ__18_OnTheProp} is a restatement of the Gosper derivation of the generating function of first order $HN$. 
\end{Oss}


We conclude this paragraph by introducing the following

\begin{defn}
	We introduce HOHN umbral polynomials (for $m=1$ see ref. \cite{Harmonic}) 

\begin{equation}\begin{split}\label{GrindEQ__19_OnTheProp} 
& {}_{m} H_{n} (x)=(x+{}_{m} \hat{h})^{n} =1+\sum _{s=1}^{n}\binom{n}{s}\;x^{n-s} {}_{m} H_{s} , \\ 
& {}_{m} H_{0}(x) =1.
\end{split} \end{equation} 
\end{defn}

\section{Motzkin Numbers: an Operational Point of View}\label{secMotz}

The \textbf{\textit{Motzkin Numbers}} can be derived as coefficients of hybrid polynomials. Such an identification allows the derivation of new identities for this family of numbers and offers a tool to investigate previously unnoticed links with the theory of special functions and with the relevant treatment in terms of operational means. The use of umbral methods opens new directions for further developments and generalizations, which leads, e.g., to the identification of new Motzkin associated forms.\\

A very well known example of link between special numbers and special polynomials is provided by the so called convolution or Telephone numbers, which can be expressed in terms of Hermite polynomials coefficients \cite{Riordan}. In ref \cite{Artioli} the Padovan and Perrin numbers \cite{Perrin,Padovan} can be recognized to be associated with particular values of two variable Legendre polynomials \cite{DattoliArt}.\\

Motzkin numbers \cite{WeissteinMotz} have been also discussed in connection with a family of hybrid polynomials \cite{Blasiak,Lorenzutta} and the relevant properties have accordingly been studied.\\

The hybrid polynomials are indeed defined as \cite{Lorenzutta} 

\begin{equation}\label{KnMotz}
P_{n}^{(q)}(x,y)=n!\sum_{r=0}^{\lfloor\frac{n}{2}\rfloor }\dfrac{x^{n-2r}y^{r}}{(n-2r)!r!(r+q)!},
\end{equation}
and the relevant generating function reads $\forall t\in\mathbb{R}$

\begin{equation}\label{genKnMotz}
\sum_{n=0}^{\infty}\dfrac{t^{n}}{n!}P_{n}^{(q)}(x,y)=\dfrac{I_{q}(2 \sqrt{y}\; t)}{(\sqrt{y}\;t)^{q}}e^{xt},
\end{equation}
where $I_{q}(x)$  is the modified Bessel function of first kind of order $q$ \ref{mBffk}.\newline

Within the present framework, the Motzkin numbers sequence can be specified as \cite{Blasiak}              

\begin{equation}\begin{split}\label{key}
& m_{n} = P_{n}^{(1)}(1,1)=\sum_{s=0}^{n}m_{n,s},\\
& m_{n,s}=\binom{n}{s}\;f_{s},\\
& f_{s} = \dfrac{s!}{\Gamma\left( \dfrac{s}{2}+2\right) \Gamma\left( \dfrac{s}{2}+1\right) }\left| \cos\left( s\dfrac{\pi}{2}\right) \right|, 
\end{split}\end{equation}
where the coefficients $m_{n,s}$ can be represented as the triangle reported below \newline

\begin{table}[htbp]
	\centering
	\begin{tabular}{||c | c|c|c|c|c|c|c|c|c||c||} \hline
		\multicolumn{10}{||c||}{\centering $m_{n,s}$ \textbf{coefficients}}& $ m_{n}\; \textbf{Motzkin} $ \\ \hline
		\toprule
		\multirow{2}{*}{\textbf{Parameter}}  & \multicolumn{9}{c||}{\textbf{s}}&\multirow{2}{*}{$\sum_{s=0}^{n}m_{n,s}$}\\ 
		\cmidrule(lr){3-10}
		&  & \textbf{0} & \textbf{1} & \textbf{2 } & \textbf{3} & \textbf{4} & \textbf{5} & \textbf{6} & \textbf{7}&\\  
		\midrule	
		\multirow{9}{*}{\textbf{n}}  & \textbf{0} & 1 &  &  &  & & & & & \textbf{1} \\ 
		& \textbf{1}  & 1 & 0 &  &  & & &  & & \textbf{1} \\ 
		& \textbf{2} & 1 & 0 & 1 & & & &  & & \textbf{2} \\ 
		& \textbf{3} & 1 & 0 & 3 & 0 &  &  & & & \textbf{4}\\ 
		& \textbf{4} & 1 & 0 & 6 & 0 & 2 &  & & & \textbf{9}\\ 
		& \textbf{5} & 1 & 0 & 10 & 0 & 10 & 0  &  &  & \textbf{21}\\ 
		& \textbf{6} & 1 & 0 & 15 & 0 & 30 & 0 & 5 &  & \textbf{51}\\ 
		& \textbf{7} & 1 & 0 & 21 & 0 & 70 & 0 & 35 & 0  & \textbf{127}\\ 
		& \dots &  \dots & \dots & \dots    & & &   & & & \dots\\ \hline
		\bottomrule
	\end{tabular}
	\caption{Motzkin Numbers and their  Coefficients.}
	\label{table1Motz}
\end{table}

\noindent in which $m_{n,2}$ corresponds, in OEIS, to the sequence $A000217$, $m_{n,4}$ to $A034827$, $m_{n,6}$ to $A000910$ and so on.\newline  

According to eq. \ref{genKnMotz}, the Motzkin numbers can also be defined as the coefficients of the following series expansion

\begin{equation}\label{genMnMotz}
\sum_{n=0}^{\infty}\dfrac{t^{n}}{n!}m_{n}=\dfrac{I_{1}(2  t)}{t}e^{t}.
\end{equation}

In the following we show how some progresses in the study of the relevant properties can be done by the use of a formalism of umbral nature.\newline

\subsection{Motzkin Numbers and Umbral Calculus}

In order to simplify most of the algebra associated with the study of the properties of the Motzkin numbers and to get new relevant identities, we apply our methods of umbral nature. \\

To this aim we remind umbral form of $\nu$-order Tricomi-Bessel function \ref{TrBnu}, $C_{\nu}(x)=\hat{c}^{\nu} e^{-\hat{c}x}\varphi_{0}$. We note that by the use of eqs. \ref{GrindEQ__11_3_Geg} $C_{\nu } (x)=\left(\dfrac{1}{x} \right)^{\frac{\nu }{2} } J_{\nu } (2 \sqrt{x} )$ and \ref{mBffkJnu} $I_{\nu}(ix)=e^{i\nu\frac{\pi}{2}}J_{\nu}(x)$, we can write $\forall q\in\mathbb{Z}$ the identity 

\begin{equation}\label{TricMotz}
C_{q}(-x)=\dfrac{I_{q}(2 \sqrt{x})}{(\sqrt{x})^{q}}=\sum_{r=0}^{	\infty}\dfrac{x^{r}}{r!(q+r)!}.
\end{equation}

\begin{lem}
The use of this formalism allows to restyle the hybrid polynomials in the form

\begin{equation}\label{KnOpMotz}
P_{n}^{(q)}(x,y)=\hat{c}^{q}  H_{n}(x,\hat{c}\;y),
\end{equation}
where $H_{n}(x,y)$ are the two variables HP.
\end{lem}

We can accordingly use the wealth of properties of this family of polynomials to derive further and new relations regarding those of the Motzkin numbers family, e.g.

\begin{prop}
 By recalling the generating function \ref{Hnlgf}\\ $\sum_{n=0}^{\infty}\dfrac{t^{n}}{n!}H_{n+l}(x,y)=H_{l}(x+2yt,y)e^{xt+yt^{2}}$, we find

\begin{equation}\label{genmnlMotz}
\sum_{n=0}^{\infty}\dfrac{t^{n}}{n!}m_{n+l}=\hat{c}\; H_{l}(1+2\hat{c}t,\hat{c})e^{t+\hat{c}t^{2}},
\end{equation}
which, after using eq.  \ref{GrindEQ__11_3_Geg}, finally yields

\begin{equation}\begin{split}\label{muMotz}
& \sum_{n=0}^{\infty}\dfrac{t^{n}}{n!}m_{n+l}=\mu_{l}(t)\;e^{t},\\
& \mu_{l}(t)=l!\sum_{r=0}^{\lfloor\frac{l}{2}\rfloor }\dfrac{1}{r!}\sum_{s=0}^{l-2r}\dfrac{2^{s}}{s!(l-2r-s)!}\dfrac{I_{s+r+1}(2t)}{t^{r+1}}.
\end{split}\end{equation}
\end{prop}

Furthermore, the same procedure and the use of the Hermite polynomials duplication formula \cite{L.C.Andrews}

\begin{equation}\label{HdMotz}
H_{2n}(x,y)=\sum_{r=0}^{n}\binom{n}{r}^{2}r!\;(2y)^{r}\left( H_{n-r}(x,y)\right)^{2}, 
\end{equation}
yields the following identity for Motzkin numbers

\begin{prop}
\begin{equation}\begin{split}\label{mduplMotz}
m_{2n}& =\hat{c} \; \sum_{r=0}^{n}r!\;\binom{n}{r}^{2}(2\hat{c})^{r} \;  H_{n-r}(1,\hat{c}) \, H_{n-r}(1,\hat{c})=\\
& = \sum_{r=0}^{n}\binom{n}{r}^{2}2^{r}r!(n-r)!\sum_{s=0}^{\lfloor\frac{n-r}{2}\rfloor }\dfrac{m_{n-r}^{(r+s+1)}}{(n-r-2s)!s!},
\end{split}\end{equation}
where

\begin{equation}\label{mnmopMotz}
m_{n}^{(q)}=P_{n}^{(q)}(1,1)=\hat{c}^{q} \; H_{n}(1,\hat{c})
\end{equation}
are \textbf{associated Motzkin numbers} \cite{Blasiak}.
\end{prop}

\begin{cor}
The identification of Motzkin numbers as in eq. \ref{mnmopMotz}, along with the use of the recurrences of Hermite polynomials, yields, e.g., the identities   

\begin{equation}\begin{split}\label{recmnMotz}
& m_{n+1}^{(q)}=m_{n}^{(q)}+2\;n\;m_{n-1}^{(q+1)},\\
& m_{n+p}=\sum_{s=0}^{\min[n,p]}2^{s}s!\;\binom{p}{s}\;\binom{n}{s}M_{p-s,\;n-s,\;s},\\
& M_{p,\;n,\;t}=p!\sum_{r=0}^{\lfloor\frac{p}{2}\rfloor}\dfrac{m_{n}^{(t+r+1)}}{(p-2r)!r!},
\end{split}\end{equation}
in which, the second identity, has been derived from the Nielsen formula for $H_{n+m}(x,y)$ \cite{Nielsen}.
\end{cor}

\subsection{Telephone Numbers}

In this last part we have shown that a fairly straightforward extension of the formalism put forward in ref. \cite{Blasiak}, allows non trivial progresses in the theory of Motzkin numbers. Further relations can be easily obtained by applying the method we have envisaged as, e.g., 

\begin{lem}
\begin{equation}\label{prmnMotz}
\sum_{s=0}^{n}m_{n-s}\;m_{s}=2\;(n+1)\;m_{n}^{(2)}
\end{equation}
is a discrete self-convolution of Motzkin numbers.
 \end{lem}

We have also mentioned the existence of the associated Motzkin numbers 

\begin{equation}\label{mnqMotz}
m_{n}^{(q)}=P_{n}^{(q)}(1,1),
\end{equation}
touched on in ref. \cite{Blasiak}. In  the present context they have been introduced on purely algebraic grounds.  Strictly speaking they are not integers and therefore they are not amenable for a combinatorial interpretation . However, redefining them through 

\begin{lem}
	We recast the associated Motzkin numbers as
\begin{equation}\label{tildemMotz}
\tilde{m}_{n}^{(q)}=\dfrac{(n+q)!}{n!}P_{n}^{(q)}(1,1).
\end{equation}
\end{lem}
\noindent we obtain for $q=2$ the sequences in OEIS  $(A014531)$, while for $q=3$ the sequences $(A014532)$ and so on.\\

We have mentioned in section \ref{secMotz} the  theory of \textit{\textbf{Telephone numbers}}  $T(n)$ \cite{Knuth}, whose importance in chemical Graph theory has been recently emphasized in ref. \cite{Hatz}.  As well known, they can be expressed in terms of ordinary Hermite polynomials, however the use of the two variable extension is more effective. They can indeed be expressed as $T(n)=H_{n}(1,\frac{1}{2})$.

\begin{lem}
The use of Hermite polynomials properties, like the index duplication formula, yields

\begin{equation}\label{T2nMotz}
T(2n)=\sum_{r=0}^{n}r!\;\binom{n}{r}^{2}\;T(n-r)^{2}.
\end{equation}
\end{lem}

\begin{prop}
The use of the Hermite numbers $h_{s}$ \cite{G.Dattoli} allows the derivation of the following further expression

\begin{equation}\begin{split}\label{hsMotz}
& T(n)=\sum_{s=0}^{n}t_{n,s},\\
& t_{n,s}=\binom{n}{s}\;h_{s}\left( \dfrac{1}{2}\right) ,\\
& h_{s}(y)=y^{\frac{s}{2}}\Gamma\left( \dfrac{s}{2}+2\right)f_{s}=\dfrac{y^{\frac{s}{2}}s!}{\Gamma\left( \dfrac{s}{2}+1\right) }\left| \cos\left( s\;\dfrac{\pi}{2}\right) \right|.  
\end{split}\end{equation}
\end{prop}
The coefficients $t_{n,s}$ of the telephone numbers can arranged in the following triangle

\begin{table}[htbp]
	\centering
	\begin{tabular}{||c | c|c|c|c|c|c|c|c|c||} \hline
		\multicolumn{10}{|c|}{\centering $t_{n,s}$ \textbf{coefficients}} \\ \hline
		\toprule
		\multirow{2}{*}{\textbf{Parameter}}  & \multicolumn{9}{c||}{\textbf{s}}\\ 
		\cmidrule(lr){3-10}
		&  & \textbf{0} & \textbf{1} & \textbf{2 } & \textbf{3} & \textbf{4} & \textbf{5} & \textbf{6} & \textbf{7}\\  
		\midrule	
		\multirow{9}{*}{\textbf{n}}  & \textbf{0} & 1 &  &  &  & & & & \\ 
		& \textbf{1}  & 1 & 0 &  &  & & &  &\\ 
		& \textbf{2} & 1 & 0 & 1 & & & &  &\\ 
		& \textbf{3} & 1 & 0 & 3 & 0 &  &  & &\\ 
		& \textbf{4} & 1 & 0 & 6 & 0 & 3 &  & &\\ 
		& \textbf{5} & 1 & 0 & 10 & 0 & 15 & 0  &  & \\ 
		& \textbf{6} & 1 & 0 & 15 & 0 & 45 & 0 & 15 & \\ 
		& \textbf{7} & 1 & 0 & 21 & 0 & 105 & 0 & 105 & 0 \\ 
		& \dots &  \dots & \dots & \dots    & & &   & & \\ \hline
		\bottomrule
	\end{tabular}
	\caption{Telephone Number Coefficients}
	\label{table2Motz}
\end{table}
in which the numbers belonging to the column $s=4$ $(3, 15, 45, 105, 210,…)$  are identified with OEIS $A050534$ and $s=5$ $(15, 105, 420,1260, 3150,…)$ is just a multiple of $A00910$.  The use of the identification with two variable Hermite polynomials opens further perspectives, by exploiting indeed the higher order HP (see section \ref{higherHermite}) we can introduce the following 

\begin{prop}
	We provide a generalization of telephone numbers 

\begin{equation}\label{TnmMotz}
T_{n}^{(m)}=H_{n}^{(m)}\left( 1,\dfrac{1}{m}\right), 
\end{equation} 
with generating function

\begin{equation}\label{gentnMotz}
\sum_{n=0}^{\infty}\dfrac{t^{n}}{n!}T_{n}^{(m)}=e^{t+\frac{1}{m}t^{m}},
\end{equation}
which satisfy the recurrence

\begin{equation}\label{tnmpMotz}
T_{n+1}^{(m)}=T_{n}^{(m)}+\dfrac{n!}{(n-m+1)!}T_{n-m}^{(m)}.
\end{equation}
\end{prop}
In the case of $m=3$ the numbers $T_{n}^{(3)}= (1, 1, 1, 3, 9, 21, 81, 351, 1233…)$ are identified with OEIS $A001470$, while for $m=4$, the series
$1, 1, 1, 1, 7, 31, 91, 211,$ $1681, 12097$ corresponds to $A118934$. For $m=5$ the associated series appears to be $A052501$ but should be more appropriately identified with the coefficients of the expansion \ref{TnmMotz}, finally the sequence $n=6$ is not reported in OEIS.\newline

A more accurate analysis of this family of numbers and the relevant interplay with Motzkin will be discussed elsewhere.

\begin{appendices}
	\chapter{}\label{AppA}

 The Appendix reported below covers some topics touched in the main body of the Chapter \ref{Chapter1} without the necessary details.

\begin{itemize}
\item [$\star$]We in particuar devote \ref{AppABorel} to the \textit{Lamb-Bateman} equation.
\item [$\star$] Appendix \ref{AppARMT} is addressed to \textit{Ramanujan Master Theorem} and to its use within the framewor of umbral theory. 
\item [$\star$]  \ref{AppAML}  is devoted to computational details relevant to the \textit{Mittag-Leffler Function} and to its applications.
\end{itemize}

\section{Lamb-Bateman equation}\label{AppABorel}

The umbral formalism, or better the umbral operational methods developed in this thesis, can be put on rigorous basis by exploiting different mathematical procedures. \\

\noindent We have already noted that by the realization \ref{propertCb} of the umbral operator and of the associated vacuum one can obtain in a fairly straightforward way the set of rules underlying our point of view to umbral calculus.\\

\noindent Defining the umbral operator in terms of shift operators, we could write, e.g,  the $0$-order cylindrical Bessel function  \ref{J0op} as 

\begin{equation*}\begin{split}\label{key}
J_0(x)&=e^{-\left( \frac{x}{2}\right)^2 e^{\partial_z}}\dfrac{1}{\Gamma(z+1)}\mid_{z=0}=
\sum_{r=0}^{\infty}\dfrac{(-1)^r}{r!}\left( \dfrac{x}{2}\right)^{2r}e^{r\partial_z}\dfrac{1}{\Gamma(z+1)}\mid_{z=0}=\\
& =\sum_{r=0}^{\infty}\dfrac{(-1)^r}{r!}\left( \dfrac{x}{2}\right)^{2r}\dfrac{1}{\Gamma(z+r+1)}\mid_{z=0}=
\sum_{r=0}^{\infty}\dfrac{(-1)^r}{(r!)^2}\left( \dfrac{x}{2}\right)^{2r} .\\
\end{split}\end{equation*}

Regarding the integration procedure we have also notet that \ref{IntJ0cop}

\begin{equation*}\label{key}
\int_{-\infty }^{\infty}J_0(x)dx=\sqrt{\pi}\;e^{-\frac{1}{2}\partial_z}\dfrac{1}{\Gamma(z+1)}\mid_{z=0}=
\sqrt{\pi}\dfrac{1}{\Gamma\left( z+\frac{1}{2}\right) }\mid_{z=0}=
\sqrt{\pi}\dfrac{1}{\left( \frac{\sqrt{\pi}}{2}\right) }=2.
\end{equation*}
An interesting byproduct of this point of view is the solution of the \textit{\textbf{Lamb-Bateman}} equation \cite{Martin} originally proposed to treat the diffraction of a solitary wave \cite{Lamb}, namely

\begin{equation}\label{key}
\int_{0}^{\infty}u(x-y^2)dy=f(x).
\end{equation}
The unknown function $u(x)$  can be obtained by noting that

\begin{equation}\label{key}
u(x-y^2)=e^{-y^2 \partial_{x}}u(x),
\end{equation}
accordingly we find

\begin{equation}\label{key}
\left( \int_{0}^{\infty}e^{-y^2 \partial_{x}}dy\right) u(x)=f(x).
\end{equation}
Furthermore, since by using \ref{GWi} we get

\begin{equation}\label{key}
\int_{0}^{\infty}e^{-y^2 \partial_{x}}dy=\dfrac{1}{2}\sqrt{\pi}\partial_{x}^{-\frac{1}{2}},
\end{equation}
we end up with the solution

\begin{equation}\label{key}
u(x)=\dfrac{2}{\sqrt{\pi}}\left(\partial_{x}^{\;\frac{1}{2}} f(x) \right) ,
\end{equation}
namely, the solution of the Lamb-Bateman problem is just the derivative of order $\dfrac{1}{2}$ of the function  $f(x)$.

\section{The Ramanujan Master Theorem}\label{AppARMT}

We have used the umbral methods to evaluate integrals and we have obtained a framing of the procedure within an adequate mathematical context.
The results we have obtained traces however back to the so called \textbf{\textit{Ramanujan Master Theorem}} $(RMT)$ \cite{On Ramanujan}, which is illustrated in this Appendix.\\

The Theorem can be stated as it follows.

\begin{thm}
	\textbf{(RMT)} If the function $ f(x) $ satisfies the series expansion
	\begin{equation}\label{espSerie}
	\mathbf{f(x)=\sum_{n=0}^{\infty}f_{n}\dfrac{(-1)^{n}x^{n}}{n!}}
	\end{equation}
	in a neighborhood of the origin, with
	$f_0=\varphi_0$, then 
	
	\begin{equation}\label{RMTint}
	\int_{0}^{\infty} x^{\nu-1}f(x)\;dx=\Gamma(\nu)f_{-\nu}\;.
	\end{equation}
\end{thm}
The use of the umbral formalism (see e.g. the property \ref{Opc}) allows indeed to write 

\begin{equation}\label{key}
f(x)=e^{-\hat{c}\;x}\;f_{0},
\end{equation}
being $\hat{c}^n \;f_{0}:=f_n$ and therefore the integral \ref{RMTint} in the form 

\begin{equation}
\int_{0}^{\infty} x^{\nu-1}e^{-\hat{c}x}\;dx\;f_{0}=\Gamma(\nu)\hat{c}^{-\nu}f_{0}=\Gamma(\nu)f_{-\nu}.
\end{equation}
The same procedure can be used to prove 

\begin{equation}\begin{split}
& \int_{0}^{\infty} x^{\nu-2}f(x)dx=\dfrac{1}{m}\Gamma\left( \dfrac{\nu-1}{m}\right)f_{\frac{1-\nu}{m}},   \\[1.1ex]
& f(x)=\sum_{n=0}^{\infty}f_{n}\dfrac{(-1)^{n}x^{mn}}{n!}\;,
\end{split}\end{equation}
obtained after noting that

\begin{equation}\label{key}
\int_{0}^{\infty} x^{\nu-2}e^{-\hat{c}\;x^m}dx\;f_{0}=\dfrac{1}{m}\int_{0}^{\infty} t^{\left( \frac{\nu-1}{m}-1\right) }e^{-\hat{c}\;t}dt\;f_{0}=\dfrac{1}{m}\Gamma\left(\dfrac{\nu-1}{m} \right) f_{-\frac{\nu-1}{m}}.
\end{equation}

We can use, e.g., the $RMT$ to prove the identities

\begin{equation}
\int_{0}^{\infty} x^{\nu-1}C_{\alpha}(x)dx=\dfrac{\Gamma(\nu)}{\Gamma(\alpha-\nu+1)},
\end{equation}
where $C_{\alpha}(x)$ is the Tricomi-Bessel function of order $\alpha$ (see \ref{TrBalfa1}) which in umbral form \ref{TrBnu} writes $C_\nu(x)=\hat{c}^{\;\nu} e^{-\hat{c}\, x}\varphi_{0}$. We note in fact that, according to the umbral formalism,

\begin{equation}\label{key}
\int_{0}^{\infty} x^{\nu-1}C_{\alpha}(x)dx=\hat{c}^{\alpha}\int_{0}^{\infty} x^{\nu-1}e^{-\hat{c}\;x}dx\;\varphi_{0}=\Gamma(\nu)\hat{c}^{\alpha-\nu}\varphi_{0}=\dfrac{\Gamma(\nu)}{\Gamma(\alpha-\nu+1)}.
\end{equation}

\section{Mittag Leffler Function and Computational Technicalities}\label{AppAML}

$\star$ Proof of eq. \ref{MLBW}

\begin{equation*}\begin{split}
& I_{\alpha,\;\beta}=\int_{-\infty}^{\infty}E_{\alpha,\beta}(-x^2) dx=\left( \sqrt{\pi}\hat{c}^{\beta-\frac{\alpha}{2}-1}\int_{0}^{\infty}e^{-s}s^{-\frac{1}{2}}ds\right) \varphi_{0}=\\
& =\sqrt{\pi}\;\Gamma\left( \dfrac{1}{2} \right)\hat{c}^{\beta-\frac{\alpha}{2}-1}\varphi_{0} =\dfrac{\pi}{\Gamma\left(\beta-\frac{\alpha}{2}\right) }, \quad \forall \alpha,\beta\in\mathbb{R}^+. 
\end{split}\end{equation*}

\begin{proof}[\textbf{Proof.}]
The proof follows the same lines we have outlined in the main body of the Chapter and in these Appendices. The $ML$ Gaussian like integral $\forall \alpha,\beta\in\mathbb{R}^+$ can, accordingly (by the use of eqs. \ref{EW}, \ref{BWamu}, \ref{GWi},  \ref{FunzGamma}, \ref{Gpropb} and \ref{Opc}), be written as 

\begin{equation}\begin{split}\label{solMLBW}
 I_{\alpha,\;\beta}&=\int_{-\infty}^{\infty}E_{\alpha,\beta}(-x^2) dx=
\int_{-\infty}^{\infty}\int_{0}^{\infty}e^{-s}W_{\beta-1}^{\alpha}(-x^2 s)ds\;dx=\\
& =
\int_{-\infty}^{\infty}\int_{0}^{\infty}e^{-s}\hat{c}^{\beta-1}e^{-\hat{c}^{\alpha}x^2 s}\varphi_0\; ds\;dx=
 \hat{c}^{\beta-1}\int_{-\infty}^{\infty}\int_{0}^{\infty}e^{-s}e^{-\hat{c}^{\alpha}x^2 s} ds\;dx\; \varphi_0=\\
 & =\hat{c}^{\beta-1}\int_{0}^{\infty}e^{-s}ds\int_{-\infty}^{\infty}e^{-\hat{c}^{\alpha} s\; x^2 }dx\; \varphi_0=
\hat{c}^{\beta-1}\int_{0}^{\infty}e^{-s}ds\; \sqrt{\dfrac{\pi}{\hat{c}^{\alpha} s}}\;\varphi_0=\\
& = \sqrt{\pi}\;\hat{c}^{\beta-\frac{\alpha}{2}-1}\int_{0}^{\infty}e^{-s}s^{-\frac{1}{2}}ds\;\varphi_{0}=
\sqrt{\pi}\;\Gamma\left( \dfrac{1}{2} \right)\hat{c}^{\beta-\frac{\alpha}{2}-1}\varphi_{0}=\\
& =\dfrac{\pi}{\Gamma\left(\beta-\frac{\alpha}{2}\right) }. 
\end{split}\end{equation}
\end{proof}

\textbf{Lemma 5} \textit{Let $E_{n,1}(x)$ the \textit{ML} function, $\forall n\in\mathbb{N}$, $\forall x\in\mathbb{R}$ and let $\lambda \in \mathbb{R} \Rightarrow E_{n,1}(x)$ is eigenfunction of $\partial_{x}^n$.}\\

\begin{proof}[\textbf{Proof.}]
	$\forall n\in\mathbb{N},\; \forall x,\lambda\in\mathbb{R}$
\begin{equation}
	\begin{split}
 \partial_{x}^n E_{n,1}(\lambda x^n)&= 
  \partial_{x}^n \sum_{r=0}^{\infty} \dfrac{\lambda^r x^{nr}}{\Gamma(nr+1)}=
  \sum_{r=0}^{\infty}\partial_{x}^n  \dfrac{\lambda^r x^{nr}}{\Gamma(nr+1)}=\\[1.1ex] 
& = \sum_{r=0}^{\infty}\lambda^r \dfrac{(nr)!}{(nr-n)!} \dfrac{ x^{nr-n}}{\Gamma(nr+1)}
 =\lambda \sum_{r=1}^{\infty} \dfrac{\lambda^{r-1} x^{n(r-1)}}{(n(r-1))!}=\lambda E_{n,1}(\lambda x^n).
\end{split}
\end{equation}
We have interchanged the order derivative and summation in view of the uniform convergence of the series defining the \textit{ML} function.
\end{proof}
\vspace{1.1cm}
\noindent \textbf{Corollary 4}
$\forall \alpha,x,\lambda\in\mathbb{R}\Rightarrow \partial_{x}^{\alpha }\; E_{\alpha,1 } (\lambda \, x^{\alpha } )=\lambda \, E_{\alpha,1 } (\lambda x^{\alpha } ) +\dfrac{x^{-\alpha}}{\Gamma(1-\alpha)}$. \\

\begin{proof}[\textbf{Proof.}]
	$\forall x,\alpha,\lambda\in\mathbb{R}$
\begin{equation} \begin{split}\label{soleq16} 
\partial_{x}^{\alpha }\; E_{\alpha,1 } (\lambda \, x^{\alpha } )&=
\partial_{x}^{\alpha }\; 1+\partial_{x}^{\alpha }\; \sum_{r=1}^{\infty}\dfrac{(\lambda x^{\alpha})^r}{\Gamma(1+\alpha r)}
=\dfrac{x^{-\alpha}}{\Gamma(1-\alpha)}+\sum_{r=1}^{\infty}\partial_{x}^{\alpha }\; \dfrac{(\lambda x^{\alpha})^r}{\Gamma(1+\alpha r)}=\\
& =
\lambda \, E_{\alpha,1 } (\lambda x^{\alpha } ) +\dfrac{x^{-\alpha}}{\Gamma(1-\alpha)}. 
\end{split}\end{equation}
The interchange between the summation and derivative is ensured by the uniform convergence of the series for $\alpha>0$ .\\
The last term is due to the fact that the fractional derivative in the sense of the Riemann-Liouville \ref{ERL} acts on a costant and does not vanish.
\end{proof}

%
%
%

	\chapter{}\label{AppB}
\markboth{\textsc{Appendix B}}{}

The Appendix reported below cover some topics touched in the main bodies of the other Chapters without the necessary details and extend the \textit{Hermite and Laguerre Calculus.}

\section{Umbra and Higher Order Hermite Polynomials}\label{higherHermite}

The higher order Hermite polynomials, also called \textit{\textbf{Lacunary HP}} \cite{DattLoren,Babusci}, defined through the operational identity 

\begin{equation}\label{key}
H_n^{(m)}(x,y)=e^{y\partial_{x}^m}x^n, \quad \forall m\in\mathbb{N},
\end{equation}
 specified by the series

\begin{equation}\label{key}
H_n^{(m)}(x,y)=n!\sum_{r=0}^{\lfloor\frac{n}{m}\rfloor}\dfrac{x^{n-mr}y^r}{(n-mr)!r!},
\end{equation}
with \textbf{generating function}

\begin{equation}\label{key}
\sum _{n=0}^{\infty }\frac{t^{n} }{n!}  H_{n}^{(m)} (x,y)=e^{x\, t+y\, t^{m} } ,
\end{equation}
 can be reduced to the $n^{th}$ power of a bynomial by introducing the umbral operator ${}_{y}\hat{h}_m^r$ such that, $\forall r\in\mathbb{R}$, 

\begin{equation}\label{key}
\begin{split}
& {}_{y}\hat{h}_m^r\; {}_m\theta_{0}:={}_m \theta_r=\dfrac{y^{\frac{r}{m}}r!}{\Gamma\left(\frac{r}{m}+1 \right) }A_{m,r},\\
& A_{m,r}=\left\lbrace \begin{array}{l}
1  \\0        \end{array}         \;          r=ms \begin{array}{l} s\in\mathbb{Z}, \\
  s\notin\mathbb{Z},
\end{array}\right. 
\end{split}
\end{equation}
which allows to define them as

\begin{equation}\label{key}
H_n^{(m)}(x,y)=\left(x+ {}_{y}\hat{h}_m\right)^n\;{}_m\theta_{0}. 
\end{equation}
For $m=2$ we recover eq. \ref{eq1HermLagbis}.\\

It is clearly evident that not too much effort is necessary to study the relevant properties and associated functions, which can be derived using the same procedure adopted for the second order case $(m=2)$ as, e.g., \textit{\textbf{higher negative order}} 

\begin{equation} \label{eq44HermLag} 
H_{-\nu}^{(m)}(x,y) = (x+{}_{-\mid y \mid} \hat{h}_m )^{-\nu} {}_m\theta_0  =\frac{1}{\Gamma (\nu )}\int_0^\infty s^{\nu -1}e^{-sx}e^{\;{}_{-\mid y \mid}\hat{h}_m s}{}_m\theta_0\; ds
\end{equation}     
or

\begin{equation} \label{eq45HermLag}       
H_{-\nu}^{(m)}(x,y) =\displaystyle \frac{1}{\Gamma (\nu )}\int_0^\infty s^{\nu -1}e^{-sx}e^{-ys^m}ds, \quad
Re(y)> 0  .              
\end{equation} 

It is therefore evident that

\begin{equation} \label{eq46HermLag} 
\begin{split}   
& I^{(m,1)}(x,y)=\displaystyle \int_0^\infty e^{-sx}e^{-s^m y}ds =H_{-1}^{(m-1)}(x,y),\\
& I^{(m,2)}(x,y)=\displaystyle \int_0^\infty e^{-s^2x}e^{-s^m y}ds =\sqrt{\pi}\;H_{-1/2}^{(m-2)}(x,y),\\
& ...\\
I& ^{(m,n)}(x,y)=\displaystyle \int_0^\infty e^{-s^nx}e^{-s^m y}ds =\frac{\Gamma \Bigr( \frac{1}{n}\Bigl )}{n}H_{-1/n}^{(m-n)}(x,y),\\
&m>n.
\end{split}                        
\end{equation}             .

\textit{\textbf{Super-Gaussian}} ($SG$) are used in optics to describe the so called flattened beams. Limiting ourselves to the one dimensional case they are represented by an exponential function of the type

\begin{equation} \label{eq47HermLag}
S_G(x,m)=e^{-x^m},                   
\end{equation}  
where, for simplicity, $m$ is assumed to be an even integer. The relevant propagation has been treated by the use of effective methods superposition employing gaussian beams \cite{Gori}, which amount to the approximation of a $SG$ beam with a superposition of Gauss beam, whose transformation through a lens like device are well known.  The use of an alternative approach (even though less efficient than the flattened beam method) is  suggested by the means of a Fresnel transform \cite{Born}, which writes 

\begin{equation} \label{eq48HermLag}    
S_G(x,m;A,B,D)= \frac{1}{\sqrt{2\pi iB}}\int_{-\infty}^{+\infty} e^{\frac{i}{2B}(Ay^2+2xy+Dx^2)}S_G(y,m)\;dy,
\end{equation}
where $A,B,D$ are constants accounting for the optical elements constituting the transport line.\\

 It is evident, according to the previous formalism that the above integral can be cast in the form

\begin{equation} \label{eq49HermLag} 
S_G(x,m;A,B,D)= \frac{e^{\frac{iD}{2B}x^2}}{\sqrt{2\pi iB}}
\left[  H_{-1}^{(4,2)}\left( \frac{ix}{B},\frac{iA}{2B},1\right) +H_{-1}^{(4,2)}\left( - \frac{ix}{B},\frac{iA}{2B},1\right) \right] ,
\end{equation}               
where the Hermite function on the right corresponds to the polynomial

\begin{equation} \label{eq50HermLag}         
H_n^{(4,2)}(x,y,z)=n! \sum_{r=0}^{\lfloor\frac{n}{4}\rfloor}    \frac{H_{n-4r}(x,y)\;z^r}{(n-4r)!r!}   .      
\end{equation}        
They are also defined by the operational identity

\begin{equation} \label{eq51HermLag} 
H_n^{(4,2)}(x,y,z)=e^{z\;\partial_x^{\;4}}  H_n (x,y)                      
\end{equation}                                      
and extended to the associated functions
defined as

\begin{equation} \label{eq52HermLag}   
H_{-\nu}^{(4,2)}(x,y,z)=\displaystyle  \frac{1}{\Gamma (\nu )}\int_0^\infty s^{\nu -1}e^{xt+yt^2-zt^4}ds, \quad
Re(y)> 0     .              
\end{equation}            

The possibility we have envisaged of using Hermite functions to study the propagation of $SG$ beams needs various refinements to become an effective tool, notwithstanding it provides a further proof on the possibility offered by these techniques.\\

Before closing this part we provide a further example concerning the Gauss-Weierstrass transform \ref{Gii} 

\begin{equation}\label{intfxiHermLag}
e^{\;y\;\partial_x^{2}}f(x)=\dfrac{1}{2\sqrt{\pi y}}\int_{-\infty }^\infty e^{-\frac{(x-\xi)^2}{4y}}f(\xi)\;d\xi=
\dfrac{1}{2\sqrt{\pi y}}\int_{-\infty }^\infty e^{-\frac{\xi^2}{4y}}e^{-\frac{x^2}{4y}+\frac{x\xi}{2y}}f(\xi)\;d\xi.
\end{equation}
The use of the generating function of Hermite polynomials \ref{genfunctH} allows to cast eq. \ref{intfxiHermLag} in the form

\begin{equation}\begin{split}\label{key}
& e^{\;y\;\partial_x^{2}}f(x)=\dfrac{1}{2\sqrt{\pi y}} \sum_{n=0}\dfrac{x^n}{n!}\left({}_H \hat{f}(y) \right),\\
& {}_H \hat{f}(y):= \dfrac{1}{2\sqrt{y}}\int_{-\infty }^\infty e^{-\frac{\xi^2}{4y}} \left( \dfrac{\xi}{2y}+{}_{-\frac{1}{4\mid y \mid}}\hat{h}\right)^n f(\xi)\;d\xi\;{}_{\bar{y}}\theta_0,\\
& y>0,
\end{split} 
\end{equation}
where $ {}_H \hat{f}(y)$ is the \textit{\textbf{Hermite transform}} of the function $f(x)$.\\

The same point of view can be followed to introduce the \textit{\textbf{Laguerre transform}} which can be derived from the  identity 

\begin{equation}\label{LagTransff}
e^{-y\;\partial_x x\partial_x}f(x)=\int_0^\infty e^{-\sigma}e^{-yx}C_0(x\sigma)f(\sigma)\;d(\sigma).
\end{equation}

The use of the generating function leads therefore to the identification of the Laguerre transform ${}_L \hat{f}(y)$, as shown below 

\begin{equation}\begin{split}\label{key}
& \int_0^\infty e^{-\sigma}e^{-yx}C_0(x\sigma)f(\sigma)\;d(\sigma)=\sum_{n=0}^\infty \dfrac{x^n}{n!}\;{}_L \hat{f}_n(y),\\
& {}_L \hat{f}_n(y):= \int_0^\infty e^{-\sigma}\left( y-\hat{c}\sigma\right)^n f(\sigma)\;d\sigma\;\varphi_0 .
\end{split} \end{equation}

Either Laguerre and Hermite transform are linked to the orthogonal properties of these family of polynomials and the relevant definitions can be extended to higher order polynomials or to multivariate Hermite. 
It is explore the relevant applications to the theory of imaging processing.\\ 

We provide now a furthermore example.\\

We can combine the $Hermitian$ umbra to get further generalizations, as for the three variable third order $HP$, which, according to the previous formalism, can be defined as

\begin{equation}\label{key}
H_n^{(3)}(x,y,z)=\left(x+ {}_{y}\hat{h}_2+{}_{z}\hat{h}_3 \right)^n\theta_{0,y} \theta_{0,z},
\end{equation}
thereby we find

\begin{equation}\label{key}
\begin{split}
& H_n^{(3)}(x,y,z)=\sum_{s=0}^n\binom{n}{s}{}_{(y,z)}\hat{h}_{(2,3)}^s \;x^{n-s}\theta_{0,y} \theta_{0,z},\\
& {}_{(y,z)}\hat{h}_{(2,3)}^s=\sum_{r=0}^s\binom{s}{r}{}_{z}\hat{h}_3^{s-r}\;{}_{y}\hat{h}_2^r.
\end{split}
\end{equation}

The extension of the method to bilateral generating functions is quite straightforward too. We consider indeed the generating function

\begin{equation}\label{key}
\begin{split}
 G(x,y;z,w\mid t)&=\sum_{n=0}^{\infty}\dfrac{t^n}{n!}H_n(x,y)H_n(z,w)=\sum_{n=0}^{\infty}\dfrac{t^n}{n!}\left(x+\;{}_y\hat{h} \right)^n H_n(z,w)\theta_0=\\
 & =e^{\left(x+\;{}_y\hat{h} \right)zt+\left[ \left(x+\;{}_y\hat{h} \right)t\right]^2w }\theta_0.
\end{split}
\end{equation}
The use of our technique yields

\begin{equation}\label{key}
G(x,y;z,w\mid t)=\dfrac{1}{\sqrt{1-4yt^2w}}e^{\frac{\left(x^2w+yz^2 \right)t^2+xtz }{1-4yt^2w}}.
\end{equation}

Exotic generating functions involving e.g. products of $LP$ and $HP$ can also be obtained.\\

The use of umbral methods looks much promising to develop a new point of view on the theory of special polynomials and of special functions as well.\\

\end{appendices}

\chapter*{\textit{Conclusions}}
\addcontentsline{toc}{chapter}{Conclusions}{}
\markboth{\textsc{Conclusions}}{}

This thesis has considered different aspects of the \textbf{\textit{Umbral Calculus}} applied to many various topics in Pure and Applied Mathematics.
 We have used the term umbral even though it is not appropriate, because it is not exactly the same formalism proposed by Rota and coworkers. We have stressed that the point of view developed here shares even more analogies with the Heaviside symbolic calculus and preserves the relevant spirit because it has been conceived for applications.\\

The most significant benefits introduced by such a technique are associated with the simplifications it provides in analytical computations and in the possibility it yields of getting a thread between seemingly uncorrelated topics. We have indeed treated such disparate points, ranging from fractional derivatives, Bessel functions, special polynomials and harmonic numbers using always the same formalism which has been exploited as a tool kit to disclose the relevant properties of different mathematical entities.\\

Although the thesis is rather long, we have been obliged to leave out many interesting subjects and applications. In particular we did not mention the theory of multi-variable Bessel functions which, within the present context, acquires a particularly tasty flavor allowing noticeable simplifications for the study of problems associated, e.g., with the treatment of synchrotron radiation and Free Electron Laser.\\

We hope that the form we have presented yields an idea of the “wildly wide” applicability of the method and of the possibility it may offer.
\chapter*{Bibliography}
\addcontentsline{toc}{chapter}{Bibliography}{}
\markboth{\textsc{Bibliography}}{}

\newpage
\begin{flushright}
	\textsl{"If people do not believe that
mathematics is simple, it is only
because they do not realize how
complicated life is".}\\
John von Neumann
\end{flushright}
\newpage
\begin{flushright}
	\textit{... e al Padre, creatore del visibile e dell'invisibile, che sempre mi sostiene. Senza di Lui, nulla è possibile.}
\end{flushright}


\begin{thebibliography}{}
		
	\bibitem[SL1]{MotzkinWolfram} Marcello Artioli, Giuseppe Dattoli, Silvia Licciardi; \textit{“Motzkin Numbers and their Geometrical Interpretation”}, Wolfram Demonstrations Project , \textbf{2017}.
	
	\bibitem[SL2]{Motzkin} Marcello Artioli, Giuseppe Dattoli, Silvia Licciardi, Simonetta Pagnutti; \textit{“Motzkin numbers: an operational point of view"}; arXiv:1703.07262 \textbf{2017}, submitted for publication to Online Electronic Integer Sequences.
	
	\bibitem[SL3]{FelHigh} Marcello Artioli, Giuseppe Dattoli, Silvia Licciardi, Simonetta Pagnutti; \textit{“Fractional Derivatives, Memory kernels and solution of Free Electron Laser Volterra type equation”},  Mathematics \textbf{2017}, 5(4), 73; doi: 10.3390/math5040073. Selected for Special Issue Cover.
	
	\bibitem[SL4]{CDR} M. Artioli et al; \textit{“A 250 GHz Radio Frequency CARM Source for Plasma Fusion”}, Conceptual Design Report, ENEA, pp. 154, 2016, ISBN: 978-88-8286-339-5.
	
	\bibitem[SL5]{Babusci} Danilo Babusci, Giuseppe Dattoli, Mario Del Franco, Silvia Licciardi; \textit{"Mathematical Methods for Physics"}, invited Monograph by World Scientific, Singapore, \textbf{2017}, in press.\\
	 D. Babusci, G. Dattoli, M. Del Franco; \textit{"Lectures on Mathematical Methods for Physics"}, RT/2010/58/ENEA.
	
	
	
	\bibitem[SL6]{FromCircular} Giuseppe Dattoli, Emanuele Di Palma, Silvia Licciardi, Elio Sabia; \textit{"From circular to Bessel functions: a transition through the umbral method"}; Fractal Fract \textbf{2017}, 1(1), 9; doi:10.3390/fractalfract1010009.
	
	\bibitem[SL7]{ProdB} Giuseppe Dattoli, Emanuele Di Palma, Elio Sabia, Silvia Licciardi; \textit{"Products of Bessel Functions and Associated Polynomials"}; Applied Mathematics and Computation, vol.  266,  Issue C, September \textbf{2015}, pages 507-514, Elsevier Science Inc. New York, NY, USA.
	
	\bibitem[SL8]{Fisher} Giuseppe Dattoli, Emanuele Di Palma, Elio Sabia and Silvia Licciardi; \textit{"Quasi Exact Solution of the Fisher Equation"}; Applied Mathematics, Vol. 4 No. 8A, \textbf{2013}, pp. 7-12. doi: 10.4236/am.2013.48A002. 
		
	\bibitem[SL9]{HermCalc} Giuseppe Dattoli, Bruna Germano, Silvia Licciardi, Maria Renata Martinelli; \textit{“Hermite Calculus”};  Modeling in Mathematics, Atlantis Transactions in Geometry, vol 2. pp. 43-52, J. Gielis, P. Ricci, I. Tavkhelidze (eds), Atlantis Press, Paris, Springer \textbf{2017}. 
	
	\bibitem[SL10]{Gegenbauer} Giuseppe Dattoli, Bruna Germano, Silvia Licciardi, Maria Renata Martinelli; \textit{"On an Umbral Treatment of Gegenbauer, Legendre and Jacobi Polynomials"}; International Mathematical Forum, vol. 12, \textbf{2017}, no. 11, pp. 531-551. 
	
	\bibitem[SL11]{Harmonic}  Giuseppe Dattoli, Bruna Germano, Silvia Licciardi, Maria Renata Martinelli;
	\textit{"Umbral Methods and Harmonic Numbers"};
	submitted for publication to Mediterranean Journal of Mathematics, \textbf{2017}.
	
	\bibitem[SL12]{ML} Giuseppe Dattoli, Katarzyna Gorska, Andrzej Horzela, Silvia Licciardi, Rosa Maria Pidatella; \textit{“Comments on the Properties of Mittag-Leffler Function”}, arxiv.org/abs/1707.01135 [math-ph] \textbf{2017}, submitted for publication to Journal Reviews in Mathematical Physics, 2017.
	
	\bibitem[SL13]{ExpMatrices} Giuseppe Dattoli, Silvia Licciardi, Federico Nguyen, Elio Sabia; \textsl{“Evolution equations involving Matrices raised to non-integer exponents”}; Modeling in Mathematics, Atlantis Transactions in Geometry, vol 2. pp. 31-41, J. Gielis, P. Ricci, I. Tavkhelidze (eds), Atlantis Press, Paris, Springer \textbf{2017}.
	
	\bibitem[SL14]{Airy} Giuseppe Dattoli, Silvia Licciardi, Rosa Maria Pidatella; \textit{“Theory of Generalized Trigonometric functions: From Laguerre to Airy forms“}; arXiv: 1702.08520, 2017, submitted for publication to Electronic Journal of Differential Equations (EJDE) \textbf{2017}.
	
	\bibitem[SL15]{GenTrFun} Giuseppe Dattoli, Silvia Licciardi, Elio Sabia; \textit{"Generalized Trigonometric Functions and Matrix Parameterization”}; Int. J. Appl. Comput. Math \textbf{2017}, pp. 1-14,  https://doi.org/10.1007/s40819-017-0427-0.
		
	\bibitem[SL16]{CarmFel} Emanuele Di Palma, Elio Sabia, Giuseppe Dattoli, Silvia Licciardi and Ivan Spassovsky; \textit{“Cyclotron auto resonance maser and free electron laser devices: a unified point of view”}, Journal of Plasma Physics, Volume 83, Issue 1 , February \textbf{2017}. 
	
	------------------------------------------------\\
	
	
	
	\bibitem{Abramovitz} M. Abramovitz and I.A. Stegun, (Eds); "\textit{Parabolic Cylinder Function}", Ch. 19 in Handbook of Mathematical Functions with Formulas, Graphs and Mathematical Tables, 9th printing, New York: Dover, pp. 685-700, 1972. 
	
	\bibitem{L.C.Andrews} L.C. Andrews; \textit{"Special Functions For Engeneers and Applied mathematicians"}; Mc Millan New York 1985.
	
	\bibitem{Ansari} A.H. Ansari, X. Liu, V.N. Mishra; \textit{"On Mittag-Leffler function and beyond"}, Nonlinear Science Letters A, 8(2)(2017), 187-199.
	
	\bibitem{Appell} P. App\'ell, J. Kamp\'e de F\'eri\'et; \textit{”Fonctions Hypergeometriques and Hyperspheriques. Polynomes d'Hermite”}; Gauthiers-Villars, Paris, 1926.
	
	\bibitem{DattoliArt} M. Artioli and G. Dattoli, Geometry of two-variable Legendre polynomials,  \textit{Wolfram Demonstrations Project-A Wolfram Web Resource}, demonstrations.wolfram.com/GeometryOfTwoVariableLegendrePolynomials.
	
	\bibitem{ArtioliDatt} M. Artioli,G. Dattoli; \textit{" 
	Geometric interpretation of Perrin and Padovan numbers"}, 
	http://demonstrations.wolfram.com/
	GeometricInterpretationOfPerrinAndPadovanNumbers/,
	\textit{Wolfram Demonstrations Project}, Sept. 14, (2016).
		
	\bibitem{Artioli} M. Artioli, G. Dattoli; \textit{"The Geometry of Hermite Polynomials"}, Wolfram Demonstrations Project, Mar. 2015, http://demonstrations.wolfram.com/TheGeometryOfHermitePolynomials/ .	
		
	\bibitem{On Ramanujan} D. Babusci, G. Dattoli; \textit{"On Ramanujan Master Theorem"}; arXiv:1103.3947 [math-ph].	
	
	\bibitem{BabDatPa}
	D. Babusci, G. Dattoli, E. Di Palma, E.N. Petropoulou; \textit{"The Humbert-Bessel functions, Stirling numbers and probability distributions in coincidence problems"}, Far East Journal of Mathematical Sciences, in printing.
	
	\bibitem{BabusciSab} D. Babusci, G. Dattoli, E. Di Palma , E. Sabia, \textit{"Complex-Type Numbers and Generalizations of the Euler Identity"}, Adv. Appl. Clifford Algebras, vol. 22,  (2), pp. 271-281,  June (2012).
	
	\bibitem{lacunary} D. Babusci, G. Dattoli, K. Gorska, K. Penson;
	\textit{"Lacunary Generating Functions for Laguerre Polynomials"}; Seminaire Lotharingien de Combinatoire (2017), Article B76b.
	
	\bibitem{DBabusci} D. Babusci, G. Dattoli, K. Gorska, K. Penson; \textit{"Symbolic methods for the evaluation of sum rules of Bessel functions"}, J. Math. Phys. \textbf{54}, 073501 (2013).
	
	\bibitem{D.Babusci} D. Babusci, G. Dattoli, K. Gorska, K.A. Penson;
	\textit{"The spherical Bessel and Struve functions and operational methods"}; Applied Mathematics and Computation, Volume 238, 1 July 2014, Pages 1-6.
	
	\bibitem{Quattromini} D. Babusci, G. Dattoli and M. Quattromini; \textit{"Relativistic equations with fractional and pseudodifferential operators"}, Phys. Rev. A 83, 062109, (2011);\\
	Kowalski and J. Rembielinski; \textit{"Salpeter equation and probability current in the relativistic quantum mechanics"}, Phys. Rev. A 84, 012108,  (2011).
	
	\bibitem{Sacchetti} D. Babusci, G. Dattoli, D. Sacchetti; \textit{"The Airy transform and the associated polynomials"}, centr.eur.j.phys. (2011) 9: 1381,  doi:10.2478/s11534-011-0057-9.
	
	\bibitem{Banderier} C. Banderier, C. Krattenthaler, A. Krinik, D. Kruchinin,
	V.r Kruchinin, D. Nguyen and M. Wallner, \textit{"Explicit formulas for enumeration of lattice paths: basketball and the Kernel method"},  Dev. Math., Springer, (2017).
	
	\bibitem{Barkai} E. Barkai; \textit{"Fractional Fokker-Planck equation, solution, and application"}, Phys. Rev. E 63, 046118 (2001).
	
	\bibitem{Baumgarten} C. Baumgarten; \textit{" Use of real Dirac matrices in two-dimensional coupled linear optics"}, Physical Review Special Topics - Accelerators and Beams, 14, 114002 (2011).
	
	\bibitem{Baym} G. Baym; \textit{"Lectures on Quantum Mechanics"}, W.A. Benjamin (1969).
	
	\bibitem{Bell} E.T. Bell, \textit{"The History of Blissard's Symbolic Method, with a Sketch of its Inventor's Life"}, The American Mathematical Monthly 45:7 (1938), pp. 414–421. 
	
	\bibitem{Bender} C.M. Bender, D.C. Brody, B.K. Meister; \textit{"On powers of Bessel functions"}, J. Math. Phys. 44, pp. 309-314 (2003).
	
	\bibitem{Ramanujan} B.C. Berndt;  \textit{Ramanujan's Notebooks: Part I}. New York: Springer-Verlag, p. 298, 1985. 
	
	\bibitem{Berry} M.V. Berry and C.J. Howls; \textit{"Integrals with coalescing saddles"}, in: NIST Handbook of Mathematical
	Functions, Cambridge University Press, Cambridge, 2010, pp. 775-793 (Chapter 36).
	
	\bibitem{Birkhoff} G. Birkhoff, S. Mac Lane; \textit{"A Survey of Modern Algebra"}, Third Edition, Chapter X, § 6 (Macmillan, 1965).
	
	\bibitem{Blasiak} P. Blasiak, G. Dattoli, A. Horzela, K. A. Penson  and K. Zhukovsky,, Motzkin numbers, central trinomial coefficients and hybrid polynomials, \textit{J. Integer Seq}.  11, Art. 08.1.1, (2008).
	
	\bibitem{Boas} R.P. Boas, R.C. Buck; \textit{"Polynomial expansions of analytic functions"}, Ergebnisse der Mathematik und ihrer Grenzgebiete. Neue Folge., 19, Berlin, New York: Springer-Verlag, MR 0094466, (1958).
	
	\bibitem{J.Bohacik} J. Bohacik, P. Augustin and P. Presnajder; \textit{"Non-perturbative anharmonic correction to Mehler's presentation of the harmonic oscillator propagator"}, Ukr. J. Phys. 59, 179 (2014).
	
	\bibitem{Born} M. Born, E. Wolf; \textit{"Principles of Optics"},  Pergamon Press 1970.
	
	\bibitem{Brychkov}	Y.A. Brychkov; \textit{"On multiple sums of special functions"}, Int. Trans. Spec. F., Vol. 21(12), December 2010, 877-884.
	
	\bibitem{Cartier} P. Cartier; \textit{""Mathemagics" (A tribute to L. Euler and R. Feynman)"}, Seminaire Lotharingien de Combinatoire 44 (2000), Article B44d.
	
	\bibitem{Cheick} Y.B. Cheick; 1998, \textit{"Decomposition of Laguerre polynomials with respect to the cyclic group"}, J. Comp. Appl. Math. 99, 55-66.
	
	\bibitem{T.S.Chihara} T.S. Chihara; \textit{"An introduction to orthogonal polynomials"}; Dover Pub. Inc. Mineola, New York (2011).
	
	\bibitem{Cholewinski} F.M. Cholewinski; \textit{"The Finite Calculus associated with Bessel Functions"}, J. Amer. Math. Soc. \textbf{75}, (1988).
	
	\bibitem{FMCholewinski} F.M. Cholewinski, J.A.Reneke, \textit{"The Generalized Airy Diffusion Equation"}, Electronic Journal of Differential Equations, , No. 87, pp. 1-64, ISSN: 1072-6691, Vol. 2003(2003).
	
	\bibitem{Ciocci} F. Ciocci, A. Torre and G. Dattoli; \textit{"Insertion Devices for Synchrotron Radiation and Free Electron Laser"}, World Scientific, River Edge, NJ, 2000.

	\bibitem{Cocolicchio} D. Cocolicchio, G. Dattoli and H.M. Srivastava; \textit{"Advanced special functions and applications"}, (Melfi, 1999), in:Proc. Melfi Sch. Adv. Top. Math. Phys., vol. 1, Rome, 2000, pp. 147-164, D. Cocolicchio, G. Dattoli and H.M. Srivastava, eds.
	
	\bibitem{Coffee} M.W. Coffee; \textit{"Expressions for Harmonic Number Generating Functions"}, Contemporary Mathematics, 517, Gems In Experimental Mathematics, T. Amdeberhan, E.T. Simos, V.H. Moll eds., AMS Special Session Experimental Mathematics, (2009). 
	
	\bibitem{Colson} W. B. Colson; "Classical Free Electron Laser theory", Laser Handbook Vol. VI ed. By W. B. Colson, C. Pellegrini and A. Renieri, North Holland (Amsterdam) 1990. 
	
	\bibitem{Conway} J.H.  Conway, R.K. Guy; \textit{"The Book of Numbers"}, Copernicus, New York, (1996).
	
	\bibitem{Cvijovic} D. Cvijovi\'c; \textit{"The Dattoli-Srivastava Conjectures Concerning Generating Functions Involving the Harmonic Numbers"}, Appl. Math. Comput., 215(9), pp. 4040-4043, (2010).
	
	\bibitem{Germano} G. Dattoli; \textit{"Generalized polynomials, operational identities and their applications"}; Journal of Computational and Applied mathematics, Elsevier
	Volume 118, Issues 12, 1 June 2000, Pages 111-123.
	
	\bibitem{Laguerre} G. Dattoli; \textit{"Hermite-Bessel and Laguerre-Bessel functions: a by-product of the monomiality principle"}; Advanced Special Functions and Applications (Melfi, 1999)-- (D. Cocolicchio, G. Dattoli and H.M. Srivastava, Editors), Aracne Editrice, Rome, pp. 147-164 (2000) \\
	G. Dattoli, \emph{Laguerre and generalized Hermite polynomials: the point of view of the operational method}, Int. Trans. Spec. F. Vol. \textbf{15}(2) (2004).
	
	\bibitem{DattArt} G. Dattoli, M. Artioli; \textit{"Vedic Math, Magic Math,...Damn Math"}, RT/2015/6/ENEA. 
	
	\bibitem{DelFranco} G. Dattoli, M. Del Franco; \textit{" The Euler Legacy to Modern Physics"}, Lecture Notes of "Seminario Interdisciplinare di Matematica" 9 (2010), pp 1-24.
	
	\bibitem{DiPalma} G. Dattoli, E. Di Palma, F. Nguyen, E. Sabia, \textit{"Generalized Trigonometric functions and
		elementary application"}, Int. J. Appl. Comput. Math, (2016). doi:10.1007/s40819-016-0168-5
	
	\bibitem{Borel} G. Dattoli, E. Di Palma, E. Sabia, K. Gorska, A. Horzela, K.A. Penson;
	\textit{"Operational versus umbral methods and the Borel transform"}; Int. J. Appl. Comput. Math. 2017 (2017) 1-22.
	
	\bibitem{Weyl} G. Dattoli, J.C. Gallardo, A. Torre; \textit{"An Algebraic View to the Operatorial Ordering and its Applications to Optics"}, Riv. Nuovo Cimento (3) 11, pp. 1-79, 1988.
	
	\bibitem{DGMR} G. Dattoli, B. Germano, M.R. Martinelli, P.E. Ricci; \textit{"A novel theory of Legendre Polynomials"},	Mathematical and Computer Modelling, vol.54, 2011, pp.80-87
	
	\bibitem{Ricci} G. Dattoli, B. Germano, M.R. Martinelli,P.E. Ricci; \textit{"Integral Transforms and Special Functions"}, vol. 19, No. 4, April 2008, 259–266.

	\bibitem{G.Dattoli} G. Dattoli, B. Germano, M.R. Martinelli, P.E. Ricci; \textit{“Lacunary Generating Functions of Hermite polynomials and Symbolic methods”}, Ilirias Journal of Mathematics, ISSN: 2334-6574, Volume 4 Issue 1(2015), Pages 16-23, URL: http://www.ilirias.com. 
	
	\bibitem{Martinelli} G. Dattoli, B. Germano, M.R. Martinelli and P.E. Ricci; \textit{“The negative derivative operator"}, Integral Transforms and Special Functions, vol.19, n. 3/4 (2008), pp. 259-266.
	
	\bibitem{Khan} G. Dattoli, S. Khan and P.E. Ricci; \textit{"On Crofton-Glaisher type relations and derivation of generating functions for Hermite polynomials including the multi-index case"}, Integral Transforms Spec. Funct.19 (1), pp. 1-9, (2008).
	
	\bibitem{DattLoren} G. Dattoli, S. Lorenzutta and C. Cesarano; \textit{"From Hermite to Humbert Polynomials"}, Rend. Istit. Mat. Univ. Trieste, Vol. XXXV, pp. 37-48, 2003.
	
	\bibitem{Mancho} G. Dattoli, S. Lorenzutta, A.M. Mancho, A. Torre; \textit{"Generalized Polynomials and Associated Operational Identities"}, Journal of Computational and Applied Mathematics, (1999), 108, 209-218. 
	
    \bibitem{Lorenzutta} G. Dattoli , S. Lorenzutta, P. E. Ricci and  C. Cesarano; \textit{"On a family of hybrid polynomials"}, J. Integral Transforms and Special Functions,  vol. 15, (2004), pp.  485-490.
	 
	\bibitem{DMR} G. Dattoli, M. Migliorati, P.E. Ricci; \textit{"The Eisentein group and the pseudo hyperbolic function"}, arXiv:1010.1676 [math-ph] (2010).
	 
	\bibitem{Vazquez} G. Dattoli, P.L. Ottaviani, A. Torre, L. Vazquez; \textit{"Evolution operator equations-integration with algebraic and finite-difference methods-applications to physical problems in classical and quantum mechanics and quantum field theory"} , Rivista del Nuovo Cimento 20, 1, 1997.
	 
	\bibitem{Renieri} G. Dattoli, A. Renieri, A. Torre; \textit{"Lectures on Free Electron Laser Theory and Related Topics"}, World Scientific, Singapore (1993). 
	 
	\bibitem{PERicci}  G. Dattoli, P.E. Ricci, L. Marinelli; \textit{"Generalized Truncated Exponential Polynomials and Applications"}, Rendiconti dell’Istituto di Matematica dell’Universit\'a di Trieste, An International Journal of Mathematics, 34, pp. 9-18, (2002). 
	 
	\bibitem{DSDF} G. Dattoli, E. Sabia, M. Del Franco;
	\textit{"The Pseudo-Hyperbolic Functions and the Matrix Representation of Eisenstein Complex Numbers"}, arXiv:1003.2698v1 [math-ph]. 
	 
	\bibitem{Srivastava} G. Dattoli, H.M. Srivastava; \textit{"A Note on Harmonic Numbers, Umbral Calculus and Generating Functions"}, Appl. Math. Lett., 21(6), pp. 686-693, (2008). 
	 
	\bibitem{DatTo} G. Dattoli, A. Torre; \textit{"Operatorial methods and two variable Laguerre polynomials"}, Acc. Scienze Torino, Atti Sc. Fis.,1 (1998), p. 132. 
	 
	\bibitem{Carpanese} G. Dattoli, A. Torre, M. Carpanese; \textit{"Operational Rules and Arbitrary Order Hermite Generating Functions"}, Journal of Mathematical Analysis and Applications, Vol. 227, Issue 1, 1 November (1998), Pages 98-111. 
	 
	\bibitem{Dattoli} G. Dattoli, A. Torre, C. Centioli and M. Richetta; \textit{"Free Electron Laser Operation in the Intermediate gain Region"}, IEEE J-QE , 25, 2327 (1989).  
	 
	\bibitem{Maino} G. Dattoli, A. Torre, S. Lorenzutta, G.Maino; \textit{"Generalized Bessel functions and Kapteyn series"}, 
	Computers \& Mathematics with Applications
	Vol. 35, Issue 8, April 1998, Pages 117-125.
	 
	\bibitem{Carpanese} G. Dattoli et al., \textit{"Undulator and Free Electron Laser radiation for Fundamental Physics research"}, talk at the 7th International Conference "Charged \& Neutral Particles Channeling Phenomena Channeling 2016", https://agenda.infn.it/conferenceOtherViews.py/view=standard\& confId=10663. 
	 
	\bibitem{Doetsch} G. Doetsch; \textit{"Handbuch der Laplace Transformation"}, Birkhnauser, Basel, (1950-1956). 
	 
	\bibitem{D.E.Edmunds} D.E. Edmunds, P. Gurka, J. Lang;  \textit{"Properties og generalized trigonometric functions"}, Journal of Approximation Theory, 164, 47-56, (2012). 
	 
	\bibitem{Ehrenpreis} L. Ehrenpreis; \textit{"The Borel Transform"};  Algebraic Analysis of Differential Equations, T. Aoki, H. Majima, Y. Takei, and N. Tose (Eds.), Springer, Berlin, 2008.
	 
	\bibitem{enea} For various material about FEL formulae, see also www.fel.enea.it.  
	 
	\bibitem{Gabor} H.G. Feichtinger,  T. Strohmer (Eds); \textit{"Gabor Analysis and Algorithms  Theory and Applications"}, Birkhäuser, 1998. 
	 
	 \bibitem{Ferrari} E. Ferrari, Bollettino UMI 18B, 933 (1981).\\
	 For early suggestions see also F.D. Bugogne, Math. of Comp. 18, 314 (1964).
	 
	 \bibitem{Fjelstad} P. Fjelstad, S.G. Gal, \textit{"Two-dimensional geometries, topologies, trigonometries and physics generated by complex-type numbers"}, Adv. Appl. Clifford Algebras, 11:81, June (2001).
	 
	 \bibitem{Gori} F. Gori; \textit{"Flattened gaussian beams"}, Optics Communications, Volume 107, Issue 5-6, p. 335-341, 05/1994. 
	 
	\bibitem{K.Gorska} K. Gorska, D. Babusci, G. Dattoli, G.H.E. Duchamp, K.A. Penson; \textit{"The Ramanujan master theorem and its implications for special functions"}; Appl. Math. Comp. 218 (2012) 11466-11471.
	
	\bibitem{FFP} K. Gorska, K.A. Penson, D. Babusci, G. Dattoli, G.H.E. Duchamp; \textit{"Operator Solutions for Fractional Fokker-Planck Equations"}, Pysical Review E85, 031138, 2012.
	
	\bibitem{Gosper} R. W. Gosper; \textit{"Harmonic summation and exponential gfs."}, math-fun@cs.arizona.edu posting, Aug. 2, 1996.
	
	\bibitem{G.H.Hardy} G.H. Hardy, Ramanujan; \textit{"Twelve lectures on subjects suggested by his life and work"}; Cambridge University Press, Cambridge, 1940.
	
	\bibitem{Hatz} R. Hatz, M. Korpinen, V. Hanninen, L. Halonen; \textit{" 
	Generalized intermolecular interaction tensor applied to long-range interactions in hydrogen and coinage metal (Cu, Ag, and Au) clusters"},
	J. Phys. Chem. A, 119, (2015),  11729--11736.
	
	\bibitem{Heaviside}. O. Heaviside; \textit{"Electromagnetic induction and its propagation"}, The Electrician, 1887.
	
	\bibitem{R.Hermann} R. Hermann, \textit{"Fractional Calculus: an introduction for Physicists"} 2 edition, World Scientific, Singapore (2014). 
	
	\bibitem{Horn} R.A. Horn, C.R. Johnson; (1991), \textit{ "Topics in matrix analysis"}, Cambridge University Press. See Section 6.1.
	
	\bibitem{PHumbert} P. Humbert; \textit{"Some extensions of Pincherle's Polynomials"}, Proceedings of the Edinburgh mathematics society, 39: 21–24, (1921).
	
	\bibitem{Jansen} J. Jansen; \textit{"Rotations in three four and five dimensions"}, arXiv:1103.5263v1[math.MG].
	
	\bibitem{S.Khan} S. Khan, A.A. Al-Gonah; \textit{"Operational methods and Laguerre - Gould-Hopper polynomials"}; Applied Mathematics and Computation, 2012, Elsevier, Volume 218, Issue 19, 1 June 2012, Pages 9930-9942.
	
	\bibitem{Knuth} D.E. Knuth,  \textit{"The art of computer programming"}, \textbf{3} (1973), Sorting and Searching, Reading, Mass.: Addison-Wesley,  65--67, MR 0445948.
	
	\bibitem{Kondo} J. Kondo; \textit{"Integral Equations"}, Oxford Appl. Math. Comput. Sci. Ser. The Clarendon Press, Oxford University Press, NewYork; Kodansha, Ltd., Tokyo, 1991.

    \bibitem{Lagarias} J.C. Lagarias; \textit{"Euler’s Constant: Euler’s Work and Modern Developments"}, Bulletin (New Series) of the American Mathematical Society Volume 50, Number 4, October 2013, Pages 527–628. Article electronically published on July 19, 2013 

	\bibitem{Lamb} H. Lamb; \textit{”On the diffraction of a solitary wave”}, Proc. London Math. Soc.8, 422 (1910).
	
	\bibitem{Laskin} N. Laskin; \textit{"Fractional Poisson process"}, Commun. Nonlinear Sci. Numer. Simul. 8, 201 2003.\\
	N. Laskin;  \textit{"Some applications of the fractional Poisson probability distribution"},  J. Math. Phys. 50, 113513 (2009).
					
	\bibitem{Louiselle} W.H. Louisell; \textit{"Quantum Statistical Properties of Radiation"}, John Wiley \& Sons Canada, Limited, 1973.
	
	\bibitem{Lopez} J.L. L\'opez and P.J. Pagola; \textit{"Analytic formulas for the evaluation of the Pearcey integral"}, arXiv:1601.03615 [mat.NA].				
		
	\bibitem{Mansour} T. Mansour; \textit{"Commutation Relations, Normal Ordering, and Stirling Numbers"}, September 21, 2015 by Chapman and Hall/CRC
	ISBN 9781466579880.	
		
	\bibitem{Majorana} Ettore Majorana, Nuovo Cimento14, 171 (1937).	
					
    \bibitem{Mittag-Leffler} M.G. Mittag-Leffler; \textit{”Une g\'en\'eralisation de l’int\'egrale de Laplace-Abel”}, Comptes Rendus Hebdomadaires des S\'eances de l’Acad\'emie des Sciences, 136, pp. 537539, 1903.			
		
	\bibitem{Martin} P.A. Martin; \textit{”Harry Bateman: from Manchester to Manuscript Project”}, www.ima.org.uk/.../mtapril10/\_harry\_bateman\_from\_manchester
	\_to\_manuscript\_project.pdf	
	
	\bibitem{Mezo} I. Mezo; \textit{"Exponential Generating Function of Hyper-Harmonic Numbers Indexed by Arithmetic Progressions"}, Cent. Eur. J. Math., 11(3), pp. 931-939, (2013).
	
	
	\bibitem{Vignat} V.H. Moll, C. H. Vignat; \textit{"On polynomials connected to powers of Bessel functions"}, ArXiv:1306.1224v1 [math-ph], 5 June 2013.	
	
	\bibitem{Murray} R.M. Murray, Z. Li, S.S. Sastry; \textit{"A Mathematical Introduction to Robotic Manipulation"}, Boca Raton, FL: CRC Press, 1994. 
	
	\bibitem{Naber} M. Naber; \textit{"Time fractional Schrodinger equation"},  J. Math. Phys. Vol 45, No. 8 pp. 3339-3352, Aug. 2004.
	
	\bibitem{Nahin} P.J. Nahin; \textit{"Oliver Heaviside: The Life, Work, and Times of an Electrical Genius of the Victorian Age"}, Baltimore: Johns Hopkins University Press (1988). 
		
	\bibitem{Nielsen} N. Nielsen; \textit{"Recherches sur les polynomes d'Hermite"}, Mathematisk-Fysiske Meddelelser, 1, (1918), 79, Det. Kgl, Danske Videnskabernes Selskab.	
				
	\bibitem{Nieto} M.M. Nieto, D. Rodney Truax; 1995, \textit{"Arbitrary-order Hermite generating functions for obtaining arbitrary-order coherent and squeezed states"}, Phys. Lett. A208, 8-16. 			
							
	\bibitem{arabi}
	K.S. Nisar, S. R. Mondal, P. Agarwal, M. Al-Dhaifallah;
	\textit{"The Umbral operator and the integration
		involving generalized Bessel-type function"}; Open Math. 2015; 13: 426-435.
	
	\bibitem{Oldham} K.B. Oldham, J. Spanier; \textit{"The Fractional Calculus: Theory and Applications of Differentiation and Integration to Arbitrary Order"}, Mathematics in Science and Engineering, vol 111, 1974.
	
	\bibitem{Olive} K.A. Olive et al.; \textit{"Review of particle Physics, (Particle Data Group)"}, Chin. Phys. C, 38, 090001, (2014) and 2015 update.
	
	\bibitem{Pagnini} G. Pagnini, R.K. Saxena; \textit{”A note on the Voigt profile function”}, arXiv:0805.2274 [math-ph].
	
	\bibitem{Pashaev} O.K. Pashaev, Z.N. Gurkan; \textit{"Abelian Chern–Simons vortices and holomorphic
		Burgers hierarchy"}, Theoretical and Mathematical Physics, 152(1): 1017–1029 (2007).
	
	\bibitem{Pathan} M.A. Pathan; \textit{"The Multivariable Voigt Functions and their representations"}, \textit{Scientia}, series A: Mathematical Sciences, 12, pp. 9-15, (2006). 
		
	\bibitem{M.A.Pathan} M.A. Pathan, M. Kamarujjama, M. Khursheed Alam; \textit{”On multiindices and multivariables presentation of the Voigt functions”}, Journal of Computational and Applied Mathematics, Volume 160, Issues 12, 1 November 2003, Pages 251-257.	
			
	\bibitem{K.A.Penson} K.A. Penson, P. Blasiak, A. Horzela, A.I. Solomon, G.H.E. Duchamp; \textit{"Laguerre-type derivatives: Dobinski relations and combinatorial identities"};
	Journal of Mathematical Physics, 50(8) 083512.
	
	\bibitem{Penson} K. Penson, K. Gorska; \textit{"Exact and explicit probability densities for one-sided Levy stable distributions"}, Phys. Rev. Lett., 105-210604, (2010).
	
	\bibitem{Qin} H. Qin, R. Davidson;  \textit{"Courant-Snyder theory for coupled transverse dynamics of charged particles in electromagnetic focusing lattices"}, Physical Review Special Topics-Accelerators and Beams, vol. 12, n. 6, (2009).
	DOI: 10.1103/PhysRevSTAB.12.064001 .
	
	\bibitem{Rafelt} G. Raffelt, L. Stodolsky; \textit{"Mixing of the photon with low-mass particles"}, Phys. Rev. D 37, 1237, (1988).
	
	\bibitem{Howard R.Reiss} H.R. Reiss; \textit{"Effect of an intense electromagnetic field on a weakly bound system"}; Phys. Rev. A 22, 1786, Published 1 November 1980.
	
	\bibitem{R.Reiss} H.R. Reiss; \textit{"Relativistic strong-field ionization"}; J. Opt. Soc. Am. B 7, 574-586 (1990).
	
	\bibitem{Ricci} P.E.Ricci; 1978, \textit{"Le funzioni Pseudo Iperboliche e Pseudo Trigonometriche"}, Pubblicazioni dell'Istituto di Matematica Applicata, N 192.
	
	\bibitem{Riordan}  J. Riordan,  \textit{Introduction to combinatorial analysis}, Dover (2002),  85--86.
	
	
	\bibitem{Rochowicz}  J.A. Rochowicz Jr; \textit{"Harmonic Numbers: Insights, Approximations and Applications"}, Spreadsheets in Education (eJSiE): Vol. 8: Iss. 2, Article 4, (2015).
	
	\bibitem{Roman} S. Roman, \textit{"The theory of the umbral calculus"}, J. Math. Anal. Appl., 87  (1982), 58.
	
	\bibitem{S.Roman} S. Roman; \textit{"The Umbral Calculus"}; Dover Publications, New York, 2005. 
	
	\bibitem{SMRoman} S.M. Roman and Gian-Carlo Rota; \textit{"The umbral calculus"}, Advances in Math. 27 (1978), no. 2, 95–188.
	
	\bibitem{Schmidt}  M.D. Schmidt; \textit{"Zeta Series Generating Function Transformations Related to Polylogarithm Functions and the k-Order Harmonic Numbers"}, Online Journal of Analytic Combinatorics, Issue 12, Article 2 (2017). 
	
	\bibitem{Sneddon} I.N. Sneddon; \textit{"Fourier Transforms"}, Courier Corporation, 1995.
	
	\bibitem{Sondow} J. Sondow, E.W. Weisstein, \textit{"Harmonic Number"}, Math World A Wolfram Web Resource, http://mathworld.wolfram.com/
	HarmonicNumber.html.
	
	\bibitem{H.M.Srivastava} H.M. Srivastava, L. Manocha; \textit{"A treatise On Generating Functions"}, Bull. Amer. Math. Soc. (N.S.), Volume 19, Number 1 (1988), 346-348.
	
	\bibitem{Srivastava} H.M. Srivastava, E.A. Miller; \textit{”A unified presentation of the Voigt functions”}, Astrophysics and Space Science (ISSN 0004-640X), vol. 135, no. 1, July 1987, p. 111-118.
	
	\bibitem{Steffen} K. Steffen; \textit{"Fundamentals of Accelerator Optics"}, CERN Accelerator School, Synchrotron 
	Radiation and Free Electron Lasers, April 1989, S. Turner ed. 90-93 (1990).
	
	\bibitem{Strehl} V. Strehl; \textit{"Lacunary Laguerre series from a combinatorial perspective"}, Séminaire Lotharingien de Combinatoire 76 (2017), Article B76c.
	
	\bibitem{Sun} J. Sun, J. Wang, C. Wang; 1991, \textit{"Orthonormalized eigenstates of cubic and higher powers of the annihilation operator"}, Phys. Rev. 44A, 3369.
	
	\bibitem{Kapteyn} R.C. Tautz, I. Lerche,  D. Dominici;
	\textit{"Methods for Summing General Kapteyn Series"}; Journal of Physics A-Mathematical and Theoretical, vol. 44, no. 38, 2011.
	
	\bibitem{Thorne} A. Thorne, U. Litz\'en, S. Johansson; \textit{”Spectrophysics: Principles and Applications”}, Springer, 1999, ISBN 13: 9783540651178.
	
	\bibitem{Tricomi} F.G. Tricomi; \textit{”Funzioni Speciali”}, pp. 408, Gheroni, (1959).
	
	\bibitem{Uchaikin} V. Uchaikin, R. Sibatov; \textit{"Fractional Kinetics in Solids: Anomalous Charge Transport in Semiconductors, Dieletrics and Nanosystems"}, World Scientific, par. 1.3.2, 2013. 
	
	\bibitem{Zolotarev} V.V. Uchaikin, V.M. Zolotarev; \textit{"Chance and Stability: Stable Distributions and their Applications"}, Walter de Gruyter, 01 gen 1999.		
		
	\bibitem{Vajda} S. Vajda; \textit{"Fibonacci and Lucas Numbers and the Golden Section: Theory and Applications"}, Halsted, Press, New York, pp. 9-10, 18-20, (1989).	
		
	\bibitem{Eric} E.W. Weisstein; \textit{"Hermite Number"}, from MathWorld--A Wolfram Web Resource,  http://mathworld.wolfram.com/HermiteNumber.html .	
		
	\bibitem{EWW} E.W. Weisstein; \textit{"Legendre Polynomial"}, From MathWorld--A Wolfram Web Resource http://mathworld.wolfram.com/about/author.html.	
		
	\bibitem{MLWolfram} E.W. Weisstein; \textit{"Mittag-Leffler Function"} From MathWorld-A Wolfram Web Resource, http://mathworld.wolfram.com/Mittag-LefflerFunction.html. 	
	
	\bibitem{WeissteinMotz} E.W. Weisstein, \textit{"Motzkin number"}, MathWorld-A Wolfram Web Resource, http://mathworld.wolfram.com/MotzkinNumber.html.
	
	\bibitem{Padovan} E.W. Weisstein; \textit{"Padovan sequence"}, Wolfram MathWorld-A Wolfram Web Resource, mathworld.wolfram.com/PadovanSequence.html.
	
	\bibitem{Weisstein} E.W. Weisstein; \textit{"Parabolic Cylinder Function"}, from MathWorld--A Wolfram Web Resource http://mathworld.wolfram.com/ParabolicCylinderFunction.html
		
	\bibitem{Perrin} E.W. Weisstein; \textit{"Perrin sequence"}, Wolfram MathWorld-A Wolfram Web Resource, mathworld.wolfram.com/PerrinSequence.html.
	
	\bibitem{Widder} D.V. Widder; \textit{"The heat equation"}, Academic Press. 1976.
		
	\bibitem{E.M.Wright} E.M. Wright; \textit{"The Asymptotic Expansion of the Generalized Bessel Functions"}; Proc. London math. Soc. 2, 38, pp. 257-270, (1935).
	
	
	\bibitem{Yamaleev} R.M. Yamaleev; \textit{"Complex algebras on n-order polynomials and generalizations of trigonometry, oscillator model and Hamilton dynamics"},  Adv. Appl. Clifford Algebras, vol. 15, (1), pp. 123-150, March (2005).
	
	\bibitem{R.M.Yamaleev} R.M. Yamaleev; \textit{"Hyperbolic Cosines and Sines Theorems for the Triangle Formed by Arcs of Intersecting Semicircles on Euclidean Plane"}, Journal of Mathematics, Article ID 920528, 10 pages, vol. (2013).
	
	\bibitem{Zhukovsky} K. Zhukovsky, G. Dattoli; \textit{"Umbral Methods, Combinatorial Identities and Harmonic Numbers"}, Applied Mathematics, 1, 46 (2011).
		
\end{thebibliography}
\end{document}